\pdfoutput=1

\documentclass[reqno]{amsart}
\usepackage{etex}
\usepackage[a4paper,hmarginratio=1:1]{geometry}
\usepackage[dvipsnames]{xcolor} 

\usepackage{amssymb,amsfonts,amsmath}
\usepackage{comment} 
\usepackage[all,arc]{xy}
\usepackage{enumerate}
\usepackage{enumitem}
\usepackage{mathrsfs,mathtools}
\usepackage{todonotes,booktabs}
\usepackage{stmaryrd}
\usepackage{marvosym}
\usepackage{graphicx}
\usepackage{pdflscape} 	
\usepackage{caption}
\usepackage{bbm}
\usepackage{microtype}
\usepackage{wasysym}

\usepackage{mathdots}
\usepackage{tikz}
\usepackage{tikz-cd}
\usetikzlibrary{trees}
\usetikzlibrary[shapes]
\usetikzlibrary[arrows]
\usetikzlibrary{patterns}
\usetikzlibrary{fadings}
\usetikzlibrary{backgrounds}
\usetikzlibrary{decorations.pathreplacing}
\usetikzlibrary{decorations.pathmorphing}
\usetikzlibrary{positioning}
\usetikzlibrary{calc}

\usepackage{graphicx}
\SelectTips{cm}{10}

\usepackage[pagebackref]{hyperref}
\usepackage{bookmark}

\definecolor{dark-red}{rgb}{0.5,0.15,0.15}
\definecolor{dark-blue}{rgb}{0.15,0.15,0.6}
\definecolor{dark-green}{rgb}{0.15,0.6,0.15}
\hypersetup{
    colorlinks, linkcolor=dark-red,
    citecolor=dark-blue, urlcolor=dark-green
}

\renewcommand*{\backref}[1]{}
\renewcommand*{\backrefalt}[4]{%
  \ifcase #1 %
No citations.
  \or
(cit. on p. #2).%
  \else
(cit on pp. #2).%
  \fi%
}

\usepackage[nameinlink,capitalise,noabbrev]{cleveref}

\definecolor{Blue1}{HTML}{7ce8ff}
\definecolor{Blue2}{HTML}{55d0ff}
\definecolor{Blue3}{HTML}{00acdf}
\definecolor{Blue4}{HTML}{0080bf}

\usepackage{stackengine}
\usepackage{scalerel}

\newtheorem{thmx}{Theorem}

\newtheorem{thm}{Theorem}[section]
\newtheorem{theorem}[thm]{Theorem}
\newtheorem{corollary}[thm]{Corollary}
\newtheorem{proposition}[thm]{Proposition}
\newtheorem{lemma}[thm]{Lemma}

\newtheorem*{theorem*}{Theorem}

\newtheorem*{conjecture*}{Conjecture}

\newtheorem*{mainthmI}{Integral main results}
\newtheorem*{mainthmQ}{Rational main results}

\theoremstyle{definition}
\newtheorem{definition}[thm]{Definition}
\newtheorem{example}[thm]{Example}

\newtheorem{notation}[thm]{Notation}
\newtheorem{hypothesis}[thm]{Hypothesis}

\theoremstyle{remark}
\newtheorem{remark}[thm]{Remark}
 
\newtheorem{recollection}[thm]{Recollection} 

\makeatletter
\let\c@equation\c@thm
\makeatother
\numberwithin{equation}{section}

\makeatletter
\@namedef{subjclassname@2020}{%
  \textup{2020} Mathematics Subject Classification}
\makeatother

\newcommand{\noloc}{\;\mathord{:}\,}

\newcommand{\cpctRecollement}[3]{
\xymatrix@C=4em{{#1} \ar[r]|-{#2} & {#3}}
}

\newcommand{\recollement}[5]{
\xymatrix@C=4em{{#1} \ar@{<-}[r]|-{#2} & #3 \ar@{<-}[r]|-{#4} \ar@{<-}@<1.5ex>[l]^-{{#2}_!} \ar@{<-}@<-1.5ex>[l]_-{{#2}^*} & #5, \ar@{<-}@<1.5ex>[l]^-{{#4}!} \ar@{<-}@<-1.5ex>[l]_-{{#4}^*}
}}
\let\lim\relax

\DeclareMathOperator{\Cat}{Cat}
\DeclareMathOperator{\lim}{lim}

\DeclareMathOperator{\op}{op}
\DeclareMathOperator{\proj}{proj}
\DeclareMathOperator{\can}{can}

\newcommand{\F}{\mathbb{F}}
\newcommand{\N}{\mathbb{N}}
\newcommand{\Q}{\mathbb{Q}}
\newcommand{\R}{\mathbb{R}}
\newcommand{\Z}{\mathbb{Z}}

\DeclareMathOperator{\cA}{\mathcal{A}}
\newcommand{\cB}{\mathcal{B}}

\DeclareMathOperator{\cE}{\mathcal{E}}

\newcommand{\cU}{\mathcal{U}}

\newcommand{\fp}{\mathfrak{p}}

\DeclareMathOperator{\Top}{Top}
\DeclareMathOperator{\Spc}{Spc}
\DeclareMathOperator{\Spch}{Spc^\mathrm{h}}
\DeclareMathOperator{\Hom}{Hom}

\DeclareMathOperator{\End}{End}
\DeclareMathOperator{\colim}{colim}

\DeclareMathOperator{\Spec}{Spec}
\DeclareMathOperator{\Mod}{Mod}

\DeclareMathOperator{\Ch}{Ch}

\DeclareMathOperator{\fib}{fib}

\DeclareMathOperator{\Ho}{Ho}
\DeclareMathOperator{\Thick}{Thick}

\newcommand{\thickid}[1]{\Thick_{#1}^{\otimes}}
\DeclareMathOperator{\Ind}{Ind}

\DeclareMathOperator{\ind}{ind}
\DeclareMathOperator{\res}{res}
\DeclareMathOperator{\coind}{coind}
\DeclareMathOperator{\Fun}{Fun}
\DeclareMathOperator{\Ob}{Ob}

\DeclareMathOperator{\id}{id}

\DeclareMathOperator{\unit}{\mathbbm{1}}

\DeclareMathOperator{\Map}{Map}

\DeclareMathOperator{\supp}{supp}
\DeclareMathOperator{\Supp}{Supp}
\DeclareMathOperator{\supph}{supp^\mathrm{h}}
\DeclareMathOperator{\Supph}{Supp^\mathrm{h}}

\DeclareMathOperator{\Cosupp}{Cosupp}

\DeclareMathOperator{\CAlg}{CAlg}

\DeclareMathOperator{\inv}{inv}

\DeclareMathOperator{\Sub}{Sub}
\DeclareMathOperator{\cons}{con}
\DeclareMathOperator{\const}{const}

\DeclareMathOperator{\rank}{rank}

\DeclareMathOperator{\gen}{gen}
\DeclareMathOperator{\infl}{infl}
\newcommand{\fin}{\mathrm{fin}}
\newcommand{\qcopen}[1]{\overset{\circ}{\mathcal{K}}({#1})}
\newcommand{\Thc}[1]{\mathrm{Th}^c({#1})}

\newcommand{\Spectral}{\Top_{\mathrm{Spec}}}

\newcommand{\Prism}{\mathsf{Prism}}

\newcommand{\Ord}{\mathrm{Ord}}

\newcommand{\cbrank}{\rank^{\mathrm{CB}}}
\newcommand{\sfB}{\mathsf{B}}
\newcommand{\sfC}{\mathsf{C}}
\newcommand{\sfD}{\mathsf{D}}
\newcommand{\sfF}{\mathsf{F}}
\newcommand{\sfG}{\mathsf{G}}
\newcommand{\sfK}{\mathsf{K}}
\newcommand{\sfL}{\mathsf{L}}
\newcommand{\sfT}{\mathsf{T}}
\newcommand{\sfS}{\mathsf{S}}
\newcommand{\sfP}{\mathsf{P}}
\newcommand{\sfQ}{\mathsf{Q}}

\newcommand{\Sp}{\mathsf{Sp}}

\newcommand{\Sh}{\mathsf{Shv}}

\newcommand{\Spfausk}{\Sp^{\mathrm{orth}}}
\newcommand{\SpBarwick}{\Sp^{\mathrm{Mack}}}
\newcommand{\SpBH}{\Sp^{\mathrm{par}}}

\newcommand{\cSpfausk}{\Sp^{\mathrm{orth},\omega}}
\newcommand{\cont}{\mathrm{cont}}

\newcommand{\Ab}{\mathrm{Ab}}

\newcommand{\Weyl}{\mathrm{Weyl}}
\newcommand{\WeylSheaves}{\Shv^{\Weyl}}

\DeclareMathOperator{\Shv}{Shv}

\DeclareMathOperator{\incfunctor}{\mathrm{inc}}
\DeclareMathOperator{\weylfunctor}{\mathrm{Weyl}}

\newcommand{\lra}{\longrightarrow}

\colorlet{rainbowOne}{red}
\colorlet{rainbowTwo}{green}
\colorlet{rainbowThree}{blue}
\globalcolorstrue
\colorlet{rainbowFour}{Mulberry}

\title{Profinite equivariant spectra and their tensor-triangular geometry}

 \author{Scott Balchin}
 \address{Mathematical Sciences Research Centre, Queen's University Belfast, UK}
 \email{s.balchin@qub.ac.uk}

 \author{David Barnes}
 \address{Mathematical Sciences Research Centre, Queen's University Belfast, UK}
 \email{d.barnes@qub.ac.uk}

 \author{Tobias Barthel}
 \address{Max Planck Institute for Mathematics, Vivatsgasse 7, 53111 Bonn, Germany}
 \email{tbarthel@mpim-bonn.mpg.de}

\date{\today}

\begin{document}

\begin{abstract}
We study the tensor-triangular geometry of the category of equivariant $G$-spectra for $G$ a profinite group, $\mathsf{Sp}_G$. Our starting point is the construction of a ``continuous'' model for this category, which we show agrees with all other models in the literature. We describe the Balmer spectrum of finite $G$-spectra up to the ambiguity that is present in the finite group case; in particular, we obtain a thick subcategory theorem when $G$ is abelian. By verifying the bijectivity hypothesis for $\mathsf{Sp}_G$, we prove a nilpotence theorem for all profinite groups.

Our study then moves to the realm of rational $G$-equivariant spectra. By exploiting the continuity of our model, we construct an equivalence between the category of rational $G$-spectra and the algebraic model of the second author and Sugrue, which improves their result to the symmetric monoidal and $\infty$-categorical level. Furthermore, we prove that the telescope conjecture holds in this category. Finally, we characterize when the category of rational $G$-spectra is stratified, resulting in a classification of the localizing ideals in terms of conjugacy classes of subgroups. 

To facilitate these results, we develop some foundational aspects of pro-tt-geometry. For instance, we establish and use the continuity of the homological spectrum and introduce a notion of von Neumann regular tt-categories, of which rational $G$-spectra is an example.
\end{abstract}

\maketitle

\setcounter{tocdepth}{1}

\vspace{-1em}
\begin{figure}[h]
    \centering{}
    \includegraphics[scale=0.275]{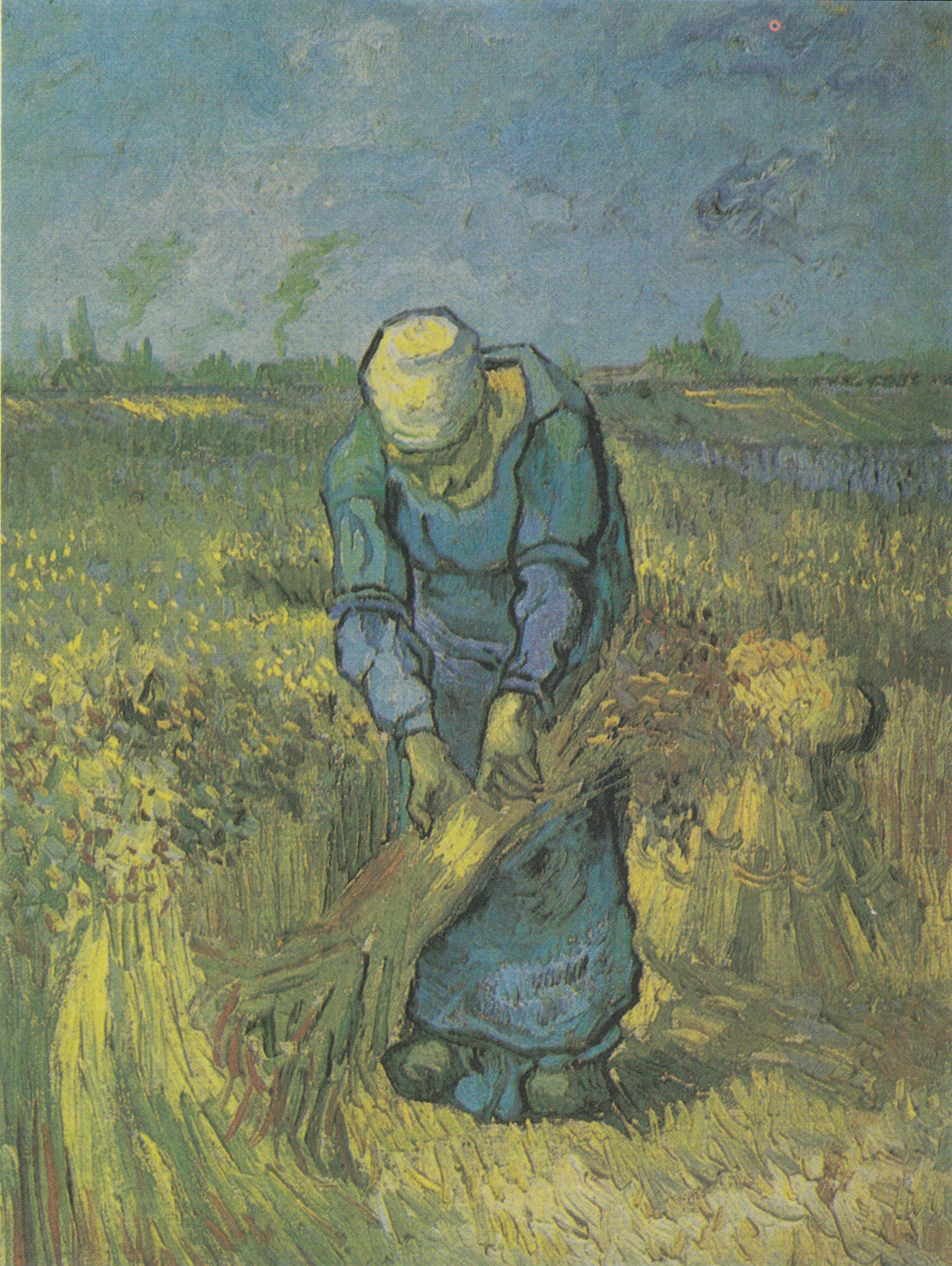}\caption*{\textit{A Peasant Woman Binding Sheaves (after Millet)}, van Gogh, 1889.}
\end{figure}
\vspace{-1em}

\newpage
\tableofcontents

\vspace{-3em}

\section*{Introduction}

The principal aim of this paper is to describe structural properties of the category of  $G$-equivariant cohomology theories where $G$ is a profinite group. That is, we study the category of genuine $G$-spectra, $\Sp_G$. We will also be interested in the category $\Sp_{G,\Q}$ of rational $G$-spectra, which corresponds to those cohomology theories which take rational values.  
Broadly speaking, the main goals of this paper are as follows:
\begin{itemize}[leftmargin=*]
    \item Formulate and construct a continuous model of $\Sp_G$, prove that it is equivalent to all existing models of profinite spectra in the literature, and to establish the fundamental properties of this category. This extends earlier work for the case of finite groups, and is analogous to the situation for compact Lie groups \cite{lmms_86,GreenleesMay1995_survey,mandell_may} or proper equivariant homotopy theory for not necessarily compact Lie groups \cite{DHLPS_proper}.
    \item Study the tensor-triangular geometry of $\Sp_G$, generalizing previous works for finite groups \cite{balmersanders,BHNNNS2019} and in analogy to the case of compact Lie groups \cite{BarthelGreenleesHausmann2020}. Highlights include a nilpotence theorem and a thick subcategory theorem, the latter relative to the ambiguity present for finite groups. While these results take a form familiar to experts, novel approaches are developed due to the open question of joint conservativity of the geometric fixed point functors in the profinite context.
    \item Determine the essential tensor-triangular features of the category of rational $G$-spectra in terms of the profinite group $G$, including stratification, the telescope conjecture, and the construction of an algebraic model, based on the development of abstract ``pro-tt-geometry''. 
    In particular, this establishes a strong form of the profinite counterpart of Greenlees's conjecture \cite{greenleesreport,BarnesKedziorek2022_survey} for compact Lie groups by extending earlier work of Barnes and Sugrue \cite{barnes_zp,sugrue_thesis,barnessugrue_spectra}.
\end{itemize}

\subsection*{Context}

Let $G$ be a finite group. Then the global structure of the category $\Sp_G$ of genuine equivariant spectra is, by now, well understood via the toolbox of tensor-triangular geometry. For example there is a nilpotence theorem \cite{Strickland2012pp,balmersanders}, and we know the Balmer spectrum as a set for all $G$, with the topology being described for abelian groups and extra-special 2-groups \cite{BHNNNS2019, KuhnLloyd2020pp}. The situation is even more clear when one moves to the rational setting (i.e., only considering those cohomology theories which take rational values). In this case, the Balmer spectrum is simply the discrete space whose points are conjugacy classes of subgroups: $\Sub(G)/G$. The simplicity of the category is reflected by the classification of localizing ideals of the category, i.e., the category is \emph{stratified}. Moreover, the discreteness of the spectrum allows us to decompose the category as a product which then leads to the development of an algebraic model \cite{greenlees_may, barnesfinite,wimmer_model}.

To capture more general actions, one would like to be more adventurous and move beyond the case of finite groups. One direction of generalization is to compact Lie groups. In this setting, by \cite{BarthelGreenleesHausmann2020}, the Balmer spectrum is once again known as a set, with topology determined up to the ambiguity left in the case of finite groups, and the equivariant nilpotence theorem holds. There has also been extensive work on the rationalization of this category. For example it is a conjecture of Greenlees that this category always admits an algebraic model \cite{greenleesreport} with the cases of 
$SO(2)$~\cite{greenleesrationals1,shipleys1,BGKSs1},
$O(2)$~\cite{greenleeso2,barneso2},
$SO(3)$~\cite{greenleesso3,kedziorekso3}, and
tori~\cite{greenleestorus,GStorus}
being resolved. Recent work of the first and third author with Greenlees has extensively studied the tensor-triangular geometry of this category and provided significant evidence towards this conjecture \cite{BalchinBarthelGreenlees2023pp}. 

This paper considers a different generalization of finite groups. While compact Lie groups might be seen as adding height to the finite case, we will consider \emph{profinite groups} which can intuitively be viewed as adding ``width'' to the finite case.

There are several reasons why a concrete understanding of the category of $G$-equivariant spectra for profinite $G$ is of interest. The most prominent of these reasons is that profinite groups are abundant, indeed, they are exactly the groups arising as Galois groups of infinite field extensions \cite{waterhouse}. Secondly, the category $\Sp_G$ for $G$ a profinite group is of use in the theory of motivic homotopy theory. Specifically, by considering categories of modules over a suitable ring object in the category $\Sp_G$ one can model \emph{Artin motives}, which have been the subject of recent work of Balmer--Gallauer \cite{BGperm1, BalmerGallauerPerm, BGperm2}.

Finally, the study of equivariant spectra for profinite $G$ can be seen as a dual to the \emph{proper} equivariant homotopy theory of Degrijse--Hausmann--L\"uck--Patchkoria--Schwede \cite{DHLPS_proper}. Proper stable equivariant homotopy theory is concerned with building a model structure on the category of $G$-equivariant spectra using the information at \emph{sub}groups which are well-behaved (in this case, compact subgroups). In the profinite case the model is built out of those \emph{quotient} groups which are well-behaved (in this case, finite quotients). In fact, there is also a common generalization of the profinite and the compact Lie group case, these are exactly the \emph{compact Hausdorff groups}.

In summary, we have the following families of groups:
\[
\xymatrix@R=3em{
\text{0-dimensional} \ar@{~>}[d] && \text{discrete} \ar[d]&\text{finite} \ar[l] \ar[r] \ar[d]& \fbox{\text{profinite}} \ar[d] \\
\text{higher-dimensional} && \text{Lie}  & \ar[l] \text{compact Lie} \ar[r] & \text{\parbox{0.75in}{\centering compact Hausdorff}}
}
\]
Given our understanding of the category $\Sp_G$ for compact Lie groups, alongside the developments in this paper for the profinite case, we are at a point where serious studies of the compact Hausdorff setting is possible, with techniques from both sides being required.

The way in which we will approach the study of the category $\Sp_G$ can be seen as a categorification of results of Dress. Dress proved in \cite{dress_notes} that, for $G$ a profinite group, constructions such as the space of closed subgroups and the Burnside ring are continuous in the sense that if $G = \lim_i G_i$, then $\Sub(G) \cong \lim_i \Sub(G_i)$ and $A(G) \cong \colim_i A(G_i)$. We will see that these results are a shadow of a categorical continuity result, that is, $\Sp_G \simeq \colim_i \Sp_{G_i}$, where the colimit is taken over inflations and has to be interpreted in an appropriate $\infty$-category of $\infty$-categories. 
 
While working with profinite groups allows us to bring to bear ``continuous'' constructions and methods that prove to be highly effective, it also offers new phenomena and difficulties when compared to the compact Lie group case. One source of complications for profinite groups is that it is unclear how to induct on closed subgroups or quotient groups to prove results such as the conservativity of geometric fixed points, thereby necessitating the introduction of new ideas.

\subsection*{Main results}

Tensor-triangular methods have proved effective in classification problems since its axiomatic inception by Balmer \cite{balmer_spectrum}. This general toolbox takes insight from earlier threads of research in commutative algebra, chromatic homotopy theory, and representation theory. 

Balmer's main classification theorem provides an abstract classification of the radical thick ideals of an essentially small tensor-triangulated subcategory via the Thomason subsets of a topological space: the \emph{Balmer spectrum}. Guided by classical work in commutative algebra and modular representation theory, Barthel--Heard--Sanders provide sufficient and necessary conditions for when Balmer's result can be extended to a classification of localizing ideals of some ambient rigidly-compactly generated tensor-triangulated category via the theory of \emph{stratification} \cite{BarthelHeardSanders2023}.  

A particularly fruitful ground for tensor-triangular techniques comes from considering various categories of equivariant spectra---this is the topic of \emph{equivariant tt-geometry}. This paper is a contribution to this subject, providing the first systematic study of the tt-geometry of spectra equipped with a genuine action by a profinite group $G$. The reader as this point may want to hold the example of $G = \Z_p$ in their mind.

A summary of the main properties that we will prove for this category are as follows:

\begin{mainthmI}
Let $G$ be a profinite group and $\Sp_G$ the category of equivariant $G$-spectra. 
\begin{itemize}
    \item The nilpotence theorem holds for $\Sp_G$ unconditionally. That is, the geometric functors $(K_p(n)_\ast \Phi_G^H)_{H,p,n}$ jointly detect $\otimes$-nilpotence of morphisms in $\Sp_G$ with compact source.
    \item The spectrum of the compacts $\Sp_G^\omega$ is known as a set in general, with topology determined up to the ambiguity left in the case of finite groups.
    \item If $G$ is moreover abelian, then there is a full classification of thick ideals of $\Sp_G^\omega$.
\end{itemize}
\end{mainthmI}

After our exploration into the category $\Sp_G$, we move to a more tame yet still rich category by only considering those spectra which represent cohomology theories taking values in rational vector spaces. That is, we localize $\Sp_G$ at the rational sphere spectrum. A synopsis of our results for the category $\Sp_{G,\Q}$ is:

\begin{mainthmQ}
Let $G$ be a profinite group. The category of rational $G$-spectra $\Sp_{G,\Q}$ is a tensor-triangulated category with Balmer spectrum 
\[
\Spc(\Sp_{G,\Q}^{\omega})\cong \Sub(G)/G.
\]
It always satisfies the telescope conjecture and it admits a symmetric monoidal algebraic model constructed from equivariant sheaves. Moreover, the following conditions are equivalent:
    \begin{enumerate}
        \item $\Sp_{G,\Q}$ is stratified;
        \item $\Sub(G)/G$ is countable;
        \item $A_{\Q}(G)$, the rational Burnside ring of $G$, is semi-Artinian.
    \end{enumerate}
If $G$ is abelian, then these conditions hold if and only if $G$ is isomorphic to $A\times \Z_{p_1} \times \cdots \times \Z_{p_r}$ for pairwise distinct primes $p_1,\ldots,p_r$ and $A$ a finite abelian group.
\end{mainthmQ} 

These structural results are proven by a combination of general statements on the tensor-triangular geometry of filtered colimits of tt-categories and results specific to the context of (rational) $G$-spectra for profinite $G$. The general theory is developed in \cref{part:profinitettgeom} of the paper. The major results of \cref{part:profinitettgeom} are best illustrated by their application to (rational) $G$-spectra, which we detail in the remainder of this introduction.

\subsection*{Profinite spectra}

\cref{part:integral} of the paper is concerned with the category of genuine $G$-equivariant spectra for a profinite group $G$. The first hurdle one must overcome is deciding what is meant by $G$-spectrum. There are several equivalent models for such a category (reviewed in \cref{app:equivariantmodels}), due to Fausk~\cite{Fausk2008}, Barwick~\cite{Barwick}, and Bachmann--Hoyois~\cite{bachmannhoyois_norms}.  However, none of these models are immediately amenable to the computational methods that we have at our disposal. 

Our model of choice takes a hint from a structural theorem of profinite groups themselves: any profinite group $G$ can be presented by a cofiltered limit $\lim_i G_i $ of a collection of finite groups $G_i$ (the $G_i$ can be taken to be the system $G/U_i$ as $U_i$ runs over the open normal subgroups of $G$). For a finite group $G_i$, we can consider the category $\Sp_{G_i}$ of genuine $G_i$ spectra \cite{mmss, mandell_may}. Working in the realm of $\infty$-categories, we can then assemble a filtered system $(\Sp_{G_i}, \mathrm{infl}^{G_i}_{G_j})$ and define the category of profinite equivariant spectra to be the colimit of this system. That is:
\begin{equation}\tag{$\star$}\label{eq:defofspec}
    \Sp_G \coloneqq \colim^\omega_i \Sp_{G_i}.
\end{equation}

We will see that defining our category in this \emph{continuous} fashion will prove advantageous to computational methods. However, having such a model is fruitless if it does not coincide with the existing models, so our first theorem asserts that this continuous model is indeed a model. We can then apply a continuity result of Gallauer to obtain an (abstract) handle on the spectrum:

\begin{thmx}[\cref{thm:gsp_contfausk}]\label{thmx:a}
Let $G$ be a profinite group. The category $\Sp_{G}$ is equivalent to Fausk's model of equivariant $G$-spectra for any profinite group $G$. Moreover, if $G = \lim_{i \in I}G_i$ with $G_i$ finite, then the Balmer spectrum is the limit of the spectra for the finite quotient groups:
\[
\Spc(\Sp_{G}^{\omega}) \cong \lim_{i}\Spc(\Sp_{G_i}^{\omega}).
\]
\end{thmx}

Of critical importance to the theory of equivariant stable homotopy theory is the existence of geometric fixed point functors, which allow one to separate out information at different subgroups, which is then governed by spectra with (naive) Weyl group action. Making use of our continuous model we construct geometric fixed points for profinite groups through a continuous extension of the usual geometric fixed points for finite groups (\cref{def:gfp}). We note that understanding geometric fixed points in this setting is a more subtle undertaking due to the fact that non-open subgroups may fail to be of finite index. 

We use these geometric fixed points to prove a generalization of the nilpotence theorem of Strickland \cite{Strickland2012pp} and Balmer--Sanders \cite{balmersanders} to the profinite realm, which strengthens these results even for finite groups by allowing the target to be not-necessarily compact. Of worthwhile note is our proof strategy which first relies on a computation of the \emph{homological spectrum} via a new continuity result (\cref{thm:spchcontinuity}) and using this to deduce the desired result. In contrast, usually it is the nilpotence theorem that allows for a computation of the homological spectrum.

\begin{thmx}[\cref{cor:nilpotence}]
    Let $G$ be a profinite group. Then the homological functors $(K_p(n)_\ast \Phi_G^H)_{H,p,n}$ jointly detect $\otimes$-nilpotence of morphisms in $\Sp_G$ with compact source.
\end{thmx}

In \cref{thmx:a} we observed that it was possible to abstractly describe the Balmer spectrum of $\Sp_G$ using the Balmer spectra of the $\Sp_{G_i}$. However, even in the finite group case we do not have a complete understanding of the topology of the Balmer spectrum. Work of Balmer--Sanders \cite{balmersanders} and Barthel--Hausmann--Naumann--Nikolaus--Noel--Sanders \cite{BHNNNS2019} give us a complete understanding in the abelian case, and work of Kuhn--Lloyd gives us understanding in the case where $G$ is an extraspecial 2-group \cite{KuhnLloyd2020pp}. In all of these cases, the topology is described via the computation of \emph{blue shift numbers} $\beth_G(H,K, p, m)$ which determine inclusions of Balmer primes. We extend these results to the profinite context, showing that for any profinite group $G$, the topology on $\Spc(\Sp_{G}^{\omega})$ is explicitly determined by the Hausdorff topology on $\Sub(G)/G$ and the blueshift numbers $\beth_{G_i}(H,K, p, m)$ for the finite quotient groups $G_i$ of $G$; see \cref{ssec:gsp_prism} for a precise statement. Furthermore, we resolve the case of abelian profinite group completely:

\begin{thmx}[\cref{thm:prism_abelian}]\label{thmx:thmxc}
    Let $A$ be an abelian profinite group. For closed subgroups $H,K \leqslant A$, primes $p,q$, and $n,m \in \N\cup\{\infty\}$, the following conditions are equivalent:
        \begin{enumerate}
            \item there is an inclusion of prime tt-ideals $\sfP_A(K,p,n) \subseteq \sfP_A(H,q,m)$;
            \item $K$ is a subgroup of $H$, the quotient $H/K$ is a pro-$p$-group, and
                \[
                \begin{cases}
                    n \geqslant m + \rank_p(H/K) & \text{if } m = 1; \\
                    n \geqslant m + \rank_p(H/K) \text{ and } p=q & \text{if } m>1.
                \end{cases}
                \]
        \end{enumerate}
    Consequently, if $H/K$ is a pro-$p$ group, then $\beth_A(H,K,p,m) = \rank_p(H/K)$.
\end{thmx}

In \cref{ex:prism_zp} we will explicitly describe how to undertake the computation of \cref{thmx:thmxc} in the case that $G = \Z_p$. The resulting schematic for the Balmer spectrum, along with the comparison map to the Zariski spectrum of the Burnside ring, is given in \cref{fig:speczp} below to whet the appetite of the curious reader.

\begin{figure}[h]
    \centering
\begin{equation*}
\label{eq:Spc(C_p)}%
\xy
{\ar@{->}_{\rho_{\Sp_{\Z_p}^\omega}} (-45,7.5)*{};(-45,-42.5)*{}};
(-45,10)*{\Spc(\Sp_{\Z_p}^\omega) = };
(-45,-45)*{\Spec(A(\Z_p)) = };
{\ar@{-} (-17,0)*{};(12,-15)*{}};
{\ar@{-} (-12,0)*{};(14.5,-15)*{}};
{\ar@{-} (-7,0)*{};(17,-15)*{}};
{\ar@{-} (-2,0)*{};(19.5,-15)*{}};
{\ar@{-} (5,0)*{};(27.5,-15)*{}};
{\ar@{-} (-12,0)*{};(12,-15)*{}};
{\ar@{-} (-7,0)*{};(12,-15)*{}};
{\ar@{-} (-2,0)*{};(12,-15)*{}};
{\ar@{-} (5,0)*{};(12,-15)*{}};
{\ar@{-} (-7,0)*{};(14.5,-15)*{}};
{\ar@{-} (-2,0)*{};(14.5,-15)*{}};
{\ar@{-} (5,0)*{};(14.5,-15)*{}};
{\ar@{-} (-2,0)*{};(17,-15)*{}};
{\ar@{-} (5,0)*{};(17,-15)*{}};
{\ar@{-} (5,0)*{};(19.5,-15)*{}};
{\ar@{-} (5,0)*{};(27.5,-15)*{}};
{\ar@{-} (35,0)*{};(12,-15)*{}};
{\ar@{-} (37.5,0)*{};(14.5,-15)*{}};
{\ar@{-} (40,0)*{};(17,-15)*{}};
{\ar@{-} (42.5,0)*{};(19.5,-15)*{}};
{\ar@{-} (49.5,0)*{};(27.5,-15)*{}};
{\ar@{-} (-12,5)*{};(-17,0)*{}};
{\ar@{-} (-12,10)*{};(-17,5)*{}};
{\ar@{-} (-12,15)*{};(-17,10)*{}};
{\ar@{-} (-12,19.5)*{};(-17,15)*{}};
{\ar@{-} (-12,25)*{};(-17,25)*{}};
  (-25,25)*{\scriptstyle \sfP(\Z_p,p,\infty)};
  (-25,0)*{\scriptstyle \sfP(\Z_p,p,2)};
  (-25,5)*{\scriptstyle \sfP(\Z_p,p,3)};
  (-25,10)*{\scriptstyle \sfP(\Z_p,p,4)};
  (-25,15)*{\scriptstyle \sfP(\Z_p,p,5)};
{\ar@{-} (-17,0)*{};(-17,5)*{}};
{\ar@{-} (-17,5)*{};(-17,10)*{}};
{\ar@{-} (-17,10)*{};(-17,15)*{}};
{\ar@{-} (-17,15)*{};(-17,19.5)*{}};
  {\ar@{-} (-7,5)*{};(-12,0)*{}};
{\ar@{-} (-7,10)*{};(-12,5)*{}};
{\ar@{-} (-7,15)*{};(-12,10)*{}};
{\ar@{-} (-7,19.5)*{};(-12,15)*{}};
{\ar@{-} (-7,25)*{};(-12,25)*{}};
{\ar@{-} (-2,5)*{};(-7,0)*{}};
{\ar@{-} (-2,10)*{};(-7,5)*{}};
{\ar@{-} (-2,15)*{};(-7,10)*{}};
{\ar@{-} (-2,19.5)*{};(-7,15)*{}};
{\ar@{-} (-2,25)*{};(-7,25)*{}};
  {\ar@{-} (-2,5)*{};(-12,0)*{}};
{\ar@{-} (-2,10)*{};(-12,5)*{}};
{\ar@{-} (-2,15)*{};(-12,10)*{}};
{\ar@{-} (-2,19.5)*{};(-12,15)*{}};
{\ar@{-} (-2,25)*{};(-12,25)*{}};
{\ar@{-} (-7,5)*{};(-17,0)*{}};
{\ar@{-} (-7,10)*{};(-17,5)*{}};
{\ar@{-} (-7,15)*{};(-17,10)*{}};
{\ar@{-} (-7,19.5)*{};(-17,15)*{}};
{\ar@{-} (-7,25)*{};(-17,25)*{}};
{\ar@{-} (-2,5)*{};(-17,0)*{}};
{\ar@{-} (-2,10)*{};(-17,5)*{}};
{\ar@{-} (-2,15)*{};(-17,10)*{}};
{\ar@{-} (-2,19.5)*{};(-17,15)*{}};
{\ar@{-} (-2,25)*{};(-17,25)*{}};
 (-17,0)*{\color{rainbowOne}\bullet};
 (-17,0)*{\circ};
  (-17,5)*{\color{rainbowOne}\bullet};
  (-17,5)*{\circ};
  (-17,10)*{\color{rainbowOne}\bullet};
  (-17,10)*{\circ};
  (-17,15)*{\color{rainbowOne}\bullet};
  (-17,15)*{\circ};
  (-17,23)*{\vdots};
  (-17,25)*{\color{rainbowOne}\bullet};
  (-17,25)*{\circ};
   (-6.5,31)*{\scriptstyle \sfP(p^k\Z_p,p,n)};
 (-6.5,27.5)*{\scriptscriptstyle (k \geqslant 0,\, n\geqslant 2)};
  {\ar@{-} (-1,1)*{};(-2,0)*{}};
{\ar@{-} (-1,6)*{};(-2,5)*{}};
{\ar@{-} (-1,11)*{};(-2,10)*{}};
{\ar@{-} (-1,16)*{};(-2,15)*{}};
{\ar@{-} (-1,25)*{};(-2,25)*{}};
{\ar@{-} (5,5)*{};(4,4)*{}};
{\ar@{-} (5,10)*{};(4,9)*{}};
{\ar@{-} (5,15)*{};(4,14)*{}};
{\ar@{-} (5,25)*{};(4,24)*{}};
{\ar@{-} (5,25)*{};(4,25)*{}};
{\ar@{-} (-12,0)*{};(-12,5)*{}};
{\ar@{-} (-12,5)*{};(-12,10)*{}};
{\ar@{-} (-12,10)*{};(-12,15)*{}};
{\ar@{-} (-12,15)*{};(-12,19.5)*{}};
 (-12,0)*{\color{rainbowTwo}\bullet};
 (-12,0)*{\circ};
  (-12,5)*{\color{rainbowTwo}\bullet};
  (-12,5)*{\circ};
  (-12,10)*{\color{rainbowTwo}\bullet};
  (-12,10)*{\circ};
  (-12,15)*{\color{rainbowTwo}\bullet};
  (-12,15)*{\circ};
  (-12,23)*{\vdots};
  (-12,25)*{\color{rainbowTwo}\bullet};
  (-12,25)*{\circ};
{\ar@{-} (-7,0)*{};(-7,5)*{}};
{\ar@{-} (-7,5)*{};(-7,10)*{}};
{\ar@{-} (-7,10)*{};(-7,15)*{}};
{\ar@{-} (-7,15)*{};(-7,19.5)*{}};
 (-7,0)*{\color{rainbowThree}\bullet};
 (-7,0)*{\circ};
  (-7,5)*{\color{blue}\bullet};
  (-7,5)*{\circ};
  (-7,10)*{\color{blue}\bullet};
  (-7,10)*{\circ};
  (-7,15)*{\color{blue}\bullet};
  (-7,15)*{\circ};
  (-7,23)*{\vdots};
  (-7,25)*{\color{blue}\bullet};
  (-7,25)*{\circ};
{\ar@{-} (-2,0)*{};(-2,5)*{}};
{\ar@{-} (-2,5)*{};(-2,10)*{}};
{\ar@{-} (-2,10)*{};(-2,15)*{}};
{\ar@{-} (-2,15)*{};(-2,19.5)*{}};
 (-2,0)*{\color{rainbowFour}\bullet};
 (-2,0)*{\circ};
  (-2,5)*{\color{rainbowFour}\bullet};
  (-2,5)*{\circ};
  (-2,10)*{\color{rainbowFour}\bullet};
  (-2,10)*{\circ};
  (-2,15)*{\color{rainbowFour}\bullet};
  (-2,15)*{\circ};
  (-2,23)*{\vdots};
  (-2,25)*{\color{rainbowFour}\bullet};
  (-2,25)*{\circ};
    (1.5,2.5)*{\hdots};
  (1.5,7.5)*{\hdots};
  (1.5,12.5)*{\hdots};
  (1.5,17.5)*{\hdots};
  (1.5,22.5)*{\hdots};
    (12.0,25)*{\scriptstyle \sfP(e,p,\infty)};
  (12.0,0)*{\scriptstyle \sfP(e,p,2)};
  (12.0,5)*{\scriptstyle \sfP(e,p,3)};
  (12.0,10)*{\scriptstyle \sfP(e,p,4)};
  (12.0,15)*{\scriptstyle \sfP(e,p,5)};
  {\ar@{-} (5,0)*{};(5,5)*{}};
{\ar@{-} (5,5)*{};(5,10)*{}};
{\ar@{-} (5,10)*{};(5,15)*{}};
{\ar@{-} (5,15)*{};(5,19.5)*{}};
 (5,0)*{\color{black}\bullet};
 (5,0)*{\circ};
  (5,5)*{\color{black}\bullet};
  (5,5)*{\circ};
  (5,10)*{\color{black}\bullet};
  (5,10)*{\circ};
  (5,15)*{\color{black}\bullet};
  (5,15)*{\circ};
  (5,23)*{\vdots};
  (5,25)*{\color{black}\bullet};
  (5,25)*{\circ};
 (42.5,31)*{\scriptstyle \sfP(p^k\Z_p,q,n)};
 (42.5,27.5)*{\scriptscriptstyle (k \geqslant 0,\,q\neq p,\, n\geqslant 2)};
{\ar@{-} (35,0)*{};(35,5)*{}};
{\ar@{-} (35,5)*{};(35,10)*{}};
{\ar@{-} (35,10)*{};(35,15)*{}};
{\ar@{-} (35,15)*{};(35,19.5)*{}};
 (35,0)*{\color{rainbowOne}\bullet};
 (35,0)*{\circ};
  (35,5)*{\color{rainbowOne}\bullet};
  (35,5)*{\circ};
  (35,10)*{\color{rainbowOne}\bullet};
  (35,10)*{\circ};
  (35,15)*{\color{rainbowOne}\bullet};
  (35,15)*{\circ};
  (35,23)*{\vdots};
  (35,25)*{\color{rainbowOne}\bullet};
  (35,25)*{\circ};
{\ar@{-} (37.5,0)*{};(37.5,5)*{}};
{\ar@{-} (37.5,5)*{};(37.5,10)*{}};
{\ar@{-} (37.5,10)*{};(37.5,15)*{}};
{\ar@{-} (37.5,15)*{};(37.5,19.5)*{}};
 (37.5,0)*{\color{rainbowTwo}\bullet};
 (37.5,0)*{\circ};
  (37.5,5)*{\color{rainbowTwo}\bullet};
  (37.5,5)*{\circ};
  (37.5,10)*{\color{rainbowTwo}\bullet};
  (37.5,10)*{\circ};
  (37.5,15)*{\color{rainbowTwo}\bullet};
  (37.5,15)*{\circ};
  (37.5,23)*{\vdots};
  (37.5,25)*{\color{rainbowTwo}\bullet};
  (37.5,25)*{\circ};
{\ar@{-} (40,0)*{};(40,5)*{}};
{\ar@{-} (40,5)*{};(40,10)*{}};
{\ar@{-} (40,10)*{};(40,15)*{}};
{\ar@{-} (40,15)*{};(40,19.5)*{}};
     (40,0)*{\color{blue}\bullet};
	 (40,0)*{\circ};
  (40,5)*{\color{blue}\bullet};
  (40,5)*{\circ};
  (40,10)*{\color{blue}\bullet};
  (40,10)*{\circ};
  (40,15)*{\color{blue}\bullet};
  (40,15)*{\circ};
  (40,23)*{\vdots};
  (40,25)*{\color{blue}\bullet};
  (40,25)*{\circ};
{\ar@{-} (42.5,0)*{};(42.5,5)*{}};
{\ar@{-} (42.5,5)*{};(42.5,10)*{}};
{\ar@{-} (42.5,10)*{};(42.5,15)*{}};
{\ar@{-} (42.5,15)*{};(42.5,19.5)*{}};
(42.5,0)*{\color{rainbowFour}\bullet};
(42.5,0)*{\circ};
  (42.5,5)*{\color{rainbowFour}\bullet};
  (42.5,5)*{\circ};
  (42.5,10)*{\color{rainbowFour}\bullet};
  (42.5,10)*{\circ};
  (42.5,15)*{\color{rainbowFour}\bullet};
  (42.5,15)*{\circ};
  (42.5,23)*{\vdots};
  (42.5,25)*{\color{rainbowFour}\bullet};
  (42.5,25)*{\circ};
  (46.25,2.5)*{\hdots};
  (46.25,7.5)*{\hdots};
  (46.25,12.5)*{\hdots};
  (46.25,17.5)*{\hdots};
  (46.25,22.5)*{\hdots};
{\ar@{-} (49.5,0)*{};(49.5,5)*{}};
{\ar@{-} (49.5,5)*{};(49.5,10)*{}};
{\ar@{-} (49.5,10)*{};(49.5,15)*{}};
{\ar@{-} (49.5,15)*{};(49.5,19.5)*{}};
 (49.5,0)*{\color{black}\bullet};
 (49.5,0)*{\circ};
  (49.5,5)*{\color{black}\bullet};
  (49.5,5)*{\circ};
  (49.5,10)*{\color{black}\bullet};
  (49.5,10)*{\circ};
  (49.5,15)*{\color{black}\bullet};
  (49.5,15)*{\circ};
  (49.5,23)*{\vdots};
  (49.5,25)*{\color{black}\bullet};
  (49.5,25)*{\circ};
   (12,-15)*{\color{rainbowOne}\bullet};
   (12,-15)*{\circ};
   (14.5,-15)*{\color{rainbowTwo}\bullet};
   (14.5,-15)*{\circ};
   (17,-15)*{\color{blue}\bullet};
   (17,-15)*{\circ};
   (19.5,-15)*{\color{rainbowFour}\bullet};
   (19.5,-15)*{\circ};
   (23.25,-15)*{\hdots};
   (27.5,-15)*{\color{black}\bullet};
   (27.5,-15)*{\circ};
  (19,-17.5)*{\scriptstyle \sfP(p^k\Z_p,0,1)};
  (19,-21.0)*{\scriptscriptstyle (k \geqslant 0)};
  (-6,-18)*{\underbrace{\ \hspace{20.5mm}}};
{\ar@{|->} (-6,-23)*{};(-6,-29)*{}};
{\ar@{-} (-6,-40)*{};(12,-50)*{}};
{\ar@{-} (-6,-40)*{};(14.5,-50)*{}};
{\ar@{-} (-6,-40)*{};(17,-50)*{}};
{\ar@{-} (-6,-40)*{};(19.5,-50)*{}};
{\ar@{-} (-6,-40)*{};(27.5,-50)*{}};
{\ar@{-} (35,-40)*{};(12,-50)*{}};
{\ar@{-} (37.5,-40)*{};(14.5,-50)*{}};
{\ar@{-} (40,-40)*{};(17,-50)*{}};
{\ar@{-} (42.5,-40)*{};(19.5,-50)*{}};
{\ar@{-} (49.5,-40)*{};(27.5,-50)*{}};
(-6,-40)*{\includegraphics[scale=0.003]{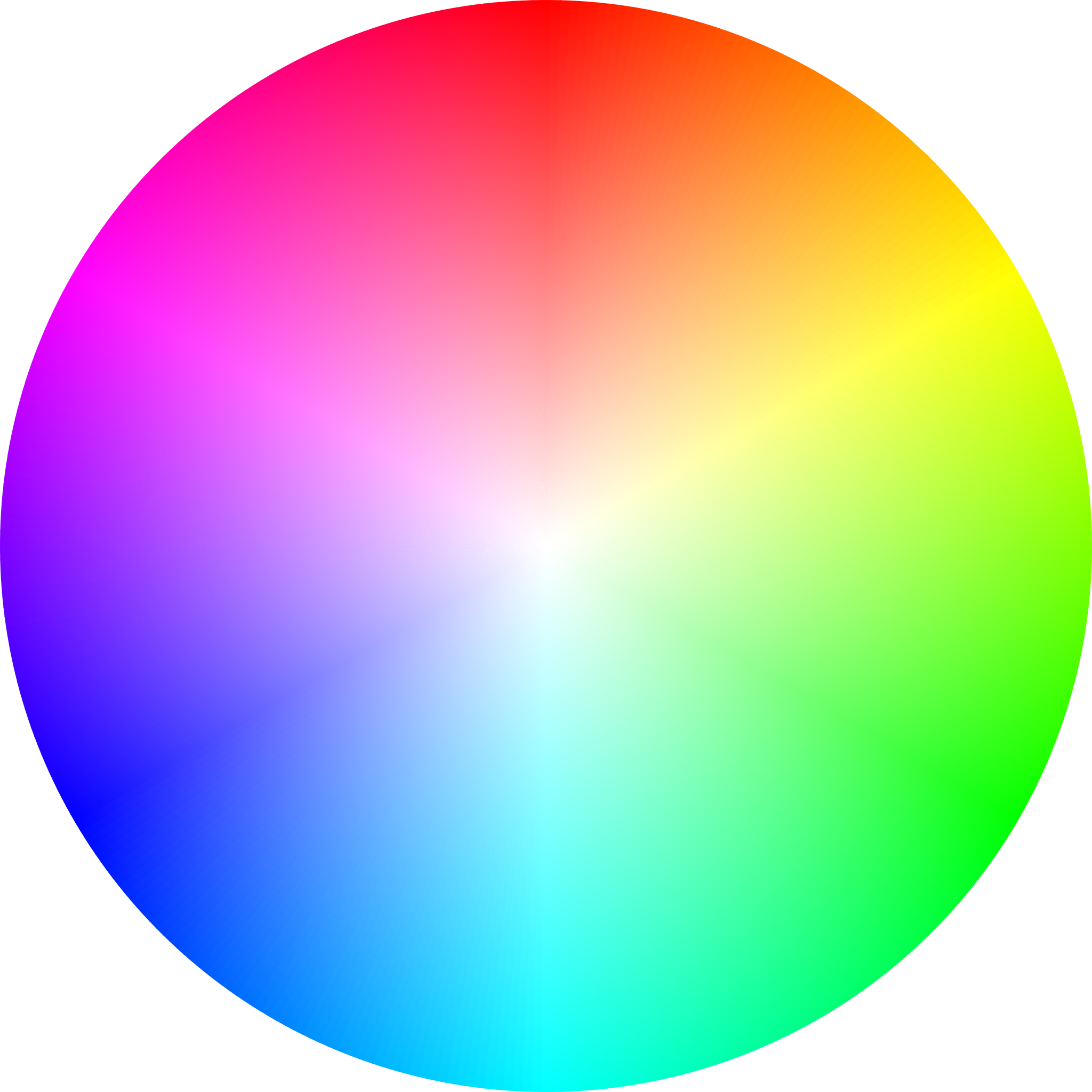}};
   (-6,-40)*{\Circle};
(12,-50)*{\color{rainbowOne}\bullet};
   (12,-50)*{\circ};
   (14.5,-50)*{\color{rainbowTwo}\bullet};
   (14.5,-50)*{\circ};
   (17,-50)*{\color{blue}\bullet};
   (17,-50)*{\circ};
   (19.5,-50)*{\color{rainbowFour}\bullet};
   (19.5,-50)*{\circ};
   (23.25,-50)*{\hdots};
   (27.5,-50)*{\color{black}\bullet};
   (27.5,-50)*{\circ};
(35,-40)*{\color{rainbowOne}\bullet};
   (35,-40)*{\circ};
   (37.5,-40)*{\color{rainbowTwo}\bullet};
   (37.5,-40)*{\circ};
   (40,-40)*{\color{blue}\bullet};
   (40,-40)*{\circ};
   (42.5,-40)*{\color{rainbowFour}\bullet};
   (42.5,-40)*{\circ};
   (46.25,-40)*{\hdots};
   (49.5,-40)*{\color{black}\bullet};
   (49.5,-40)*{\circ};
(43,-33.5)*{\scriptstyle \mathfrak{p}(p^k\Z_p,q)};
  (43,-37.0)*{\scriptscriptstyle (k \geqslant 0, q \neq p)};
(19,-52.5)*{\scriptstyle \mathfrak{p}(p^k\Z_p,0)};
  (19,-56.0)*{\scriptscriptstyle (k \geqslant 0)};
  (-7,-33.5)*{\scriptstyle \mathfrak{p}(p^k\Z_p,p) = \mathfrak{p}(p^{l}\Z_p,p)};
  (-7,-37.0)*{\scriptscriptstyle (k,l \geqslant 0)};
  {\ar@{.} (32,34)*{};(32,-43)*{}};
{\ar@{.} (53,34)*{};(53,-43)*{}};
{\ar@{.} (53,34)*{};(32,34)*{}};
{\ar@{.} (53,-43)*{};(32,-43)*{}};
  {\ar@{.} (31,35)*{};(31,-44)*{}};
{\ar@{.} (54,35)*{};(54,-44)*{}};
{\ar@{.} (54,35)*{};(31,35)*{}};
{\ar@{.} (54,-44)*{};(31,-44)*{}};
\endxy
\end{equation*}
\caption{A schematic for $\Spc(\Sp_{\Z_p}^\omega)$ with a comparison to the Zariski spectrum of the Burnside ring of $\Z_p$. There is a copy of the contents of the dotted box for each prime $q \neq p$, and these do not interact with one another. The dots in the upper half of the schematic correspond to tt-primes, and lines correspond to inclusions, with smaller primes going up and to the right on the page.}\label{fig:speczp}
\end{figure}

\subsection*{Rational profinite spectra}

In \cref{part:qgsp} we consider the category of rational $G$-spectra, $\Sp_{G,\Q}$. This is the full category of $G$-spectra that are local with respect to rational homology. 
Work of the second author in collaboration with Sugrue proves the existence of an algebraic model for the category $\Sp_{G,\Q}$ at the model category theory level. The model is defined in terms of  
chain complexes of rational $G$-equivariant sheaves over a $G$-space $X$, $\Ch (\Shv_{G,\Q}(X))$. 
The first results of \cref{part:qgsp} develop the necessary machinery to promote this to an equivalence at the level of \emph{symmetric monoidal} $\infty$-categories. An integral part of this passage is the existence of a suitable model category on the category of chain complexes of equivariant sheaves:

\begin{thmx}[\cref{cor:sheavescompgen}, \cref{thm:eqsheaves_continuity}]\label{thmx:sheaves}
Let $G$ be a profinite group which acts continuously on a profinite space $X$. There is a projective model structure on $\Ch (\Shv_{G,\Q}(X))$ which is finitely generated and monoidal. As such, the underlying symmetric monoidal $\infty$-category $\mathsf{D}(\Shv_{G,\Q}(X))$ is compactly generated. Moreover, if there is system of $G_i$-spaces $X_i$ such that $\lim_i G_i = G$ and $\lim_i X_i = X$ then the structure maps in the system induce a geometric equivalence 
\[
\colim_i^\omega \sfD(\Shv_{G_i , \Q}(X_i))  \xrightarrow[\quad]{\sim} \sfD(\Shv_{G,\Q}(X)).
\]
\end{thmx}

While the continuity statement of \cref{thmx:sheaves} mirrors our construction of profinite spectra in \eqref{eq:defofspec}, the proof of this result in-fact crucially uses rational coefficients. We suspect that the continuity of sheaves fails to be true in general if we instead take integral coefficients.

Using \cref{thmx:sheaves}, we focus on the case where $X=\Sub(G)$ and prove the existence of an algebraic model for $\Sp_{G,\Q}$ by restricting to the subcategory of $\mathsf{D}(\Shv_{G,\Q}(\Sub(G)))$ 
consisting of Weyl sheaves: those equivariant sheaves whose stalk at $H$ is $H$-fixed. This provides a \emph{monoidal} derived version of \cite{barnessugrue_spectra}. 
Again, the proof of this result takes advantage of the continuous description of the category in a fundamental way, extending the approach to algebraic models in the finite group case by Wimmer \cite{wimmer_model} to the profinite level:

\begin{thmx}[\cref{thm:sheafalgebraicmodel}]
Let $G$ be a profinite group. Then there is a natural symmetric monoidal equivalence 
\[
\xymatrix{\Sp_{G,\Q} \ar[r]^-{\sim} & \sfD (\Shv_{G,\Q}^{\Weyl}(\Sub(G)))}
\]
between rational $G$-spectra and the derived category of rational equivariant Weyl-$G$-sheaves on $\Sub(G)$.
\end{thmx}

Equipped with the algebraic model we then begin our exploration of the tensor-triangular geometry of $\Sp_{G,\Q}$. Firstly, combining \cref{thmx:a} with our understanding of the Balmer spectrum of $\Sp_{G_i,\Q}$ for $G_i$ a finite group along with results of Dress we can conclude that the Balmer spectrum of $\Sp_{G,\Q}$ can be described as the space $\Sub(G)/G$ equipped with the Hausdorff metric topology. This space is profinite, i.e., compact Hausdorff and totally disconnected. The next step is to understand various localizations of the category with respect to its Balmer primes.  To do so, in \cref{sec:ttstalks} and \cref{sec:zerodimtt} we develop foundational material about filtered colimits of tensor-triangulated categories, with a special emphasis on their local structure:

\begin{thmx}[\cref{cor:ttstalkformula}, \cref{cor:continuousttstalk}]
Suppose $\sfT$ is a rigidly-compactly generated tt-category and let $\sfP \in \Spc(\sfT^{\omega})$ be a point in its spectrum. 
    \begin{enumerate}
        \item The stalk of $\sfT$ at $\sfP$ can be computed as $\sfT_{\sfP} \simeq \colim_{U \ni \sfP}^\omega\sfT(U)$, where the colimit is taken over the quasi-compact opens in $\Spc(\sfT^{\omega})$ containing $\sfP$.
        \item If $\sfT = \colim_{i}^\omega\sfT_{i}$ is a filtered colimit of rigidly-compactly generated tt-categories and $\sfP_i \in \Spc(\sfT_i^\omega)$ is the image of $\sfP$ under the induced map on spectra, then we have
        \[
        \sfT_{\sfP} = (\colim_{i}^\omega\sfT_i)_{\sfP} \simeq \colim_{i}^\omega(\sfT_i)_{\sfP_i}.
        \]
    \end{enumerate}
\end{thmx}

Using this abstract theory, we are able to identify the stalks of $\Sp_{G,\Q}$ in the algebraic model. Here we show that each stalk has a particularly nice form, it is a \emph{tensor-triangular field} in the sense of \cite{BKSruminations}. Taking a leaf from the book of commutative algebra, we will say that a tt-category $\sfT$ is \emph{von Neumann regular} if every stalk of $\sfT$ is a tt-field.  We have:

\begin{thmx}[\cref{prop:qgsp_ttstalk}]
Let $G$ be a profinite group, $H \leqslant G$ a closed subgroup, and write $W_G(H) = N_G(H)/H$ for the corresponding Weyl group. Then the stalk of $\Sp_{G,\Q}$ at $H$ is given by
\[
(\Sp_{G,\Q})_H \simeq \sfD\Mod_{\Q}^{\delta}(W_G(H)),
\]
where $\sfD\Mod_{\Q}^{\delta}(W_G(H))$ is the category of discrete rational $W_G(H)$-modules (see \cref{def:discretemodules_general}). This category is a tt-field and as such $\Sp_{G,\Q}$ is von Neumann regular.
\end{thmx}

One key fact regarding tt-fields is that they are minimal in that they are generated by any non-zero object. As such, to resolve the classification of localizing ideals for $\Sp_{G,\Q}$ using the theory of stratification from \cite{BarthelHeardSanders2023} we are left only to determine when the local-to-global principle holds. That is, in a von Neumann regular tt-category, stratification holds if and only if the local-to-global principle holds. Our next result is somewhat surprising, in that it links the local-to-global principle for $\Sp_{G,\Q}$ to the condition of $\Sub(G)/G$ being countable. Using work of Gartside--Smith we are able to provide many examples (and counterexamples) of this behaviour:

\begin{thmx}[\cref{thm:scatterediffcountable}]
The category $\Sp_{G,\Q}$ is stratified if and only if the profinite space $\Sub(G)/G$ is countable. Furthermore, we have the following special cases:
    \begin{enumerate}
        \item If $G$ is abelian, then $\Sub(G)$ is countable if and only if $G$ is topologically isomorphic to $A\times \Z_{p_1} \times \cdots \times \Z_{p_r}$ for pairwise distinct primes $p_1,\ldots,p_r$ and $A$ a finite abelian group.
        \item If $G = \operatorname{SL}_n(\mathbb{Z}_p)$, then $\Sub(G)/G$ is countable if and only if $n \leqslant 2$.
    \end{enumerate}
\end{thmx}

Since $\Sp_{G,\Q}$ has a profinite spectrum, it follows formally from \cite[Theorem 9.11]{BarthelHeardSanders2023} that the telescope conjecture holds in $\Sp_{G,\Q}$ whenever it is stratified, as characterized in the previous theorem. However, we can do significantly better and in fact prove that that the telescope conjecture is true for $\Sp_{G,\Q}$ unconditionally. This relies on a new abstract criterion for the telescope conjecture based on the homological spectrum, established as \cref{prop:zerodim_telescopeconjecture}, and uses much of the structural properties we prove about $\Sp_{G,\Q}$ as well as its algebraic model. 

\begin{thmx}[\cref{thm:qgsp_tc}]
    The telescope conjecture holds in $\Sp_{G,\Q}$ for any profinite group $G$.
\end{thmx}

This result may be viewed as a multi-generator generalization of work on the telescope conjecture for derived categories of rings of small weak global dimension by Stevenson \cite{stevenson_absolutelyflatrings} and Bazzoni--\v{S}t'ov\'{\i}\v{c}ek \cite{BS_telescope}; in fact, our proof gives a different approach to their result for commutative von Neumann rings.

\subsection*{Conventions}

\begin{itemize}
    \item Unless stated otherwise $\sfT, \sfS$ will denote rigidly-compactly generated tensor-triangulated (tt-)categories with an enhancement given by a symmetric monoidal $\infty$-category. We write $\sfT^{\omega}$ for the full subcategory of $\sfT$ on its compact objects.
    \item For a collection of objects $\cE \subseteq \sfT$ we write $\langle \cE \rangle$ for the smallest localizing ideal containing $\cE$, and likewise for the Serre analogue in an abelian setting.
    \item For a space $X$ and a subset $U \subseteq X$ we write $U^c = X\setminus U$ for the complement.
    \item $\Sub(G)$ is the space of closed subgroups of a profinite group $G$, equipped with the Hausdorff metric topology. $\Sub(G)/G$ is the quotient of $\Sub(G)$ by the continuous conjugation action, and we equip it with the quotient topology.
    \item We write $\Hom$ for the usual hom between objects, and $\underline{\hom}$ for the internal hom.
    \item If $\mathcal{M}$ is a Quillen model category, we write $\mathsf{N}(\mathcal{M})$ for the associated $\infty$-category.
\end{itemize}

\subsection*{Acknowledgements }

We are grateful to Markus Hausmann, Drew Heard, Beren Sanders, and Greg Stevenson for conversations related to the results of this project, and thank Drew Heard and Martin Gallauer for useful comments on a preliminary draft of this paper. 
SB and TB were supported by the European Research Council (ERC) under Horizon Europe (grant No.~101042990). SB, DB, and TB would like to thank the Max Planck Institute for Mathematics for its hospitality. The authors would also like to thank the Hausdorff Research Institute for Mathematics for its hospitality and support during the trimester program `Spectral Methods in Algebra, Geometry, and Topology', funded by the Deutsche Forschungsgemeinschaft under Germany's Excellence Strategy – EXC-2047/1 – 390685813.

\newpage
\part{Pro-tensor-triangular geometry}\label{part:profinitettgeom}

The first part of this paper is concerned with formulating abstract results for tensor-triangulated (tt-)categories whose Balmer spectrum is zero-dimensional, with a particular focus on \emph{pro-tt-geometry}, which is the study of colimits of compatible systems of tt-categories in a suitable $\infty$-category of $\infty$-categories.

To do so, we will begin in \cref{sec:ttgeom} by recalling the necessary theory from tt-geometry, including that of (co)stratification and the telescope conjecture.  In \cref{sec:ttcontinuity} we then move towards our first pro-tt result by proving a continuity statement for the homological spectrum (after introducing the necessary definitions), mirroring the analogous result for the Balmer spectrum. In \cref{sec:ttstalks} we undertake an investigation of \emph{tt-stalks}. Finally in \cref{sec:zerodimtt} we combine all of this theory to build a collection of tools which are applicable to zero-dimensional tt-categories.

\section{Tensor-triangular geometry and the theory of stratification}\label{sec:ttgeom}

In this preliminary section we will recall the main constructions from tt-geometry that we will require from~\cite{balmer_spectrum,balmerfavi_idempotents}, along with the abstract theory of (co)stratification of tt-categories from~\cite{BarthelHeardSanders2023, BCHS_cosupport}. 

\begin{hypothesis}\label{hyp:hyp}
Throughout, $(\sfT, \otimes, \unit)$ will denote a rigidly-compactly generated tt-category. The full subcategory $\sfT^\omega$ of compact objects in $\sfT$ forms an essentially small tt-category in which all objects are rigid. We will always assume that $\sfT$ is the homotopy category of a  compactly generated $\infty$-category, which via abuse of notation we will also denote by $\sfT$. In this setting, the categories $\sfT^\omega$ and $\sfT$ determine each other. Indeed, forming the ind-category $\Ind(\sfT^\omega)$ provides an equivalence $\Ind(\sfT^\omega) \simeq \sfT$. This observation forms the basis of an equivalence of $\infty$-categories
\[\xymatrix@C=4em{
\Cat_\infty^{\natural} \ar@<1ex>[r]^-{\Ind} \ar@{}[r]|-{\sim}& \ar@<1ex>[l]^-{(-)^\omega} \Pr^{L,\omega}
}\]
between essentially small $\infty$-categories which are idempotent complete, and compactly generated $\infty$-categories~\cite[Proposition 5.5.7.8]{htt}. Here, a functor in $\Pr^{L,\omega}$ is taken to be both colimit preserving and compact object preserving. This equivalence is compatible with the adjectives stable and monoidal. We will return to these observations in \cref{sec:ttcontinuity}.

A \emph{geometric functor} is a tt-functor $f^\ast \colon \sfT \to \sfS$ between rigidly-compactly generated tt-categories that preserves direct sums. Under these conditions $f^\ast$ possesses a right adjoint $f_\ast$ which has a further right adjoint $f^{!}$ \cite{bds_wirth}. That is, we have a triple of adjoints
\[
\xymatrix@C=4em{
\sfT \ar@<1.25ex>[r]^{f^\ast} \ar@<-1.25ex>[r]_{f^{!}}& \sfS \ar[l]|{f_\ast} \rlap{.}
}
\]
A tt-functor which is an equivalence will also be called a \emph{geometric equivalence}.
\end{hypothesis}

\subsection{The spectrum}

Let $\sfT$ be as in \cref{hyp:hyp}. Associated to $\sfT^\omega$ we have the \emph{Balmer spectrum} of prime thick ideals, which we denote $\Spc(\sfT^\omega)$, along with a universal support theory $\supp$ which assigned to any object $t \in \sfT^\omega$ the set
    \begin{equation}\label{eq:littlesupport}
        \supp(t) = \{ \sfP \in \Spc(\sfT^\omega) \mid t \not\in \sfP \}.
    \end{equation}
The collection of subsets $\supp(t) \subseteq \Spc(\sfT^{\omega})$ forms the basis of closed sets of a topology on $\Spc(\sfT^\omega)$. This topology makes $ \Spc(\sfT^\omega) $ into a spectral space, i.e., a sober topological space for which the quasi-compact opens are closed under finite intersection and form a basis for the topology. The main result of \cite{balmer_spectrum} is that this topological space solves the classification problem for thick ideals; namely there is a bijection
    \begin{equation}\label{eq:suppttclassification}
        \supp \colon \{\text{thick ideals of }  \sfT^\omega \} \xrightarrow[\quad \quad]{\cong} \{\text{Thomason subsets of } \Spc(\sfT^\omega)\}
    \end{equation}
defined as $\supp(\mathsf{I}) = \bigcup_{t \in \mathsf{I}} \supp(t)$. Here a \emph{Thomason subset} is a subset of the form $\bigcup_{\lambda} Y_\lambda$ with each $Y_\lambda$ closed with quasi-compact complement.

\begin{remark}\label{rem:comparisonmaps}
    In \cite{balmer_3spectra}, Balmer constructs a \emph{comparison map} between the spectrum of an essentially small tt-category $\sfK$ and the Zariski spectrum of the degree 0 endomorphisms of its monoidal unit:
        \[
            \xymatrix{\rho_{\sfK}\colon \Spc(\sfK) \ar[r] & \Spec(\End_{\sfK}^0(\unit)).}
        \]
    In favorable situations this map is a homeomorphism, as for example when $\sfK$ is the derived category of perfect complexes over a commutative ring (by a theorem of Thomason \cite{thomasonclassification}).
\end{remark}

Under the equivalence \eqref{eq:suppttclassification}, any Thomason subset $Y$ corresponds bijectively to a thick ideal of $\sfT^\omega$. We denote this assignment as $Y \mapsto \sfT_Y^\omega := \supp^{-1}(Y)$, which extends to a localization sequence of categories
\begin{equation}\label{eqn:finiitelocncompact}
         \xymatrix{\sfT_{Y}^\omega \ar[r] & \sfT^\omega \ar[r]^-{\lambda_{Y^c}} & \sfT(Y^c)^\omega.}
\end{equation}
In particular, for any such $Y$ we obtain a triangle in $\sfT$
    \[
        e_Y  \to \unit \to f_Y
    \]
    with $f_Y \coloneqq  \lambda_{Y^c}(\unit) \in \sfT(Y^c)$ and $e_Y \in \sfT_{Y}$. Intuitively, one should see the role of $e_Y$ as detecting those objects ``supported on $Y$'' and $f_Y$ as detecting those objects ``supported away from $Y$''. Let us record some useful properties of the $\otimes$-idempotents $e_Y$ and $f_Y$ which provide a guiding intuition for their behaviour, see {\cite[Lemma 1.27]{BarthelHeardSanders2023}}.

\begin{lemma}\label{prop:bhs-strat-idempotents}
Let $\sfT$ be a rigidly-compactly generated tt-category. For any two Thomason subsets $Y_1, Y_2 \subseteq \Spc(\sfT^\omega)$ we have
\begin{itemize}
\item $e_{Y_1} \otimes f_{Y_2} = 0$ if and only if $Y_1 \subseteq Y_2$;
\item $e_{Y_1} \otimes e_{Y_2} = e_{Y_1 \cap Y_2}$;
\item $f_{Y_1} \otimes f_{Y_2} = f_{Y_1 \cup Y_2}$.
\end{itemize}
\end{lemma}

\begin{remark}\label{rem:naturalinU}
Let $U = Y^c$ and $V = Z^c$. Then the above construction of the finite localization $\lambda_U$ is natural in $U$, in the sense that there is a commutative diagram
\begin{equation}\label{eq:ttrestriction}\begin{gathered}
    \xymatrix{& \sfT^\omega \ar[dl]_{\lambda_V} \ar[rd]^{\lambda_U} \\
    \sfT(V)^\omega \ar[rr]_-{\res^V_U \coloneqq -\otimes f_{U^c}} & & \sfT(U)^\omega}
    \end{gathered}
\end{equation}
for any inclusion $U \subseteq V$.
 \end{remark}

In \cref{def:support} we will use this naturality to extend the support theory of compact objects to all objects of $\sfT$. For now, we recall some point-set topological facts that we will need for this purpose. 

We say that a point $\sfP \in \Spc(\sfT^\omega)$ is \emph{weakly visible} if it can be written as $\{\sfP\} = Y \cap Z^c$ where $Y, Z \subseteq \Spc(\sfT^\omega)$ are Thomason subsets. 

\begin{definition}
    Let $X$ be a spectral space. If all points of $X$ are weakly visible then we say that $X$ is \emph{weakly Noetherian}.
\end{definition}

Weakly Noetherian spaces are, as the name suggests, a generalization of Noetherian spectral spaces. We will also require an intermediate class of spectral spaces that we will introduce shortly, after recalling some facts regarding the generalization closure of subsets in a spectral space. 

\begin{remark}\label{rem:gen}
    Let $X$ be a spectral space. Recall that the \emph{generalization closure} of subset $S \subseteq X$ is defined as $\gen(S) := \{x \in X\mid \overline{\{x\}}\cap S \neq \varnothing\}$, i.e., the set of elements that generalize an element in $S$. By \cite[Theorem 4.1.5(iii)]{book_spectralspaces}, 
     \begin{equation}\label{eqn:gen}
        S \text{ is quasi-compact } \iff  \gen(S) = \bigcap_{S \subseteq U} U,
     \end{equation}
where $U$ runs through the quasi-compact open subsets of $X$ containing $S$. We will make repeated use of this formula. 

A closely related fact worth recording is the equivalence of the following three conditions on a subset $V \subseteq X$:
        \begin{enumerate}
            \item $V$ is quasi-compact and generalization closed;
            \item $V$ is the complement of a Thomason subset in $X$;
            \item $V = \bigcap_{U\supseteq V} U$, where the intersection over all quasi-compact opens in $X$ containing $V$.
        \end{enumerate}
    In particular, it follows from this equivalence that the intersection of arbitrary quasi-compact opens is quasi-compact.
    
    For reference, we include a sketch of the argument. Using \eqref{eqn:gen} along with the observation that intersections of opens are closed under generalization, we obtain the equivalence of (1) and (3). The equivalence of (2) and (3) follows from the identity $X \setminus (\bigcap_{U\supseteq V}  U) = \bigcup_{U\supseteq V} (X \setminus U)$ which recovers the description of Thomason subsets when the $U$ are quasi-compact open as in our hypothesis. 
\end{remark}

\begin{definition}\label{def:genNoeth}
    Let $X$ be a spectral space. If for each $x \in X$ the subspace $\mathrm{gen}(x)$ is Noetherian, then we say that $X$ is \emph{generically Noetherian}.
\end{definition}

We have already claimed that weakly Noetherian spaces are a generalization of Noetherian spaces. The generically Noetherian spaces form an intermediate class between these two notions:

\begin{lemma}[{\cite[Lemma 9.9]{BarthelHeardSanders2023}}]\label{lem:genimpweak}
    For $X$ spectral, we have the following implications:
        \[
            X \text{ Noetherian } \implies X \text{ generically Noetherian } \implies X \text{ weakly Noetherian}.
        \]
\end{lemma}

\subsection{Stratification and costratification in tensor-triangular geometry}\label{ssec:ttstratification}

Returning to abstract tt-geometry, suppose that $\sfP \in \Spc(\sfT^\omega)$ is weakly visible, witnessed by Thomason subsets $Y,Z$ with $\{\sfP\} = Y \cap Z^c$. Then we can construct a $\otimes$-idempotent $g({\sfP})  := e_Y \otimes {f_Z} $. We note that the construction of $g({\sfP})$ is not sensitive to the choice of witness Thomason subsets \cite[Corollary 7.5]{balmerfavi_idempotents}. We are now in a position to define support for all objects of $\sfT$.

\begin{definition}\label{def:support}
Let $\sfT$ be a rigidly-compactly generated tt-category with $\Spc(\sfT^\omega)$ weakly Noetherian. The \emph{support} of an object $t \in \sfT$ is defined as
\[
\operatorname{Supp} (t) = \{ \sfP \in \Spc(\sfT^\omega) \mid g({\sfP})  \otimes t \neq 0 \}.
\] 
\end{definition}
The support of \cref{def:support} extends that of \eqref{eq:littlesupport} in that if $t \in \sfT^\omega$ then $\supp(t) = \Supp(t)$.

Recall that we write $\langle \cE \rangle$ for the smallest localizing ideal generated by a collection $\cE$ objects in $\sfT$. Associated to the $\otimes$-idempotent $g({\sfP})$ is the corresponding category $\langle g({\sfP}) \rangle$ of those objects supported exactly at the prime $\sfP$. For notational clarity we will write $\Gamma_{\sfP} \sfT := \langle g({\sfP}) \rangle$. Balmer's classification result tells us that the support theory of compact objects can be used to detect all thick ideals of $\sfT^\omega$ via the Thomason subsets. One might wonder if the large support theory defined above is also of use in classification results. Under the assumption of a weakly Noetherian spectrum, the support theory provides us with a surjection
\begin{equation}\label{eq:bigsupp}
\operatorname{Supp} \colon \{ \text{localizing ideals of } \sfT \} \xrightarrow[\quad \quad]{}  \{ \text{subsets of } \Spc(\sfT^\omega) \}.
\end{equation}
The theory of \emph{stratification} is concerned with when this map is moreover an injection, providing a classification of localizing ideals via subsets of the Balmer spectrum.

\begin{definition}\label{defn:minimal}
Let $\sfT$ be a rigidly-compactly generated tt-category. If the map \eqref{eq:bigsupp} is a bijection then we say that $\sfT$ is \emph{stratified}.
\end{definition}

Unlike the classification of thick ideals, it is not the case that $\sfT$ will always be stratified and that we have a classification of the localizing ideals. Instead, it is a property of the tt-category. Work of the third author along with Heard and Sanders addresses the problem of when stratification holds as we shall now discuss~\cite{BarthelHeardSanders2023}. We first of all recall the definition of the local-to-global principle and minimality for a localizing ideal. 

\begin{definition}\label{defn:local2global}
Let $\sfT$ be a rigidly-compactly generated tt-category with $\Spc(\sfT^\omega)$ weakly Noetherian. We say that $\sfT$ \emph{satisfies the local-to-global principle} if
\[
\langle t \rangle = \langle t \otimes g({\sfP}) \mid \sfP \in \operatorname{Supp}(t) \rangle
\]
for every $t \in \sfT$.
\end{definition}

\begin{definition}\label{defn:strat}
Let $\sfT$ be a rigidly-compactly generated tt-category with $\Spc(\sfT^\omega)$ weakly Noetherian. Let $\sfP \in \Spc(\sfT^\omega)$. We say that $\Gamma_{\sfP} \sfT$ is a \emph{minimal localizing ideal} if there are no nonzero proper localizing ideals contained in it.
\end{definition}

\begin{proposition}[\cite{BarthelHeardSanders2023}]\label{prop:bhs-strat}
Let $\sfT$ be a rigidly-compactly generated tt-category with $\Spc(\sfT^\omega)$ weakly Noetherian. Then the following are equivalent:
\begin{enumerate}
\item $\sfT$ satisfies the local-to-global principle and for each $\sfP \in \Spc(\sfT^\omega)$, $\Gamma_{\sfP} \sfT$ is a minimal localizing ideal of $\sfT$.
\item $\sfT$ is stratified.
\end{enumerate}
\end{proposition}

As such, to establish the classification of localizing ideals by the theory of support, it suffices to check minimality at all points and the local-to-global principle. 

We now recall a further property that a tt-category may or may not possess. 

\begin{definition}\label{def:tc}
Let $\sfT$ be a rigidly-compactly generated tt-category. Then $\sfT$ satisfies the \emph{telescope conjecture} if every smashing ideal of $\sfT$ is generated by a collection of compact objects. Here, a localizing ideal $\sfL \subseteq \sfT$ is said to be \emph{smashing} if the right adjoint to the associated quotient functor $\sfT \to \sfT/\sfL$ preserves direct sums. 
\end{definition}
 
\begin{proposition}[{\cite[Theorem 9.11]{BarthelHeardSanders2023}}]\label{prop:tc}
Let $\sfT$ be a stratified rigidly-compactly generated tt-category such that the Balmer spectrum $\Spc(\sfT^\omega)$ is generically Noetherian in the sense of \cref{def:genNoeth}. Then $\sfT$ satisfies the telescope conjecture.
\end{proposition}

There is a dual theory to that of stratification which is concerned with the  classification of \emph{colocalizing coideals} of $\sfT$, that is, those thick subcategories $\mathsf{C}$ of $\sfT$ which are moreover closed under products and having the property that $\underline{\hom}(t,c) \in \mathsf{C}$ for all $t \in \sfT$ and for all $c \in \mathsf{C}$. To end this subsection we will briefly recall the necessary ingredients and formalism from~\cite{BCHS_cosupport}. The first thing that we require is a definition of tt-cosupport. 

\begin{definition}\label{def:cosupport}
Let $\sfT$ be a rigidly-compactly generated tt-category with $\Spc(\sfT^\omega)$ weakly Noetherian. The \emph{cosupport} of an object $t \in \sfT$ is defined as
\[
\operatorname{Cosupp} (t) = \{ \sfP \in \Spc(\sfT^\omega) \mid \underline{\hom}(g({\sfP}), t) \neq 0 \}.
\] 
\end{definition}

Analogous to the  story for stratification, under the weakly Noetherian assumption we once again obtain a surjection
\begin{equation}\label{eq:bigcosupp}
\operatorname{Cosupp} \colon \{ \text{colocalizing coideals of } \sfT \} \xrightarrow[\quad \quad]{}  \{ \text{subsets of } \Spc(\sfT^\omega) \},
\end{equation}
and we ask when this surjection is actually a bijection. When this holds, we say that $\sfT$ is \emph{costratified}. There are clear dual notions of the local-to-global principle and minimality in this setting. However, in place of  $\Gamma_{\sfP}\sfT$ one is lead to consider the category $\Lambda^{\sfP} \sfT$ which is the colocalizing coideal consisting of those objects $t \in \sfT$ such that $\underline{\hom}(g(\sfP),t) \simeq t$.

\begin{remark}\label{rem:stalkcostalkduality}
For any weakly visible prime $\sfP$ of $\Spc(\sfT^\omega)$ there is a symmetric monoidal equivalence of categories $\Lambda^{\sfP} \sfT \cong \Gamma_{\sfP} \sfT$~\cite[Remark 5.11]{BCHS_cosupport}.
\end{remark}

All in all, this allows us to state one of the main results of~\cite{BCHS_cosupport} which dual to \cref{prop:bhs-strat}.

\begin{proposition}[{\cite[Theorem 7.7]{BCHS_cosupport}}]\label{prop:bhs-costrat}
Let $\sfT$ be a rigidly-compactly generated tt-category with $\Spc(\sfT^\omega)$ weakly Noetherian. Then the following are equivalent:
\begin{enumerate}
\item $\sfT$ satisfies the colocal-to-global principle and, for each $\sfP \in \Spc(\sfT^\omega)$, $\Lambda^{\sfP} \sfT$ is a minimal colocalizing coideal of $\sfT$.
\item $\sfT$ is costratified.
\end{enumerate}
\end{proposition}

There are several relations between the notions appearing in stratification and costratification which we will find of use in our main application.

\begin{proposition}[{\cite[Theorem 6.4]{BCHS_cosupport}}]\label{prop:l2giffcol2g}
Let $\sfT$ be a rigidly-compactly generated tt-category with $\Spc(\sfT^\omega)$ weakly Noetherian. Then the following are equivalent:
\begin{enumerate}
\item $\sfT$ satisfies the local-to-global principle.
\item $\sfT$ satisfies the colocal-to-global principle.
\item Cosupport detects trivial objects in $\sfT$.
\end{enumerate}
\end{proposition}

\begin{proposition}[{\cite[Theorem 7.19]{BCHS_cosupport}}]\label{prop:costractimpstrat}
Let $\sfT$ be a rigidly-compactly generated tt-category with $\Spc(\sfT^\omega)$ weakly Noetherian. If $\sfT$ is costratified then it is also stratified.
\end{proposition}

The categories $\Gamma_{\sfP} \sfT$ and $\Lambda^{\sfP}\sfT$ that we encounter in our application to profinite $G$-spectra will have a particularly nice form. They are \emph{tt-fields} in the sense of~\cite{BKSruminations} as we now recall.

\begin{definition}\label{defn:ttfield}
A \emph{tt-field} is a rigidly-compactly generated tt-category $\sfF$ which satisfies the following two properties:
    \begin{enumerate}[label=(F\arabic*)]
        \item every object in $\sfF$ is a coproduct of compact objects in $\sfF$;
        \item every non-zero object in $\sfF$ is $\otimes$-faithful (i.e., if $t \otimes \alpha = 0$ for $0 \neq t \in \sfF$ and $\alpha$ some morphism in $\sfF$, then $\alpha=0$).
    \end{enumerate}  
\end{definition}

The advantage of having a tt-field is that stratification and costratification are immediate as we now record.

\begin{proposition}[{\cite[Corollary 8.9]{BCHS_cosupport}}]\label{prop:ttfieldstrat}
    Any tt-field $\sfF$ is stratified and costratified with spectrum $\Spc(\sfF^\omega) = \{(0)\}$.
\end{proposition}

\subsection{Presheaves of tensor-triangulated categories and their stalks}

Let $\sfT$ be a rigidly-compactly generated tt-category. Following Balmer, we recall (a big variant of) the structure presheaf of tt-categories on $X=\Spc(\sfT^{\omega})$ constructed in \cite{balmer_spectrum}.

In \eqref{eqn:finiitelocncompact} we saw how the selection of a Thomason subset $Y \subseteq \Spc(\sfT^\omega)$ with complement $U = Y^c$ provided a localization sequence of categories. Applying the ind-construction, we obtain the localization sequence for $\sfT$ itself:
\begin{equation}\label{eq:finitelocalizationsequence}
    \xymatrix{T_{U^c} \ar[r] & \sfT \ar[r]^-{\lambda_U} & \sfT(U).}
\end{equation}
Recalling the naturality of these localizations (cf., \cref{rem:naturalinU}) we are led to the following construction:

\begin{definition}[Balmer]\label{def:ttstructuresheaf}
Let $\sfT$ be a rigidly-compactly generated tt-category with spectrum $X = \Spc(\sfT^{\omega})$ whose collection of quasi-compact opens is denoted $\qcopen{X}$. The \emph{(tt-)structure presheaf} $\mathcal{O}_{\sfT}$ associated to $\sfT$ is the presheaf of tt-categories 
\[
\xymatrix{\mathcal{O}_{\sfT}\colon \qcopen{X}^{\op} \ar[r] &  \{\text{tt-categories}\}, \quad U \mapsto \mathcal{O}_{\sfT}(U) \coloneqq \sfT(U),}
\] 
with restriction maps as in \eqref{eq:ttrestriction}. 
\end{definition}

\begin{remark}
    From a support-theoretic perspective, it is natural to define the tt-structure presheaf $\sfT(-)$ of $\sfT$ on arbitrary complements of Thomason subsets of $X = \Spc(\sfT^{\omega})$. The above presheaf can be extended uniquely to a presheaf with domain $\Thc{X}$, the poset of complements of Thomason subsets of $X$. Explicitly, we will show in \cref{prop:localsectionsformula} how $T(V)$ for $V$ the complement of a Thomason subset of $X$ is uniquely determined by the local sections over quasi-compact opens containing it.
    
    Another variation is conceivable: one might wish to define $\sfT(-)$ on all open subsets of $X$ as opposed to only the quasi-compact opens. Since $X$ is spectral, its topology admits a basis of quasi-compact opens, which can be used to show that $\sfT(-)$ extends uniquely from the quasi-compact opens to arbitrary open subsets of $X$.
\end{remark}

\begin{remark}
    In fact, working with a suitable enrichment, the pair $(\Spc(\sfT^{\omega}), \mathcal{O}_{\sfT})$ forms a locally 2-ringed space, whose construction is part of forthcoming work \cite{AokiBarthelChedalavadaSchlankStevenson2023ip}. In particular, $\mathcal{O}_{\sfT}$ satisfies descent on $\Spc(\sfT^{\omega})$ and therefore is a sheaf of tt-categories on the spectrum of $\sfT$ (in a suitable sense). For the purposes of the present paper, the presheaf structure will satisfy. 
\end{remark}

Any theory of sheaves with coefficients in a category with filtered colimits affords a theory of stalks. Developing this notion for the structure presheaf of $\sfT$ will be the topic of \cref{sec:ttstalks}.

\section{Continuity of the homological spectrum}\label{sec:ttcontinuity}

The goal of this section is to establish a continuity property of the homological spectrum of a tt-category. Along the way, we recall some material on limits and colimits of tt-categories that will be instrumental throughout this paper.

\subsection{Limits and colimits of $\infty$-categories}\label{ssec:limcoliminfty}

Before we delve into the result promised in the title of this section, we first need to revisit the construction of limits and colimits of $\infty$-categories. In particular we will eventually require some dexterity moving between various categories of $\infty$-categories. The material that we discuss here is taken from \cite{htt,ha} and also neatly summarised in \cite[\S 15]{bachmannhoyois_norms}.

We fix the following notation choices, some of which we have already encountered:

\begin{itemize}
    \item $\widehat{\Cat}_\infty$ -- the category of (not necessarily small) $\infty$-categories.
    \item $\Cat_\infty$ -- the category of essentially small $\infty$-categories.
    \item $\Pr^L$ -- the category of presentable $\infty$-categories and left adjoints between them.
    \item $\Pr^R$ -- the category of presentable $\infty$-categories and right adjoints between them.
    \item $\Pr^{L,\omega}$ -- the category of compactly generated presentable $\infty$-categories and compact object preserving left adjoints between them.
    \item $\Cat_{\infty}^\natural$ -- the category of idempotent complete essentially small $\infty$-categories.
\end{itemize}
It is clear that all of these categories have suitable stable and monoidal variants, but we will not carry around this subscript here. 

The first observation that we make is that there is an equivalence of categories $\Pr^R \simeq (\Pr^L)^{\op}$ obtained by passing to adjoints. By \cite[Propositions 5.5.3.13 and 5.5.3.18]{htt}, both $\Pr^L$ and $\Pr^R$ admit limits and the forgetful functors $\Pr^L \to \widehat{\Cat}_\infty$ and $\Pr^R \to \widehat{\Cat}_\infty$ preserve them. We summarise this diagrammatically as follows:

\begin{equation}\label{eq:prlprr}
\begin{gathered}
    \xymatrix@C=5em{
(\Pr^L)^{\op} \simeq \Pr^R \ar[r]^-{\text{limit}}_-{\text{preserving}} & \widehat{\Cat}_{\infty}
\rlap{.}}
\end{gathered}
\end{equation}
This implies that colimits of a diagram in $\Pr^L$ are computed by first passing to right adjoints and then taking the limit of the resulting diagram in $\Pr^R$, which may be computed in $\widehat{\Cat}_{\infty}$. 

Moving to the other categories of interest, we remind the reader (see \cref{hyp:hyp}) that there is an equivalence of categories between $\Cat^\natural_\infty$ and $\Pr^{L,\omega}$ furnished by taking idempotent completion in one direction and taking compact objects in the other. By \cite[Proposition 5.5.7.6]{htt} the forgetful functor $\Pr^{L,\omega} \to \Pr^L$ preserves colimits (the stable version is the content of \cite[Theorem 1.1.4.4]{ha}). The forgetful functor $\Cat^\natural_\infty \to \Cat_\infty$ does not preserves arbitrary colimits, but it does preserve filtered colimits by \cite[Lemma 7.3.5.10]{ha}. The forgetful functor $\Pr^L \to \widehat{\Cat}_\infty$ does not preserve colimits in general. In summary, we have the following diagram of relations:
\begin{equation}\label{eq:colimitcomparisons}
\begin{gathered}
\xymatrix@C=5em{
\Pr^{L,\omega} \ar[r]^{\text{colimit}}_{\text{preserving}} \ar@<1ex>[d] \ar@{<-}@<-1ex>[d] \ar@{}[d]|{\simeq}& \Pr^L \ar[r]^{\text{not colimit}}_{\text{preserving}}  & \widehat{\Cat}_\infty \\
\Cat_{\infty}^\natural \ar[r]^{\text{filt. colimit}}_{\text{preserving}}  & \Cat_\infty 
\rlap{,} }
\end{gathered}
\end{equation}
in which the horizontal functors are the forgetful ones.

\subsection{Filtered colimits of tt-categories and their spectrum}\label{ssec:colimits}

In this section we will provide a construction via which many examples, including our example of interest, can be obtained.  For now, we shall work in $\Cat_\infty^{\natural}$. We will begin by explicitly describing filtered colimits in the category $\Cat_\infty$ which is sufficient for our purposes by \eqref{eq:colimitcomparisons}.

To this end, suppose that $F \colon I \to \Cat_\infty$ is a filtered diagram of $\infty$-categories with colimit $\mathcal{C} = \colim_i F(i)$ and write $\mathcal{C}_i = F(i)$. We wish to describe the objects of $\mathcal{C}$ along with the mapping spaces between any two objects.

The collection of objects can be computed in the category of sets. That is, 
\[
\Ob(\mathcal{C}) = (\coprod_I \Ob(\mathcal{C}_i))/\sim,
\]
where $c_{i_0} \sim d_{i_1}$ if and only if there exists some $i_0 \xrightarrow{\theta_0} j \xleftarrow{\theta_1} i_1$ in $I$ with $F(\theta_0)(c_{i_0}) \cong F(\theta_1)(d_{i_1})$. 

In particular, the objects of $\mathcal{C}$ all occur in some $\mathcal{C}_i$. Therefore, using the fact that $I$ is filtered, for any $c \in \mathcal{C}_i$ we can coherently pick a system of representatives $c_i \in \mathcal{C}_i$ for all $i \in I$.

\begin{lemma}[{\cite[Lemma 3.12]{bss}}]\label{lem:colimitofmaps}
Suppose $F \colon I \to \Cat_\infty$ is a filtered diagram of $\infty$-categories with colimit $\mathcal{C}$. For two objects $c,d \in \mathcal{C}$ represented by $(c_i)$ and $(d_i)$ there is a canonical equivalence
\begin{equation}\label{eq:colimitmap}
    \colim_i \Map_{\mathcal{C}_i}(c_i,d_i) \xrightarrow{\quad\sim\quad}  \Map_{\mathcal{C}}(c,d),
\end{equation}
induced by the transition maps in the diagram $F$. In particular by taking homotopy there are isomorphisms $\colim_i \pi_\ast \Map_{\mathcal{C}_i} (c_i,d_i) \xrightarrow{\, \cong \,} \pi_\ast \Map_{\mathcal{C}}(c,d)$ where the filtered colimit is taken in the category of sets.
\end{lemma}

Our first application of the above discussion is a description of unions of thick ideals.

\begin{lemma}\label{lem:unionascolimit}
Let $\sfK$ be an essentially small tt-category and write $X = \Spc(\sfK)$. Suppose that $(U_i)_{i \in I}$ is a directed system of quasi-compact open subsets of $X$ and consider the induced system of thick ideals and inclusions, $(\{t \in \sfK\mid \supp(t) \subseteq U_i\})_{i\in I}$. Then we have
\[
\colim_{i}\{t \in \sfK\mid \supp(t) \subseteq U_i^c\} = \bigcup_{i \in I}\{t \in \sfK\mid \supp(t) \subseteq U_i^c\},
\]
as full subcategories of $\sfK$, where the union is taken as a full subcategory in $\sfK$.
\end{lemma}
\begin{proof}
We begin by showing that the objects coincide. Unpacking the above description of the objects of the filtered colimit we see that the objects of the left-hand side are given by
\[
    (\coprod_I \{t \in \sfK\mid \supp(t) \subseteq U_i^c\}) / \sim,
\]
where $c_{i_0} \sim d_{i_1}$ if and only if there exists some $i_0 \xrightarrow{\theta_0} j \xleftarrow{\theta_1} i_1$ with $F(\theta_0)(c_{i_0}) \cong F(\theta_1)(d_{i_1})$. However, $F(\theta_i)$ is the canonical inclusion functor, so we deduce that the objects in the filtered colimits coincide with those in the union. It remains to check that the morphism spaces coincide as well. However, this follows again from fully faithfulness of the transition functors in conjunction with \cref{lem:colimitofmaps}.
\end{proof}

The following result of Gallauer is a fundamental continuity result for the Balmer spectrum.

\begin{proposition}[{\cite[Proposition 8.2]{gallauer_filteredmodules}}]\label{prop:spccontinuity}
Let $(\sfK_i, f_{ij} \colon \sfK_i \to \sfK_j)$ be a filtered diagram of essentially small tt-categories. Then there is a canonical homeomorphism
\[
\Spc(\colim_i \sfK_i)  \xrightarrow{\sim} \lim_i \Spc(\sfK_i).
\]
\end{proposition}

\begin{corollary}\label{cor:continuous_comparisonmaps}
    Let $\sfK = \colim_{i}\sfK_i$ be a filtered colimit of essentially small tt-categories. Then the corresponding comparison maps (\cref{rem:comparisonmaps}) fit into a commutative diagram
        \[
            \xymatrix{\Spc(\sfK) \ar[r]^-{\sim} \ar[d]_{\rho_{\sfK}} & \lim_i \Spc(\sfK_i) \ar[d]^{\lim_i \rho_{\sfK_i}} \\
            \Spec(\End_{\sfK}^0(\unit)) \ar[r]^-{\sim} & \lim_i \Spec(\End_{\sfK_i}^0(\unit)).}
        \]
\end{corollary}
\begin{proof}
    The existence of the commutative square follows from the naturality of the comparison maps, which was established in \cite[Corollary 5.6(b)]{balmer_3spectra}. The top horizontal map is a homeomorphism by \cref{prop:spccontinuity}, so it remains to verify that the same is true for the bottom horizontal map. We first claim that the canonical ring homomorphism
        \[
            \xymatrix{\End_{\sfK}^0(\unit) \ar[r] & \colim_i\End_{\sfK_i}^0(\unit)}
        \]
    is an isomorphism. Indeed, the functor $\End_{(-)}(\unit)$ from essentially small tt-categories to commutative ring spectra preserves filtered colimits, and so does the functor $\pi_0$. The claim thus follows from the presentation of $\sfK$ as a filtered colimit. Since $\Spec$ converts a filtered colimit of commutative rings into a cofiltered limit of topological spaces (\cite[12.5.2]{book_spectralspaces}), the bottom horizontal map in the above square is a homeomorphism.
\end{proof}

\begin{remark}\label{rem:compactcolim}
    Following \cite[\S 3.3]{bss} and \cite[Appendix A]{heuts_goodwillie}, we define the colimit of a filtered diagram $(\sfT_i, f_{ij} \colon \sfT_i \to \sfT_j)$ of compactly generated tt-categories as the ind-completion of the corresponding system of essentially small tt-categories. We will denote the colimit as $\colim^\omega_i \sfT_i$ to record the fact that we are taking the colimit in $\Pr^{L,\omega}$.
\end{remark}

\begin{proposition}\label{prop:continuity_of_ideals}
    Let $(\sfT_i,f_{ij} \colon \sfT_i \to \sfT_j)_{i \in I}$ be a filtered system of rigidly-compactly generated tt-categories and geometric functors with filtered colimit $\sfT$. Suppose that $\sfC_i \subseteq \sfT_i^\omega$ forms a compatible filtered system of subcategories, that is, we have commuting squares
    \[
    \xymatrix{
    \sfC_i \ar[d]_{f_{ij}} \ar@{^(->}[r] & \sfT_i \ar[d]^{f_{ij}}  \\
    \sfC_j \ar@{^(->}[r] & \sfT_j \rlap{.}
    }
    \]
    Then there is an equivalence of tt-ideals in $\sfT$, where $f_i \colon \sfT_i \to \sfT$:
            \begin{equation}\label{eq:continuousP_prop}
             \colim_{i  }^{\omega}\langle \sfC_i\rangle = \bigcup_{i \in I}\langle f_i\sfC_i \rangle.
        \end{equation}
\end{proposition}
\begin{proof}
    Since all objects in $\sfC_i$ are compact in $\sfT_i$, it suffices to restrict attention to the full subcategories of compact objects throughout. Let us write $\sfK_i = \sfT_i^{\omega}$ and $\sfK = \sfT^{\omega}$, and $\thickid{\sfK}(\cE)$ for the thick ideal of $\sfK$ generated by a collection of objects $\cE \subseteq \sfK$. It follows from the description of mapping spaces in filtered colimits of $\infty$-categories in \cref{lem:colimitofmaps} that $\colim_{i \in I}\thickid{\sfK_i}(\sfC_i)$ can naturally be identified via $(f_i)$ with a full subcategory of $\sfK$, while $\bigcup_{i \in I}\thickid{\sfK}(f_i\sfC_i)$ is so by definition. 
    
    In fact, we claim that both subcategories are thick ideals in $\sfK$. The union is closed under tensoring with objects in $\sfK$ by construction, so it remains to show it is also thick. Consider $X,Y \in \bigcup\thickid{\sfK}(f_i\sfC_i)$ and assume first that $X = f_i(X_i)$ and $Y = f_j(Y_j)$ for some $X_i \in \sfC_i$ and $Y_i \in \sfC_j$. Since $I$ is filtered, we can find a common refinement $k \in I$ with $i < k$ and $j <k$ and verify the thickness in $\sfC_j$. The general case then follows from a thick subcategory argument. For the analogous statement for the colimit, we use that any object in $\colim_{i \in I}\thickid{\sfK_i}(\sfC_i)$ as well as any object in $\sfK$ comes from a finite stage of the diagram to verify the desired properties in a common refinement. Applying the geometric transition functors then gives the claim. 

    In order to finish the proof, we observe that both sides of \eqref{eq:continuousP_prop} are generated, as thick ideals inside $\sfK$, by the collection of objects $\{f_i\sfC_i\mid i \in I\}$. Indeed, this holds for $\bigcup\thickid{\sfK}(f_i\sfC_i)$ by construction, while the result for the colimit follows again by the description of objects in a filtered colimit given above and the fact that the transition functors are geometric.
\end{proof}

\subsection{Recollections on the homological spectrum and its support}\label{ssec:hspc+hsupp}

In this section, we will discuss the homological spectrum and the homological support for tt-categories, whose roots can be traced back to \cite{BKSruminations} and whose theory was developed further in \cite{balmer_homological,balmer_nilpotence,BalmerCameron2021,BarthelHeardSanders2023a}. To do so we first of all need to introduce the Freyd envelope $\mathcal{A}(\sfT^\omega)$:

\begin{definition}
Let $\sfT$ be a tt-category and denote by $\mathrm{Mod}\text{-}\sfT^\omega$ the category of additive functors $(\sfT^\omega)^{\op} \to \Ab$ equipped with the Day convolution symmetric monoidal structure. Denote by $h$ the restricted Yoneda embedding $\sfT \to \mathrm{Mod}\text{-}\sfT^\omega$. The \emph{Freyd envelope of $\sfT$}, denoted $\mathcal{A}^{\mathrm{fp}}(\sfT^\omega)$, is the tensor-abelian subcategory of $\mathrm{Mod}\text{-}\sfT^\omega$ spanned by the finitely presented objects; that is, those functors $F$ which admit presentations
\[
h(t) \longrightarrow h(s) \longrightarrow F(-) \longrightarrow 0
\]
for objects $s,t \in \sfT^\omega$.
\end{definition}

There is also a symmetric monoidal Yoneda embedding on $\sfT^{\omega}$ defined by $h(t) = \Hom_{\sfT}(-,t)|_{\sfT^\omega}$, fitting into the following commutative diagram
\[
\xymatrix{
\sfT^\omega \ar@{^(->}[d] \ar@{^(->}[r]^-{h} & \ar@{^(->}[d]  \mathcal{A}^{\mathrm{fp}}(\sfT^\omega) \\
\sfT \ar[r]_-h & \mathrm{Mod}\text{-}\sfT^\omega \rlap{.}
 }
\]

We recall that a \emph{Serre subcategory} of an abelian category $\mathcal{A}$ is a non-empty full subcategory $\mathcal{C} \subseteq \mathcal{A}$ which is closed under subobjects, quotients, and extensions.

\begin{definition}
    Let $\sfT$ be a tt-category. The \emph{homological spectrum}, $\Spch(\sfT^\omega)$, consists of all maximal Serre ideals of $\mathcal{A}^{\mathrm{fp}}(\sfT^\omega)$. Associated to any object $t \in \sfT^\omega$ is the \emph{homological support} 
    \[
    \supp^\mathrm{h}(t) = \{\mathcal{B} \in \Spch(\sfT^\omega) \mid h(t) \not\in \mathcal{B} \}.
    \]
    We will endow $\Spch(\sfT^\omega)$ with the topology generated by $\supp^\mathrm{h}(t)$ for $t$ ranging through $\sfT^{\omega}$.
\end{definition}

The above defined homological support theory can in fact be obtained by pulling back the usual tt-support. For this, we require information regarding the comparison between homological and tt-primes. One can verify that the pair $(\Spch(\sfT^\omega),\supp^{\mathrm{h}}(-))$ is a support data in the sense of \cite[Definition 3.1]{balmer_spectrum}, and as such there exists a canonical continuous surjection
\begin{equation}\label{eq:balmer_map}
    \phi_{\sfT} \colon \Spch(\sfT^\omega) \twoheadrightarrow \Spc(\sfT^\omega)
\end{equation}
defined as $\mathcal{B} \mapsto h^{-1}(\mathcal{B})$.  In particular, for any $t \in \sfT$, we have
\begin{equation}\label{eq:homologicalsupport}
\supp^{\mathrm{h}}(t) = \phi_{\sfT}^{-1}(\supp(t)).
\end{equation}

Just as we have a notion of big support $\Supp(t)$ for all $t \in \sfT$ (\cref{def:support}) that extends the triangular support $\supp$, we similarly have a big version of the homological support. One noteworthy difference is that the `big' triangular support $\Supp$ is constructed under the additional assumption that $\Spc(\sfT^c)$ is weakly Noetherian, while the extension of the homological support works unconditionally. To construct it, note that every Serre ideal of $\mathcal{A}^{\mathrm{fp}}(\sfT^\omega)$ generates a localizing Serre ideal $\langle \mathcal{B} \rangle$ of $\Mod\text{-} \sfT^\omega$. There is a corresponding pure-injective object $E_{\mathcal{B}} \in \sfT$ such that $\langle \mathcal{B} \rangle = \ker(h(E_\mathcal{B}) \otimes -)$.

\begin{definition}\label{def:homological_support}
    Let $\sfT$ be a rigidly-compactly generated tt-category. The \emph{(big) homological support} of an object $t \in \sfT$ is defined as
    \[
    \Supph = \{ \mathcal{B} \in \Spch(\sfT^\omega) \mid  \underline{\hom}(t, E_\mathcal{B}) \neq 0\}.
    \]
\end{definition}

This big homological support $\Supph(-)$ extends the small homological support $\supph(-)$ in the sense that for any compact object $t \in \sfT^\omega$ we have $\Supph(t) = \supph(t)$. One advantage of the homological support over the tensor-triangular support is that it always satisfies the \emph{tensor-product formula} by \cite[Theorem 1.2]{balmer_homological}: Let $\sfT$ be a rigidly-compactly generated tt-category. Then $\Supph(-)$ satisfies the tensor-product formula, that is, 
    \begin{equation}\label{eq:homological_tensor}
        \Supph(t \otimes t')  = \Supph(t) \cap \Supph(t').
    \end{equation}
Using the comparison map $\phi_\sfT$ from \eqref{eq:balmer_map}, one can show that there is always an inclusion 
    \begin{equation}\label{eq:suppcomparisoninclusion}
        \phi_{\sfT}\Supph(t) \subseteq \Supp(t),
    \end{equation}
see \cite[Proposition 3.10]{BarthelHeardSanders2023a}. We warn, however, that even in the case that $\phi_{\sfT}$ is a bijection this inclusion need not be an equality. This can happens if $\Supph(t)$ fails to detect zero objects; a concrete example is given in \cite[Example 5.5]{BarthelHeardSanders2023a}.

As observed by Balmer \cite[Remark 5.15]{balmer_nilpotence}, in all known cases where the homological spectrum has been identified, the comparison map $\phi_{\sfT}$ is not just a surjection, but it also an injection, that is, the homological and Balmer spectra coincide. For example, it is a bijection for $\sfT$ the 
derived category of a quasi-compact quasi-separated scheme \cite[Corollary 5.11]{balmer_nilpotence} or for the $G$-equivariant stable homotopy category for $G$ a compact Lie group \cite[Corollary 5.10]{balmer_nilpotence}. There are also general results in this direction. For example, the results of the third  author with Heard and Sanders tell us that $\phi_{\sfT}$ is a bijection whenever $\sfT$ is stratified \cite[Theorem B]{BarthelHeardSanders2023a}. Moreover, an equivalent nilpotence-type criterion for local tt-categories is established in \cite[Theorem A.1]{balmer_homological}. We provide a name for this bijectivity phenomenon, which is also known as Balmer's \emph{nerves of steel conjecture}, returning to it at the end of this section:

\begin{definition}\label{def:bijhyp}
    Let $\sfT$ be a tt-category. Then $\sfT$ is said to satisfy the \emph{bijectivity hypothesis} if the canonical surjection $\phi_{\sfT} \colon \Spch(\sfT^\omega) \twoheadrightarrow \Spc(\sfT^\omega)$ is a bijection. 
\end{definition}

\begin{remark}
    In the case that $\sfT$ satisfies the bijectivity hypothesis, then \cite[Theorem A]{BarthelHeardSanders2023a} further tells us that $\phi_{\sfT}$ is in fact a homeomorphism.
\end{remark}

\subsection{Continuity for the homological spectrum}\label{ssec:homological_continuity}

We now shift our attention to proving the announced continuity result for the homological spectrum. As such, we first recall results regarding the functoriality of the homological spectrum.

To this end, fix a geometric functor $f^\ast \colon \sfT \to \sfS$ between rigidly-compactly generated tt-categories with right adjoint $f_\ast$. By the discussion in \cite[\S 5.1]{balmer_homological}, building off work of Krause \cite{Krause2000}, there is an induced adjunction $F^\ast \dashv F_\ast $ with both adjoints being exact and $F^\ast$ symmetric monoidal and such that the following diagram commutes:
\begin{equation}\label{eq:homologicalsquare}
\begin{gathered}
    \xymatrix{
        \sfT \ar[r]^-{h_\sfT} \ar@<-0.75ex>[d]_{f^\ast} \ar@<0.75ex>@{<-}[d]^{f_\ast} & \mathrm{Mod}\text{-}\sfT^\omega  \ar@<-0.75ex>[d]_{F^\ast} \ar@<0.75ex>@{<-}[d]^{F_\ast} \\
        \sfS \ar[r]_-{h_\sfS} & \mathrm{Mod}\text{-}\sfS^\omega \rlap{.}
    }
\end{gathered}
\end{equation}

\begin{lemma}\label{lem:spchcont}
Let $f^\ast \colon \sfT \to \sfS$ be a geometric functor. Then the map
    \begin{align*}
    \varphi^{\mathrm{h}} \coloneqq \Spch(f^\ast) \colon \Spch(\sfS^\omega) &\longrightarrow \Spch(\sfT^\omega) \\
    \mathcal{C} & \longmapsto (F^\ast)^{-1}(\mathcal{C})
    \end{align*}
is continuous.
\end{lemma}

\begin{proof}
    That the map $\varphi^{\mathrm{h}}$ is well-defined is the subject of \cite[Theorem 5.10]{balmer_homological}, where it is also proved that it fits into the following commutative diagram
        \begin{equation}\label{eq:differentspc}
        \begin{gathered}
    \xymatrix@C=4em{
        \Spch(\sfS^\omega) \ar[r]^{\varphi^{\mathrm{h}}} \ar@{->>}[d]_{\phi_{\sfS}} & \Spch(\sfT^\omega) \ar@{->>}[d]^{\phi_{\sfT}} \\
        \Spc(\sfS^\omega) \ar[r]_{\varphi} & \Spc(\sfT^\omega) \rlap{ .}
    }
    \end{gathered}
    \end{equation}

    Recall that the basic closed sets of $\Spch(\sfT^\omega)$ are given by the homological supports $\supp^\mathrm{h}(t)$ for $t \in \sfT^\omega$. As such it suffices to show that $(\varphi^{\mathrm{h}})^{-1} \supp^\mathrm{h}(t)$ is closed.  Using \eqref{eq:homologicalsupport}, \eqref{eq:differentspc}, and the functoriality of the Balmer spectrum we obtain the following
    \begin{align*}
       (\varphi^{\mathrm{h}})^{-1} \supp^\mathrm{h}(t) &= (\varphi^{\mathrm{h}})^{-1} \phi^{-1}_{\sfT}(\supp(t)) \\
        &= \phi_{\sfS}^{-1} \varphi^{-1}(\supp(t)) \\
        &= \phi^{-1}_{\sfS}(\supp(f^\ast t))\\
        &= \supp^{\mathrm{h}}(f^\ast t)
    \end{align*}
    which is a basic closed of $\Spch(\sfS^\omega)$ as required.
\end{proof}

In order to mirror the constructions of \cref{ssec:colimits} we will change our focus to essentially small tt-categories for the time being. Let us collect the following auxiliary result which will be required in the proof of the continuity of the homological spectrum. The first one tells us that the Freyd envelope construction commutes with filtered colimits. 

\begin{proposition}\label{prop:Freydcont}
    Let $(\sfK_i, f_{ij} \colon \sfK_i \to \sfK_j)$ be a filtered diagram of essentially small tt-categories. Then the functors $f_{ij}$ induce a symmetric monoidal equivalence
    \[
    \colim_i \mathcal{A}^{\mathrm{fp}}(\sfK_i) \xrightarrow[\quad]{\sim} \mathcal{A}^{\mathrm{fp}}(\colim_i \sfK_i).
    \]
\end{proposition}

\begin{proof}
    As noted in the paragraph before \eqref{eq:homologicalsquare}, the canonical comparison map $\colim_i \mathcal{A}^{\mathrm{fp}}(\sfK_i) \to \mathcal{A}^{\mathrm{fp}}(\colim_i \sfK_i)$ is compatible with the symmetric monoidal structure, so it remains to show that it is an additive equivalence. The Freyd envelope enjoys the following universal property \cite{freyd}: it is initial among all homological functors $F \colon \sfK \to \mathsf{A}$, i.e., for any such functor there is a unique factorization
    \begin{equation}\label{eq:freyduniversal}
    \begin{gathered}
    \xymatrix{
    \sfK \ar[rr]^{F}  \ar[dr]_h && \mathsf{A} \\
    & \mathcal{A}^{\mathrm{fp}}(\sfK) \ar@{-->}[ur]_{\exists !}
    }
    \end{gathered}
    \end{equation}
    with $\mathcal{A}^{\mathrm{fp}}(\sfK) \to \mathsf{A}$ unique up to canonical equivalence.
    
    We show that $\colim_i \mathcal{A}^{\mathrm{fp}}(\sfK_i)$ has the universal property of \eqref{eq:freyduniversal}. To that end, suppose that $F \colon \colim_i \sfK_i \to \mathsf{A}$ is a homological functor.  For each $i \in I$ we obtain a functor $F_i \colon \sfK_i \to \mathsf{A}$ which factors by the universal property as $\sfK_i \to \mathcal{A}^{\mathrm{fp}}(\sfK_i) \to \mathsf{A}$. Using the compatibility of \eqref{eq:homologicalsquare} we can take colimits to obtain a  factorization
    \[
    \xymatrix{
    \colim_i \sfK_i \ar[rr]^{F}  \ar[dr] && \mathsf{A} \\
    & \colim_i \mathcal{A}^{\mathrm{fp}}(\sfK_i) \ar[ur]_{\alpha}
    }
    \]
    such that $\alpha$ is exact, and $\colim_i \sfK_i \to \colim_i \mathcal{A}^{\mathrm{fp}}(\sfK_i)$ is homological. It remains to show that $\alpha$ is unique up-to canonical isomorphism among all factorizations of the given functor $F$. Suppose not, that is there exists some some exact $\beta \colon \colim_i \mathcal{A}^{\mathrm{fp}}(\sfK_i) \to \mathsf{A}$ factorizing the given functor $F$ such that $\alpha \neq \beta$. By construction of maps out of colimits it follows that on restriction to some $j$ there are maps $\alpha_j, \beta_j \colon \mathcal{A}^{\mathrm{fp}}(\mathsf{K}_j) \to \mathsf{A}$ with $\alpha_j \neq \beta_j$. However, this is a contradiction as $\mathcal{A}^{\mathrm{fp}}(\mathsf{K}_j)$ satisfies the universal property of \eqref{eq:freyduniversal}.

    We conclude that $\colim_i \mathcal{A}^{\mathrm{fp}}(\sfK_i)$ satisfies the required universal property and is therefore equivalent to $ \mathcal{A}^{\mathrm{fp}}(\colim_i \sfK_i)$. 
\end{proof}

\begin{lemma}\label{lem:pullbackofserre}
    Let $(\sfK_i, f_{ij} \colon \sfK_i \to \sfK_j)$ be a filtered diagram of essentially small tt-categories and $\mathcal{B}$ be a Serre ideal (not necessarily maximal) in $\colim_i  \mathcal{A}^{\mathrm{fp}}(\sfK_i)$. Then $F_j^{-1}(\mathcal{B})$ is a Serre ideal in $\mathcal{A}^{\mathrm{fp}}(\sfK_j)$ where $F_j \colon \mathcal{A}^{\mathrm{fp}}(\sfK_j) \to \colim_i \mathcal{A}^{\mathrm{fp}}(\sfK_i)$ is the canonical functor.
\end{lemma}

\begin{proof}
    We need show that $F_j^{-1}(\mathcal{B}) = \{x \in \mathcal{A}^{\mathrm{fp}}(\sfK_j) \mid F_j(x) \in \mathcal{B} \}$ is both an ideal and a Serre subcategory. For the former, let $x \in F_j^{-1}(\mathcal{B})$ and $a \in \mathcal{A}^{\mathrm{fp}}(\sfK_j)$. Then $F_j(x \otimes a) = F_j(x) \otimes F_j(a)$ using that $F_j$ is monoidal. It is clear that the latter lies in $\mathcal{B}$ as $\mathcal{B}$ is assumed to be an ideal.
    
    Similarly, assume that $x \to a \to y$ is an exact sequence in $\mathcal{A}^{\mathrm{fp}}(\sfK_j)$ with $x,y \in F_j^{-1}(\mathcal{B})$. Applying the exact functor $F_j$ we obtain an exact sequence $F_j(x) \to F_j(a) \to F_j(y)$ in $\colim_i \mathcal{A}^{\mathrm{fp}}(\sfK_i)$. By assumption $F_j(x)$ and $F_j(y)$ lie in the Serre subcategory $\mathcal{B}$, so $F_j(x) \in \mathcal{B}$ as required.
\end{proof}

We now state and prove our main theorem of this section. We highlight that the proof of \cref{thm:spchcontinuity} follows a similar argument to that of Gallauer's continuity argument for the Balmer spectrum, using the continuity of the Freyd envelope which is afforded to us by \cref{prop:Freydcont}.

\begin{theorem}\label{thm:spchcontinuity}
Let $(\sfK_i, f_{ij} \colon \sfK_i \to \sfK_j)$ be a filtered diagram of essentially small tt-categories. Then there is a canonical homeomorphism
\[
\tau \colon \Spch(\colim_i \sfK_i)  \xrightarrow{\sim} \lim_i \Spch(\sfK_i).
\]
\end{theorem}
\begin{proof}
Write $\sfK = \colim_i \sfK_i$, so that $\mathcal{A}^{\mathrm{fp}}(\sfK) = \mathcal{A}^{\mathrm{fp}}(\colim_i \sfK_i) \simeq \colim_i\mathcal{A}^{\mathrm{fp}}(\sfK_i) $ by \cref{prop:Freydcont}. Appealing to \cref{lem:pullbackofserre}, given a maximal Serre ideal $\mathcal{B} \in \Spch (\sfK)$, we obtain a compatible system of maximal Serre ideals $\mathcal{B}_i \coloneqq F_i^{-1}(\mathcal{B}) \in \Spch(\sfK_i)$ for $i$ ranging through $I$. The map $\tau$ in the statement of the theorem is then determined by the assignment
    \[
        \tau\colon \mathcal{B} \mapsto (\mathcal{B}_i)_i.
    \]
We begin with the proof that $\tau$ is injective. Suppose that $\mathcal{B} \neq \mathcal{C} \in \Spch(\colim_i \sfK_i)$. Then we can pick some $x \in \mathcal{B} \setminus \mathcal{C}$. As $I$ is filtered, we can find some $i \in I$ such that $x_i \in \mathcal{A}^{\mathrm{fp}}(\mathcal{\sfK}_i)$ with $F_i(x_i) = x$.  But then $x_i \in F_i^{-1}( \mathcal{B}) \setminus F_i^{-1}(\mathcal{C})$ which implies that $\tau(\mathcal{B}) \neq \tau(\mathcal{C})$ as required.

We now shift our focus to proving surjectivity. Pick some $(\mathcal{B}_i)_I \in \lim_i \Spch(\mathsf{K}_i)$. Associated to this element we define the following subsets of $\mathcal{A}(\sfK)$:
    \begin{align*}
    \mathcal{B} & \coloneqq \{ b \in \mathcal{A}^{\mathrm{fp}}(\sfK) \mid \exists i \in I,  b_i \in \mathcal{B}_i: b \cong F_i(b_i) \} , \\
\mathcal{B}' & \coloneqq \{ b \in \mathcal{A}^{\mathrm{fp}}(\sfK) \mid \forall i \in I , b_i \in \mathcal{A}(\sfK_i) : b \cong F_i(b_i) \Rightarrow b_i \in \mathcal{B}_i \}.
\end{align*}

We begin by showing that the sets $\mathcal{B}$ and $\mathcal{B}'$ are equal. If $b \in \mathcal{B}'$, we may choose $i \in I$ and $b_i \in \mathcal{A}^{\mathrm{fp}}(\sfK_i)$ such that $b \cong F_i(b_i)$. By construction of $\mathcal{B}'$ we have that $b_i \in \mathcal{B}_i$ and therefore $b \in \mathcal{B}$. Conversely, let $b \in \mathcal{B}$, so $b \cong F_i(b_i)$ with $b_i \in \mathcal{B}_i$ for some $i \in I$. Suppose that we are given $b'_j \in \mathcal{A}^{\mathrm{fp}}(\sfK_j)$ such that $b \cong F_j(b_j')$. As $I$ is filtered we can find $k \in I$ along with maps $u_i \colon i \to k$ and $u_j \colon j \to k$ in $I$. We then have $F_k F_{ik} (b_i) \cong F_i(b_i) \cong b \cong F_j (b_j') \cong F_k F_{jk} (b_j')$. By the description of objects in a filtered colimit, we can find some $u \colon k \to l$ in $I$ such that $F_{il}(b_i) \cong F_{kl}F_{ik} (b_i) \cong F_{kl}F_{jk} (b_j') \cong F_{jl}(b_j')$. Since $F_{il}(b_i) \in F_l^{-1}(b) \in  \mathcal{B}_l$, it follows that $F_{jl}(b_j') \in \mathcal{B}_l$ as well. This implies that $b_j' \in \mathcal{B}_j$ as required.

The observation that $\mathcal{B} =\mathcal{B}'$ is useful as it allows us to conclude that $F_i^{-1}(\mathcal{B})$ is $\mathcal{B}_i$ for all $i \in I$. Indeed, if we write $\pi_i \colon \lim_i \Spch(\sfK_i) \to \Spch(\sfK_i)$ for the canonical projection then we have $\pi_i \tau = F_i^{-1}$ so that
\begin{equation}\label{eqn:pullbackofB}
    \pi_i \tau(\mathcal{B}) = F_i^{-1}(\mathcal{B}) = F_i^{-1}(\mathcal{B}') = \mathcal{B}_i.
\end{equation}

We now wish to show that $\mathcal{B}$ is a maximal Serre ideal. That is, we have to show that:
\begin{enumerate}[label=(\arabic*)]
    \item\label{item:serre} $\mathcal{B}$ is a Serre subcategory;
    \item\label{item:ideal} $\mathcal{B}$ is an ideal;
    \item\label{item:maximal} $\mathcal{B}$ is maximal.
\end{enumerate}

For \cref{item:serre}, let $E \colon a \to b \to c$ be an exact sequence $\mathcal{A}^{\mathrm{fp}}(\sfK)$ with $a, c$ in $\mathcal{B}$. By assumption there exists some $i \in I$ and an exact sequence $E_i \colon a_i \to b_i \to c_i$ in $\mathcal{A}^{\mathrm{fp}}(\sfK_i)$ such that $F_i (E_i) \cong E$. Therefore, we must have $a_i, c_i \in \mathcal{B}_i$, and hence $b_i \in \mathcal{B}_i$. Then we have $b \cong F_i(b_i) \in \mathcal{B}$ as required. A similar argument of pulling back to a finite stage $\mathcal{A}^{\mathrm{fp}}(\sfK_i)$ provides a proof that $\mathcal{B}$ is an ideal, proving \cref{item:ideal}

Finally, for \cref{item:maximal} we must show that $\mathcal{B}$ is a \emph{maximal} Serre ideal. Pick $x \in \mathcal{A}^{\mathrm{fp}}(\sfK) \setminus \mathcal{B}$ and define $\mathcal{C} = \langle \mathcal{B}, x \rangle$ to be the smallest Serre ideal containing both $\mathcal{B}$ and $x$. There exists some $i \in I$ and $x_i \in \mathcal{A}^{\mathrm{fp}}(\sfK_i)$ with $F_i(x_i) \cong x$. We may then consider the Serre ideal $\mathcal{C}_i = \langle \mathcal{B}_i, x_i\rangle \subseteq \mathcal{A}^{\mathrm{fp}}(\sfK_i)$. As $\mathcal{B}_i$ is maximal and by construction $x_i \not\in \mathcal{B}_i$, we have that $\mathcal{C}_i  = \mathcal{A}^{\mathrm{fp}}(\sfK_i)$. Observe that $\mathcal{A}^{\mathrm{fp}}(\sfK_i) = \langle \mathcal{B}_i, x_i\rangle = \langle  F_i^{-1}(\mathcal{B}),x_i\rangle \subseteq F_i^{-1}(\langle \mathcal{B}, x \rangle)$, where the second equality follows from \eqref{eqn:pullbackofB}. Therefore, we have $F^{-1}_i(\mathcal{C}) = \mathcal{A}^{\mathrm{fp}}(\sfK_i)$. A similar cofinality argument as before then shows that $\mathcal{C} = \mathcal{A}^{\mathrm{fp}}(\sfK)$, hence $\mathcal{B}$ is maximal as required.

We have now concluded the proof that $\tau$ is a bijection. By \cref{lem:spchcont} we know that $\tau$ is continuous, so it now suffices to prove that it is an open map. Recall that a basis for the topology of $\Spch(\colim_i \sfK_i)$ is given by the complements of the homological support: $U(x) \coloneqq \Spch(\colim_i \sfK_i) \setminus \supp^{\mathrm{h}}(x)$ where $x$ runs through the objects of $\colim_i \sfK_i$. 

Fix such an $x$, say $x \cong F_i(x_i) $ for some $x_i \in \sfK_i$. We need show that $\tau(U(x)) = \pi^{-1}_i(U(x_i))$. Let $\mathcal{B} \in U(x)$, which implies that $F_i(x_i) \cong x \in \mathcal{B}$, equivalently that $x_i \in F_i^{-1}(\mathcal{B}) = \pi_i \tau (\mathcal{B})$. That is, $\tau (\mathcal{B}) \in \pi^{-1}_i(U(x_i))$. In the other direction, suppose that $(\mathcal{B}_i)_i \in \pi_i^{-1}(U(x_i))$, that is, $x_i \in \mathcal{B}_i$. Then by the proof of surjectivity of $\tau$ we have $(\mathcal{B}_i)_i = \tau(\mathcal{B})$ with $x \in \mathcal{B}$. That is, $(\mathcal{B}_i)_i \in \tau(U(x))$ as required.
\end{proof}

The following corollary is then a consequence of \cref{thm:spchcontinuity} combined with \cref{prop:spccontinuity}.

\begin{corollary}\label{cor:bijhyp_continuity}
    Let $(\sfK_i, f_{ij} \colon \sfK_i \to \sfK_j)$ be a filtered diagram of essentially small tt-categories. If each $\sfK_i$ satisfies the bijectivity hypothesis then so does $\colim_i \sfK_i$.
\end{corollary}

\section{The geometry of tensor-triangular stalks}\label{sec:ttstalks}

In \cref{def:ttstructuresheaf} we introduced the structure presheaf of a rigidly-compactly generated tt-category. In this section we will take this point of view in earnest and discuss the theory of stalks of tt-categories. While we are afforded the theory of stalks via this sheaf point of view, there is also a natural construction of a ``stalk'' of a tt-category at a prime tt-ideal as the quotient $\sfT / \langle \sfP \rangle$. A key result of this section is that these two notions of stalks coincide. The remainder of the section is then dedicated to studying the behaviour of stalks under colimit constructions.

\subsection{Basic properties of tt-stalks}\label{subsec:basicttstalk}

We begin by introducing the notion of a tt-stalk at a prime tt-ideal before comparing it to the sheaf theoretic definition while also providing some guiding examples along the way.

\begin{definition}\label{def:ttstalks}
Let $\sfT$ be a rigidly-compactly generated tt-category. The \emph{tt-stalk} of $\sfT$ at $\sfP \in \Spc(\sfT^{\omega})$ is defined as the quotient $\sfT_{\sfP} = \sfT/\langle \sfP \rangle$. The \emph{stalk} of an object $t \in \sfT$ at a point $\sfP$ is defined as the image of $t$ under the corresponding finite localization $t_{\sfP}\in \sfT_{\sfP}$.
\end{definition}

Recall that $\gen(\sfP)$ denotes the generalization of the element $\sfP \in \Spc(\sfT^{\omega})$. As explained in \cite[Definition 1.25]{BarthelHeardSanders2023}, the stalk of $\sfT$ can be identified in terms of local sections of the tt-presheaf, as follows:
\begin{equation}\label{eq:ttstalklocalsection}
    \sfT_{\sfP} = \sfT/\langle \sfP \rangle \simeq \sfT(\gen(\sfP)).
\end{equation}
In other words, passage to stalks $(-)_{\sfP}\colon\sfT \to \sfT_{\sfP}$ coincides with the finite localization functor away from the subset $\gen(\sfP)^c \subset X=\Spc(\sfT^{\omega})$. This subsets admits a useful description as 
\begin{equation}\label{eq:suppP}
    \supp(\sfP) = \gen(\sfP)^c  = \textstyle\bigcup_{\sfP \in U \in \qcopen{X}} U^c,
\end{equation}
where the union is over all quasi-compact opens in $X$ which contain $\sfP$. Here, the second one can be seen on complements: Indeed, in any spectral space $X$, we have $\gen(x) = \bigcap_{x\in U \in \qcopen{X}} U$ for $x \in X$, see \cref{rem:gen}. 

\begin{remark}\label{rem:stalkambiguity}
    In the literature, there are competing notions of stalks for arbitrary tt-categories, referring for instance to $\Gamma_{\sfP}\sfT$ which in general behaves more like a formal neighborhood around $\sfP$. To make this point explicit, let us include an example illustrating the phenomena that arise:
\end{remark}

\begin{example}
Let $\sfT = \sfD(\Z)$ be the derived category of abelian groups, a rigidly-compactly generated tt-category with $\Spc(\sfT^{\omega}) \cong \Spec(\Z)$. Under this identification, the stalk of $\sfT$ at $p$ is then given by $\sfT_p \simeq \sfD(\Z_{(p)})$, whereas $\Gamma_{p}\sfT$ is equivalent to the derived category of $p$-power torsion abelian groups. 
\end{example}

Stalks are natural in geometric functors: let $f=f^*\colon \sfS \to \sfT$ be a geometric functor and write $\varphi\colon \Spc(\sfT^{\omega}) \to \Spc(\sfS^{\omega})$ for the induced map on spectra. For any $\sfP \in \Spc(\sfT^{\omega})$, there is a commutative square
\[
\xymatrix{\sfS \ar[r]^-{f} \ar[d]_{\lambda_{\varphi(\sfP)}} & \sfT \ar[d]^{\lambda_{\gen(\sfP)}} \\ 
\sfS/\langle \varphi(\sfP) \rangle \ar@{-->}[r] & \sfT/\langle \sfP \rangle.}
\]
Indeed, we have $f(\langle \varphi(\sfP) \rangle) = f(\langle f^{-1}(\sfP) \rangle) \subseteq \langle \sfP \rangle$, so the composite $\lambda_{\gen(\sfP)} \circ f$ factors canonically (and uniquely) through the localization $\sfS/\langle \varphi(\sfP) \rangle$.

\begin{definition}\label{def:ttstalkmap}
With notation as above, $f\colon \sfS \to \sfT$ induces a map on stalks defined as
\[
f_{\sfP}\colon \sfS_{\varphi(\sfP)} = \sfS/\langle \varphi(\sfP) \rangle \longrightarrow \sfT/\langle \sfP \rangle = \sfT_{\sfP}.
\]
\end{definition}
Note that, by construction, $f_{\sfP}\colon \sfS_{\varphi(\sfP)} \to \sfT_{\sfP}$ is a geometric functor.

\begin{remark}\label{rem:smallttstalks}
    The construction of the structure presheaf in \cref{def:ttstructuresheaf} as well as of stalks and morphisms between them transposes without effort to the setting of essentially small tt-categories, keeping in mind appropriate idempotent completion whenever necessary. The latter is the setting chosen in Balmer's writings \cite{balmer_spectrum,balmer_3spectra,balmer_icm}, and we refer to these papers for comparison.
\end{remark}

\begin{proposition}\label{prop:localsectionsformula}
Let $\sfT$ be a rigidly-compactly generated tt-category and write $X = \Spc(\sfT^{\omega})$ for its spectrum. For any Thomason subset $V^c \subseteq X$ with complement $V$, the category of local sections of the tt-presheaf of $\sfT$ over $V$ are given by
    \[
        \xymatrix{\underset{V \subseteq U \in \qcopen{X}}{\colim^{\omega}}\sfT(U) \ar[r]^-{\simeq} & \sfT(V),}
    \]
where the colimit is taken over the restriction functors \eqref{eq:ttrestriction} and is indexed on quasi-compact open subsets of $X$ containing $V$.
\end{proposition}
\begin{proof}
    For simplicity, we write the indexing diagram as $U \supseteq V$, implicitly assuming that the $U$s range through quasi-compact opens of $X$. We first claim that there are equalities of full subcategories in $\sfT$:
        \begin{equation}\label{eq:thickP}
            \{t \in \sfT^{\omega} \mid \supp(t) \subseteq V^c\} = \bigcup_{U \supseteq V}\{t \in \sfT^{\omega} \mid \supp(t) \subseteq U^c\} = \underset{U \supseteq V}{\colim}\{t \in \sfT^{\omega} \mid \supp(t) \subseteq U^c\}.
        \end{equation}
    The second equality is an instance of \cref{lem:unionascolimit}, so it remains to prove the first one. Since each of the subcategories $\{t \in \sfT^{\omega} \mid \supp(t) \subseteq U^c\}$ is a thick  ideal, namely the preimage of $U^c$ under the function $\supp$ from \eqref{eq:suppttclassification}, so is the union in \eqref{eq:thickP}. It follows that we can prove the desired equality by showing both sides have the same support. The Thomason subset $V^c$ is a union of subsets with quasi-compact open complement, $V^c = \bigcup_{U \supseteq V}U^c$, cf.~\cref{rem:gen}. By the basic properties of the support function, the relevant supports are then given by
    \[
    \supp\{t \in \sfT^{\omega} \mid \supp(t) \subseteq V^c\} = V^c = \bigcup_{U \supseteq V}U^c = \supp \bigcup_{U \supseteq V}\{t \in \sfT^{\omega} \mid \supp(t) \subseteq U^c\}.
    \]
    This implies the claim.
    
    Now recall the Verdier sequence \eqref{eqn:finiitelocncompact} which determines $\sfT(U)^{\omega}$. By naturality of finite localizations, these form a filtered system of Verdier sequences indexed on $\{U \supseteq V\}$. It then follows from \cite[Lemma 5.25]{BlumbergGepnerTabuada2013} that the filtered colimit over this system,
    \[
    \xymatrix{\underset{U \supseteq V}{\colim}\{t \in \sfT^{\omega} \mid \supp(t) \subseteq U^c\} \ar[r] & \sfT^{\omega} \ar[r] & \underset{U \supseteq V}{\colim}\sfT(U)^{\omega}}
    \]
    is again a Verdier sequence. Upon passage to ind-categories, we thus obtain a commutative diagram 
    \[
    \xymatrix{\langle\{t\in \sfT^{\omega}\mid \supp(t) \subseteq V^c\}\rangle \ar[r] & \sfT \ar[r] & \sfT(V) \\
    \underset{U \supseteq V}{\colim^{\omega}}\langle\{t \in \sfT^{\omega} \mid \supp(t) \subseteq U^c\}\rangle \ar[r] \ar[u]^{\simeq} & \sfT \ar[r] \ar@{=}[u] & \underset{U \supseteq V}{\colim^{\omega}}\sfT(U) \ar@{-->}[u]_{\simeq}}
    \]
    in which the two rows are localization sequences. The left vertical functor is a geometric equivalence by \eqref{eq:thickP}, hence so is the right vertical functor. 
\end{proof}

\begin{corollary}\label{cor:ttstalkformula}
    With notation as in \cref{prop:localsectionsformula}, for any $\sfP \in X$, the tt-stalk of $\sfT$ at $\sfP$ is given by
    \begin{equation}\label{eq:ttstalkformula}
        \xymatrix{\underset{\sfP \in U \in \qcopen{X}}{\colim^{\omega}}\sfT(U) \ar[r]^-{\simeq} & \sfT_{\sfP}.}
    \end{equation}
\end{corollary}
\begin{proof}
     Recall from \eqref{eq:ttstalklocalsection} that $\sfT_{\sfP} = \sfT(\gen(\sfP))$. The description \eqref{eq:ttstalkformula} of the tt-stalk at $\sfP$ then follows from the special case $V = \gen(\sfP)$ of \cref{prop:localsectionsformula}, by observing that any quasi-compact open $U$ in $X$ is closed under generalization, so we have $\gen(\sfP) \subseteq U$ if and only if $\sfP \in U$.
\end{proof}

\begin{remark}\label{rem:ttstalkgeometric}
    Undoubtedly, this result has been known to some tt-experts at least in spirit. In fact, Balmer verifies in \cite[Lemma 6.3]{balmer_3spectra} that the graded homomorphism groups in $\sfT_{\sfP}$ have the desired colimit description and then states the triangulated version of \cref{cor:ttstalkformula} in \cite[Theorem 25(b)]{balmer_icm}. In any case, this equivalence vindicates the choice of referring to $\sfT_{\sfP}$ as the tt-geometric stalk of $\sfT$ at $\sfP$, as it categorifies the usual definition of stalk in algebraic geometry. 
\end{remark}

\begin{corollary}\label{cor:objectwisestalk}
    With notation as in \cref{cor:ttstalkformula}, the stalk of any object $t \in \sfT$ at $\sfP$ is 
    \[
    \underset{\sfP \in U \in \qcopen{X}}{\colim} t \otimes f_{U^c} \simeq t_{\sfP}.
    \]
\end{corollary}
\begin{proof}
    One could formally deduce this result from \cref{cor:ttstalkformula}, but we include a short direct argument. By definition and keeping in mind \eqref{eq:ttstalklocalsection}, the stalk of $t$ at $\sfP$ is $t_P = t \otimes f_{\gen(\sfP)^c}$. Therefore, it is enough to prove that the canonical map $\colim_{U \ni \sfP}f_{U^c} \to f_{\gen(\sfP)^c}$ is an equivalence. Equivalently, we can show that the canonical map of the corresponding left idempotents
    \[
    \xymatrix{\colim_{U \ni \sfP}e_{U^c} \ar[r] & e_{\gen(\sfP)^c}}
    \]
    is an equivalence. This reduces again to checking both sides have equal support, and we conclude by invoking \eqref{eq:suppP}. 
\end{proof}

\subsection{Continuity of finite localizations and stalks}\label{ssec:finitelocalizations_continuity}

For the remainder of this section, let $(\sfT_i,f_{ij} \colon \sfT_i \to \sfT_j)_{i \in I}$ be a filtered system of rigidly-compactly generated tt-categories and geometric functors. We denote the rigidly-compactly generated colimit of this system by $\sfT$ and write $f_i\colon \sfT \to \sfT_i$ for the canonical geometric functor, with induced map on spectra $\varphi_{i}\colon \Spc(\sfT^{\omega}) \to \Spc(\sfT_i^{\omega})$.

\begin{proposition}\label{prop:finitelocalizations_continuity}
    Suppose $\sfC \subseteq \sfT^{\omega}$ is a thick ideal and let $\sfC_i = f_i^{-1}(\sfC) \subseteq \sfT_i^{\omega}$ be the preimage in $\sfT_i^{\omega}$. Then there is a canonical geometric equivalence
        \[
            \xymatrix{\colim_{i}^{\omega}\sfT_i/\langle\sfC_i\rangle \ar[r]^-{\sim} & \sfT/\langle\sfC\rangle,}
        \]
    where the colimit is taken by the geometric functors induced by $f_{ij}$ on finite localizations. 
\end{proposition}
\begin{proof}
    As in the proof of \cref{prop:localsectionsformula}, it suffices to show 
        \begin{equation}\label{eq:continuousP}
             \colim_{i}^{\omega}\langle \sfC_i\rangle = \bigcup_{i}\langle f_i\sfC_i \rangle =  \langle \sfC \rangle ,
        \end{equation}
    because this is equivalent to the equivalence of the corresponding local categories by passing to Verdier quotients. For the first equality, we employ \cref{prop:continuity_of_ideals}, so it remains to verify the second equality. Working now intrinsic to $\sfT$, it suffices to check the following support-theoretic identity:
        \[
            \bigcup_{i}\supp(f_i\sfC_i) = \supp(\sfC).
        \]
    For any $i\in I$, we have $f_if_i^{-1}(\sfC) \subseteq \sfC$, so $\bigcup_i\supp(f_i\sfC_i) \subseteq \supp(\sfC)$. For the reverse inclusion, consider some $x \in \sfC$. By construction of $\sfT^{\omega}$ as a filtered colimit, there exists an $i \in I$ and some $x_i \in \sfT_i^{\omega}$ such that $f_i(x_i) \simeq x$, i.e., $x_i \in \sfC_i$. Therefore, $x \in f_i\sfC_i$ and thus $\supp(x) \subseteq \bigcup_i\supp(f_i\sfC_i)$. Varying over all $x \in \sfC$ then establishes the claim.
\end{proof}

As an application, we deduce the continuity of tt-stalks, i.e., their compatibility with passage to filtered colimits.

\begin{corollary}\label{cor:continuousttstalk}
    Let $\sfP \in \Spc(\sfT^{\omega})$ and write $\sfP_i = \varphi_i(\sfP)$. Then there is a canonical geometric equivalence
    \[
    \xymatrix{\colim_{i}^{\omega}(\sfT_i)_{\sfP_i} \ar[r]^-{\sim} & \sfT_{\sfP},}
    \]
    where the transition maps in the colimit are the induced maps on stalks as in \cref{def:ttstalkmap}.
\end{corollary}
\begin{proof}
    Indeed, this is the special case of \cref{prop:finitelocalizations_continuity} when $\sfC = \sfP$ is a prime tt-ideal, using the description of tt-stalks given in \eqref{eq:ttstalklocalsection}. 
\end{proof}

\begin{remark}\label{lem:gencontinuity}
    Let $(X_i)_{i \in I}$ be a filtered inverse system of spectral spaces with limit $X$ and write $\phi_i\colon X \to X_i$ for the canonical maps. The point-set topological content of the proof of \cref{prop:finitelocalizations_continuity} is the fact that a basis of opens for $X$ is given by the collection
        \[
            \mathcal{B}  = \{\phi_i^{-1}(B_i) \mid i \in I \text{ and } B_i \in \mathcal{B}_i\},
        \]
    where $\mathcal{B}_i$ is the set of quasi-compact opens in $X_i$. 
    
    Translated to tt-geometry and passing to the inverse topology, any Thomason subset $Z \subseteq \Spc(\sfT^c)$ can be written as $Z = \bigcup_{\alpha}Z_{\alpha}$ for $Z_{\alpha} = \varphi_{\alpha}^{-1}(C_{\alpha})$ for some $\alpha \in I$ and $C_{\alpha} \subseteq \Spc(\sfT_{\alpha}^{\omega})$ Thomason closed. For the support of the tt-stalk of $\sfT$ at $\sfP$, there is then the explicit formula
    \[
    \gen(x) = \bigcap_{i \in I}\bigcap_{U_i \ni \sfP_i}\varphi_i^{-1}U_i,
    \]
    where the inner intersection is indexed on quasi-compact opens in $\Spc(\sfT_i^{\omega})$ containing $\sfP_i$.
\end{remark}

\section{Zero-dimensional tensor-triangular geometry}\label{sec:zerodimtt}

While we have seen that the Balmer spectrum of an arbitrary essentially small tt-category is a spectral space, in one of our applications of interest, that is, rational $G$-equivariant spectra for $G$ a profinite group we will see that the Balmer spectrum is a particularly nice topological space. It is a spectral space which is  zero-dimensional. We recall that the \emph{(Krull) dimension} of a topological space $X$ is defined to be the supremum of the lengths of chains of irreducible closed subsets of $X$.

In this section we will explore some of the theory that is afforded to us in the case when the Balmer spectrum of $\sfT^\omega$ is zero-dimensional, where $\sfT$ is rigidly-compactly generated.

\subsection{Profinite spaces}\label{ssec:profinitespaces}

We begin with a quick reminder on profinite spaces. Recall that those spectral spaces which are moreover zero-dimensional are necessarily Hausdorff. These spaces coincide with the class of \emph{profinite spaces}:

\begin{definition}\label{def:profinite}
A topological space $X$ is said to be \emph{profinite} if it is Hausdorff, quasi-compact, and totally disconnected. 
\end{definition}

The collection of profinite spaces and continuous maps between them assembles into a category $\Top_{\mathrm{Pro}}$. We will record here that limits of profinite spaces are computed in the underlying category $\Top$ \cite[Corollary 2.3.8]{book_spectralspaces}.

One way to characterize profinite spaces among all spectral spaces is as fixed points of an involution on the category of spectral spaces, given by the inverse topology: 

\begin{definition}\label{def:inversetopology}
For any spectral space $X$ there is another topology on the underlying set of $X$, the \emph{inverse topology}, whose open sets are given by the Thomason subsets of $X$. We denote this space $X_{\inv}$. This construction lifts to an involution $\inv \colon \Spectral \to \Spectral$. 
\end{definition}

\begin{lemma}\label{lem:charofpro}
Let $X$ be a spectral space. Then the following are equivalent:
	\begin{enumerate}
        \item $X$ is zero-dimensional;
	\item $X$ is Hausdorff;
	\item $X$ is profinite;
	\item $X = X_{\inv}$.
	\end{enumerate}
\end{lemma}

We will also require a further characterization of the category of profinite spaces, justifying the terminology.

\begin{lemma}\label{lem:proflim}
A topological space $X$ is profinite if and only if it is is a limit of a pro-object in the category of finite sets. That is, it is homeomorphic to a limit of a cofiltered system of finite discrete spaces.
\end{lemma}

In \cref{sec:ttgeom}, we have seen the importance of pointset topological properties of the Balmer spectrum when constructing $\otimes$-idempotents. The following results will explore the richness of the pointset topology of profinite spaces for tt-geometry.

\begin{remark}\label{rem:openisThom}
    Let $X$ be a profinite space. Then by \cref{lem:charofpro}(4) we conclude that a subset $S \subseteq X$ is Thomason if and only if it is open.
\end{remark}

\begin{remark}\label{rem:complementisthom}
Let $X$ be a profinite space. Then for any $x \in X$, we have that $\{x\}^\mathrm{c} := X \setminus \{x\}$ is open, and therefore Thomason by \cref{rem:openisThom}. In particular every $x \in X$ is the complement of a Thomason subset. We will make repeated use of this fact throughout the remainder of this paper.

More generally, if $X$ is an arbitrary spectral space, then $Y = \{x\}^c$ is Thomason if and only if $\gen(x) = \{x\}$. Indeed, $Y$ is Thomason if and only if it is inversely open. Unwrapping definitions, this asks that $Y$ is open in $X_{\cons}$ and specialization closed in $X$. The first condition is always satisfied because $X_{\cons}$ is Hausdorff so that $\{x\} \subseteq X_{\cons}$ is closed, and as such we conclude that $Y$ is Thomason if and only if it is specialization closed, if and only if $\{x\}$ is generically closed.
\end{remark}

\begin{lemma}\label{lem:weaklynoeth}
Let $X$ be a profinite space. Then $X$ is generically Noetherian, and as such weakly Noetherian. In particular, for every $x \in X$ we isolate $\{x\} = X \cap \{x\}$.
\end{lemma}

\begin{proof}
Let $x \in X$. Then we have that $\mathrm{gen}(x) = \{x\}$ is a Noetherian subspace. The claim regarding $X$ being weakly Noetherian then follows from \cref{lem:genimpweak} with the explicit description of $\{x\}$ as a weakly visible subset being facilitated by \cref{rem:complementisthom}.
\end{proof}

\begin{remark}
A profinite space is Noetherian if and only if it is a finite discrete space.
\end{remark}

The final ingredient that we will need from pointset topology for our study is the \emph{Cantor--Bendixson} filtration on a profinite space. In \cref{sec:qgsp_lgp} this will play a key role in resolving stratification for our categories of interest. 

\begin{definition}\label{defn:cbrank}
Let $X$ be a topological space and write $\delta X$ for the set of non-isolated points of $X$. The set $\delta X$ is called the \emph{Cantor--Bendixson derivative} of $X$. For $\lambda \in \Ord$ we define $\delta^{\lambda} X$ as a subset of $X$ via transfinite recursion on $\lambda$:
    \begin{itemize}
        \item $\delta^0 X = X$;
        \item $\delta^{\lambda + 1} X = \delta \delta^{\lambda} X$;
        \item If $\lambda$ is a limit ordinal then $\delta^{\lambda} X = \cap_{\nu < \lambda} \delta^{\nu} X$.
    \end{itemize}
Moreover, we set $\delta^{\infty} X = \bigcap_{\lambda \in \Ord}\delta^{\lambda} X$. For a point $x \in X$ we then define the \emph{Cantor--Bendixson} rank of $x$ as follows:
\[
\cbrank_X(x) = 
\begin{cases}
    \infty & \text{if } x \in \delta^{\infty}X; \\
    \sup\{\lambda \in \Ord \mid x \in \delta^{\lambda}X\} & \text{otherwise}.
\end{cases}
\]
If the ambient space is clear from context, we also write $\cbrank(x)$ for $\cbrank_X(x)$. Finally, the \emph{Cantor--Bendixson rank} of a non-empty topological space $X$, denoted $\cbrank(X)$, is the supremum of all $\lambda \in \Ord$ with $\delta^{\lambda}X \neq \varnothing$ or $\infty$ if the supremum does not exist.
\end{definition}

\begin{remark}\label{rem:BDscattered}
We note that, by construction, having Cantor--Bendixson rank is equivalent to the space $X$ being \emph{scattered}. That is, every non-empty closed subset $S \subseteq X$ contains an isolated point with respect to the subspace topology. In general, the Cantor--Bendixson process terminates in the \emph{perfect hull} of $X$, which is a perfect space. We remind the reader that a space $X$ is said to be \emph{perfect} if it contains no isolated singletons.
\end{remark}

We conclude with an observation regarding when certain compact spaces are scattered by considering the size of their minimal basis. For $X$ a topological space, write $w(X)$ for its weight (i.e., the cardinality of the minimal sized basis of $X$). A space $X$ is \emph{countably based} if $w(X)$ is countable. 

\begin{proposition}[{\cite[Theorem 8.2.3]{book_spectralspaces}}]\label{prop:countbased}
    Let $X$ be an infinite compact space. Then the following are equivalent:
    \begin{enumerate}
        \item $X$ is countable.
        \item $X$ is scattered and is countably based.
    \end{enumerate}
\end{proposition}

This becomes particularly useful in conjunction with \cref{prop:scatteredimpl2g} below.

\subsection{Zero-dimensional tensor-triangulated categories}\label{ssec:zerodimttcats}

Now that we have an understanding of those spectral spaces which are zero-dimensional, in this subsection we will begin our exploration of those tt-categories whose Balmer spectrum is zero-dimensional. While covering the case of rational $G$-spectra for $G$ a profinite group, there are many other naturally occurring examples of interest. First, let us define the concept of \emph{dimension} for a tt-category before specializing to the zero-dimensional case.

\begin{definition}\label{def:ttdimension}
    We define the \emph{(Krull--Balmer) dimension} $\dim(\sfK)$ of an (essentially small) tt-category $\sfK$ as the dimension of its spectrum $\Spc(\sfK)$. If $\sfT$ is a big tt-category, we set $\dim(\sfT) = \dim(\sfT^{\omega})$, and say that $\sfT$ is $\dim(\sfT)$-dimensional.
\end{definition}

Our focus in this paper is on zero-dimensional rigidly-compactly generated tt-categories $\sfT$. We collect a couple of examples.

\begin{example}\label{ex:rings}
    Let $R$ be a commutative ring and write $\sfD(R)$ for its derived category. Thomason proved in \cite{thomasonclassification} that $\Spc(\sfD(R)^{\omega})$ is homeomorphic to $\Spec(R)$; in fact, this homeomorphism is realized by Balmer's comparison map \cite{balmer_3spectra}. Consequently, we obtain an equality
    \[
    \dim(R) = \dim(\sfD(R)).
    \]
    In the zero-dimensional case, profinite spaces can be realized naturally as the spectrum of a commutative ring and hence of a tt-category: If $X$ is a profinite space, then the ring of locally constant functions $C(X,\Q)$ has $\Spec (C(X,\Q)) \cong X$. In general, Hochster \cite{hochster} proved that any spectral space appears as the spectrum of a commutative ring, but this cannot be made functorial (see \cite[Proposition 3]{hochster} or \cite[12.6.2]{book_spectralspaces}).
\end{example}

\begin{example}\label{ex:sheaves}
    If $X$ is any spectral space, then the derived category $\sfD(\Sh_{\Q}(X))$ of sheaves of rational vector spaces on $X$ is a compactly generated tt-category under the pointwise tensor-product and has spectrum of compacts\footnote{As a caveat, note that we do not claim that $\sfD(\Sh_{\Q}(X))$ is rigidly-compactly generated; however, the full subcategory of compact objects does form a tt-category.} homeomorphic to the constructible topology on $X$, see \cite{Aoki2023}. In other words, $\sfD(\Sh_{\Q}(X))$ is zero-dimensional.
\end{example}

For a zero-dimensional example from homotopy theory and the main tt-category of interest in this paper, see \cref{sec:qgsp_lgp}. In passing, we mention the following useful observation:

\begin{remark}\label{lem:zerodim_basicclosed}
Let $\sfT$ be a zero-dimensional tt-category. For a subset $C \subseteq \Spc(\sfT^{\omega})$, the following statements are equivalent:
    \begin{enumerate}
        \item $C$ is quasi-compact open;
        \item $C$ is closed and open;
        \item $C = \supp(x)$ for some $x \in \sfT^{\omega}$.
    \end{enumerate}
Indeed, the subsets of $\Spc(\sfT^\omega)$ of the form $\supp(x)$ for $x$ compact are precisely the \emph{Thomason closed} subsets, i.e., those that have quasi-compact open complement; see for example \cite[Remark 1.14]{BarthelHeardSanders2023}. The result then follows from \cite[Corollary 1.3.15]{book_spectralspaces}.
\end{remark}

The zero-dimensional assumption on the spectrum allows for a tight control over the local tt-geometry of $\sfT$, which we explore here and then further in the next subsection. We say that a point $x$ in a spectral space $X$ is \emph{generically closed} if $\gen(x) = \{x\}$, i.e., if it is minimal with respect to the specialization order. 

\begin{proposition}\label{prop:stalkcostalk}
Suppose $\sfT$ is a rigidly-compactly generated tt-category and let $\sfP \in \Spc(\sfT^{\omega})$ be a generically closed point. The functor $\Gamma_{\sfP}$ is then the finite localization away from $\{\sfP\}^c \subseteq \Spc(\sfT^{\omega})$ and has kernel $\langle\sfP \rangle$. Moreover, there are equalities 
\[
\sfT_{\sfP} = \Gamma_{\sfP}\sfT = \Lambda^{\sfP}\sfT
\]
of localizing ideals in $\sfT$. In particular, if $\sfT$ is zero-dimensional then this holds for all points of the spectrum.
\end{proposition}
\begin{proof}
    Let $X = \Spc(\sfT^{\omega})$. From \cref{rem:complementisthom} and our assumption we have that $\{\sfP\}^c$ is Thomason and as such we can  write $\{\sfP\}$ as $X \cap (\{\sfP\}^c)^c$, an intersection of a Thomason subset and the complement of a Thomason subset. By definition, we thus obtain equivalences
    \begin{equation}\label{eq:stalkidempotents}
        g({\sfP}) \simeq e_{X} \otimes f_{\{\sfP\}^c} \simeq f_{\{\sfP\}^c}.
    \end{equation}
    In particular, this proves that $\Gamma_{\sfP} = g({\sfP}) \otimes -$ identifies with the finite localization away from $\{\sfP\}^c$. This also implies that we may determine the kernel of $\Gamma_{\sfP}$ via support:
    \[
    \Supp (\langle \sfP\rangle ) = \gen(\sfP)^c = \{\sfP\}^c,
    \]
    where the first equality is \eqref{eq:suppP}, while the second equality uses again that $\{\sfP\}$ is generically closed. It follows that the kernel of $\Gamma_{\sfP}$ is $\langle\sfP\rangle$ and that $\sfT_{\sfP} = \Gamma_{\sfP}\sfT$. 
    
    In order to identify the colocalization, first note that \eqref{eq:stalkidempotents} dualizes to an equivalence 
    \[
    \Lambda^{\sfP} = \Hom(g({\sfP}),-) \simeq \hom(f_{\{\sfP\}^c},-).
    \]
    By abstract local duality theory \cite[Theorem 2.21]{BarthelHeardValenzuela2018}, the endofunctors $f_Y\otimes -$ and $\hom(f_Y,-)$ have the same essential image in $\sfT$ for any Thomason $Y \subseteq \Spc(\sfT^{\omega})$, hence $\Gamma_{\sfP}\sfT = \Lambda^{\sfP}\sfT$.
\end{proof}

\begin{remark}\label{rem:unambiguous}
    This result enables us to talk \emph{unambiguously} about the stalks of objects in $\sfT$ as well as of $\sfT$ itself at least whenever $\sfT$ is zero-dimensional; cf.~the discussion in \cref{rem:stalkambiguity}.
\end{remark}

Once we have some understanding of the local structure of $\sfT$, we turn towards the question if and how $\sfT$ can be assembled from its tt-stalks. As explained in \cref{ssec:ttstratification}, this is governed by the local-to-global principle. We recall the following result of Stevenson \cite{stevenson_localglobal} that we will return to in our example of interest in \cref{sec:qgsp_lgp}.

\begin{proposition}[Stevenson]\label{prop:scatteredimpl2g}
Let $\sfT$ be a zero-dimensional rigidly-compactly generated tt-category. If $\Spc(\sfT^{\omega})$ is scattered in the sense of \cref{rem:BDscattered}, then $\sfT$ satisfies the local-to-global principle.
\end{proposition}

For later use, we conclude this subsection with another consequence of zero-dimensionality:

\begin{remark}\label{rem:telescope}
Any profinite space is generically Noetherian in the sense of \cref{def:genNoeth}. Therefore, by \cref{prop:tc}, the telescope conjecture holds in any stratified zero-dimensional rigidly-compactly generated tt-category. A more general criterion is the subject of the next subsection.
\end{remark}

\subsection{The telescope conjecture}\label{ssecLzerodimtt_telescopeconjecture}

In this subsection, we prove a criterion for when the telescope conjecture (\cref{def:tc}) holds in a zero-dimensional tt-category $\sfT$, based on the homological spectrum and its support function (see \cref{def:homological_support}). We will begin with a lemma on the support of smashing ideals; here, we will write $\sfL^\perp$ for the \emph{right-orthogonal} of a given localizing ideal $\sfL \subseteq \sfT$, i.e., the full subcategory of $\sfT$ containing all objects $t \in \sfT$ with $\Hom(\sfL,t) = 0$.

\begin{lemma}\label{lem:smashingcomplements}
    Let $\sfT$ be a rigidly-compactly generated tt-category such that $\sfT$ satisfies the bijectivity hypothesis (\cref{def:bijhyp}). Then for any smashing ideal $\sfS \subseteq \sfT$ we have $\Supp(\sfS)^c = \Supp(\sfS^\perp)$.
\end{lemma}
\begin{proof}
    Consider the smashing localization $L$ of $\sfT$ with kernel $\sfS$, and write $f=L{\unit}$ and $e = \fib(\unit \to f)$ for the associated left and right idempotents, respectively. We will show that $\Supp(e)$ is the complement of $\Supp(f)$ in $\Spc(\sfT^{\omega})$, which then implies the claim. On the one hand, $e \otimes f = 0$, so 
        \[
        \varnothing = \Supph(e \otimes f) = \Supph(e) \cap \Supph(f),
        \]
    using Balmer's tensor product formula for the homological support (see \eqref{eq:homological_tensor}). Since the comparison map $\phi_{\sfT}$ from the homological to the triangular spectrum is injective by the bijectivity hypothesis for $\sfT$, using \eqref{eq:suppcomparisoninclusion} we deduce that $\varnothing = \Supp(e) \cap \Supp(f)$ as well. On the other hand, the triangle $e \to \unit \to f$ implies that $\Supp(e) \cup \Supp(f)$ contains $\Supp(\unit) = \Spc(\sfT^{\omega})$, so $\Supp(e) \cup \Supp(f) = \Supp(\unit) = \Spc(\sfT^c)$.
\end{proof}

\begin{proposition}\label{prop:zerodim_telescopeconjecture}
    Let $\sfT$ be a zero-dimensional rigidly-compactly generated tt-category such that
        \begin{enumerate}
            \item $\sfT$ satisfies the bijectivity hypothesis (\cref{def:bijhyp}); 
            \item and the homological support has the detection property, i.e., $\Supph(t) = \varnothing$ implies $t= 0$ for any $t \in \sfT$.
        \end{enumerate}
    Then the telescope conjecture holds in $\sfT$.
\end{proposition}
\begin{proof}
    As $\sfT$ satisfies the bijectivity hypothesis the comparison map between the tensor-triangular and homological spectra is a bijection, which we may take as the identity. By the assumption that $\Supph$ has the detection property, we can apply \cite[Corollary 3.12]{BarthelHeardSanders2023a} 
    to conclude that the homological and the tensor-triangular support coincide and that both have the detection property. Therefore, as $\Supph$ satisfies the tensor-product formula by \eqref{eq:homological_tensor}, so does $\Supp$. That is, for all $s,t \in \sfT$ we have $\Supp(s \otimes t) = \Supp(s) \cap \Supp(t)$. 
    
    Now, consider a smashing localization $L$ of $\sfT$ with corresponding triangle of idempotents $e \to \unit \to f=L\unit$. By \cref{lem:smashingcomplements} we know that $\Supp(e)$ is the complement of $\Supp(f)$ in $\Spc(\sfT^{\omega})$. Moreover, \cite[Proposition 2.14]{stevenson_absolutelyflatrings} shows that  $\Supp(f)$ is closed in $\Spc(\sfT^\omega)$, so $\Supp(e) = \Supp(\ker(L))$ is open, hence Thomason in $\Spc(\sfT^\omega)$.

    Consider the finite localization $L^\fin$ corresponding to the Thomason subset $\Supp(e)$ under Balmer's classification theorem. By construction, this functor has the property that $\ker(L^\fin) \subseteq \ker(L)$ and $\Supp(\ker(L^\fin)) = \Supp(\ker(L))$, thus $\Supp(L^\fin{\unit}) = \Supp(f)$ as well. We now wish to show that $L^\fin \simeq L$ which finishes the proof as $L$ was an arbitrary smashing localization.

    By construction there is a natural transformation $L^{\fin} \to L$, whose cofiber evaluated on $\unit$ we denote by $c$. We have $\Supp(c) \subseteq \Supp(f) \cup \Supp(L^\fin{\unit})$. As $\Supp(L^\fin{\unit}) = \Supp(f)$, we conclude that $\Supp(c) \subseteq \Supp(f)$. Now, tensoring with $f$ yields another triangle $f \otimes L^{\fin}{\unit} \to f \otimes f \to f \otimes c$. Since $f$ is idempotent we have $f \otimes f \simeq f$. Moreover, $L\circ L^{\fin} \simeq L$ because $\ker(L^{\fin}) \subseteq \ker(L)$, so the canonical map $L^{\fin}{\unit} \otimes f \to f\otimes f$ is an equivalence. It follows that $c \otimes f \simeq 0$. Taking supports and using the tensor-product formula we see that $\varnothing = \Supp(f) \cap \Supp(c)$. Therefore, $\Supp(c) \subseteq \Supp(f)^c$. However, combining this with the aforementioned observation that $\Supp(c) \subseteq \Supp(f)$, we conclude that $\Supp(c) = \varnothing$. As  support has the detection property, this is equivalent to $c = 0$. That is, $L^{\fin}{\unit} \simeq f$ and hence $L^{\fin} \simeq L$ as required.
\end{proof}

\begin{example}
    If $R$ is a commutative von Neumann ring (see the discussion at the start of \cref{ssec:vNr_ttcats}), then $\sfT = \sfD(R)$ satisfies the conditions of \cref{prop:zerodim_telescopeconjecture} by the results of \cite[Section 4]{stevenson_absolutelyflatrings}. Therefore, the telescope conjecture holds for $\sfD(R)$, as was also proven in \cite{BS_telescope} using different techniques, generalizing \cite[Theorem 4.25]{stevenson_absolutelyflatrings}.
\end{example}

\begin{remark}
    In fact, the same argument as used in the proof of \cref{prop:zerodim_telescopeconjecture} strengthens \cite[Theorem 9.11 and Proposition 9.12]{BarthelHeardSanders2023} for arbitrary rigidly-compactly generated tt-categories $\sfT$, by replacing the stratification assumption there with the condition that $\sfT$ has the tensor-product formula for support. 
\end{remark}

\subsection{Von Neumann regular tt-categories}\label{ssec:vNr_ttcats}

Among the zero-dimensional commutative rings, a special role is played by the reduced ones. These rings are the \emph{von Neumann regular rings} and can be characterized in a variety of ways, see for example \cite[Proposition 3.2]{stevenson_absolutelyflatrings}. The most relevant of these characterizations is that a commutative ring $R$ is von Neumann regular if and only if all of its localizations $R_{\fp}$ at prime tt-ideals $\fp \in \Spec(R)$ are fields.

We introduce and study a categorification of this notion in the setting of tt-categories:

\begin{definition}\label{def:vNr_ttcats}
    A rigidly-compactly generated tt-category $\sfT$ is called \emph{von Neumann regular} if the localizations $\sfT_{\sfP}$ are tt-fields for all $\sfP \in \Spc(\sfT^{\omega})$.
\end{definition}

\begin{remark}\label{rem:vNr_categorification}
Viewing a rigidly-compactly generated tt-category as a categorified commutative ring (aka.~\emph{commutative 2-ring}), \cref{def:vNr_ttcats} is a natural categorification of the concept of von Neumann regularity from commutative algebra. The next example provides further justification for our choice of terminology.
\end{remark}

\begin{example}\label{ex:vNr_categorification}
Let $R$ be a von Neumann regular commutative ring, then its derived category $\sfD(R)$ is von Neumann regular in the sense of \cref{def:vNr_ttcats}. To see this, we identify $\Spc(\sfD(R)^{\omega})$ with the Zariski spectrum $\Spec(R)$. Under this identification, the finite localization $\sfD(R) \to \sfD(R)_{\sfP}$ sends $R$ to the localization $R_{\fp}$ at the prime ideal $\fp$ corresponding to $\sfP \in \Spc(\sfD(R)^{\omega})$, see \cite[Lemma 4.2]{stevenson_absolutelyflatrings}. It follows that $\sfD(R)_{\sfP} \simeq \sfD(R_{\fp})$. Since $R$ is von Neumann regular, $R_{\fp}$ is a field, hence $\sfD(R_{\fp})$ is a tt-field, as claimed.
\end{example}

\begin{lemma}\label{lem:vNr_zerodim}
    If $\sfT$ is a von Neumann regular tt-category, then it is zero-dimensional. 
\end{lemma}
\begin{proof}
    On spectra, the finite localization $\sfT \to \sfT_{\sfP}$ induces the canonical inclusion $\Spc(\sfT_{\sfP}^{\omega}) \cong \gen(\sfP) \subseteq \Spc(\sfT^{\omega})$, see for instance \cite[Definition 1.25]{BarthelHeardSanders2023}. Since $\sfT_{\sfP}$ is a tt-field by assumption, its spectrum is a point (\cite[Proposition 5.15]{BKSruminations}), hence $\gen(\sfP) = \{\sfP\}$. Varying over all $\sfP \in \Spc(\sfT^{\omega})$, this implies that the spectrum of $\sfT$ is zero-dimensional. 
\end{proof}

\begin{remark}\label{rem:lem:vNr_zerodim}
    The converse to \cref{lem:vNr_zerodim} fails. For example, let $k$ be a field, then $\sfD(k[\epsilon]/\epsilon^2)$ is zero-dimensional but not von Neumann regular. 
\end{remark}

\begin{proposition}\label{prop:vNr_stratification}
For a von Neumann regular tt-category $\sfT$, the following conditions are equivalent: 
    \begin{enumerate}
        \item $\sfT$ is stratified;
        \item $\sfT$ is costratified;
        \item $\sfT$ satisfies the local-to-global principle.
    \end{enumerate}
\end{proposition}
\begin{proof}
    By \cref{lem:vNr_zerodim}, $\sfT$ is zero-dimensional, so \cref{prop:stalkcostalk} provides identifications 
    \[
    \sfT_{\sfP} = \Gamma_{\sfP}\sfT = \Lambda^{\sfP}\sfT 
    \]
    for any $\sfP \in \Spc(\sfT^{\omega})$.  Since these subcategories are tt-fields by hypothesis on $\sfT$, they are minimal both as a localizing ideal and as a colocalizing ideal of $\sfT$, see \cite[Theorem 18.4]{BCHS_cosupport}. It follows that $\sfT$ is stratified (resp., costratified) if and only if $\sfT$ satisfies the local-to-global principle (resp., the colocal-to-global principle). But the local-to-global principle is equivalent to the colocal-to-global principle for any rigidly-compactly generated tt-categories by \cite[Theorem 6.4]{BCHS_cosupport}, so we obtain the equivalence of the three conditions of the proposition. 
\end{proof}

\begin{example}\label{ex:absoluteflat}
Continuing \cref{ex:vNr_categorification}, we  deduce from \cref{prop:vNr_stratification} and \cref{prop:scatteredimpl2g}  that $\sfD(R)$ is stratified (equivalently: costratified) for a von Neumann regular ring $R$ if $\Spec(R)$ is scattered. However, a stronger result holds: In \cite{stevenson_absolutelyflatrings, stevenson_localglobal}, Stevenson proves that stratification for $\sfD(R)$ is in fact equivalent to $\Spec(R)$ being scattered, and shows that this is the case if and only if $R$ is Artinian. In summary, the following are equivalent for a von Neumann regular ring $R$:
    \begin{enumerate}
        \item $D(R)$ is stratified;
        \item $D(R)$ is costratified;
        \item $D(R)$ satisfies the local-to-global principle;
        \item $\Spec(R)$ is scattered;
        \item $R$ is semi-Artinian.
    \end{enumerate}
We will return to this example in \cref{sec:qgsp_lgp}.
\end{example}

\newpage
 
\part{Profinite equivariant spectra}\label{part:integral}

We now turn our attention to the main example of interest in this paper, namely the category $\Sp_G$ of $G$-spectra for a profinite group $G$. After a preliminary section (\cref{sec:profinitegroups}) on profinite groups, in \cref{sec:equivariantspectra} we construct our continuous model of $\Sp_G$, compare it to Fausk's model based on orthogonal spectra, and describe various change of group functors in this context. The geometric fixed point functors will then be used in \cref{sec:gsp_tt} to determine the Balmer spectrum of compact $G$-spectra, up to the ambiguity present in the case of finite groups. \cref{part:integral} then culminates in \cref{sec:gsp_nilpotence} in proofs of the bijectivity hypothesis and the nilpotence theorem for $\Sp_G$.

\section{Profinite groups}\label{sec:profinitegroups}

After a quick review of some of the basic notions and constructions for profinite groups, we will study pro-$p$-subnormal groups and verify the continuity of Weyl groups in this section. For a more thorough introduction to the theory, we refer the interested reader to \cite{wilson,ribes_zalesskii}.

\subsection{Basic constructions}

A \emph{profinite group} $G$ is a group object in the category of profinite spaces. Equally, we may describe a profinite group as a limit of an object in the pro-category of finite groups, i.e., as a cofiltered limit of finite groups. 

By a \emph{subgroup} of a profinite group we will always mean a closed subgroup. Special privilege is given to the open normal subgroups, which are of finite index in $G$. Consequently, if $G$ is an abstract profinite group, then one can obtain a presentation of $G$ as a cofiltered limit by considering the system $\lim_i G/U_i$ as $U_i$ runs through the collection of open normal subgroups in $G$. We will often write $G_i := G/U_i$ for this collection of finite quotient groups. If each $G_i$ has some property $\mathcal{C}$ then we shall say that $G$ is a \emph{pro-$\mathcal{C}$-group}. Particular instances of this will be the notions of \emph{pro-$p$-groups} and \emph{pro-cyclic groups}.

The above constructed cofiltered system interacts well with closed subgroups of $G$. Indeed if $H \leqslant G$ is a closed subgroup and $G = \lim_i G_i$, then
\begin{equation}\label{eq:limofsubgroups}
H \cong \lim_i H_i \text{ with } H_i = \pi_i H = H U_i/ U_i \cong H/(H \cap U_i),
\end{equation}
where $\pi_i \colon G \to G_i$ is the projection map to the quotient \cite[Corollary 1.1.8]{ribes_zalesskii}. We write $\Sub(G)$ for the space of closed subgroups of $G$, equipped with the Hausdorff metric topology inherited from $G$.

Write $m(F)$ for the order of a finite group $F$. For $G$, a profinite group, the \emph{order of $G$}, $m(G)$, is the supernatural number defined as
\begin{equation}\label{eq:order}
    m(G) = \mathrm{lcm}(m(G/U)\mid U\leqslant_o G)
\end{equation}
as $U$ runs over the collection of open normal subgroups of $\Gamma$. 

Another useful construction for a given profinite group $G$ and a given closed subgroup $H$ is the \emph{core} of $H$ in $G$ which is the largest normal subgroup of $G$ contained in $H$, i.e.,
    \[
    H_G = \bigcap_{g \in G} H^g.
    \]
One can verify that if $H$ happens to moreover be an open subgroup, then  $H_G$ is an open normal subgroup of $G$ contained in $H$.

During the exploration of our main application we will require the notion of \emph{pro-$p$-ranks} of pro-$p$-groups to discuss the phenomena of \emph{blueshift} (cf., \cref{ssec:gsp_blueshift}). Before we define this, we recall the extension of the notion of Frattini subgroups to profinite groups:

\begin{recollection}\label{rec:frattini}
    The \emph{Frattini subgroup} of a profinite group $G$ is defined as the intersection of all maximal open subgroups,  $\Phi(G) = \bigcap_{U \leqslant_{o,\text{max}} G} U$. It is a normal closed subgroup of $G$ which consists precisely of the nongenerators of $G$, i.e., those elements that can be omitted from every generating set of $G$. 
\end{recollection}

\begin{definition}\label{def:prank}
    Let $P$ be a pro-$p$ group for a given prime $p$. We define the \emph{pro-$p$-rank} of $P$ as
        \[
        \rank_p(P) \coloneqq \log_p(m(P/\Phi(P))) \in \N \cup \{\infty\},
        \]
    the base-$p$ logarithm of the order of the quotient of $P$ by its Frattini subgroup $\Phi(P)$.
\end{definition}

\subsection{Subnormal subgroups}

We dedicate this subsection to exploring when subgroups $H$ of a profinite group $G$ have a specific $p$-subnormality condition when restricted to the finite quotient groups $G_i$. We set the stage with the following notation and definition.

\begin{notation}\label{nota:pnormal}
    Let $G$ be a profinite group and $p$ a prime number. A closed subgroup $N$ in $G$ is said to be \emph{$p$-normal}, in symbols $N \vartriangleleft_p G$, if $N$ is normal in $G$ and $G/N$ is a pro-$p$-group.\footnote{Beware that this choice of notation differs from other sources such as \cite{balmersanders}, where $\vartriangleleft_p$ instead requires the quotient to be of order $p$.}
\end{notation}

\begin{definition}
    Let $G$ be a finite group. A subgroup $H$ is said to be \emph{$p$-subnormal in $G$} if there exists a finite chain
    \begin{equation}\label{eq:subnormalchain}
H = L_k \vartriangleleft_p L_{k-1}\vartriangleleft_p \cdots \vartriangleleft_p L_1 = G
    \end{equation}
    such that $L_{i+1}$ is normal in $L_{i}$ with quotient a finite $p$-group.
\end{definition}

\begin{remark}
    Assume that $H$ is $p$-subnormal in $G$. By refining the chain \eqref{eq:subnormalchain} if necessary, we can assume that each quotient $L_{i+1}/L_i$ is a cyclic $p$-group. Indeed, let $G$ be a group with normal subgroup $N$ such that $G/N = Q$ is a $p$-group of order $p^k$. Then $Q$ contains a normal subgroup $Q'$ of order $p^{k-1}$ and we can form the following commutative diagram:
    \[
    \xymatrix{
    1 \ar[r] & N \ar[r] \ar@{=}[d] & H \ar@{}[dr]|<<<<{\lrcorner} \ar[d] \ar[r] & Q' \ar[r] \ar[d] & 1 \\
    1 \ar[r] & N \ar[r] & G \ar[r] \ar[d] & Q \ar[r] \ar[d] & 1 \\
    && \Z/p \ar@{=}[r] &\Z/p
    }
    \]
    where $H$ is defined to be the pullback. As $Q'$ is normal in $Q$, and using the fact that pullbacks of normal subgroups are once again normal, we see that $H$ is normal in $G$ and moreover that the quotient is $\Z/p$. We can then apply the same procedure to $H$ and iteratively  obtain a chain of length $n$ between $N$ and $G$ where the quotients are all $\Z/p$. 
\end{remark}

\begin{remark}\label{rem:pnormal_intersection}  
    Suppose $N_1,N_2 \vartriangleleft_p G$. Then the diagonal map exhibits $G/(N_1 \cap N_2)$ as a closed subgroup of $G/N_1 \times G/N_2$, hence a pro-$p$-group as well. In other words, $N_1 \cap N_2 \vartriangleleft_p G$, so the family $\{N \vartriangleleft_p G\}$ is filtered, hence also filtered from below in the sense of \cite[Page 24]{ribes_zalesskii}.
\end{remark}

\begin{definition}\label{def:pperfect}
    Let $G$ be  profinite group and fix a prime $p$; then we define
    \[
    \mathrm{O}^p(G) \coloneqq \bigcap_{N \vartriangleleft_p G} N,
    \]
    where the intersection is indexed on all $p$-normal subgroups $N$ of $G$.
\end{definition}

\begin{remark}\label{rem:pperfectquotient}
    Note that the quotient $G/\mathrm{O}^p(G)$ is a cofiltered limit of pro-$p$-groups, hence a pro-$p$-group itself. 
\end{remark}

\begin{lemma}\label{lem:pperfect_continuity}
    Fix a prime $p$. Let $G$ be a profinite group and $G \twoheadrightarrow	Q$ a finite quotient group, and write $\pi\colon \Sub(G) \to \Sub(Q)$ for the induced map on closed subgroups. Then
    \[
    \pi(\mathrm{O}^p(G)) = \mathrm{O}^p(Q).
    \]
    In particular, if $G = \lim_{i}G_i$ is presented as a cofiltered limit of a cofinal system of finite quotient groups, then $\mathrm{O}^p(G) \cong \lim_{i}\mathrm{O}^p(G_i)$.
\end{lemma}
\begin{proof}
    First observe that, if $N \vartriangleleft_p G$, then $\pi(N) \vartriangleleft_p Q$ as well. Conversely, we claim that we can lift any $M \vartriangleleft_p Q$ to a normal subgroup $N$ of $G$ with $p$-group quotient and such that $\pi(N) = M$. Indeed, set $N = \pi^{-1}(M)$. This is a normal subgroup of $G$, so it suffices to show that the quotient is a $p$-group. Consider the following commutative diagram with exact rows:
        \[
        \xymatrix{1 \ar[r]  & N \ar[r] \ar@{->>}[d] & G \ar[r] \ar@{->>}[d] & G/N \ar[r] \ar@{-->>}[d] &  1 \\
        1 \ar[r] & M \ar[r] & Q \ar[r] & Q/M \ar[r] & 1}
        \]
    The induced map on the right is surjective, and we will show by a diagram chase that it is also injective. To this end, let $x \in G/N$ that maps to the trivial element in $Q/M$. Consider a lift $g \in G$. The image $q$ of $g$ in $Q$ maps to the trivial element in $Q/M$, so there exists a lift of $q$ to $M$, call it $m$. Since $N$ is the pullback of the left square, we can find a (unique) element $n \in N$ that maps to $g$, hence witnessing the triviality of $x$. Therefore, $G/N \cong Q/M$ is a (finite) $p$-group and, by construction, $\pi(N) = M$. This verifies the claim. As a consequence, we see that the two families $\{\pi(N)\mid N \vartriangleleft_p G\}$ and $\{M \vartriangleleft_p Q\}$ coincide.

    We can use this to make the following computation:
        \[
        \pi(\mathrm{O}^p(G)) = \pi\left (\bigcap_{N \vartriangleleft_p G} N \right )= \bigcap_{N \vartriangleleft_p G} \pi(N) = \bigcap_{M \vartriangleleft_p Q} M = \mathrm{O}^p(Q).
        \]
    The first and last equality hold by definition, the second one uses \cite[Proposition 2.1.4(b)]{ribes_zalesskii}, which applies by \cref{rem:pnormal_intersection}, and the third one follows from the preliminary observations made above. Consequently, there are isomorphisms
        \[
        \mathrm{O}^p(G) \cong \lim_{i}\pi_i(\mathrm{O}^p(G)) \cong \lim_{i}\mathrm{O}^p(G_i),
        \]
    where $\pi_i \colon \Sub(G) \to \Sub(G_i)$ denotes the map induced by the quotient map $G \to G_i$.
\end{proof}

\begin{proposition}\label{prop:pperfect}
    Let $H$ be a closed subgroup of a profinite group $G$ and let $p$ be a prime number. The following statements are equivalent:
        \begin{enumerate}
            \item $\mathrm{O}^p(H) = \mathrm{O}^p(G)$;
            \item $\mathrm{O}^p(H_i) = \mathrm{O}^p(G_i)$ for all $i \in I$;
            \item $\mathrm{O}^p(G) \leqslant H$;
            \item $\mathrm{O}^p(G_i) \leqslant H_i$ for all $i \in I$;
            \item $H_i$ is a $p$-subnormal subgroup of $G_i$ for all $i \in I$.
        \end{enumerate}
\end{proposition}
\begin{proof}
    If $\mathrm{O}^p(H) = \mathrm{O}^p(G)$, then the first part of \cref{lem:pperfect_continuity} implies $\mathrm{O}^p(H_i) = \mathrm{O}^p(G_i)$ for all $i \in I$. For the converse, we deduce from the second part of the same lemma, applied respectively to $G$ and $H$, that
    \[
    \mathrm{O}^p(H) \cong \lim_{i}\mathrm{O}^p(H_i) \cong \lim_{i}\mathrm{O}^p(G_i) \cong \mathrm{O}^p(G).
    \]
    This shows that $(1)$ is equivalent to $(2)$. A similar argument proves that $(3)$ holds if and only if $(4)$ is true, because $\lim_{i \in I}$ preserves inclusions of closed subgroups. Finally, the equivalence of $(2)$, $(4)$, and $(5)$ is proven in \cite[Lemma 3.3]{balmersanders}.
\end{proof}

\begin{notation}\label{nota:propsubnormal}
    A closed subgroup $H \leqslant G$ is said to be \emph{pro-$p$-subnormal} if it satisfies the equivalent conditions of \cref{prop:pperfect}.  
\end{notation}

\subsection{Spaces of closed subgroups}
We will now study the space $\Sub(G)$ of (closed) subgroups of a profinite group $G$ in more detail. Important for us will be the following description of a neighbourhood basis of a closed subgroup $H \leqslant G$:
    \begin{equation}\label{eq:nbhdbasis}
    (\cB(H,N) \coloneqq \{K \in \Sub(G) \mid NK = NH\})_{N},
    \end{equation}
where $N$ ranges through the open normal subgroups of $G$. The group $G$ acts on $\Sub(G)$ via conjugation, providing a quotient map $q \colon \Sub(G) \to \Sub(G)/G$.

\begin{remark}\label{rem:quotientisopen}
    We record for later that the map $q \colon \Sub(G) \to \Sub(G)/G$ is an open topological map. Indeed, a sufficient condition for $q$ to be an open map is the following: if for every $U \subseteq \Sub(G)$ open, $q^{-1}(q(U)) \subseteq \Sub(G)$ is open. In our situation $q^{-1}(q(U)) = \cup_{g \in G} U^g$ is clearly open. As such, we can obtain an explicit basis for $\Sub(G)/G$ by pushing forward the basis of $\Sub(G)$. In particular we have that $(q (\cB(H,N)) = \{q(K) \mid K \in \Sub(G)\colon NK = NH\})_N$ is a neighbourhood basis for $q(H)$ in $\Sub(G)/G$. 
\end{remark}

Although we now have an understanding of the space $\Sub(G)$ and its associated quotient space $\Sub(G)/G$, its construction does not yet fit into our continuous mindset. Given that $G$ is built from its finite quotient groups $G_i \coloneqq G/U_i$ where $U_i$ is an open normal subgroup, one may ask if the space $\Sub(G)$ (and $\Sub(G)/G$) admit a similar description. We note that the projection maps $\pi_{G/U_i} \colon G \to G/U_i$ give rise to projection maps $\pi_i \colon \Sub(G) \to \Sub(G/U_i)$ which send $H$ to $HU_i/U_i \cong H/(H \cap U_i)$. These maps are compatible with conjugation in the sense that the following diagram commutes:
    \begin{equation}\label{eq:squareofspaces}
    \begin{gathered}
        \xymatrix{
        \Sub(G) \ar[r] \ar[d] & \Sub(G/U_i) \ar[d] \\
        \Sub(G)/G \ar[r] & \Sub(G/U_i)/(G/U_i).
        }
    \end{gathered}
    \end{equation}

The following result of Dress gives the desired continuous description of the space of subgroups.

\begin{proposition}[\cite{dress_notes}]\label{prop:profinitesubg}
Let $G = \lim_i G/U_i$ be a profinite group, where the $U_i$ run through a cofiltered system of open normal subgroups of $G$. Then there are natural homeomorphisms
\[
\Sub(G) \cong \lim_i \Sub(G/U_i) \qquad \text{and} \qquad \Sub(G)/G \cong \lim_i\Sub(G/U_i)/(G/U_i).
\]
where $\Sub(G/U_i)$ (and $\Sub(G/U_i)/(G/U_i)$) are equipped with the discrete topology. In particular, both $\Sub(G)$ and $\Sub(G)/G$ are profinite spaces. 
\end{proposition}

The next lemma provides a feeling for how the open subgroups behave with respect to the topology. In particular, we see that the $\Sub(H)$ as $H$ runs through the open subgroups yield a neighbourhood basis for the identity element.

\begin{lemma}[{\cite[4.3.1]{sugrue_thesis}}]
If $H \leqslant G$ is open, then $\Sub(H) \subseteq \Sub(G)$ is clopen. In particular $\Sub(H) \subseteq \Sub(G)$ is quasi-compact and open in $\Sub(G)$.
\end{lemma}

\begin{example}
    It is high time for a guiding example in the world of profinite groups. Consider the cofiltered limit of the groups $\Z/{p^i}$ for $i \in \N$. Then this cofiltered limit realises the $p$-adic integers which we will denote $\Z_p$. The order of this group in the sense of \eqref{eq:order} is the supernatural number $m(\Z_p) = p^\infty$.
    
    The closed subgroups of $\Z_p$ are the $p^i \Z_p$ with $i \in \N$ and the identity element $e = p^\infty \Z_p$. All of these subgroups are open except for the trivial subgroup. It follows that the space $\Sub(\Z_p)$ can be identified with the one-point compactificaton of the integers, $\N^*$, which here should be thought of as the colimit of the intervals $[0,i]$. The canonical projection map $\pi_i \colon \Sub(\Z_p) \to \Sub(\Z/p^i)$ is then given by:
    \[
    \begin{cases}
        \pi_i(p^k\Z_p) = p^k\Z/p^i & 0 \leqslant k < i; \\
        \pi_i(p^k\Z_p) =e & \text{else}.
    \end{cases}
    \]
\end{example}

\begin{example}\label{ex:cantorspec}
    We provide a second guiding example. Fix a prime $p$ and consider the group $G = \prod_{\N} \Z/p$, the \emph{elementary abelian $p$-group of rank $|\N|$}. This is indeed a profinite group and can be constructed as a cofiltered limit of $\prod_i \Z/p$ for $i \in \N$.

    It is clear that the open subgroups of $G$ are classified by those sequences $\{a_0, a_1, \dots , a_k , \dots\}$ where only finitely many of the $a_i$ are zero. As each copy of $\Z/p$ contributes an interval $\{0,1\}$ from $\Sub(\Z/p)$, we see that $\Sub(G) \cong \{0,1\}^\N$, the Cantor set.
\end{example}

\subsection{Continuity of profinite Weyl groups}\label{ssec:weylgroups}

We finish this section with results on the continuity of normalizers and Weyl groups. Recall that for $H$ a subgroup of a group $G$, the \emph{Weyl group} $W_G(H)$ is the subgroup $N_G(H)/H$. We also recall that for $H$ a closed subgroup of $G$, the projection of $H$ in a finite quotient group $G_i \coloneqq G/U_i$ is the subgroup $H_i \coloneqq HU_i/U_i = H/(H \cap U_i)$. The goal of this section is the prove that the construction of the normalizer and the Weyl group are continuous. To do so, we require some preparatory lemmas.

\begin{lemma}\label{lem:normaliserofquotient}
    Let $G$ be a profinite group. If $K$ is an open subgroup of $G$, and $U$ is an open normal subgroup of $G$ contained in $K$, then the projection map $\pi \colon G \to G/U$ restricts to an isomorphism $N_G(K)/U \xrightarrow{\sim} N_{G/U}(K/U)$.
\end{lemma}

\begin{proof}
    We begin by showing that the map is well defined. Pick some element $g \in N_g(K)$,and consider $\pi(g) = gU \in G/U$. For any $k \in K$, we see that $(gU)(kU)(g^{-1}U) = (gkg^{-1} )U = k' U$ for some $k' \in K$. As such $\pi$ restricts to a map $\pi \colon N_G(K) \to N_{G/U}(K/U)$. Because $U \leqslant K$ by assumption, the kernel of this map is precisely $U$. Therefore, $\pi$ induces an injective map $\pi \colon N_G(K)/U \to N_{G/U}(K/U)$, and it remains to prove surjectivity. To this end, pick some element $h \in N_{G/U}(K/U)$. We have $h = gU$ for some $g \in G$ and for $k \in K$ we have that $(gU)(kU)(g^{-1}U) = k'U$ for some $k' \in K$. In particular, we see that $gkg^{-1} = k' u \in K$ for some $u \in U$. Thus, $g \in N_G(K)$ and $\pi(g) =h \in N_{G/U}(K/U)$.
\end{proof}

\begin{lemma}\label{lem:limitisintersection}
    Let $(G_i)_i$ be a system of profinite groups where all maps are inclusions. Then $\lim_i G_i = \cap_i G_i$. 
\end{lemma}

\begin{proof}
    This follows from the fact that the forgetful functor from profinite groups to sets preserves limits together with the statement that the limit of a decreasing sequence of sets is the intersection.
\end{proof}

\begin{lemma}[{\cite[Proposition 2.2.4]{ribes_zalesskii}}]\label{lem:limitisexact}
    Let
    \[
    (\xymatrix{1 \ar[r] & Q_i \ar[r] & G_i \ar[r] & K_i \ar[r] & 1})_{i \in I}
    \]
    be a cofiltered system of exact sequences of profinite groups. Then the limit of this sequence is once again exact.
\end{lemma}

\begin{proposition}\label{prop:normalisercontinuous}
    Let $G$ be a profinite group with closed subgroup $H$. Then the construction of normalizers is continuous. That is, there is an isomorphism
    \[
    N_G(H) \cong \lim_i N_{G_i}(H_i),
    \]
    where $G_i$ runs over the finite quotient groups of $G$ and the maps of the system are restrictions of the quotient maps 
    defining $G$ as the limit of the $G_i$. 
\end{proposition}

\begin{proof}
    Consider the system of exact sequences
    \[
    (\xymatrix{1 \ar[r] & U_i \ar[r] & N_G(HU_i) \ar[r] & N_G (HU_i)/U_i \ar[r] & 1})_{i \in I}.
    \]
    By \cref{lem:limitisexact}, taking the limits of this system gives another exact sequence
    \[
    \xymatrix{1 \ar[r] & \lim_i U_i \ar[r] & \lim_i N_G(HU_i) \ar[r] & \lim_i N_G (HU_i)/U_i \ar[r] & 1.}
    \]
    By \cref{lem:limitisintersection} we can identify the first term as the intersection of all the 
    open normal subgroups $U_i$, which is the trivial subgroup $\{e\}$ by \cite[Corollary 1.2.4]{wilson}. For the middle term, first observe that $N_GH$ is a subgroup of $N_G(HU_i)$ as the $U_i$ are normal, hence $N_G(H)$ is contained in $\cap_i N_G(HU_i)$. Likewise $N_G(HU_i)$ is a subgroup of $N_G(HU_j)$ whenever $U_i \leqslant U_i$, so an application of \cref{lem:limitisintersection} yields that $\bigcap_i N_G(HU_i) \cong \lim_i N_G(HU_i)$. Conversely, consider some $g$ which normalizes each $HU_i$. Since
    \[
    H = \bigcap_i (HU_i) = \bigcap_i g(HU_i)g^{-1} = g(\bigcap_i HU_i) g^{-1} = gHg^{-1}
    \]
    we conclude that $g$ normalizes $H$ itself, so $g \in N_G(H)$. In summary, we get that
    \[
        N_G(H) = \bigcap_i N_G(HU_i) \cong \lim_i N_G(HU_i).
    \]
    Applying \cref{lem:normaliserofquotient} on the last term, we have identified the exact sequence as
    \[
    \xymatrix{1 \ar[r] & \{e\} \ar[r] & N_G(H) \ar[r] & \lim_i N_{G_i} (HU_i/U_i) \ar[r] & 1}
    \]
    As such, it follows that the quotient maps induce an isomorphism
    \[
    N_G(H) \cong \lim_i N_{G_i} (HU_i/U_i) = \lim_i N_{G_i} (H_i)
    \]
    as required.
\end{proof}

Now that we have seen that the normalizers behave continuously, we are able to conclude that Weyl groups also behave continuously.

\begin{proposition}\label{prop:profiniteweyl}
    Let $G$ be a profinite group with closed subgroup $H$. Then the construction of Weyl group is continuous. That is, there is an isomorphism
    \[
    W_G(H) \cong \lim_i W_{G_i}(H_i),
    \]
    where $G_i$ runs over the finite quotient groups of $G$.
\end{proposition}
\begin{proof}
    Consider the system of short exact sequences of finite groups
        \[
        (\xymatrix{1 \ar[r] & H_i \ar[r] & N_{G_i}(H_i) \ar[r] & W_{G_i}(H_i) \ar[r] & 1})_{i \in I},
        \]
    with transition maps as defined above. The limit of this diagram is an exact sequence by \cref{lem:limitisexact}. In light of \eqref{eq:limofsubgroups}  and \cref{prop:normalisercontinuous}, we thus obtain a short exact sequence 
        \[
        \xymatrix{1 \ar[r] & H \ar[r] & N_{G}(H) \ar[r] & \lim_i W_{G_i}(H_i) \ar[r] & 1,}
        \]
    which proves that $W_G H \cong \lim_i W_{G_i}(H_i)$ as required.
\end{proof}

\section{Stable equivariant homotopy theory for a profinite group}\label{sec:equivariantspectra}

The goal of this section is to introduce a convenient model for the category of (genuine) equivariant $G$-spectra for $G$ a profinite group and to construct base-change functors for it. The key idea is that a (finite) $G$-spectrum should be ``continuous'' in the sense that it a filtered colimit of $G_i$-spectra for $G_i$ the finite quotient groups of $G$. We make this intuition precise in \cref{ssec:gsp}, our approach being based on and compared to Fausk's work \cite{Fausk2008}. We refer to \cref{app:equivariantmodels} for a more thorough discussion and a comparison with other models of the equivariant stable homotopy theory of profinite groups. Our ``continuous model'' results in a computation of the Balmer spectrum of the tt-category of compact equivariant $G$-spectra for all profinite $G$ in terms of the Balmer spectra for finite groups. In \cref{ssec:gfp}, we then apply continuity to construct geometric fixed point functors for profinite groups and to establish their characteristic properties.

\subsection{The category of equivariant $G$-spectra for profinite $G$}\label{ssec:gsp}

Let $G$ be a profinite group, fixed throughout this section and presented as a cofiltered limit of finite quotient groups, $G = \lim_{i\in I}G_i$. For example, we may take $G_i = G/U_i$ for $(U_i)_{i \in I}$ a cofinal system of open normal subgroups of $G$. Given a finite group $F$, we write $\Sp_{F}$ for the symmetric monoidal stable $\infty$-category of genuine $F$-equivariant spectra. Moreover, for our profinite group $G$, we denote by 
\begin{equation}\label{eq:gsp_inftyfausk}
    \Spfausk_{G} \coloneqq \mathsf{N}(G\Sp)
\end{equation}
the symmetric monoidal stable $\infty$-category underlying Fausk's monoidal model category, as recalled in \cref{app:equivariantmodels}. The transition maps in the system $(G_i)_{i\in I}$ induce compatible restriction functors on the categories of $G_i$-spectra: For any map $f_{ij}\colon G_i \to G_j$, there are geometric functors
\begin{equation}\label{eq:gsp_restriction}
\xymatrix{f_{ij}^* \colon \Sp_{G_j} \to \Sp_{G_i}}
\end{equation}
given by restriction the $G_j$-action to an $G_i$-action along $f_{ij}$. We note that these restriction functors are naturally coherent with respect to the system $(G_i)_{i\in I}$, in the expected way. When $f_{ij}$ is surjective, $f_{ij}^*$ is usually called \emph{inflation} and denoted $\infl = \infl_{G_j}^{G_i}$, while for an inclusion of subgroups, the induced functor is called \emph{restriction} and denoted $\res=\res_{G_j}^{G_i}$.

\begin{definition}\label{def:gsp_contmodel}
For a profinite group $G$, we define the symmetric monoidal stable $\infty$-category of \emph{continuous $G$-spectra} as the filtered colimit over the geometric functors \eqref{eq:gsp_restriction}:
\[
\Sp_G^{\cont} \coloneqq \colim_{i}^{\omega}\Sp_{G_i}.
\]
Here, the filtered colimit is taken in the $\infty$-category of rigidly-compactly generated tt-categories, as studied in \cref{ssec:colimits}.
\end{definition}

\begin{remark}\label{rem:gsp_limitmodel}
    Suppose $f\colon F \to Q$ is an epimorphism between finite groups with kernel $N$. The right adjoint $\Sp_{F} \to \Sp_{Q}$ to inflation $f^*$ is then given by the (categorical) $N$-fixed points. Following \cref{ssec:limcoliminfty}, we may describe the colimit in \cref{def:gsp_contmodel} equivalently as the limit
        \[
            \Sp_G^{\cont} \simeq \lim_i^{\Pr^R}\Sp_{G_i}
        \]
    in $\Pr^R$ taken over the corresponding fixed point functors. This limit can also be computed in $\widehat{\Cat}_{\infty}$, see the discussion around \eqref{eq:prlprr}, so any $G$-spectrum is given by a system of $G_i$-spectra $X_i$ such that $X_i^{\ker (f_{ij})} \simeq X_j$ compatibly. This limit formula, however, has the disadvantage over the continuous description above in that it does not interact well with the symmetric monoidal structures present.
\end{remark}

By construction, $\Sp_{G}^{\cont} = (\Sp_{G}^{\cont},\otimes,S_G^0)$ is a rigidly-compactly generated tt-category with monoidal structure $\otimes$ and unit the \emph{$G$-equivariant sphere spectrum} $S_G^0 = \colim_i S_{G_i}^0$. If $F$ is a finite group, then the suspension spectra $F/E_+$ of the transitive $F$-sets forms a set of compact generators for $\Sp_{F}$. It then follows from the description of filtered colimits in \cref{ssec:colimits} that $(G/U_+)_U$ is a set of a compact generators for $\Sp_{G}^{\cont}$, where $U$ ranges through the open normal subgroups of $G$. Finally, we note that for finite $G$, the system $(G_i)_i$ has a final object, so $\Sp_G^{\cont} = \Sp_G$ in this case. 

\begin{theorem}\label{thm:gsp_contfausk}
Let $G$ be a profinite group. The inflation functors induce a geometric equivalence
\[
\xymatrix{\Sp_G^{\cont} \ar[r]^-{\sim} & \Spfausk_{G}.}
\]
\end{theorem}
\begin{proof}
    We begin with the construction of the comparison functor between the two models. Because both categories are rigidly-compactly generated, it suffices to perform this construction on the level of compact objects; the result then follows by passing to the associated ind-categories. For any finite quotient group $G \to G_i$, there is an inflation functor $\infl\colon \Sp_{G_i}^{\omega} \to \cSpfausk_G$; for example, this functor may be produced as the underlying functor of symmetric monoidal $\infty$-categories of a monoidal Quillen functor. Varying over the system $(G_i)_{i \in I}$ of quotient groups of $G$, we obtain an induced functor
    \[
    \xymatrix{\psi\colon \Sp_G^{\cont,\omega} \simeq \colim_{i}\Sp_{G_i}^{\omega} \ar[r] & \cSpfausk_G}
    \]
    of symmetric monoidal $\infty$-categories. 
    
    Let $\sfG = (G/U_+\mid U \triangleleft_o G)$ be the full subcategory of $\Sp_G^{\cont,\omega}$ on the set of generators exhibited above. In order to show that $\psi$ is an equivalence, we need to verify the following two claims:
        \begin{enumerate}
            \item $\psi(\sfG)$ forms a set of generators for $\cSpfausk_G$.
            \item $\psi$ is fully faithful on $\sfG$.
        \end{enumerate}
    In fact, both claims can be checked on the level of triangulated categories, so we may tacitly reduce to the corresponding homotopy categories. Claim (1) is then a direct consequence of the construction of Fausk's model category, see \cref{thm:fausk}. Keeping in mind the formula \eqref{eq:colimitmap} for the mapping spaces in a filtered colimit of $\infty$-categories, the second claim translates to the statement that, for open normal subgroups $V,W$ in $G$, the functor $\psi$ induces an isomorphism
        \[
        \colim_i\pi_*\Hom_{\Sp_{G_i}}(G/V_+,G/W_+) \cong \pi_*\Hom_{\Spfausk_G}(G/V_+,G/W_+),
        \]
    where $G_i$ runs over those finite quotient groups of $G$ which stabilize both $G/V$ and $G/W$, i.e., $G_i = G/U_i$ with $U_i \leqslant V \cap W$. 
    The desired isomorphism was established \cite[Corollary 3.10]{barnessugrue_spectra}, so we conclude that $\psi$ is an equivalence of tt-categories. 
\end{proof}

The previous theorem justifies the following choice of:

\begin{notation}\label{nota:gsp}
    For $G$ a profinite group, we write $\Sp_G$ for the rigidly-compactly generated tt-category of \emph{(genuine) $G$-equivariant spectra}, modelled by either of the two equivalent categories $\Sp_{G}^{\cont}$ or $\Spfausk_G$.
\end{notation}

Given a group homomorphism $f_{ij}\colon G_i \to G_j$ in the system $(G_i)_{i\in I}$, the corresponding inflation functor $\infl_{G_j}^{G_i}$ induces a continuous map on Balmer spectra which we denote by
\begin{equation}\label{eq:gsp_iota}
    \xymatrix{\iota_{ij}\colon \Spc(\Sp_{G_i}^{\omega}) \ar[r] & \Spc(\Sp_{G_j}^{\omega}); & \iota_G\colon \Spc(\Sp_G^{\omega}) \ar[r] & \lim_{i \in I}\Spc(\Sp_{G_i}^{\omega}).}
\end{equation}
The map $\iota_G$ displayed on the right is then the induced map on limits. Since the colimit diagram~$I$ is filtered, by the continuity of the spectrum (\cref{prop:spccontinuity}), we obtain:

\begin{corollary}\label{cor:gsp_contspectrum}
    Let $G$ be a profinite group, then inflation induces a homeomorphism
    \[
    \xymatrix{\iota_G\colon\Spc(\Sp_{G}^{\omega}) \ar[r]^-{\cong} & \lim_{i\in I}\Spc(\Sp_{G_i}^{\omega}).}
    \]
\end{corollary}

In \cref{ssec:gsp_ttprimes}, we will give a more explicit description of the points of $\Spc(\Sp_{G}^{\omega})$.

\subsection{Geometric fixed points for profinite groups}\label{ssec:gfp}

We continue to exploit the continuity of $\Sp_G$, now in order to construct geometric fixed point functors for profinite groups and establish a characterization through their basic properties. We freely refer to \cite{lmms_86, mandell04} as references for geometric fixed points for finite groups; see also the $\infty$-categorical treatment in \cite{MathewNaumannNoel2017}. A different, but ultimately equivalent approach to the construction of geometric fixed point functors for profinite groups has been given by Bachmann and Hoyois in \cite[Section 9]{bachmannhoyois_norms}.

Recall our conventions that $G$ is a profinite group presented as $G = \lim_{i\in I}G_i$ with $I$ cofiltered and $G_i=G/U_i$ finite for all $i \in I$. Let $H$ be a closed subgroup of $G$ and consider its image $H_i = HU_i/U_i$ under the quotient map $\Sub(G) \to \Sub(G_i)$ induced by $G \to G_i$. Note that $H \cong \lim_i H_i$ by \eqref{eq:limofsubgroups}. The system of subgroups $(H_i)_{i \in I}$ is compatible with $(G_i)_{i\in I}$ in the sense that, given a surjective homomorphism $G_i \to G_j$, there is a commutative diagram 
\[
\begin{tikzcd}
    G_i = G/U_i \arrow[d, shift left=4,two heads,swap,"f_{ij}"] & HU_i/U_i \cong H/U_i \cap H = H_i \arrow[l,hook] \ar[d, shift right = 16,two heads,"f_{ij}"] \\
    G_j = G/U_j & HU_j/U_j \cong H/U_j \cap H = H_j. \arrow[l,hook]
\end{tikzcd}
\]
Passing to the corresponding restriction and inflation functors yields the commutative square which forms the left part of the following diagram:
\begin{equation}\label{eq:gfp_finitecompatibility}
\begin{tikzcd}[column sep=1.5cm]
    \Sp_{G_i} \arrow[r,swap,"\res"] \arrow[rr, bend left=20,"\Phi_{G_i}^{H_i}"] & \Sp_{H_i} \arrow[r,swap,"\Phi_{H_i}^{H_i}"] & \Sp  \\
    \Sp_{G_j} \arrow[r,"\res"]\arrow[rr, bend right=20,swap, "\Phi_{G_j}^{H_j}"] \arrow[u,"\infl"] & \Sp_{H_j} \arrow[r,"\Phi_{H_j}^{H_j}"] \arrow[u,"\infl"] & \Sp. \arrow[u,swap,"="]
\end{tikzcd}
\end{equation}

The right square of the diagram, displaying the interaction of the \emph{geometric fixed points} $\Phi_{H_i}^{H_i}$ and $\Phi_{H_j}^{H_j}$ with inflation, commutes as well, see for example \cite[Section 2(K)]{balmersanders}. Finally, the top and the bottom triangles commute by construction of the absolute geometric fixed point functors as a composite of restriction followed by geometric fixed points.

Observe that all functors in \eqref{eq:gfp_finitecompatibility} are geometric. This allows us to take the compactly generated filtered colimit of \eqref{eq:gfp_finitecompatibility} over $I$, resulting in the bottom half of the following diagram:
\begin{equation}\label{eq:gfp_profinitecompatibility}
\begin{tikzcd}[column sep=1.5cm]
    \Sp_G \arrow[r,dashed,"\res"] \arrow[rr,dashed, bend left=20,"\Phi_G^H"] & \Sp_H \arrow[r,dashed,"\Phi_H^H"] & \Sp  \\
    \colim_{i}^{\omega}\Sp_{G_i} \arrow[r,"\colim^{\omega}\res"]\arrow[rr, bend right=20,swap, "\colim^{\omega}\Phi_{G_i}^{H_i}"] \arrow[u,"\infl","\cong"'] & \colim_{i}^{\omega}\Sp_{H_i} \arrow[r,"\colim^{\omega}\Phi_{H_i}^{H_i}"] \arrow[u,"\infl","\cong"'] & \Sp. \arrow[u,swap,"="]
\end{tikzcd}
\end{equation}
The vertical geometric equivalences are obtained from \cref{thm:gsp_contfausk} (or by construction, \cref{def:gsp_contmodel}, respectively), so we can define the top part of the diagram \eqref{eq:gfp_profinitecompatibility} so that it becomes commutative.

\begin{definition}\label{def:gfp}
    Let $H$ be a closed subgroup of $G$ and write $(H_i)_{i \in I}$ for the associated cofiltered system comprised of subgroups $H_i$ of $G_i$. Upon identifying $\Sp_G$ with $\colim_{i}^{\omega}\Sp_{G_i}$ via inflation, and similarly for $\Sp_H$, we define:
        \begin{enumerate}
            \item the \emph{restriction} from $G$ to $H$ as $\res_{H}^G\coloneqq \colim_{i}^{\omega}\res_{H_i}^{G_i}$;
            \item the \emph{geometric fixed points} at $H$ as $\Phi_{H}^H \coloneqq \colim_{i}^{\omega}\Phi_{H_i}^{H_i}$;
            \item and the \emph{(absolute) geometric fixed point functor} at $H$ as $\Phi_G^H \coloneqq \Phi_{H}^H\circ \res_{H}^G$.
        \end{enumerate}
    If the ambient profinite group $G$ is clear from context, we may also omit the corresponding subscript on these functors. 
\end{definition}

In summary, the geometric fixed points for profinite groups are extended via continuity from their finite counterparts; as such, they inherit their excellent formal properties. The next proposition isolates their two key characteristics:

\begin{proposition}\label{prop:gfp_properties}
    The geometric fixed point functors $(\Phi^H)_{H \in \Sub(G)}$ of \cref{def:gfp} are uniquely determined by the following two properties:   
        \begin{enumerate}
            \item the functor $\Phi^H\colon \Sp_G \to \Sp$ is geometric;
            \item for every open normal subgroup $U \leqslant G$ there is a natural equivalence of geometric functors
            \[
            \xymatrix{\Phi^H \circ \infl_{G/U}^G \simeq \Phi_{G/U}^{HU/U}\colon \Sp_{G/U} \ar[r] & \Sp.}
            \]
        \end{enumerate}
\end{proposition}
\begin{proof}
    By its construction as a colimit taken in rigidly-compactly generated tt-categories, the geometric fixed points are obtained and thus determined as ind-extensions from their restriction to compact objects. Since both constituents in the composite
    \[
    \xymatrixcolsep{4pc}{
    \xymatrix{\colim_{i}\Sp_{G_i}^{\omega} \ar[r]^-{\colim_i \res_{H_i}^{G_i}} & \colim_{i} \Sp_{H_i}^{\omega} \ar[r]^-{\colim_{i}\Phi_{H_i}^{H_i}} & \Sp^{\omega}}}
    \]
    are filtered colimits of tt-functors, so is $\Phi_{H}^G\colon \Sp_{G}^{\omega} \to \Sp^{\omega}$. This verifies the first property. 
    
    To verify (2), by a cofinality argument, we may take $G_i = G/U$. The diagrams \eqref{eq:gfp_finitecompatibility} and \eqref{eq:gfp_profinitecompatibility} provide a commutative diagram
    \[
    \begin{tikzcd}[column sep=1.5cm]
    \Sp_{G} \arrow[r,swap,"\res_H^G"] \arrow[rr, bend left=20,"\Phi_{G}^{H}"] & \Sp_{H} \arrow[r,swap,"\Phi_{H}^{H}"] & \Sp  \\
    \Sp_{G_i} \arrow[r,"\res_{H_i}^{G_i}"]\arrow[rr, bend right=20,swap, "\Phi_{G_i}^{H_i}"] \arrow[u,"\infl_{G_i}^{G}"] & \Sp_{H_i} \arrow[r,"\Phi_{H_i}^{H_i}"] \arrow[u,"\infl_{H_i}^{H}"] & \Sp \arrow[u,swap,"="]
    \end{tikzcd}
    \]
    which establishes the second property. 
    
    Finally, unwinding the construction of $\Sp_{G}^{\omega}$ as a filtered colimit over inflation functors, we see that these two properties determine the geometric fixed points uniquely. 
\end{proof}

\begin{corollary}\label{cor:gfp_unstable}
    Let $H$ be a closed subgroup of $G$. Let $X$ be a finite $G$-CW complex which is fixed by some $G_i$ and write $X_+$ for its $G$-suspension spectrum. Then we have $\Phi^H(X_+) \simeq (X^{H_i})_+$.
\end{corollary}
\begin{proof}
    If $G_i$ fixes $X$, then $X_+$ can be viewed as the inflation from $G_i$ to $G$ of the corresponding suspension spectrum in $\Sp_{G_i}$. The claim then follows from \cref{prop:gfp_properties}(2) together with the corresponding property of geometric fixed points for finite groups. 
\end{proof}

\begin{remark}\label{rem:gfp_conjugacy}
    Suppose $H \sim_{G} K$ are two $G$-conjugate closed subgroups of $G$. As in \cite[Section 2(L)]{balmersanders}, one observes that the corresponding inner automorphism of $G$ induces a natural equivalence $\Phi_{G}^{H} \simeq \Phi_{G}^{K}$. This uses that $K \simeq_G H$ implies $K_i \simeq_{G_i} H_i$ for all $i \in I$, see \cref{prop:profinitesubg}.
\end{remark}

\begin{remark}\label{rem:gfp_coherence}
    The reader might wonder how coherent the constructions of this section are. There are two answers: Firstly, one might restrict oneself to working up to homotopy, i.e., at the level of tt-categories. In this case, the diagrams are rendered commutative by the corresponding compatibilities for finite groups. Secondly, in order to obtain commutativity at the level of $\infty$-categories requires to either lift the basic constructions for finite groups to left Quillen functors as in \cite[Chapter V]{mandell_may}, see \cite[Section 6]{MathewNaumannNoel2017}, or to use the parametrized approach taken by Bachmann--Hoyois \cite{bachmannhoyois_norms}. In \cref{app:equivariantmodels}, we compare the different models for the stable equivariant homotopy theory of profinite groups, thereby reconciling both points of view.
\end{remark}

\section{The tensor-triangular geometry of \texorpdfstring{$G$}{G}-spectra}\label{sec:gsp_tt}

Now that we have introduced our category of interest, namely $\Sp_G$ for $G$ a profinite group, it is time to apply the theory of \cref{part:profinitettgeom} to it. In particular, in this section we will describe the Balmer spectrum of this category, providing an explicit computation in the case that $G = \Z_p$.

\subsection{Equivariant prime tt-ideals}\label{ssec:gsp_ttprimes}

Let $G$ be a profinite group, presented as $G = \lim_{i}G_i$ for a cofiltered system $(G_i)_i$ of finite quotient groups of $G$. In \cref{cor:gsp_contspectrum}, we identified the spectrum of $\Sp_G^{\omega}$ as the inverse limit of the spectra of $\Sp_{G_i}^{\omega}$ over $I$. Our next goal is to give a more explicit description of the points of $\Spc(\Sp_{G}^{\omega})$ utilizing the geometric fixed point functors constructed in \cref{ssec:gfp}. We begin by briefly recalling the non-equivariant case:

\begin{recollection}\label{rec:sp_ttprimes}
    Let $K(n)=K_p(n)$ be the $n$-th Morava $K$-theory spectrum at the prime $p$ and height $n$; by convention, $K_p(0) \coloneqq H\Q$ for all $p$, while $K_p(\infty) \coloneqq H\F_p$. By the work of Devinatz, Hopkins, and Smith \cite{nilpotence1,nilpotence2} as expressed tt-geometrically by Balmer \cite[Theorem 9.1]{balmer_3spectra}, the points of $\Spc(\Sp^{\omega})$ are given precisely by the prime tt-ideals
    \[
    \sfP(p,n) \coloneqq \{x\in \Sp^{\omega}\mid K_p(n-1) \otimes x = 0\},
    \]
    where $p$ ranges through prime numbers and $n$ is an positive integer or $\infty$. Here, note that $\sfP(p,1)$ is the full subcategory of torsion finite spectra for any prime $p$, hence $\sfP(p,1) = \sfP(q,1)$ for all primes $p$ and $q$. The topology on $\Spc(\Sp^{\omega})$ turns out to be determined by the inclusions among the prime tt-ideals, and we have
    \[
    \sfP(p,n) \subseteq \sfP(q,m) \iff 
        \begin{cases}
            n\geqslant m & \text{if } m=1; \\
            n\geqslant m \text{ and } p=q & \text{if } m>1.
        \end{cases}
    \]
\end{recollection}

The geometric fixed point functors (\cref{ssec:gfp}) induce maps on spectra, which we denote by:
\begin{equation}\label{eq:gsp_varphi}
    \xymatrix{\varphi_{G}^{H}\colon \Spc(\Sp^{\omega}) \ar[r] & \Spc(\Sp_{G}^{\omega}); & \varphi\colon \Sub(G)/G \times \Spc(\Sp^{\omega}) \ar[r]^-{\varphi_{G}^{(-)}} & \Spc(\Sp_{G}^{\omega}).}
\end{equation}
Here, the notation $\varphi_{G}^{(-)}$ indicates that for any closed subgroup of $H$ in $G$, the $H$-component of the map $\varphi$ is given by $\varphi_G^H = \Spc(\Phi_G^H)$. Because $\Phi_G^H$ depends only on the $G$-conjugacy class of $H$ (see \cref{rem:gfp_conjugacy}), this map is well-defined. Now since the geometric fixed point functors are geometric, we obtain a candidate collection of equivariant prime tt-ideals from their non-equivariant counterparts via pullback:

\begin{definition}\label{def:gsp_primett}
    For every closed subgroup $H$ in $G$, any prime number $p$, and any  $1 \leqslant n \leqslant \infty$, we define a point in $\Spc(\Sp_G^{\omega})$ by
        \[
        \sfP_G(H,p,n) \coloneqq \varphi_{G}^H(\sfP(p,n)) = \{x \in \Sp_G^{\omega} \mid K_p(n-1)_*(\Phi_G^H(x))=0\}.
        \]
\end{definition}

The next result, extending theorems of Strickland and Balmer--Sanders \cite[Theorems 4.9 and 4.14]{balmersanders} from the case of finite groups to profinite groups, states that the prime tt-ideals of \cref{def:gsp_primett} capture all points of $\Spc(\Sp_G^{\omega})$.

\begin{proposition}\label{prop:gsp_spectrumset}
    Let $G$ be a profinite group and continue to use the notation introduced in \eqref{eq:gsp_varphi}. Then the geometric fixed point functors induce a bijection
    \[
    \xymatrix{{\Sub(G)/G} \times \Spc(\Sp^{\omega}) \ar[r]_-{\simeq}^-{\varphi} & \Spc(\Sp_{G}^{\omega}),\quad (H,\sfP(p,n)) \mapsto \sfP_G(H,p,n).}
    \]
    In particular, all prime tt-ideals of $\Sp_G$ are of the form $\sfP_G(H,p,n)$ as defined in \cref{def:gsp_primett}, and the following two conditions are equivalent for any two of them:
        \begin{enumerate}
            \item $\sfP_G(H,p,n) = \sfP_G(K,q,m)$;
            \item $H$ is $G$-conjugate to $K$, $n=m$, and, if $m>1$, then $p=q$.
        \end{enumerate}
\end{proposition}
\begin{proof}
    Keeping in mind the notation introduced in \eqref{eq:gsp_varphi}, we begin with the construction of the following diagram and then verify that it is in fact commutative: 
    \begin{equation}\label{eq:gsp_spectrumset}
        \begin{gathered}
        \xymatrix@C=2em{
            &  \Sub(G)/G \times \Spc(\Sp^{\omega})    \ar[rr]^-{\varphi_G} \ar[d]_{\can_G} && \ar[d]^{\iota_G} \Spc(\Sp_G^{\omega})  \ar `r[r] `[dd]^{\iota_{G_i}} [dd] & \\
            & \lim_{i}\Sub(G_i)/G_i \times \Spc(\Sp^{\omega}) \ar[rr]^-{\lim_i\varphi_{G_i}} \ar[d]_{\proj_i} && \lim_{i}\Spc(\Sp_{G_i}^{\omega}) \ar[d]^{\proj_i}& \\
            &\ar@{<-} `l[l] `[uu]^{\pi_i \times \id} [uu]  \Sub(G_i)/G_i \times \Spc(\Sp^{\omega}) \ar[rr]_-{\varphi_{G_i}} && \Spc(\Sp_{G_i}^{\omega})&
        }
        \end{gathered}
    \end{equation}
    Consider first the outer part of the diagram, which we produce for any finite quotient group $G_i$ of $G$. Recall that $\iota_{G_i}$ is the map induced on spectra by the inflation functor $\infl_{G_i}^G$, see \eqref{eq:gsp_iota}, and similar for $\iota_G$. Fixing a closed subgroup $H$ of $G$, the relation $\Phi^H \circ \infl_{G_i}^G \simeq \Phi_{G_i}^{H_i}$ established in \cref{prop:gfp_properties}(2) implies that $\iota_{G_i} \circ \varphi_G^H \cong \varphi_{G_i}^{H_i} = \varphi_{G_i}^{\pi_iH}$. Letting $H$ vary in $\Sub(G)/G$ then gives the commutativity of the outer rectangle. 
    
    Since this diagram is compatible with the maps in the given system $(G_i)_{i\in I}$, as explained on the level of geometric functors in \cref{ssec:gfp}, we obtain the inner part of the diagram from the universal property of limits. Here, the maps labelled $\proj_i$ denote the projections on the $i$-th part of the limit diagram, while $\can_G$ is the canonical map. This concludes the construction of \eqref{eq:gsp_spectrumset} as a commutative diagram.
    
    Let us now focus on the upper square in \eqref{eq:gsp_spectrumset}. By \cite[Theorems 4.9 and 4.14]{balmersanders}, the map $\varphi_{G_i}$ is a bijection for any finite group $G_i$, hence so is the middle horizontal map. The left vertical map $\can_G$ is a homeomorphism by \cref{prop:profinitesubg}, and we showed in \cref{cor:gsp_contspectrum} that $\iota_G$ is a homeomorphism. Therefore, $\varphi_G$ is bijective, as desired. 
    
    Finally, unpacking the bijectivity of $\varphi_G$ then gives the statement about prime tt-ideals: surjectivity corresponds to all tt-primes being of the form $\sfP_G(H,p,n)$, while injectivity translates into the characterization of when two such prime tt-ideals coincide, keeping in mind the classification of tt-primes in $\Sp^{\omega}$, see \cref{rec:sp_ttprimes}.
\end{proof}

\subsection{The Balmer spectrum, revisited}\label{ssec:gsp_prism}

The geometric fixed point functors allow us to turn the  homeomorphism of \cref{cor:gsp_contspectrum} into a more explicit presentation of the Balmer spectrum of $\Sp_{G,\Q}$. We begin with some point-set topological preliminaries. 

\begin{recollection}\label{rec:priestleyspaces}
    Let $X$ be a spectral topological space. Naturally associated to $X$ are two pieces of data on its underlying set: 
        \begin{itemize}
            \item the constructible topology on $X$, forming a profinite space $X_{\cons}$;
            \item the specialization order $\rightsquigarrow$ of $X$, giving rise to a poset $(X,\rightsquigarrow)$.
        \end{itemize}
    In fact, $X$ is uniquely determined by the pair $(X_{\cons},\rightsquigarrow)$; the latter has the structure of a \emph{Priestley space}, i.e., an ordered profinite space satisfying a certain separation axiom. This observation can be promoted to an isomorphism between the category of spectral spaces and spectral maps on one side, and the category of Priestley spaces and monotone continuous maps on the other. We refer to interested reader to \cite[Section 1]{book_spectralspaces} for a detailed review of this circle of ideas.
    
    The Balmer spectrum $\Spc(\sfK)$ of any tt-category $\sfK$ is a spectral space, so we may consider its associated Priestley space $(\Spc(\sfK)_{\cons},\rightsquigarrow)$. In this case, the specialization order $\rightsquigarrow$ corresponds precisely to inclusion among the prime tt-ideals of $\sfK$:
        \[
        \sfP \rightsquigarrow \sfQ \iff \sfQ \subseteq  \sfP.
        \]
    Especially in situations where the topology of $\Spc(\sfK)$ is complicated, its Priestley space provides a convenient presentation that separated the combinatorial features from the topological ones. Following \cite{BalchinBarthelGreenlees2023pp}, we refer to the Priestley space of $\Spc(\sfK)$ also as the \emph{prism} $\Prism(\sfK)$ of $\sfK$.
\end{recollection}

\begin{example}\label{ex:sp_prism}
    For any prime $p$, the prism of the category $\Sp_{p}^{\omega}$ of compact $p$-local spectra is $(\N^*,\leqslant)$, where $\N^* = \N \cup \{\infty\}$ denotes the one-point compactification of the natural numbers. The identification is induced by $\sfP(p,n) \mapsto n$. The integral case readily reduces to the $p$-local one, by gluing along the common point $\sfP(0,1)$; cf.~\cref{rec:sp_ttprimes}.
\end{example}

In order to describe the prism of $\Sp_{G}^{\omega}$ for profinite groups $G$, we first need to determine the constructible topology on its spectrum. The next lemma describes the answer in terms of the topological space $\Sub(G)/G$ and the constructible topology in the non-equivariant case, given in \cref{ex:sp_prism}.

\begin{lemma}\label{ssec:gsp_cons}
    For any profinite group $G$, the geometric fixed points induce a homeomorphism 
    \[
    \xymatrix{{\Sub(G)/G} \times \Spc(\Sp^{\omega})_{\cons} \ar[r]_-{\cong}^-{\varphi} & \Spc(\Sp_{G}^{\omega})_{\cons},\quad (H,\sfP(p,n)) \mapsto \sfP_G(H,p,n).}
    \]
\end{lemma}
\begin{proof}
    We first prove the claim in the case of finite groups, say for one of the finite quotient groups $G_i$ of $G$. Note that the space $\Sub(G_i)/G_i$ is discrete. For any subgroup $H_i$ of $G_i$, the geometric fixed point functor $\Phi_{G_i}^{H_i}$ induces a spectral map $\varphi_{G_i}^{H_i}\colon \Spc(\Sp^{\omega}) \to \Spc(\Sp_{G_i}^{\omega})$. Since finite coproducts in the category of spectral spaces are computed in topological spaces and using the finite group case of \cref{prop:gsp_spectrumset}, these maps assemble into a bijective spectral map $\varphi_{G_i}\colon\bigsqcup_{H_i \in \Sub(G_i)/G_i}\Spc(\Sp^{\omega}) \to \Spc(\Sp_{G_i}^{\omega})$. Rewriting the domain as a product and passing to the constructible topology, we thus obtain a continuous bijection
        \[
        \xymatrix{\varphi_{G_i,\cons}\colon \Sub(G_i)/G_i\times \Spc(\Sp^{\omega})_{\cons} \ar[r] & \Spc(\Sp_{G_i}^{\omega})_{\cons}}
        \]
    between profinite spaces. Therefore, $\varphi_{G_i,\cons}$ is a homeomorphism.
    
    The maps $\varphi_{G_i}$ are natural in the system $(G_i)_i$, see the discussion surrounding \eqref{eq:gsp_spectrumset}, providing a commutative square of profinite spaces and continuous maps
        \[
        \xymatrixcolsep{4pc}{
        \xymatrix{\Sub(G)/G\times \Spc(\Sp^{\omega})_{\cons} \ar[r]^-{\varphi_{G,\cons}} \ar[d] & \Spc(\Sp_{G}^{\omega})_{\cons} \ar[d] \\
        \lim_{i}(\Sub(G_i)/G_i\times \Spc(\Sp^{\omega})_{\cons}) \ar[r]_-{\cong}^-{\lim_i(\varphi_{G_i,\cons})} & \lim_{i}(\Spc(\Sp_{G_i}^{\omega})_{\cons}).}}
        \]
    The bottom horizontal map is a homeomorphism as a consequence of the claim for finite groups. The left vertical map is a homeomorphism because limits commute with products and using \cref{prop:profinitesubg}. For the right vertical map, we need the fact that limits of spectral spaces commute with passage to the constructible topology, which are computed in topological spaces. This combined with \Cref{cor:gsp_contspectrum} gives homeomorphisms
        \[
        \Spc(\Sp_{G}^{\omega})_{\cons} \cong (\lim_{i} \Spc(\Sp_{G_i}^{\omega}))_{\cons} \cong \lim_{i}(\Spc(\Sp_{G_i}^{\omega})_{\cons}),
        \]
    so the right vertical map is a homeomorphism as well. The commutativity of the diagram then implies that $\varphi_{G,\cons}$ is a homeomorphism, as desired.
\end{proof}

\begin{lemma}\label{lem:gsp_inclusions}
    Let $G$ be a finite group with a cofinal system $(G_i)_{i \in I}$ of finite quotient groups.
    Consider two equivariant tt-primes $\sfP_G(H,p,n)$ and $\sfP_G(K,q,m)$, and write $H_i$ and $K_i$ for the images of $H$ and $K$ in $G_i$. Then
        \[
        \sfP_G(K,p,n) \subseteq \sfP_G(H,q,m) \iff \sfP_{G_i}(K_i,p,n) \subseteq \sfP_{G_i}(H_i,q,m) \text{ for all } i \in I. 
        \]
\end{lemma}
\begin{proof}
    The spectral space $X=\Spc(\Sp_G^{\omega})$ is homeomorphic to the limit of the spectral spaces $X_i = \Spc(\Sp_{G_i}^{\omega})$ over the filtered system $I$, see \cref{cor:gsp_contspectrum}. By the discussion in \cite[Section 2.3.9]{book_spectralspaces}, the specialization order on $\Spc(\Sp_G^{\omega})$ is computed componentwise, i.e., $x \rightsquigarrow y$ in $X$ if and only if $x_i \rightsquigarrow y_i$ in $X_i$ for all $i \in I$. In the Balmer spectrum, the specialization relation between points corresponds precisely to the (reversed) inclusion relation among the prime tt-ideals (\cref{rec:priestleyspaces}), so we conclude by the parametrization of \cref{prop:gsp_spectrumset} along with the formula $\iota_{G_i}(\sfP_G(H,p,n)) = \sfP_{G_i}(H_i,p,n)$ obtained from the commutative diagram \eqref{eq:gsp_spectrumset}.
\end{proof}

In the remainder of this section, we implicitly use the bijection $\varphi_G$ of \cref{prop:gsp_spectrumset} to identify the underlying set of $\Spc(\Sp_G^{\omega})$ with $\Sub(G)/G \times \Spc(\Sp^{\omega})$.

\begin{theorem}\label{thm:prism}
    Let $G$ be a profinite group. The Priestley space of $\Spc(\Sp_G^{\omega})$ is given by 
        \[
        \Prism(\Sp_{G}^{\omega}) = ({\Sub(G)/G} \times \Spc(\Sp^{\omega})_{\cons},\supseteq),
        \]
    where all the inclusions between equivariant tt-primes are of the form $\sfP_G(K,p,n) \subseteq \sfP_G(H,p,m)$ for $K$ conjugate in $G$ to a pro-$p$-subnormal (in the sense of \cref{nota:propsubnormal}) subgroup of $H$ and $n \geqslant m + \beth_G(H,K,p,m)$ for some $\beth_G(H,K,p,m) \in \N \cup\{\infty\}$. In particular, a subset in $\Spc(\Sp_G^{\omega})$ is closed (resp., Thomason closed) if and only if it is closed (resp., clopen) in ${\Sub(G)/G} \times \Spc(\Sp^{\omega})_{\cons}$ and down-closed under inclusion.
\end{theorem}
\begin{proof}
    Keeping in mind \cref{rec:priestleyspaces}, the identification of the Priestley space of $\Spc(\Sp_G^{\omega})$ follows from \cref{ssec:gsp_cons}. The characterization of closed (resp.,~Thomason closed) subsets is then an instance of the general relation between spectral spaces and their Priestley spaces, as stated for example in \cite[Theorem 1.5.11]{book_spectralspaces}. It remains to establish the necessary conditions for the containment between equivariant tt-primes.
    
    To this end, we will appeal to \cref{lem:gsp_inclusions} to reduce to the analogous conditions for finite groups summarized in \cite[Theorem 1.4]{balmersanders}: Consider an inclusion among equivariant tt-primes, say $\sfP_G(K,p,n) \subseteq \sfP_G(H,q,m)$ for closed subgroups $H,K$, primes $p,q$, and $m,n\in \N\cup\{\infty\}$. By \cref{lem:gsp_inclusions}, this is equivalent to the statement that $\sfP_{G_i}(K_i,p,n) \subseteq \sfP_{G_i}(H_i,q,m)$ for all $i \in I$. For a given finite group $G_i$, such an inclusion is possible if and only if the following conditions are satisfied:
        \begin{enumerate}
            \item $p=q$ if $m>1$;
            \item $n \geqslant m + \beth_{G_i}(H_i,K_i,p,m)$ for some integer $\beth_{G_i}(H_i,K_i,p,m) \geqslant 0$;
            \item $K_i$ is $G_i$-conjugate to a $p$-subnormal subgroup of $H_i$. 
        \end{enumerate}
    Note that, if $m=1$, then $\sfP_G(H,q,1) = \sfP_G(H,p,1)$ by convention, so we may consider $m>1$ and take $p=q$. Setting 
    \begin{equation}\label{eq:pro_blueshiftnumbers}
        \beth_G(H,K,p,m) \coloneqq \sup_{i \in I}\big\{\beth_{G_i}(H_i,K_i,p,m)\big\} \in \N \cup \{\infty\}
    \end{equation}
    implies $n \geqslant m + \beth_G(H,K,p,m)$. It remains to deduce that $K$ is $G$-conjugate to a $p$-subnormal subgroup of $H$. But this follows from (3): \cref{prop:profinitesubg} shows that $K$ is $G$-conjugate to a subgroup $K'$ of $H$, and $K'$ is then pro-$p$-subnormal in $H$ by \cref{prop:pperfect}.
\end{proof}

\begin{figure}[h]
    \centering
    \begin{tikzpicture}[yscale=0.75, darkstyle/.style={circle,draw,fill=gray!70, ,minimum size=5, inner sep=0pt}]
        \foreach \x in {0,...,1}
            \foreach \y in {0,...,3} 
                \node [darkstyle]  (\x\y) at (2*\x,\y) {};
        \node at ($(00)+(0,-0.5)$) {$H$};
        \node at ($(00)+(1,-0.5)$) {$\supseteq$};
        \node at ($(10)+(0,-0.5)$) {$K$};

        \node[align=right,text width=30] at ($(00)+(-1,0)$) {$m$};
        \node[align=right,text width=30] at ($(01)+(-1,0)$) {$m+1$};
        \node[align=right,text width=30] at ($(03)+(-1,0)$) {$n$};
        \foreach \x in {0,...,1}
            \foreach \y [count=\yi] in {0,...,2}  
              \draw (\x\y)--(\x\yi);

        \draw[fill=white, white] (-0.25,1.5) rectangle (2.25,2.5);
        \draw[Blue4] (00) -- (13);
        \node at ($(00)+(0,2.15)$) {$\vdots$};
        \node at ($(10)+(0,2.15)$) {$\vdots$};
        \draw [decorate, decoration = {brace,amplitude=15pt}] (2.2,3) --  (2.2,0);
        \node at ($(10)+(2,1.5)$) {$\beth_G(H,K,p,m)$};
    \end{tikzpicture}
    \caption{A schematic illustrating of the nature of blueshift.}
\end{figure}
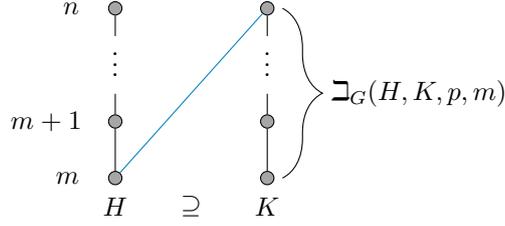

\subsection{Blueshift for profinite groups}\label{ssec:gsp_blueshift}

The previous theorem shows that the topology on $\Spc(\Sp_G^{\omega})$ is determined by the profinite space ${\Sub(G)/G} \times \Spc(\Sp^{\omega})_{\cons}$ together with the numbers $\beth_{G_i}(H_i,K_i,p,m)$ for a cofinal system $(G_i)$ of finite quotient groups of $G$. Following \cite{BHNNNS2019}, we refer to both $\beth_G(H,K,p,m)$ as well as their analogues for finite groups as the \emph{blueshift number} of the tuple $(G,H,K,p,m)$.  

\begin{remark}\label{rem:gspspc_sota}
    We will first review what is known about the blueshift numbers for a \emph{finite} group $G$. As shown in \cite{balmersanders}, we can reduce to the case that $H = G$ is a finite $p$-group and work $p$-locally, for a fixed prime $p$; we will henceforth omit the $p$-locality from the notation. Let $\Phi(G)$ be the Frattini subgroup of $G$, see \cref{rec:frattini}. If $K$ is a $p$-subnormal subgroup of $G$, define the \emph{lower blueshift number} $\beth_{G}^-(K)$ of the pair $(K,G)$ as the $p$-rank of $G/(K\Phi(G))$, and the \emph{upper blueshift number} $\beth_{G}^+(K)$ as the minimal length of a $p$-subnormal chain from $K$ to $G$. Combining the main results of \cite{balmersanders} and \cite{BHNNNS2019}, we obtain inequalities
        \begin{equation}\label{eq:blusehift_inequalities}
        \beth_{G}^-(K) \leqslant \beth_G(G,K,p,m) \leqslant \beth_{G}^+(K),
        \end{equation}
    as explained in \cite[Section 5]{KuhnLloyd2020pp}. When $G$ is finite abelian, $\beth_{G}^-(K) = \beth_{G}^+(K)$ for all subgroups $K$, so this determines the blueshift numbers and hence the topology of the Balmer spectrum of $\Sp_G^{\omega}$ completely in this case. Kuhn and Lloyd (\cite{KuhnLloyd2020pp}) then establish further improvements for certain families of non-abelian groups, proving in particular that the $\beth_G(G,K,p,m)$ coincide with the upper blueshift numbers for extraspecial 2-groups $G$. In other words, they resolve the problem of computing the topology on $\Spc(\Sp_G^{\omega})$ in these cases. 
\end{remark}

\Cref{thm:prism} reduces the computation of the topology on $\Spc(\Sp_G^{\omega})$ to understanding the inclusions $\sfP_G(K,p,n) \subseteq \sfP_G(H,p,m)$ for $K$ a pro-$p$-subnormal subgroup of $H$. As in \cite[Proposition 6.11]{balmersanders}, it suffices to consider the case of pro-$p$-groups:

\begin{lemma}\label{lem:gspspc_propreduction}
    Let $G$ be a profinite group, let $H, K \in \Sub(G)$, and consider a prime $p$ and $n,m \in \N \cup \{\infty\}$. The following two statements are equivalent:
        \begin{enumerate}
            \item there is an inclusion $\sfP_G(K,p,n) \subseteq \sfP_G(H,p,m)$;
            \item $K$ is $G$-conjugate to a pro-$p$-subnormal subgroup $K' \leqslant H$ and 
            \[
            \sfP_{H/\mathrm{O}^p(H)}(K'/\mathrm{O}^p(H),p,n) \subseteq \sfP_{H/\mathrm{O}^p(H)}(H/\mathrm{O}^p(H),p,m).
            \]
        \end{enumerate}
    Note that the quotient $K'/\mathrm{O}^p(H)$ is meaningful, thanks to \cref{prop:pperfect}.
\end{lemma} 
\begin{proof}
    By \cref{lem:gsp_inclusions} and the continuity of $\mathrm{O}^p(-)$ established in \cref{lem:pperfect_continuity}, this reduces to the case of finite $G$, which was proven in \cite[Proposition 6.11]{balmersanders}. 
\end{proof}

\begin{remark}\label{rem:burnsidering}
    Let $G=\lim_iG_i$ be a profinite group. The \emph{Burnside ring} of $G$ is defined as the endomorphism ring $A(G) =\pi_0\End_{\Sp_{G}}(S_{G})$. As observed in the proof of \cref{cor:continuous_comparisonmaps}, the inflation functors induce an isomorphism $\colim_iA(G_i) \cong A(G)$, so that we obtain a homeomorphism
        \[
            \xymatrix{\Spec(A(G)) \ar[r]^-{\sim} & \lim_i\Spec(A(G_i)).}
        \]
    Dress \cite{dress1969} computed the spectrum of $A(G)$ for any finite group $G$, which thus determines the spectrum of the Burnside ring for any profinite group as well. Indeed, every prime ideal of $A(G)$ has the form $\frak p(H,p)$ for $p$ a prime or $\frak p(H,0)$, for $H$ ranging over the closed subgroup in $G$, up to the following relations:
        \begin{enumerate}
            \item $\frak p(H,0) \subseteq \frak p(K,0)$ iff $H$ is conjugate to $K$ in $G$;
            \item $\frak p(H,p) \subseteq \frak p(K,q)$ only if $p=q$ and $\frak p(H,p) = \frak p(K,q)$;
            \item $\frak p(H,0) \subseteq \frak p(K,p)$ iff $\frak p(H,p) = \frak p(K,p)$ iff $\mathrm{O}^p(H)$ is conjugate to $\mathrm{O}^p(K)$ in $G$;
            \item $\frak p(H,0) \subset \frak p(H,p)$ and $\frak p(H,p) \nsubseteq \frak p(K,0)$.
        \end{enumerate}
    In light of \cref{cor:continuous_comparisonmaps}, we can also describe the inclusion preserving comparison map along the lines of \cite[Proposition 6.7]{balmersanders}:
        \[
            \xymatrix{\Spc(\Sp_G^{\omega}) \ar[r]^-{\rho_G} & \Spec(A(G))},\quad \rho_G(\sfP_G(H,p,m)) = 
                \begin{cases}
                    \frak p(H,0) & \text{if } m=1 \\
                    \frak p(H,p) & \text{if } m>1.
                \end{cases}
        \]
\end{remark}

The next result determines the inclusions among equivariant tt-primes and hence the topology on $\Spc(\Sp_G^{\omega})$ completely for arbitrary \emph{abelian} profinite groups.

\begin{theorem}\label{thm:prism_abelian}
    Let $A$ be an abelian profinite group. For closed subgroups $H,K \leqslant A$, primes $p,q$, and $n,m \in \N\cup\{\infty\}$, the following conditions are equivalent:
        \begin{enumerate}
            \item there is an inclusion $\sfP_A(K,p,n) \subseteq \sfP_A(H,q,m)$;
            \item $K$ is a subgroup of $H$, the quotient $H/K$ is a pro-$p$-group, and
                \[
                \begin{cases}
                    n \geqslant m + \rank_p(H/K) & \text{if } m = 1; \\
                    n \geqslant m + \rank_p(H/K) \text{ and } p=q & \text{if } m>1.
                \end{cases}
                \]
        \end{enumerate}
    Consequently, if $H/K$ is a pro-$p$ group, then $\beth_A(H,K,p,m) = \rank_p(H/K)$.
\end{theorem}
\begin{proof}
    This essentially follows from \cref{thm:prism} and the main theorem of \cite{BHNNNS2019}, as discussed in \cref{rem:gspspc_sota}; let us spell out the details. As before in the proof of \cref{thm:prism}, we can deal with the edge case $m=1$ separately, thereby reducing to $p=q$. By \cref{lem:gsp_inclusions}, there is an inclusion $\sfP_A(K,p,n) \subseteq \sfP_A(H,p,m)$ if and only if $\sfP_{A_i}(K_i,p,n) \subseteq \sfP_{A_i}(H_i,p,m)$ for some cofinal system $(A_i)_{i \in I}$ of finite quotient groups of $A$. By \cite[Proposition 6.11]{balmersanders} and \cite[Corollary 6.11]{BHNNNS2019}, this happens if and only if for all $i\in I$:
        \begin{enumerate}
            \item $K_i$ is $p$-subnormal in $H_i$; and
            \item $n \geqslant m + \rank_p(H_i/K_i)$.
        \end{enumerate}
    Because $A$ is abelian, $K_i$ is $p$-subnormal in $H_i$ if and only if $H_i/K_i$ is a finite abelian $p$-group: Indeed, we can lift any filtration of $H_i/K_i$ with cyclic quotients to a $p$-subnormal chain of $K_i$ in $H_i$. Therefore, Condition (1) for all $i \in I$ is equivalent to the statement that $H/K$ is an abelian pro-$p$-group. Let us write $P = H/K$ and $P_i$ for the image of $P$ in $\Sub(A_i)$. By \cite[Corollary 2.8.3]{ribes_zalesskii}, we have $P/\Phi(P) \cong \lim_{i}P_i/\Phi(P_i)$, hence $\#(P/\Phi(P)) = \mathrm{lcm}_{i \in I}\{\#(P_i/\Phi(P_i))\}$. This implies that $\rank_p(P) = \sup_{i \in I}\{\rank_p(P_i)\}$. In other words, we have shown that Condition (2) for all $i\in I$ is equivalent to $n \geqslant m + \rank_p(H/K)$, which concludes the proof. 
\end{proof}

\begin{remark}\label{rem:prism_abelian}
    \Cref{thm:prism_abelian} is a profinite analogue of the main theorem of \cite{BHNNNS2019} for finite abelian groups and of \cite{BarthelGreenleesHausmann2020} for abelian compact Lie groups. 
\end{remark}

We conclude this subsection with an example that showcases the general phenomenology of the topology on $\Spc(\Sp_G^{\omega})$.

\begin{example}\label{ex:prism_zp}
    Consider the example $G=\Z_p$. Any closed subgroup of $\Z_p$ is of the form $p^k\Z_p$ for some $k \in \N$ or $e$, the trivial group. This induces a bijection $\Sub(\Z_p) \to \N \cup \{\infty\}$, which in fact furnishes a homeomorphism $\Sub(\Z_p) \cong \N^*$. Since the most interesting information of the spectrum of $\Sp_{\Z_p}^{\omega}$ is concentrated at the prime $p$, let us work $p$-locally and omit the prime $p$ from the notation. Given two subgroups $H$ and $K$ in $\Z_p$, say $K \leqslant H$, and $n,m \in \N \cup \{\infty\}$, we have the following characterization of the possible inclusions among the corresponding tt-primes:
    \[
    \sfP_{\Z_p}(K,n) \subseteq \sfP_{\Z_p}(H,m) \iff
        \begin{cases}
            n \geqslant m & \text{if } [H:K] \text{ is finite}; \\
            n \geqslant m+1 & \text{otherwise}.
        \end{cases}
    \]
    In \cref{fig:speczp} in the introduction, we display a schematic picture for the Balmer spectrum of compact $\Z_p$ equivariant spectra.
\end{example}

\section{The nilpotence theorem for profinite groups}\label{sec:gsp_nilpotence}

The final goal of this part is to prove a version of the nilpotence theorem for the category of equivariant $G$-spectra. Despite the continuity of $\Sp_G$, this does not appear to be formal. As in the case of finite groups (or even compact Lie groups), one would like to reduce to the non-equivariant nilpotence theorem of Devinatz, Hopkins, and Smith, but usually this is achieved via the joint conservativity of the geometric fixed point functors. In the profinite case, this tool is not readily available, essentially because it is not clear how to induct on closed subgroups. Instead, we will first compute the homological spectrum of $\Sp_G$ and then deduce from it our nilpotence theorem, thereby reversing the approach taken in \cite{balmer_nilpotence}.

\subsection{The bijectivity hypothesis}\label{ssec:bijectivityhyp}

The next result verifies the bijectivity hypothesis (\cref{def:bijhyp}) for $\Sp_G$, extending a previous theorem by Balmer from finite to profinite groups. It will be the key ingredient in our proof of the profinite nilpotence theorem below.

\begin{proposition}\label{prop:bijectivityhyp}
    The bijectivity hypothesis holds for $\Sp_G$ for any profinite group $G$, i.e., the comparison map
        \[
        \xymatrix{\phi_G\colon \Spch(\Sp_G^{\omega}) \ar[r]^-{\cong} & \Spc(\Sp_{G}^{\omega})}
        \]
    is a homeomorphism.
\end{proposition}
\begin{proof}
    Let us first construct an auxiliary commutative diagram. The two rows are obtained by applying the homological and triangular spectrum functor to the geometric fixed points and the inflation functors, respectively. The vertical maps are all given by Balmer's (surjective) comparison map \eqref{eq:balmer_map}, whose naturality gives the commutativity of the diagram. 
    \begin{equation}\label{eq:spectracomparison}
        \begin{gathered}
            \xymatrix{
            \Sub(G)/G \times \Spch(\Sp^{\omega}) \ar[r]^-{\varphi^h} \ar[d]_-{\cong}^-{\sqcup \phi} & \Spch(\Sp_{G}^{\omega}) \ar[r]^-{\iota^h} \ar@{->>}[d]^-{\phi_G} & \lim_{i}\Spch(\Sp_{G_i}^{\omega}) \ar[d]_-{\cong}^-{\lim_i\phi_{G_i}} \\
            \Sub(G)/G \times \Spc(\Sp^{\omega}) \ar[r]_-{\simeq}^-{\varphi} & \Spc(\Sp_{G}^{\omega}) \ar[r]_-{\cong}^-{\iota} & \lim_{i}\Spc(\Sp_{G_i}^{\omega})
            } 
        \end{gathered}
    \end{equation}
    The two outer vertical maps are homeomorphisms by \cite[Corollary 5.10]{balmer_nilpotence}. We have seen that the map $\iota$ is a homeomorphism in \cref{cor:gsp_contspectrum}, while \cref{prop:gsp_spectrumset} shows that $\varphi$ is a continuous bijection.
    
    Now consider the right square of the diagram \eqref{eq:spectracomparison}. In light of \cref{thm:gsp_contfausk}, we can apply \cref{thm:spchcontinuity} to see that $\iota^h$ is a homeomorphism. By the commutativity of the square, we deduce that $\phi_G$ is a homeomorphism as well.
\end{proof}

As a consequence, the commutativity of the left square in \eqref{eq:spectracomparison} gives us the homological analogue of \cref{prop:gsp_spectrumset}:

\begin{corollary}\label{cor:gsp_hspectrumset}
    The geometric fixed point functors induce a continuous bijection
        \[
        \xymatrix{{\Sub(G)/G} \times \Spch(\Sp^{\omega}) \ar[r]_-{\simeq}^-{\varphi^h} & \Spch(\Sp_{G}^{\omega}).}
        \]
\end{corollary}

\subsection{The nilpotence theorem}\label{ssec:nilpotence}

It was shown in \cite{BCHS2023bpp} that the conclusion of \cref{cor:gsp_hspectrumset} is equivalent to the statement that the family $(\Phi_G^H)_{H\in \Sub(G)/G}$ of geometric fixed point functors detect weak ring objects in $\Sp_G$. Here, a \emph{weak ring} is an object $R \in \Sp_G$ together with a map $\eta\colon S_G^0 \to R$ such that $\eta \otimes R$ splits, and detection means that 
    \begin{equation}\label{eq:profinitenilfaithful}
        \Phi_G^H(R) = 0 \text{ for all } H \in \Sub(G)/G \iff R=0. 
    \end{equation}
Suppose $F$ is a compact $G$-spectrum. If $\Phi_G^H(F) = 0$ for all $H$, then also $\Phi_G^H(F \otimes \underline{\hom}(F,S_G^0)) = \Phi_G^H(\underline{\hom}(F,F)) = 0$ for all $H$, hence $\underline{\hom}(F,F) = 0$ by \eqref{eq:profinitenilfaithful} because it is a ring object. Since $F$ can be written as a retract of $F \otimes \underline{\hom}(F,F)$, this implies $F \simeq 0$. We conclude that the geometric fixed point functors are jointly conservative on $\Sp_G^{\omega}$.

The equivalence of \eqref{eq:profinitenilfaithful} is an equivariant formulation of the nilpotence theorem for $\Sp_G$ for profinite groups $G$, and provides a partial conservativity result for the geometric fixed points. In combination with the non-equivariant nilpotence theorem \cite{nilpotence1,nilpotence2}, we see that the collection $(K_p(n)_*\Phi_G^H)_{H,p,n}$ of homological functors detects weak ring objects in $\Sp_G$, where the indexing set ranges through closed subgroups $H$, primes $p$, and $n \in \N\cup\{\infty\}$. The equivalence \eqref{eq:profinitenilfaithful} is also a consequence of the following, stronger version of the nilpotence theorem for $\Sp_G$ applied to the map $\eta$.

\begin{theorem}\label{thm:nilpotence}
    Let $G$ be a profinite group, then the homological functors $(K_p(n)_*\Phi_G^H)_{H,p,n}$ jointly detect $\otimes$-nilpotence of morphisms in $\Sp_G$ with compact source.
\end{theorem}
\begin{proof}
    Suppose $\alpha\colon x \to t$ is a morphism in $\Sp_G$ with $x$ compact. We have to show that $\alpha$ satisfies $\alpha^{\otimes m} = 0$ for some $m \geqslant 1$ if $K_p(n)_*\Phi_G^H(\alpha) = 0$ for all $H \in \Sub(G)$, primes $p$, and $n \in \N \cup\{\infty\}$. 
    
    The proof requires some preparation. Fix a triple $(H,p,n)$ and consider the corresponding homological functor $F\coloneqq K_p(n)_*\Phi_G^H\colon \Sp_{G} \to \sfD(K(n)_*)$, taking values in the tt-field of graded $K_p(n)_*$-modules. By the universal property of $\cA(\Sp_G^{\omega}) = \Mod\text{-}\Sp_G^{\omega}$, the functor $F$ extends over the restricted Yoneda embedding $h\colon \Sp_G \to \cA(\Sp_G^{\omega})$ to a functor $\hat{F}\colon \cA(\Sp_G^{\omega}) \to \sfD(K_p(n)_*)$. By \cite[Lemma 2.2]{BalmerCameron2021} (summarizing results from \cite{balmer_homological} and \cite{BKSruminations}), the kernel of $\hat{F}$ is equal to the localizing ideal $\langle \sfB \rangle$ generated by a homological prime $\sfB \in \Spch(\Sp_{G}^{\omega})$. Therefore, we obtain a factorization $\overline{F}$ of $\hat{F}$ through the Gabriel quotient of $\cA(\Sp_G^{\omega})$ by $\langle \sfB \rangle$, as depicted in the commutative diagram below. By definition, $\overline{h}_{\sfB} = q_{\sfB}\circ h$ is the homological residue field at $\sfB$.
    \begin{equation}\label{eq:homologicalresiduefunctors}
        \begin{gathered}
            \begin{tikzcd}[column sep=1.5cm]
            \Sp_G^{\omega} \arrow[r,"h"] \arrow[dr,"F = K_p(n)_*\Phi_G^H"'] \arrow[rr, bend left=15,swap, "\overline{h}_{\sfB}"'] & \cA(\Sp_G^{\omega}) \arrow[r,"q_{\sfB}"] \arrow[d,"\hat{F}"'] & \cA(\Sp_G^{\omega})/\langle \sfB \rangle \arrow[dl,dashed,"\overline{F}"] \\
            & \sfD(K_p(n)_*). 
            \end{tikzcd}
        \end{gathered}
    \end{equation}
    We claim that $\sfB$ corresponds under the comparison map $\phi\colon \Spc^h(\Sp_G^{\omega}) \to \Spc(\Sp_G^{\omega})$ to the equivariant tt-prime $\sfP_G(H,p,n)$. Indeed, for $x \in \Sp_G^{\omega}$, we have equivalences
        \begin{align*}
            h(x) \in \sfB = \ker(\hat{F})\cap\cA^{\mathrm{fp}}(\Sp_G^{\omega}) & \iff 0 = \hat{F}h(x) \simeq K_p(n)_*\Phi_G^H(x) \\
            & \iff x \in \sfP_G(H,p,n),
        \end{align*}
    i.e., $\phi(\sfB) = \sfP_G(H,p,n)$. Since $\phi$ is a bijection by \cref{prop:bijectivityhyp}, varying $(H,p,n)$ exhausts all homological primes of $\Sp_{G}^{\omega}$. 

    Let us return to our given map $\alpha$ satisfying $K_p(n)_*\Phi_G^H(\alpha) = \overline{F}\circ\overline{h}_{\sfB}(\alpha) = 0$ for all $(H,p,n)$. The same lemma as cited above shows that $\overline{F}$ is monoidal, exact, and faithful, so $\overline{h}_{\sfB}(\alpha) = 0$ for all $\sfB \in \Spch(\Sp_{G}^{\omega})$. We are thus in the situation of Balmer's homological nilpotence theorem, specifically in the form of \cite[Corollary 4.7(a)]{balmer_nilpotence}. This implies the tensor nilpotence of $\alpha$.
\end{proof}

\begin{corollary}\label{cor:nilpotence}
    Let $G$ be a profinite group, then the geometric functors $(\Phi_G^H)_{H}$ jointly detect $\otimes$-nilpotence of morphisms in $\Sp_G$ with compact source.
\end{corollary}

\begin{remark}
    \Cref{thm:nilpotence} generalizes the equivariant nilpotence theorem of Strickland \cite{Strickland2012pp} and Balmer--Sanders \cite[Theorem 4.15]{balmersanders} from finite groups to profinite groups, and also removes the compactness hypothesis on the domain of the map present in their result. It can be thought of as the profinite analogue of the equivariant nilpotence theorem for compact Lie groups established in \cite[Theorem 3.19]{BarthelGreenleesHausmann2020}. However, as remarked at the beginning of this subsection, our proof here differs significantly from these previous ones, which ultimately relied on induction over closed subgroups of the given finite or compact Lie group. This tool does not seem to be available in the context of profinite groups, which is why we reversed the conventional approach and first identified the homological with the triangular spectrum, as opposed to deducing this identification from the nilpotence theorem. 
\end{remark}

We conclude this section with another consequence of the proof of the nilpotence theorem, namely the identification of homological support with geometric isotropy in $\Sp_G$.

\begin{corollary}\label{cor:gsp_hsupport}
    For any profinite group $G$ and $t \in \Sp_G$, we have
        \[
            \Supph(t) = \{(H,(n,p)) \in {\Sub(G)/G} \times \Spch(\Sp^{\omega})\mid K_p(n)_*\Phi_G^H(t) \neq 0\}.
        \]
\end{corollary}
\begin{proof}
    Consider an object $t \in \Sp_G$ and let $\sfB \in \Spc^h(\Sp_G^{\omega})$ be a homological prime, corresponding under the bijection  of \cref{cor:gsp_hspectrumset} to a tuple $(H,(n,p))$. Using diagram \eqref{eq:homologicalresiduefunctors} together with the aforementioned fact that $\overline{F}$ is faithful and therefore conservative (\cite[Lemma 2.2]{BalmerCameron2021}), we obtain the following equivalences:
        \[
            \overline{h}_{\sfB}(t) = 0 \iff \overline{F} \circ \overline{h}_{\sfB}(t) = 0 \iff K_p(n)_*\Phi_G^H(t) = 0.
        \]
    The claim thus follows from the alternative description (\cite[Proposition 4.2]{balmer_homological}) of the homological support of $t$ as the set of those homological primes $\sfB$ for which $\overline{h}_{\sfB}(t) \neq 0$. 
\end{proof}

\newpage
\part{Rational profinite equivariant spectra}\label{part:qgsp}

In this final part, we embark on an in-depth analysis of the structure of the category of \emph{rational} $G$-equivariant spectra.
The advantage here is the existence of an algebraic model given in terms of equivariant sheaves. This algebraic model is a symmetric monoidal and $\infty$-categorical lift of the one of Barnes--Sugrue \cite{barnessugrue_spectra}, and relies on a good derived category of $G$-equivariant sheaves on profinite spaces. The construction of this derived category is the subject of \cref{sec:equivariantsheaves}, while the construction of the algebraic model is contained in \cref{sec:qgsp_algmodel}.

The reduced complication from discarding chromatic heights allows much finer control of the category and we are able to study tensor-triangular phenomena such as stratification and the telescope conjecture. This is discussed in \cref{sec:min}. Finally in \cref{sec:qgsp_lgp} we fully resolve the question of stratification through an analysis of the local-to-global principle using observations regarding the topology of the space of closed subgroups up to conjugacy.

\section{Equivariant sheaves on profinite spaces}\label{sec:equivariantsheaves}

Suppose $G$ is a profinite group and let $R$ be a commutative ring. In this section we introduce the category of $G$-equivariant sheaves of $R$-modules over a profinite space and prove it 
is a Grothendieck abelian category with enough injectives. Setting $R=\Q$ we show that the category has a set of 
enough (finite) projectives and use this to construct a projective model structure on the category of chain complexes
of rational $G$-sheaves, where the weak equivalence are the quasi-isomorphisms. This model structure is stable, monoidal, finitely generated (in the sense of \cite[Section 7.4]{hov99}), and it agrees non-equivariantly with Hovey's flat model structure from \cite[Section 3]{hov01sheaves}. As such the $\infty$-category associated to this model structure provides us with a derived category of rational $G$-equivariant sheaves. These results are the subject of \cref{prop:equivariantsheafmodel,prop:monoidalsheafmodel,prop:gsheafhocompact}, as summarized in \cref{cor:sheavescompgen}.

With these properties established, in \cref{thm:eqsheaves_continuity} we prove a continuity result for equivariant sheaves analogous to \cref{thm:gsp_contfausk}. In \cref{sec:qgsp_algmodel} we will use this to  extend the monoidal algebraic model of Wimmer for a finite group to the profinite case in \cref{thm:finitegroups_algmodel} and compare it to the algebraic model of the second author and Sugrue \cite{barnessugrue_spectra}.

\begin{remark}\label{rem:eqsheavesalternative}
    Our approach to equivariant sheaves is motivated by our desire to connect the material of this section and the next to the perspective taken by Barnes and Sugrue \cite{barnessugrue_gsheaves,barnessugrue_weylsheaves}. However, other (less hands-on but more general) formulations are possible, such as within the framework of condensed mathematics developed by Clausen and Scholze; this will be the topic of forthcoming work by the third author with Volpe.
\end{remark}

\subsection{The abelian category of \texorpdfstring{$G$}{G}-equivariant sheaves}\label{ssec:eqsheaves_abelian}

There are several equivalent definitions of $G$-equivariant sheaves, 
see Bernstein and Lunts \cite{Bernstein_Lunts}, Scheiderer \cite[Section 8]{Scheiderer} or \cite[Section 2]{barnessugrue_gsheaves}.
We have chosen the following formulation via a sheaf space as it is most amenable to our requirements:

\begin{definition}\label{defn:eqsheaf}
Let $R$ be a commutative ring and $X$ a $G$-space. A \emph{$G$-equivariant sheaf of $R$ modules over $X$} is a map $p \colon E \to X$ such that:
\begin{enumerate}
\item  $p$ is a $G$-equivariant map $p \colon E \to X$ of spaces with continuous $G$-actions;
\item  $(E,p)$ is a sheaf space of $R$-modules;
\item  $g \colon p^{-1} (x) \to p^{-1} (g x)$ 
is a map of $R$-modules for every $x\in X$ and $g\in G$.
\end{enumerate}
A map $f \colon (E,p) \to (E',p')$ of $G$-sheaves of $R$-modules over $X$ is a 
$G$-equivariant map $f \colon E \to E'$ such that $p'f=p$ 
and $f_x \colon E_x \to E'_x$ is a map of $R$-modules for each $x \in X$.  We denote the resulting category by $\Shv_{G, R}(X)$.
\end{definition}

For $U$ an open subset of $X$, we define $E(U) \coloneqq \Gamma(U,E)$ to be the set of continuous, but not-necessarily equivariant, \emph{sections} of $U$ with respect to $p$.  The \emph{stalk} at $x \in X$ is defined as 
\[
E_x \coloneqq p^{-1}(x) \cong \colim_{U \ni x}E(U),
\]
where the colimit is taken over the open subsets of $X$ containing $x$.

The stalk over any point, as a subspace of the sheaf space, is a discrete space by the local homeomorphism assumption in (2). By Ribes and Zalesskii
\cite[Lemma 5.3.1]{ribes_zalesskii}, the stabiliser of any element of a stalk  is an open subgroup of $G$. 
For this section, we will use the fact that an $R$-module $M$, equipped with the discrete topology, 
has a continuous action of a profinite group $G$ if and only if $M$ is the colimit (increasing union) of its
fixed point submodules $M^H$ as $H$ runs over the open subgroups (or, equivalently, open normal subgroups) of $G$. Such modules are called \emph{discrete} $G$-modules. We examine discrete $G$-modules in more detail in \cref{ssec:discretemodules}. 

Stalks of $G$-equivariant sheaves will play a fundamental role in this part, especially when the $G$-space $X$ is profinite, as it will be in our main application of interest. For example, Lemma \cref{lem:epicompact} is well-known non-equivariantly, however the last two of its equivalent statements rely upon the base space being profinite. 

There is a forgetful functor from $G$-equivariant sheaves of $R$ modules over $X$ to the category of sheaves
of $R$-modules over $X$, given by forgetting the $G$-actions on the sheaf space and base space.  
As one may expect, this functor has many good properties. 

\begin{lemma}\label{lem:sheafabelianbasics}
Let $X$ be a profinite $G$-space. The category of $G$-equivariant sheaves of $R$-modules over $X$, $\Shv_{G, R}(X)$, is abelian
and has all small limits and colimits.

The forgetful functor from $\Shv_{G, R}(X)$ to the category of sheaves of $R$-modules over $X$
is faithful and commutes with taking finite limits and arbitrary colimits.
Furthermore, the forgetful functor commutes with taking sections or stalks of a $G$-equivariant sheaf.
\end{lemma}
\begin{proof}
    By \cite[Theorem C]{barnessugrue_gsheaves},  $\Shv_{G, R}(X)$ is an abelian category with all small limits and colimits. 
    The forgetful functor is faithful as maps of $G$-sheaves are in particular maps of the sheaves of $R$-modules. 
    The statement about limits and colimits follows from their constructions, see \cite[Constructions 6.2 and 6.3]{barnessugrue_gsheaves}.

    The last statement is a consequence of the defining sections of a $G$-equivariant sheaf 
    to be the not-necessarily equivariant sections. 
\end{proof}

\begin{lemma}\label{lem:epicompact}
Let $f \colon E \to F$ be a map of $G$-equivariant sheaves of abelian groups over a $G$-equivariant profinite space $X$.
The following are equivalent:
\begin{enumerate}
\item $f$ is an epimorphism.
\item $f_x$ is surjective for each $x \in X$.
\item $f(Y)$ is surjective for each compact open $Y \subseteq X$.
\end{enumerate}
Additionally, if $R=\Q$, then these conditions are equivalent to:
\begin{enumerate}
\item[(4)] $f(Y)^{G_Y \cap N}$ is surjective for each open normal subgroup $N$ of $G$ and each compact open $Y \subseteq X$.
\end{enumerate}
Moreover, for any $R$,  $f$ is an isomorphism of sheaves if and only if $f_x$ is an isomorphism for each $x \in X$.
\end{lemma}

\begin{proof}
Non-equivariantly, the first two items are equivalent, see Tennison \cite[Theorem 3.4.8]{tenn}.
We want to apply that proof to the equivariant case. This is done by noting the following facts about the forgetful functor:
it commutes with finite limits and colimits, including those defining kernels, cokernels and images by \cref{lem:sheafabelianbasics}; 
it preserves epimorphisms, as these are characterized in terms of cokernels;
and it reflects epimorphisms as it is faithful. 

The third item implies the second as filtered colimits preserve surjections. 
By \cite[Corollary A]{barnessugrue_gsheaves}, 
we have a natural isomorphism 
\[
E(Y) \cong \colim_{i} E(Y)^{G_Y \cap N_i} 
\] 
as $N_i$ runs over the open normal 
subgroups of $G$. That is, $E(Y)$ is a discrete $G_Y$-module in the sense of \cref{def:discretemodules_general}. 
Hence, the fourth item always implies the third. 
The converse requires that we work rationally, as then taking fixed points at open subgroups is an exact functor.

It remains to prove that a stalkwise surjection is a surjection on the compact open sections.  Let $f$ be a stalkwise surjection and $Y$ a compact open subset of $X$.  Let $t \in F(Y)$ be a section and for each $x \in Y$, choose an $s_x \in E_x$ which maps to $t_x$. Choose a representative section $s_{Y,x} \in E(Y)$ for each $s_x$ using extension-by-zero to extend a section over some compact open $V$ to all of $Y$ wherever needed. As $f(s_{Y,x})$ agrees with $t$ at $x$, they agree on some compact open subset $V_x$. Varying over $x \in Y$, these cover the compact open subset $Y$, hence we may choose a finite subcover $V_1, \dots, V_n$, with corresponding section $s_1, \dots, s_n$. 

Define $Y_1=V_1$, $Y_2=V_2 \setminus Y_1$, $Y_3=V_3 \setminus (Y_1 \cup Y_2)$ and continue through to $Y_n$, so that $Y$ is the disjoint union of the $Y_i$. Restrict each section $s_i$ to $Y_i$ and extend-by-zero to $Y$, giving $s_i'$. Define $s = s_1' + \cdots + s_n'$. We see that $f(s)=t$ as they agree on each $Y_i$, where $f(s)$ takes value $f(s_i)$. 

For the statement on isomorphisms, we can detect monomorphisms of equivariant sheaves via the forgetful functor, hence the results of \cite[Subsection 3.3]{tenn} imply that a map of sheaves is a monomorphism if and only it is an injection on each stalk. Adding this fact to the earlier equivalences proves that isomorphisms of sheaves are characterized as stalkwise isomorphisms.
\end{proof}

It is high time for some examples of $G$-equivariant sheaves, the first of which is the equivariant analogue of a constant sheaf:

\begin{example}\label{ex:constantsheaf}
The simplest example of an equivariant sheaf of $R$-modules is the \emph{constant sheaf} over $X$ at an $R$-module $M$, given by the projection onto the base space $M \times X \to X$. Note that the local homeomorphism condition implies that all stalks are discrete, so that we must give $M$ the discrete topology. We will denote the constant sheaf at $M$ as $C(X,M)$.
\end{example}

\begin{remark}\label{rmk:extensionexact}
    Working non-equivariantly, recall the restriction and extension-by-zero functors of sheaves, such as in Tennison \cite[Section 3.8]{tenn}. 
    We first focus on the case where  $Y$ is a clopen subset of $X$.
    Let $E$ be a sheaf over $X$, the restriction functor at $E$ is given by taking the pullback of $E$ over the inclusion $Y \to X$.
    The group of sections of the restriction of $E$ are given by $E(U)$, for $U$ an open subset of $Y$.

    The extension-by-zero of a sheaf $F$ over $Y$ is a sheaf over $X$ whose sheaf space is given by $F \coprod ((X \setminus Y) \times \{ 0 \})$.
    The sections of this sheaf are given by $F(V \cap Y)$ for $V$ an open subset of $X$. 
    By our assumption of $Y$ being clopen, the restriction and extension-by-zero functors are both left and right adjoint to each other by \cite[Lemma 3.8.7 and Remark 3.8.13]{tenn}. 

    The definitions of these functors extend readily to the equivariant setting. Moreover, the construction generalizes to any closed subset of $X$ and the formulas for sections and the statement on adjoints remain true, see \cite[Section 7]{barnessugrue_gsheaves}.
\end{remark}

In the setting of non-equivariant sheaves, a \emph{skyscraper sheaf} is a sheaf with exactly one non-zero stalk. Equivariantly, the situation is more complicated: for $g \in G$ and $x \in X$,  the stalk $E_x$ at $x$ must be isomorphic to the stalk $E_{gx}$ at $gx$. That is, instead of having one non-zero stalk, we have one non-zero orbit of stalks. The equivariant generalization of skyscraper sheaves is constructed in \cite[Example 8.4]{barnessugrue_gsheaves}; we briefly recap it here:

\begin{example}\label{ex:skycraper} 
Choose some $x \in X$ and let $G_x$ be the corresponding stabiliser subgroup of $G$.
Choose an $R$-module $M$ with a discrete action of $G_x$. Construct the space
$G \times_{G_x} M$, thought of as a sheaf over the orbit of $x$. Extending by zero gives a 
$G$-sheaf of $R$-modules over $X$ that we call an \emph{equivariant skyscraper sheaf}. 
\end{example}

We wish to show that $\Shv_{G,R}(X)$ is a Grothendieck abelian category with enough injectives. If $M$ is an injective discrete $G_x$-object in $R$ modules, then the resulting skyscraper sheaf is injective by \cite[Proposition 9.1]{barnessugrue_gsheaves}. These injective skyscraper sheaves provide enough injectives:

\begin{lemma}[{\cite[Theorem D]{barnessugrue_gsheaves}}]
Let $R$ be a commutative ring and $X$ a profinite space. 
Then the category $\Shv_{G, R}(X)$ 
has enough injectives. The class of injective equivariant skyscraper sheaves 
give enough injectives.
\end{lemma}

Next, we will prove the existence of a family of generators (in the Grothendieck sense) for $\Shv_{G,R}(X)$. To do so we will construct left adjoints from categories of $R$-modules with group actions into the category of equivariant sheaves. The construction is made in terms of constant sheaves, extension-by-zero, and change of groups functors. 

\begin{definition}\label{def:sheafinduction}
Let $X$ be a profinite $G$-space and $Y$ a compact open subspace of $G$. Let $G_Y$ be the set of those $g \in G$ such that $gY \subseteq Y$. Note that, by \cite[Lemma A.1]{barnessugrue_gsheaves}, $G_Y$ is an open subgroup of $G$ as $Y$ is compact. Hence $G_Y$ has finite index in $G$. There is an adjoint pair
\[
\ind \colon 
\Shv_{G_Y, R} (Y)
\rightleftarrows 
\Shv_{G, R} (X)
\noloc 
\res
\] 
where the right adjoint, $\res$, is the composite of the forgetful functor (changing the group)
$\Shv_{G, R} (X) \to \Shv_{G_Y, R} (X)$ and  the restriction functor (changing the space)
$\Shv_{G_Y, R} (X) \to \Shv_{G_Y, R} (Y)$.

The left adjoint $\ind$ is the composite of the two left adjoints. Explicitly, one applies the extension-by-zero functor $\Shv_{G_Y, R} (Y) \to \Shv_{G_Y, R} (X)$ and then applies the induction functor 
$G \times_{G_Y} -$ to the sheaf space; this defines a functor.
\end{definition}

We use the functor $\ind$ to create a generalized version of the constant sheaves of \cref{ex:constantsheaf} 
by starting from a constant sheaf over a compact open subspace $Y$ of $X$ and applying $\ind$. 
We define a set of such sheaves in \cref{def:freeconstantsheaf} and use the $(\ind,\res)$ adjunction in \cref{lem:freesheafadjunction,lem:grothengen} to construct a set of finitely generated generators for $\Shv_{G, R} (X)$.
Recall that an object $A$ of a cocomplete abelian category $\mathcal{A}$ is said to be \emph{finitely generated} if for any  
filtered system $B \colon I \to \mathcal{A}$ whose transition maps are monomorphisms, the following canonical map is an isomorphism
\[
\colim_i \Hom(\const_{G \cdot Y}(N),  A_i) \longrightarrow \Hom(\const_{G \cdot Y}(N), \colim_i A_i).
\]

\begin{definition}\label{def:freeconstantsheaf}
    Let $Y$ be a compact open subset of a profinite $G$-space $X$, $G_Y$  the stabilizer of $Y$, and $N$ an open normal subgroup of $G$. Define $\const_{Y}(N) \in \Shv_{G_Y, R}(Y)$ to be the constant $G_Y$-sheaf at $R[G_Y/(G_Y \cap N)]$ over the base space $Y$ and $\const_{G \cdot Y}(N) \in \Shv_{G, R}(X)$ to be $\ind( \const_{Y}(N))$. 
\end{definition}

By tracing through the adjunctions (including the constant sheaf--global sections adjunction)
we see that $\const_{G \cdot Y}(N)$ represents taking $G_Y \cap N$-fixed points of the sections at $Y$:

\begin{lemma}\label{lem:freesheafadjunction}
Let $Y$ be a compact open subset of a profinite $G$-space $X$, $G_Y$  the stabilizer of $Y$, and $N$ an open normal subgroup of $G$.
The functor $\Hom(\const_{G \cdot Y}(N), -)$ from $G$-equivariant sheaves over $X$ to abelian groups is isomorphic to the functor which sends a $G$-equivariant sheaf $E$ to 
 $E(Y)^{G_Y \cap N}$.
\end{lemma}

\begin{lemma}\label{lem:grothengen}
Let $R$ be a commutative ring and $X$ a profinite space. 
Then the set
\begin{equation}
\{ \const_{G \cdot Y}(N) \mid N \leqslant G \text{ open normal and } Y \subseteq X \text{ compact open} \}
\end{equation}
provides a family of finitely generated generators for $\Shv_{G,R}(X)$. 
\end{lemma}

\begin{proof}
Let $E$ be a $G$-sheaf of $R$-modules over $X$ and  
choose some germ $e \in E_x$. 
By \cite[Corollary A]{barnessugrue_gsheaves},
there is a section $s \in E(Y)$ representing $e$ which is 
fixed by ${G_Y \cap N}$, for some open normal subgroup $N$ of $G$. 
That is, the set of sections $E(Y)$ has a discrete action of $G_Y$ 
in the sense of \cref{def:discretemodules_general}.
By \cref{lem:freesheafadjunction}, the map $R \to E(Y)^{G_Y \cap N}$ induced by sending the unit to $s$ induces a map of sheaves
$\const_{G \cdot Y}(N) \to E$. This map has $e$ in its image. 

It remains to show that the objects $\const_{G \cdot Y}(N)$ are finitely generated.
By \cref{lem:freesheafadjunction}, it suffices to show that the functor which sends a $G$-equivariant sheaf $E$ to 
the abelian group $E(Y)^{G_Y \cap N}$ commutes with filtered colimits of monomorphisms. 
By \cite[Lemma 6.29.1]{stacks-project}, taking sections at a compact open subset $Y$ commutes with filtered colimits.
Finally, fixed points commute with filtered colimits of monomorphisms. 
\end{proof}

Equipped with our family of generators, we are now able to prove the aforementioned result regarding $\Shv_{G,R}(X)$ being a Grothendieck abelian category.

\begin{theorem}\label{thm:gshvgrothendieck}
Let $R$ be a commutative ring and $X$ a profinite $G$-space. 
Then $\Shv_{G, R}(X)$ is a Grothendieck abelian category. 
\end{theorem}
\begin{proof}
The category is abelian and possesses arbitrary coproducts by \cref{lem:sheafabelianbasics}.
We claim that filtered colimits preserve exact sequences. 
We follow the  method used at the start of the proof of 
\cref{lem:epicompact} and note as the forgetful functor preserves kernels and cokernels
it preserves monomorphisms and epimorphisms. As the forgetful functor is also faithful, 
it preserves and reflects exact sequences. 
Hence, we may extend \cite[Theorem 3.6.5]{tenn} to the equivariant setting and see that 
a sequence of $G$-sheaves is exact if and only if each induced sequence on stalks is exact.
As filtered colimits preserve stalks and exact sequences of $R$-modules, the claim holds. 
The final condition is to prove that there is a family of generators, which is the content of \cref{lem:grothengen}.  
\end{proof}

\subsection{The derived category of \texorpdfstring{$G$}{G}-equivariant sheaves}\label{ssec:eqsheaves_derived}
In the previous section, we proved for an arbitrary commutative ring $R$ that the category $\Shv_{G,R}(X)$ of $G$-equivariant sheaves over a profinite $G$-space $X$ is a Grothendieck abelian category. In this section we will restrict out attention to $R=\Q$ and construct the derived category of $\Shv_{G,\Q}(X)$ via model categorical techniques.

Explicitly, we will consider the category $\Ch (\Shv_{G, \Q}(X))$, with $X$ a profinite $G$-space, and equip it with a model structure which is finitely generated and monoidal, and as such the resulting $\infty$-category is a compactly generated stable symmetric monoidal $\infty$-category.

For the $\infty$-category of the model structure that we will construct to be rightly called the derived category, we will require the weak equivalences to be the quasi-isomorphisms, that is those maps $f \colon E \to F$ in $\Ch (\Shv_{G,\Q}(X))$ such that we have an isomorphism of sheaves for each homology sheaf. In practice, we will work with several definitions of quasi-isomorphisms whose equivalence is afforded to us by requiring that the space $X$ is profinite:

\begin{proposition}\label{prop:sectionsandstalks}
    Let $X$ be a profinite $G$-space and $f \colon E \to F$ a map in $\Ch (\Shv_{G,\Q}(X))$. Then the following are equivalent:
\begin{enumerate}
\item the map $f$ is a quasi-isomorphism;
\item the map  $f_x$ is a homology isomorphism for each $x \in X$;
\item the map $f(Y)$ is a homology isomorphism for each compact open $Y \subseteq X$;
\item the map  $f(Y)^{G_Y \cap N}$ is a homology isomorphism for each open normal subgroup $N$ of $G$ and each compact open $Y \subseteq X$.
\end{enumerate}
\end{proposition}

\begin{proof}
The first two points are equivalent by 
\cref{lem:epicompact} and the fact that taking homology commutes with filtered colimits, such as that defining stalks.

We claim the first condition implies the fourth. 
By \cref{lem:grothengen}, the functor sending an equivariant sheaf $E$ to the $R[G_Y/(G_Y \cap N)]$-module $E(Y)^{G_Y \cap N}$ 
is a right adjoint, hence it is  left exact. 
By \cref{lem:epicompact} evaluating at $Y$ preserves epimorphisms, as does taking fixed points for open subgroups when working rationally. 
We conclude that $E \mapsto E(Y)^{G_Y \cap N}$ 
is an exact functor, so the claim holds.

We claim the fourth condition implies the third. 
By \cite[Corollary A]{barnessugrue_gsheaves}, 
we have an isomorphism of chain complexes
\[
E(Y) \cong \colim_{i} E(Y)^{G_Y \cap N_i} 
\] 
as $N_i$ runs over the open normal 
subgroups of $G$. That is, $E(Y)$ is a chain complex of $G_Y$-modules 
that are discrete in the sense of \cref{def:discretemodules_general}. 
The analogous isomorphism
for $F$ holds and $f$ induces a commuting
square between $E(Y)$, $F(Y)$, and their colimit expressions. 
As homology commutes with these filtered colimits the claim holds. 

The third condition implies the second by writing 
$E_x$ as the filtered colimit of $E(Y)$
as $Y$ runs over those open compact $Y$ containing $x$ and commuting homology past this filtered colimit.
\end{proof}

With \cref{prop:sectionsandstalks} in hand we are now able to begin the construction of the desired model structure. The construction is given in terms of the generating sheaves $\const_{G \cdot Y}(N)$ from \cref{def:freeconstantsheaf} along with the commonly-used
algebraic ``sphere-to-disc'' and ``zero-to-disc'' cofibrations from the projective model structure on chain complexes. In fact, we shall see that our construction coincides with  the projective model structure on 
chain complexes of rational equivariant sheaves. Non-equivariantly, this model structure retrieves Hovey's flat model structure from \cite[Section 3]{hov01sheaves}. 

\begin{proposition}\label{prop:equivariantsheafmodel} 
Let $G$ be a profinite group and $X$ a profinite $G$-space. 
Then the category $\Ch (\Shv_{G,\Q}(X))$ admits a model structure where the:
\begin{itemize}
    \item weak equivalences are the quasi-isomorphisms;
    \item fibrations are the stalkwise epimorphisms. 
\end{itemize}
Moreover, the model structure is cofibrantly generated with generating sets 
\begin{align*}
I &= \{ (S^{n-1}\Q \to D^n \Q) \otimes \const_{G \cdot Y}(N) \mid n \in \Z, \ Y \subseteq X
\textnormal{ open compact}, N \trianglelefteqslant G \textnormal{ open}   \}, \\
J &= \{ (0 \to D^n \Q) \otimes \const_{G \cdot Y}(N) \mid n \in \Z, \ Y \subseteq X
\textnormal{ open compact}, N \trianglelefteqslant G \textnormal{ open}   \}, 
\end{align*} 
where $S^{n}\Q$ is the chain complex with $\Q$ is position $n$ and 0 elsewhere, and $D^{n}\Q$ is the chain complex with $\Q$ in positions $n$ and $n-1$ connected by the identity map, and $0$ elsewhere.
\end{proposition}

\begin{proof}
We follow Hovey's recognition theorem \cite[Theorem 2.1.19]{hov99}. 
The smallness conditions follow from the finiteness statement of \cref{lem:grothengen} as the generating cofibrations
and acyclic cofibrations are degreewise monomorphisms.
The two-out-of-three and retraction conditions for weak equivalences hold as they are quasi-isomorphisms. 

By direct examination, one observes that the maps with the right-lifting property with respect to the set $I$ are exactly those $f \colon E \to F$ such that
\[
f(Y)^N \colon E(Y)^{G_Y \cap N} \to F(Y)^{G_Y \cap N}
\]
are epimorphisms and homology isomorphisms for all $G$-invariant compact open $Y$ and all open normal subgroups $N$ of $G$.

Similarly, the maps with the right-lifting property
with respect to $J$ are exactly those $f \colon E \to F$ such that
\[
f(Y)^N \colon E(Y)^{G_Y \cap N} \to F(Y)^{G_Y \cap N}
\]
are epimorphisms.  By \cref{lem:epicompact} and \cref{prop:sectionsandstalks}, 
this gives the required conditions on the fibrations and acyclic fibrations. 

As the elements of $J$ are cofibrations, all that is left to show is
that transfinite compositions of pushouts of elements of $J$ are quasi-isomorphisms. Via \cref{prop:sectionsandstalks}, this follows from the fact that taking pushouts and filtered colimits commute with taking stalks. 
\end{proof}

The next result tells us that the model structure of \cref{prop:equivariantsheafmodel} 
is simply the usual projective model structure in disguise.

\begin{proposition}\label{prop:sheafprojective}
    Let $G$ be a profinite group and $X$ a profinite $G$-space. Then the category $\Shv_{G,\Q}(X)$ has enough projectives and the model structure of \cref{prop:equivariantsheafmodel} is the projective model structure on rational $G$-sheaves. Hence, a map is a cofibration if and only if it is a degreewise split monomorphism with cofibrant cokernel. 
    Moreover, a bounded below complex of projectives is cofibrant.     
\end{proposition}
\begin{proof}
    We want to apply the work of Christensen--Hovey, specifically \cite[Lemma 2.7 and Theorem 5.7]{ch02}.  
    In the language of that paper, we have to show that the set of equivariant sheaves 
    $\const_{G \cdot Y}(N)$, as $N$ runs over the open normal subgroups of $G$ and $Y$ runs
    over the compact open subsets of $X$, is a set of enough small projectives.

    The sheaves $\const_{G \cdot Y}(N)$ are projective by the same arguments as for \cref{lem:epicompact} and  \cref{prop:sectionsandstalks}.  
    \cref{lem:grothengen} shows there are enough such objects and that they are finitely generated, which implies they are small in the required sense.
\end{proof}

\begin{remark}
Non-equivariantly, the model structure we have constructed is exactly the 
flat model structure of \cite[Theorem 3.2]{hov01sheaves}. The generating sets of the reference are 
strictly larger than the sets we give. The additional maps are constructed from the following set of monomorphisms,
where $C(Y,\Q)$ is the constant sheaf at $\Q$ supported over $Y$: 
\[
\{ C(V,\Q) \to C(Y,\Q) \mid \textrm{ $V \subseteq Y$ clopen} \}.
\]
Using that the complement of a clopen subset in a profinite space is clopen, these maps can be identified as the direct sum of an identity map and an initial map
\[
\left( C(V,\Q) \to C(V,\Q) \right) 
\oplus
\left( 0 \to C(Y \setminus V,\Q) \right).
\]
It follows that, in the case of a profinite base space and working non-equivariantly, the flat model structure and the projective 
model structure agree.
\end{remark}

We will now establish some properties of the model structure of \cref{prop:equivariantsheafmodel}.

Recall that the category of sheaves of $\Q$-modules over $X$ has a tensor product given by 
the sheafification of the presheaf tensor product over the constant sheaf at $\Q$.
That is, for sheaves $E$ and $F$, we have the presheaf which at $Y$ takes value $E(Y) \otimes_{C(Y,\Q)} F(Y)$. The sheafification of this presheaf is denoted $E \otimes F$. The unit for this monoidal structure is the constant sheaf at $\Q$. 
 
This is a closed monoidal product, with hom-object $\underline{\hom}(-,-)$ given by 
\[
\underline{\hom}(E,F)(Y) = \Hom_{C(U,\Q)}(E_{|Y},F_{|Y}), 
\]
where the last term denotes the set of not necessarily equivariant maps of 
$C(Y,\Q)$-modules over $Y$ between the restrictions of $E$ and $F$.

This closed monoidal product extends to chain complexes in the usual manner, and is moreover homotopically meaningful. To prove this we require an auxiliary lemma:

\begin{lemma}\label{lem:tensorprojective}
Let $N,M$ be open normal subgroups of $G$, and $Y,V \subseteq X$ compact open subsets.
Then $\const_{G \cdot Y}(N) \otimes \const_{G \cdot V}(M)$ is projective.
\end{lemma}
\begin{proof}
We first reduce the problem to the behaviour of the restriction functor (restricting the space and the group) and projective objects. 
Consider the following sequence of isomorphisms of $\Q$-modules, for any $G$-sheaves $B$ and $C$
\begin{align*}
\Hom(\const_{G \cdot Y}(N) \otimes B, C) 
& \cong \Hom(\const_{G \cdot Y}(N) , \underline{\hom} (B, C) ) \\
& \cong \underline{\hom} (B, C)(Y)^{G_Y \cap N} \\
& \cong \Hom (B_{|Y}, C_{|Y}).
\end{align*}
The last term is the $\Q$-module of maps of $G_Y \cap N$-sheaves over $Y$ from 
$B_{|Y}$ to $C_{|Y}$.

The $G$-sheaf $\const_{G \cdot Y}(N) \otimes B$ is projective if and only if $\Hom(\const_{G \cdot Y}(N) \otimes B, -)$ preserves epimorphisms. 
By the above, this is equivalent to the functor 
\[
\Hom (B_{|Y}, (-)_{|Y})
\]
from $G_Y \cap N$-sheaves over $Y$ to $\Q$-modules  
preserving epimorphisms for all open compact $Y \subseteq X$ and all open normal subgroups $N$.
The restriction functor from $G$-sheaves over $X$ to $G_Y \cap N$-sheaves over $Y$ preserves epimorphisms 
as these are characterized as the stalkwise surjections by \cref{lem:epicompact}.
Thus, it suffices to prove that
the restriction functor from $G$-sheaves over $X$ to $G_Y \cap N$-sheaves over $Y$ preserves projective objects (so that $B_{|Y}$ will be projective if $B$ is). 

This follows from noting that the restriction functor is in fact both left and right adjoint to induction,
so both are exact functors. We have seen that restriction and extension-by-zero functors are both left and 
right adjoint to each other (when restricting to a compact open subset)  in 
 \cref{rmk:extensionexact}. That the restriction of groups is \emph{left adjoint} to the induction functor uses the fact that 
$G \times_{G_Y \cap N} M$ is a finite direct sum of copies of $M$, with the action of $G$ permuting the factors.
Since the category is additive, the direct sum is both a product and a coproduct.
\end{proof}

\begin{proposition}\label{prop:monoidalsheafmodel}
The model structure of  \cref{prop:equivariantsheafmodel}  is a  monoidal model structure
and the unit is cofibrant.
\end{proposition}
\begin{proof}
The unit is the constant sheaf at $\Q$, $C(X, \Q)$, which is a bounded below complex of 
projective objects and hence is cofibrant by \cref{prop:sheafprojective}.

To prove that the model structure is monoidal, it suffices to prove that: 
\begin{enumerate}
\item the pushout product $i \square i'$ is a cofibration for any generating cofibrations $i, i'$;
\item the pushout product $i \square j$ is a weak equivalence 
    for any generating cofibration $i$ and any generating acyclic cofibration $j$. 
\end{enumerate}

Write $i = \const_{G \cdot Y}(N) \otimes i_n$ for $i_n \colon S^{n-1}\Q \to D^n\Q$, some $n \in \Z$, 
some open normal subgroup $N$ and some compact open $Y \subseteq X$.
Similarly, write $i' = \const_{G \cdot V}(M) \otimes i_m$. Then 
\[
i \square i' = \const_{G \cdot Y}(N) \otimes \const_{G \cdot V}(M) \otimes (i_n \square i_m).
\]
By \cref{lem:tensorprojective}, $\const_{G \cdot Y}(N) \otimes \const_{G \cdot V}(M)$ is projective, 
hence the map $i \square i'$ is a cofibration 
by \cref{prop:sheafprojective}. 
Similarly, let $j = \const_{G \cdot V}(M) \otimes j_m$ for $j_m \colon 0 \to D^m\Q$. Then
\[
i \square j = \const_{G \cdot Y}(N) \otimes \const_{G \cdot V}(M) \otimes D^m\Q \otimes i_n
\]
which is a cofibration by the preceding argument. It is a quasi-isomorphism as all terms are tensored with $D^m \Q$.  
\end{proof}

\begin{definition}\label{def:derivedeqsheaves}
    Let $G$ be a profinite group and $X$ a profinite $G$-space. The \emph{derived category of $G$-equivariant sheaves on $X$}, denoted $\sfD(\Shv_{G,\Q}(X))$ is the symmetric monoidal $\infty$-category associated to the model structure of \cref{prop:monoidalsheafmodel}.
\end{definition}

By construction we have that $\sfD(\Shv_{G,\Q}(X))$ is stable. In the remainder of this subsection we will prove that the derived category is compactly generated in the $\infty$-category sense. 

\begin{proposition}\label{prop:gsheafhocompact}
The projective model structure on  $\Ch (\Shv_{G,\Q}(X))$ is finitely generated in the sense of  \cite[Section 7.4]{hov99}. As such, the $\infty$-category associated to this model structure is compactly generated. 
A set of generators is given by (suspensions of) the $G$-sheaves  $\const_{G \cdot Y}(N)$ from \cref{def:freeconstantsheaf} as $N$ runs over the open normal subgroups of $G$ and $Y$ runs over the compact open subsets of $X$. 
\end{proposition}
\begin{proof}
As the generating cofibrations are degreewise monomorphisms, \cref{lem:grothengen} implies that this
model structure is finitely generated. 
By \cite[Corollary 7.4.4]{hov99}, a set of compact generators for the 
homotopy category, and hence the associated $\infty$-category, can be obtained by considering the cofibers of the generating cofibrations
of \cref{prop:equivariantsheafmodel}. 
Thus, the set (of suspensions) of the objects $\const_{G \cdot Y}(N)$ as $N$ runs over the open normal subgroups of $G$ and 
$Y$ runs over the compact open subsets of $X$ forms a set of compact generators, as claimed. 
\end{proof}

We summarise the results of this section in the following theorem.

\begin{theorem}\label{cor:sheavescompgen}
    Let $G$ be a profinite group and $X$ a profinite $G$-space. Then the derived category of rational $G$-equivariant sheaves over $X$, $\sfD(\Shv_{G,\Q}(X))$, is a compactly generated symmetric monoidal stable $\infty$-category.
\end{theorem}

\subsection{Continuity for \texorpdfstring{$G$}{G}-equivariant sheaves}\label{subsec:basechanged}
In the previous section we have constructed the derived category of $G$-equivariant sheaves on $X$, where $G$ is a profinite group and $X$ is a profinite $G$-space. In this section we will prove an analogue of \cref{thm:gsp_contfausk} for this derived category. That is, we will prove that $\sfD(\Shv_{G,\Q}(X))$ is equivalent to the filtered colimit of the categories $\sfD(\Shv_{G_i,\Q}(X_i))$ where $G_i$ is a finite group such that $\lim_i G_i = G$ and $X_i$ is a finite $G_i$-space such that $\lim_i X_i \cong X$.

We will prove this at a pointset level, i.e., at the level of model structures, improving
on \cite[Theorem 11.6]{barnessugrue_weylsheaves} which considers the abelian level. 
Once we have shown that the construction is homotopically well-behaved this will provide us with the desired result for the derived category. We begin by defining the systems that we will be interested in. 

\begin{remark}\label{rem:modelcategoricalbugbears}
    We alert the reader that we will define a \emph{cofiltered} system as opposed to a filtered system. This is due to the fact that we will, instead of taking the filtered colimit over the left adjoints, form the cofiltered limit over the corresponding right adjoints. This allows us to use the well-understood theory of limits of model categories, such as that of Bergner  \cite{bergner}. When passing to the $\infty$-categorical realm, by \eqref{eq:prlprr}, we see that the limit constructions that we consider are nothing more than formulas for the appropriate colimits over the adjoint maps, which is what we are interested in.
\end{remark}

\begin{definition}\label{defn:compatible_system}
Let $I$ be a cofiltered category and suppose that we have an $I$-diagram of finite discrete groups and surjective group homomorphisms
\[
G_\bullet = (G_i, \phi_{ij}^G \colon G_i \to G_j )
\] 
and an $I$-diagram of finite discrete spaces and surjective maps
\[
X_\bullet = (X_i, \phi_{ij} \colon X_i \to X_j ).
\]
We say that \emph{$G_\bullet$ acts on $X_\bullet$} if each $X_i$ is a $G_i$-space and the 
following diagram commutes
\begin{equation}\label{eq:compatiblesquare}
\begin{gathered}
\xymatrix{
G_i \times X_i \ar[r] \ar[d]_{\phi_{ij}^G \times \phi_{ij}} &
X_i \ar[d]^{\phi_{ij}} \\
G_j \times X_j \ar[r] &
X_j \rlap{.}
}
\end{gathered}
\end{equation}
\end{definition}

The following lemma is fundamental to the theory.

\begin{lemma}
Let $G_\bullet$ act on $X_\bullet$ as in \cref{defn:compatible_system}. 
If $G=\lim_i G_i$ and $X=\lim_i X_i$, then $X$ has a continuous $G$-action.  
Moreover, the projection maps $p_i^G \colon G \to G_i$ and 
$p_i \colon X \to X_i$ fit into a commutative square
\begin{equation}\label{eq:compatiblesquare2}
\begin{gathered}
\xymatrix{
G \times X \ar[r] \ar[d]_{p_i^G \times p_i} &
X \ar[d]^{p_i} \\
G_i \times X_i \ar[r] &
X_i \rlap{.}
}
\end{gathered}
\end{equation}
\end{lemma}
\begin{proof}
    The topologies on $X$ and $G$ are the limit topologies. 
    Thus, a basic open set of $X$ is given by $p_i^{-1}(A_i)$ for 
    some open set $A_i \subseteq X_i$. We must show that the preimage
    of $p_i^{-1}(A_i)$ in $G \times X$ is open. Consider the preimage $B_i$ of $A_i$ in $G_i \times X_i$. 
    As the action of $G_i$ on $X_i$ is continuous, $B_i$ is open. 
    The maps $p_i^G$ and $p_i$ are continuous, so the preimage
    of $B_i$ in $G \times X$ is open. This is precisely  
    $p_i^{-1}(A_i)$.   
\end{proof}

An example of a system as in \cref{defn:compatible_system} is given by a profinite group $G=\lim_i G_i$ acting by conjugation on its space of subgroups $\Sub (G)$,  which is homeomorphic to $\lim_i \Sub (G_i)$ by 
 \cref{prop:profinitesubg}.

Given an action of $G_\bullet$ on $X_\bullet$, we would like to have a corresponding diagram of categories, $\Shv_{G_i,R}(X_i)$, indexed by $I$, where $G_i$ is the finite group $G/U_i$, for $U_i$ open normal in $G$. 
To do so, we will need to construct the maps in the system. In \cref{def:sheafinduction}, we introduced induction and restriction functors along subgroup inclusions; in this subsection we will need analogous functors for projections onto quotient groups.

To introduce the relevant functors for quotient groups, first recall that 
given a quotient map $\varepsilon \colon G \to G/U$ for $U$ open normal in $G$
and a $G$-space $X$, we denote the $U$-fixed points of $X$ by $X^{U}$.
This defines a functor from $G$-spaces to $G/U$-spaces.  
The left adjoint, inflation, equips a $G/U$-space $Y$ with the $G$-action
given by restriction along $\varepsilon$, denoted $\varepsilon^* Y$.

\begin{definition}\label{def:pushfixed}
Let $G$ be a profinite group with open normal subgroup $U$.
Let $X$ be a profinite $G$-space, $Y$ a profinite $G/U$-space, and 
$\phi \colon X \to \varepsilon^* Y$ a $G$-equivariant map to the inflation of $Y$.
We define 
\[
(\phi_U)_* \colon \Shv_{G,\Q}(X) \to \Shv_{G/U,\Q} (Y)
\]
as the functor which sends a $G$-sheaf $p \colon E \to X$ to the $G/U$-sheaf whose sheaf space is the
(stalkwise) $U$-fixed points of the pushforward sheaf $\phi_* E$.
\end{definition}

\begin{remark}
While the pushforward construction is standard (the sheafification of the presheaf 
pushforward), the fixed point functor 
is somewhat new in this context. We justify that this construction gives a functor as claimed.
Consider a $G$-sheaf $q \colon F \to \varepsilon^* Y$, for $Y$ a $G/U$-space. 
Any $y \in \varepsilon^* Y$ is $U$-fixed, so the stalk $F_y$ has an action of $U$.
Hence, one can consider the subspace $F^U$ of $F$ consisting of those 
points fixed under the action of $U$.  

The main step in proving that $F^U$ is a $G/U$-sheaf
is that $F^U \to F \to Y$ is a local homeomorphism.
Consider some $f \in F^U_y$, which is represented by a section
$s \colon V \to F$, for $V$ a compact open set (which exists as a profinite space has a compact open basis). By the local sub-equivariance condition
(see \cite[Corollary A]{barnessugrue_gsheaves}) we may assume that $s$ is $N$-equivariant
for some open subgroup $N$ of $G$, with $N \leqslant U$. 
Viewing $s$ as a function, for $u \in U$ we define a function $u \cdot s$ from $u V$ to $F$ by 
$v \mapsto u s(u^{-1}v)$. 
With that notation, for each $n \in N$, we have
\[
n \cdot s =s.
\]
For each $u \in U$, $u \cdot s$ and $s$ take value $f$ at $y$. 
Hence, these two sections agree on some neighbourhood of $y$.
As $U/N$ is finite, we can perform a finite sequence of restrictions and see that 
$s$ is $U$-equivariant on some $V' \subseteq V$.
As each point of $Y$ is fixed by $U$, 
\[
s= u \cdot s = u  s
\]
that is, for each $v \in V'$, $s(v)$ is $U$-fixed and so $s$ factors
as a map $s \colon V' \to F^U$. This gives a local inverse to the 
composite map $F^U \to F \to Y$.
\end{remark}

We now construct the adjoint functor to $(\phi_U)_\ast$.

\begin{definition}\label{def:pullinflation}
Let $G$ be a profinite group with open normal subgroup $U$.
Let $X$ be a profinite $G$-space, $Y$ a profinite $G/U$-space and 
$\phi \colon X \to \varepsilon^* Y$ be a $G$-equivariant map to the inflation of $Y$.
Define a functor 
\[
\phi_U^*  \colon \Shv_{G/U,\Q} (Y) \to \Shv_{G,\Q}(X)
\]
which sends a $G/U$-sheaf $q \colon F \to Y$ to the $G$-sheaf 
given by the pullback to $X$ along $\phi$ of the inflation 
\[
\varepsilon^* q  \colon \varepsilon^* F \to \varepsilon^* Y.
\]
\end{definition}

As the two functors we just defined are essentially composites of existing adjunctions
(with the context of the base spaces carefully chosen), it follows that they form an adjunction:

\begin{lemma}\label{lem:infresadjont}
Let $G$ be a profinite group with open subgroup normal $U$.
Let $X$ be a profinite $G$-space, $Y$ a profinite $G/U$-space, and 
$\phi \colon X \to \varepsilon^* Y$ be a $G$-equivariant map to the inflation of $Y$. Then there is an adjunction
\[
\phi^*_U
\colon 
\Shv_{G/U, \Q} (Y)
\rightleftarrows 
\Shv_{G, \Q} (X)
\noloc 
(\phi_U)_*.
\] 
Moreover, the left adjoint admits a simple formula for stalks: for any $x \in X$, we have
\begin{equation}\label{eq:simplestalks}
(\phi^*_U F)_x \cong \varepsilon^* F_{\phi(x)}.
\end{equation}
\end{lemma}

We now return to the situation at hand. We suppose that we have an action of $G_\bullet$ on $X_\bullet$. For $i \to j$ in $I$, we will abuse notation and write the adjoint pair of \cref{lem:infresadjont} induced by $\phi_{ij}$ and $\phi_{ij}^G$ as
\begin{equation}\label{eq:inducedonsheaves}
{\phi_{ij}^\ast} \colon 
\Shv_{G_j, \Q} (X_j)
\rightleftarrows 
\Shv_{G_i, \Q} (X_i)
\noloc 
({\phi_{ij}})_*.
\end{equation}
The adjunction of \eqref{eq:inducedonsheaves} then allows us to build the desired cofiltered system of equivariant sheaf categories:

\begin{definition}\label{defn:sysofsheaves}
Let $G_\bullet$ act on $X_\bullet$. Then a \emph{$G_\bullet$-sheaf of $\Q$-modules over $X_\bullet$} is 
a $G_i$-equivariant sheaf $D_i$ on $X_i$ of $\Q$-modules for each $i \in I$, 
with maps of sheaves of $\Q$-modules over $X_i$
\[
\alpha_{ij} \colon ({\phi}_{ij}^\ast) D_j \lra D_i
\]
for each $i \to j$ in $I$.  
A \emph{map $f_\bullet \colon D_\bullet \to D'_\bullet$ of $G_\bullet$-sheaves of $\Q$-modules over $X_\bullet$} consists of maps 
$f_i \colon D_i \to D_i'$ of $G_i$-sheaves of $\Q$-modules over $X_i$ that commute with the structure maps
\[
\xymatrix{
({\phi}_{ij}^\ast) D_j \ar[r]^-{\alpha_{ij}} \ar[d]_-{(\phi_{ij}^*) f_j} & 
D_i \ar[d]^-{f_i} \\
({\phi}_{ij}^\ast) D'_j \ar[r]_-{\alpha'_{ij}} & 
D'_i. \\
}
\]
We denote the resulting category as $\Shv_{G_\bullet,\Q} (X_\bullet)$. 
\end{definition}

\begin{remark}\label{rmk:sectionformula}
Note the contravariance introduced in \cref{defn:sysofsheaves}, the maps
of sheaves reverse the maps in the diagram $I$. 
It can be useful to write out how the functor ${{\phi}_{ij}}_\ast$ acts. 
For $V$ an open subset of $X_j$, 
\[
({{\phi}_{ij}}_\ast) D_i (V) =  D_i (\phi_{ij}^{-1} V)^{U_j/U_i}.
\]
Moreover, $U_j/U_i$ is a subgroup of $G/U_i=G_i$ which is
in the stabiliser of the set $\phi_{ij}^{-1} V$. 
\end{remark}

We now have all of the required setup to prove the desired continuity result. We suppose that we have the category $\sfD(\Shv_{G_i,\Q} (X_i))$ where $G = \lim G_i$ and $X = \lim X_i$. We wish to prove that the colimit of the categories $\sfD(\Shv_{G_i,\Q} (X_i))$ over the left adjoints is equivalent to the category $\sfD(\Shv_{G,\Q}(X))$. To be precise, we compute the colimit by passing to the corresponding right adjoints and then constructing the limit in the $\infty$-category of $\infty$-categories, which coincides with the limit taken in the $\infty$-category of presentable categories and right adjoints between them.

We will begin with an abelian level statement. Write $\lim \Shv_{G_\bullet,\Q} (X_\bullet)$ for the  full subcategory of $\Shv_{G_\bullet,\Q} (X_\bullet)$ consisting of those objects $D_\bullet$ whose adjoint structure maps
\begin{equation}\label{eq:adjointstruct}
\bar{\alpha}_{ij} \colon D_j \lra ({\phi}_{ij})_* D_i
\end{equation}
are isomorphisms. The following is a result of Barnes--Sugrue.

\begin{theorem}[{\cite[Theorem 11.6]{barnessugrue_weylsheaves}}]\label{thm:sheaves_continuous}
    Let $G_\bullet$ act on $X_\bullet$. Then there is a canonical equivalence of categories
    \[
    \lim \Shv_{G_\bullet,\Q} (X_\bullet) \xrightarrow[\quad]{\sim} \Shv_{G, \Q}(X),
    \]
    where the limit is taken over the right adjoints as in \eqref{eq:adjointstruct}.
\end{theorem}

\begin{remark}
    We warn the reader that in \cite{barnessugrue_weylsheaves}, the category $ \lim \Shv_{G_\bullet,\Q} (X_\bullet)$ is instead denoted by $\colim \Shv_{G_\bullet,\Q} (X_\bullet)$. This notation is not unreasonable given that in the derived version of this result (\cref{thm:eqsheaves_continuity}), we will use that the limit is equivalent to the colimit in presentable $\infty$-categories over the left adjoints. 
\end{remark}

We wish to make \cref{thm:sheaves_continuous} homotopical, giving us the desired result regarding the continuity of the derived category of rational $G$-equivariant sheaves. We begin by discussing that the relevant adjoint pairs are homotopical.

By \cref{lem:infresadjont} we have an adjunction $({\phi}_{i}^*,({\phi}_{i})_*)$  between $\Shv_{G_i,\Q} (X_i)$ and $\Shv_{G, \Q}(X)$ for each $i$. That result also gives the adjunction 
$({\phi}_{ij}^*,({\phi}_{ij})_*)$  between the 
categories $\Shv_{G_i,\Q} (X_i)$ and $\Shv_{G_j,\Q} (X_j)$ for varying $i$ and $j$.  We equip chain complexes over these categories with the model structure of \cref{prop:equivariantsheafmodel}, and we again can consider the adjoint pair $({\phi}_{ij}^*,({\phi}_{ij})_*)$, now between the categories of chain complexes.

\begin{lemma}\label{lem:canderivefunctors}
    The adjunction $({\phi}_{ij}^*,({\phi}_{ij})_*)$  between the model
categories $\Ch(\Shv_{G_i,\Q} (X_i))$ and $\Ch(\Shv_{G_j,\Q} (X_j))$ is a Quillen pair.
Moreover, the left adjoints ${\phi}_{ij}^*$ preserve compact objects.  
\end{lemma}
\begin{proof}
We note that all objects in the model categories are fibrant. 
By the formula of \cref{rmk:sectionformula}, we understand the behaviour of the right adjoints on sections at compact open sets. 
The weak equivalences and fibrations of the model structures can be phrased in  terms of quasi-isomorphisms and epimorphisms of chain complexes of sections at compact open sets by \cref{lem:epicompact} and \cref{prop:sectionsandstalks}. Combining these ideas, we see the right adjoints are right Quillen functors and hence the adjunctions are all Quillen adjunctions.   

The right adjoints commute with arbitrary coproducts by \cref{rmk:sectionformula}. 
As they also preserve all quasi-isomorphisms, the derived right adjoints commute with arbitrary coproducts 
at the level of homotopy categories. This implies that ${\phi}_{ij}^*$ preserves compacts.
\end{proof}

\begin{remark}\label{rem:limitmodelstructure}
    We will require the theory of limits of model categories in the proof of  \cref{thm:eqsheaves_continuity}. We recall the formalism from \cite{barwick_leftright,bergner} for the convenience of the reader. Given a collection $(\mathcal{C}_i)_i$ of (left proper, combinatorial) model categories and right Quillen functors $g_{ij} \colon \mathcal{C}_i \to \mathcal{C}_j$, we can form the category of right sections. We may then equip this section category with the projective model structure and take a left Bousfield localization so that fibrant objects $C$ in the localized model structure 
    are characterized by having (adjoint) structure maps $\overline{\alpha}_{ij} \colon C_i \to (g_{ij})_\ast C_j$ be weak equivalences. 
    For concreteness we denote this model structure $\overline{\lim}_i \mathcal{C}_i$. The notation for this model structure is not unwarranted. We can consider the $\infty$-category associated to this model structure, $\mathsf{N}(\overline{\lim}_i \mathcal{C}_i)$. This admits a comparison map to the limit of the associated $\infty$-categories, that is, $\lim_i \mathsf{N}(\mathcal{C}_i)$. This comparison map is an equivalence of $\infty$-categories, as proved in \cite[Theorem 4.1]{bergner}, so the model categorical limit as defined here models the $\infty$-categorical limit.
\end{remark}

\begin{theorem}\label{thm:eqsheaves_continuity}
     Let $G_\bullet$ act on $X_\bullet$. Then there is a canonical geometric equivalence
         \[
        \colim_{i}^\omega\sfD(\Shv_{G_i,\Q} (X_i)) \xrightarrow[\quad]{\sim} \sfD(\Shv_{G, \Q}(X)). 
    \]
\end{theorem}

\begin{proof}
Consider the filtered diagram of $\infty$-categories consisting of the terms $\sfD(\Shv_{G_i,\Q} (X_i))$ and the derived functors of the left adjoints relating them from \cref{lem:canderivefunctors}.  As the left Quillen functors of \cref{lem:canderivefunctors} are strong monoidal by the stalk formula \eqref{eq:simplestalks} of 
\cref{lem:infresadjont}, this promotes to a diagram in the $\infty$-category $\CAlg(\Pr^L)$ of presentably symmetric monoidal $\infty$-categories and symmetric monoidal left adjoints. 

We can consider the colimit of this diagram in $\CAlg(\Pr^L)$, which for clarity we denote by $\colim^{\CAlg(\Pr^L)}_i \sfD(\Shv_{G_i,\Q} (X_i))$. By the universal property of the colimit, we thus obtain a symmetric monoidal functor of presentable $\infty$-categories
    \[
    \phi \colon \colim_{i}^{\CAlg(\Pr^L)}\sfD(\Shv_{G_i,\Q} (X_i)) \xrightarrow[\quad]{} \sfD(\Shv_{G, \Q}(X)).
    \]
On the one hand, as each category appearing in the diagram is compactly generated, and as the functors are compact object preserving by \cref{lem:canderivefunctors}, our diagram in fact lives in $\CAlg(\Pr^{L, \omega})$. By \eqref{eq:colimitcomparisons} we therefore conclude that 
    \[
        \colim^{\omega}_i \sfD(\Shv_{G_i,\Q} (X_i)) \simeq \colim^{\CAlg(\Pr^L)}_i \sfD(\Shv_{G_i,\Q} (X_i))
    \]
symmetric monoidally. On the other hand, since colimits in $\CAlg(\Pr^L)$ may be computed in $\Pr^L$, it suffices to show that $\phi$ is an equivalence of presentable $\infty$-categories to prove the desired result.

We now appeal to the duality between $\Pr^L$ and $\Pr^R$ as discussed in \eqref{eq:prlprr}. That is, we observe that $\colim_{i}^{\Pr^{L}}\sfD(\Shv_{G_i,\Q} (X_i))$ is equivalent to the limit over the corresponding adjoint structure maps in $\Pr^R$. Finally by \eqref{eq:colimitcomparisons} again, we conclude that this limit is equivalent to taking the limit in the $\infty$-category $\widehat{\Cat}_\infty$ of $\infty$-categories. In summary, we have reduced to showing that the adjoint functors $(\phi_i)_*$ induce an equivalence
    \[
    \sfD(\Shv_{G, \Q}(X)) \xrightarrow{\psi} \lim_{i } \sfD(\Shv_{G_i,\Q} (X_i)), 
    \]
where the limit is taken in $\widehat{\Cat}_\infty$. This functor factors as in the following commutative diagram:
    \[
    \xymatrix{
    \mathsf{N}(\overline{\lim}_i \Ch(\Shv_{G_i,\Q} (X_i)))
     \ar[d]_{\gamma_2} & 
    \sfD(\Shv_{G, \Q}(X)) \ar[dl]^{\psi} \ar[l]_-{\gamma_1} \\
    \lim_{i } \sfD(\Shv_{G_i,\Q} (X_i)). 
    }  
    \]
Our strategy is then to prove that $\gamma_1$ and $\gamma_2$ are equivalences.

Following \cref{rem:limitmodelstructure}, we have the limit model structure $\overline{\lim}_i \Ch(\Shv_{G_i,\Q} (X_i))$. The result of Bergner (\cite[Theorem 4.1]{bergner}) provides us with the fact that $\gamma_2$ is an equivalence. It remains to prove that $\gamma_1$ is an equivalence. We prove the existence of the following Quillen equivalence
and take $\gamma_1$ to be the functor induced by $(\phi_i)_*$,
\[
\colim_i \phi_{i}^* 
\colon 
\overline{\lim}_i \Ch (\Shv_{G_i,\Q} (X_i))
\rightleftarrows 
\Ch (\Shv_{G, \Q} (X))
\noloc 
(\phi_i)_*.
\] 
The adjunction is that of \cref{thm:sheaves_continuous}. Explicitly, given an object $D$ in $\overline{\lim}_i \Ch (\Shv_{G_i,\Q} (X_i))$, the functor $\colim_i \phi_{i}^* $ forms the colimit of $D$ over the diagram of the \emph{left} adjoints.

For each $i$, we have the Quillen adjunction $({\phi}_{i}^*,({\phi}_{i})_*)$, 
hence $(\phi_i)_*$ gives a right Quillen functor into the projective model structure.
It passes to a right Quillen functor to the limit model structure as
for any $A \in \Ch(\Shv_{G, \Q} (X))$, 
the adjoint structure maps of $(\phi_i)_*(A)$ are isomorphisms 
by \cite[Construction 11.4]{barnessugrue_weylsheaves}. 

We show that the derived counit and unit are equivalences. 
Note that every object in $\Ch (\Shv_{G, \Q} (X))$ is fibrant, so that we do not need to 
take fibrant replacements before applying  $(\phi_i)_*$.
The formula for stalks of \cref{lem:infresadjont} implies that 
the functor $\colim_i \phi_{i}^*$ sends projective equivalences (quasi-isomorphisms at each $i$) to quasi-isomorphisms. 
As the functor preserves weak equivalences, we do not need to take cofibrant replacements before applying  $\colim_i \phi_{i}^*$ to compute the derived functor.

The counit of the adjunction is an isomorphism, as stated in \cite[Theorem 11.6]{barnessugrue_weylsheaves}.
The proof of the cited theorem explicitly shows that the unit of the adjunction
is an isomorphism at $C$ when the (adjoint) structure maps of $C$ are isomorphisms. 
We follow that proof line by line for an object $C$ in $\Ch (\Shv_{G_i,\Q} (X_i))$ whose 
(adjoint) structure maps $\overline{\alpha}_{ij} \colon C_i \to (\phi_{ij})_\ast C_j$ are quasi-isomorphims.
Using the facts that homology commutes with filtered colimits and 
fixed points are exact by  \cref{rem:fixedorbitequivalence}, the proof extends to 
show that the unit, and hence derived unit, is a quasi-isomorphism at the chosen object $C$. 
Since fibrant objects in the limit model structure on $\overline{\lim}_i \Ch (\Shv_{G_i,\Q} (X_i))$
are precisely those whose adjoint structure maps are quasi-isomorphims, the unit of the adjunction
is a quasi-isomorphism.
\end{proof}

\section{The algebraic model}\label{sec:qgsp_algmodel}

In \cref{sec:equivariantsheaves} we introduced the derived category of $G$-sheaves over a profinite space, and proved that its construction is continuous in the appropriate manner. In this section, we will build upon these results to construct an algebraic model for rational $G$-spectra for $G$ a profinite group at the level of symmetric monoidal $\infty$-categories, extending results of Barnes--Sugrue~\cite{barnessugrue_spectra}, and thereby establishing a strong form of the profinite counterpart of Greenlees's conjecture for compact Lie groups \cite{greenleesreport}. The starting point is the following result of Wimmer:

\begin{theorem}[{\cite[Theorem 1.1]{wimmer_model}}]\label{thm:finitegroups_algmodel}
    Let $G$ be a finite group. Then the rational geometric fixed point functors $\Phi^H$ (\cref{rem:qgsp_spc}) assemble into an equivalence 
        \[
        \xymatrix{\Phi = \Phi_G \coloneqq (\Phi^H)\colon \Sp_{G,\Q} \ar[r]^-{\sim} & \prod_{(H) \in \Sub(G)/G} \mathsf{D}(\Mod_{\Q}(W_G(H)))}
        \]
    of symmetric monoidal stable $\infty$-categories. If $\sfP(H) = (\Phi^H)^{-1}(0) \in \Spc(\Sp_{G,\Q}^{\omega})$ is the prime tt-ideal corresponding to $H \in \Sub(G)/G$, then this equivalence restricts to a geometric equivalence
        \begin{equation}\label{eq:finitegroup_stalks}
            \xymatrix{\Phi^H\colon (\Sp_{G,\Q})_{\sfP(H)} \ar[r]^-{\sim} & \mathsf{D}(\Mod_{\Q}(W_G(H)) )}
        \end{equation}
    on tt-stalks.
\end{theorem}

We will generalize \cref{thm:finitegroups_algmodel} to profinite groups by using a continuity argument. By  \cref{thm:gsp_contfausk} we can describe the model of $\Sp_G$ for $G$ profinite as a colimit of the categories
\[
\prod_{(H) \in \Sub(G_i)/G_i} \mathsf{D}(\Mod_{\Q}(W_{G_i} (H)))
\]
as $G_i$ runs over the finite quotient groups of $G$. It is reasonable to ask for a more direct description of this category with a more algebraic nature, and this is where our exploration of $G$-equivariant sheaves from \cref{sec:equivariantsheaves} is required.

We shall see that a direct description is obtained as a subcategory of $\sfD(\Shv_{G,\Q}(\Sub(G)))$ on those sheaves which satisfy a certain Weyl condition. In \cref{thm:sheafalgebraicmodel} we will prove that these \emph{Weyl-$G$-sheaves}, which we will explicitly introduce in \cref{ssec:algmodel_weylsheaves}, provide a canonical equivalence of  symmetric monoidal stable $\infty$-categories
        \[
        \xymatrix{\Sp_{G,\Q} \ar[r]^-{\sim} & \mathsf{D} (\WeylSheaves_{G, \Q} (\Sub (G))).}
        \]

\subsection{Preliminaries on rationalization}\label{ssec:rationalization}

We begin by collecting some basic facts about the rationalization of the category of $G$-equivariant spectra, for $G$ a profinite group. We continue to take a presentation of $G$ as an inverse limit of finite quotients groups $(G_i)_{i \in I}$. The category of \emph{rational $G$-equivariant spectra} is then defined as the Verdier localization of $\Sp_G$ away from the thick subcategory generated by the equivariant Moore spectra $M_G(p) \coloneqq S_G/p$ for all primes $p$:
    \[
        \Sp_{G,\Q} \coloneqq \Sp_{G}/\langle \{M_G(p) \mid p \text{ prime}\} \rangle.
    \]
Note that each $M_G(p)$ is inflated from the non-equivariant Moore spectrum $M(p) = S/p$. We are thus in a position to apply \cref{prop:finitelocalizations_continuity}, showing that the inflation functors yield a geometric equivalence
\begin{equation}\label{eq:qgsp_continuity}
\xymatrix{\Sp_{G,\Q} \ar[r]^-{\simeq} & \colim_i^{\omega}\Sp_{G_i,\Q}.}
\end{equation}
Continuity of the spectrum (\cref{prop:spccontinuity}) thus supplies a homeomorphism
    \begin{equation}\label{eq:qgsp_spcccontinuity}
        \xymatrix{\iota_{G,\Q}\colon\Spc(\Sp_{G,\Q}^{\omega}) \ar[r]^-{\cong} & \lim_{i}\Spc(\Sp_{G_i,\Q}^{\omega}),}
    \end{equation}
and our goal for the remainder of this subsection is to identify this spectrum with $\Sub(G)/G$. To this end, let us review the rational Burnside ring for profinite groups, following Dress \cite{dress_notes}. We write $S_{G,\Q}$ for the rational $G$-equivariant sphere spectrum, which is equivalent to the inflation of its non-equivariant version, $S_{G,\Q} \simeq \infl S_{\Q}$. Since inflation commutes with colimits, we have $\Q \otimes S_{G} \simeq S_{G,\Q}$.

\begin{definition}\label{def:qburnsidering}
For $G$ a profinite group, we write $A_{\Q}(G) \coloneqq \Q \otimes A(G) \cong \pi_0\End_{\Sp_{G,\Q}}(S_{G,\Q})$ for the \emph{rational Burnside ring}, where the integral Burnside ring $A(G)$ was defined in \cref{rem:burnsidering}.
\end{definition}

The rational Burnside ring of $G$ is determined by $\Sub(G)/G$ in a way that we now discuss. Write $C(X,R)$ for the set of continuous functions from a space $X$ with values in a topological commutative ring $R$. Pointwise addition and multiplication combined with the diagonal map of $X$ endow $C(X,R)$ with the structure of a commutative ring. Let $k$ be a field equipped with the discrete topology, so that continuous functions $X \to k$ are locally constant. If $X$ is a profinite space, then there is a canonical homeomorphism
\[
\Spec(C(X,k)) \cong X,
\]
see for example \cite[Example 12.1.13(v)]{book_spectralspaces}.

\begin{proposition}[{\cite[Appendix B, Theorem 2.3(a)]{dress_notes}}]\label{prop:ratburnsidering}
If $G$ is a profinite group, then the projection maps $G \to G_i$ induce natural isomorphisms
    \[
    A_{\Q}(G) \cong \colim_iA_{\Q}(G_i) \cong \colim_iC(\Sub(G_i)/G_i,\Q) \cong C(\Sub(G)/G,\Q),
    \]
where $\Q$ carries the discrete topology. Consequently, we have an induced homeomorphism of profinite spaces
    \begin{equation}\label{eq:rationalburnsidering}
        \Spec(A_{\Q}(G)) \cong \Sub(G)/G.
    \end{equation}
\end{proposition}

The next result gives the desired identification of the spectrum of $\Sp_{G,\Q}^{\omega}$ via Balmer's comparison map (\cref{rem:comparisonmaps}).

\begin{proposition}\label{prop:qgsp_spc}
    For any profinite group $G$, the comparison map induces a homeomorphism
    \[
        \xymatrix{\Spc(\Sp_{G,\Q}^{\omega}) \ar[r]_-{\rho_G}^-{\cong} & \Spec(A_{\Q}(G)) \cong \Sub(G)/G.}
    \]
\end{proposition}
\begin{proof}
    Applying \Cref{cor:continuous_comparisonmaps} to \eqref{eq:qgsp_continuity} supplies a commutative square
    \[
        \xymatrixcolsep{4pc}{
        \xymatrix{\Spc(\Sp_{G,\Q}^{\omega}) \ar[r]^-{\rho_G} \ar[d]_{\cong} & \Spec(A_{\Q}(G)) \ar[d]^{\cong} \\
        \lim_i\Spc(\Sp_{G_i,\Q}^{\omega}) \ar[r]_-{\lim_i \rho_{G_i}} & \lim_i\Spec(A_{G_i,\Q}).}}
    \]
    For each $i\in I$, the comparison map $\rho_{G_i}$ is a homeomorphism between finite discrete spaces by \cite{greenlees_may,wimmer_model}, so it follows that $\rho_G$ is a homeomorphism as well. Finally, the identification of $\Spec(A_{\Q}(G))$ with $\Sub(G)/G$ has already been observed in \eqref{eq:rationalburnsidering}. 
\end{proof} 

\begin{remark}\label{rem:qgsp_spc}
    Following the discussion of \cref{ssec:gfp}, the geometric fixed point functors on $\Sp_G$ induce $\Q$-linear geometric fixed points
        \[
            \xymatrix{\Phi_{G,\Q}^H\colon \Sp_{G,\Q} \ar[r] & \Sp_{\Q}}
        \]
    for any closed subgroup $H$ of $G$. These functors inherit the same convenient properties as their integral counterparts; in particular, they naturally take values in (naive) $W_G(H)$-equivariant rational spectra. By \cref{thm:finitegroups_algmodel} and the continuity of rational $G$-spectra \eqref{eq:qgsp_continuity} as well as \cref{prop:profinitesubg}, the rational geometric fixed points then induce a continuous bijection
    \[
        \xymatrix{\varphi_{G,\Q} = (\Spc(\Phi_{G,\Q}^H))_H \colon \Sub(G)/G \ar[r]^-{\cong} & \Spc(\Sp_{G,\Q}^{\omega}), \quad H \mapsto (\Phi_{\Q}^{G,H})^{-1}(0),}
    \]
    which is necessarily a homeomorphism, because both sides are profinite spaces. One can check that the composite $\rho_{G} \circ \varphi_{G,\Q}$ is the identity.

    When the context is clear, we will keep the rationalization or ambient group implicit and write $\Phi_{\Q}^H$ or just $\Phi^H$ for the functor $\Phi_{G,\Q}^H$, respectively.
\end{remark}

\begin{corollary}\label{cor:ttideal_classification}
    For any profinite group $G$, support induces a bijection
        \[
            \begin{Bmatrix}
                \text{Thick ideals} \\ \text{of } \Sp_{G,\Q}^\otimes
            \end{Bmatrix} 
                \xymatrix@C=2pc{ \ar@<0.75ex>[r]^{\supp} &  \ar@<0.75ex>[l]^{\supp^{-1}}}
            \begin{Bmatrix}
                \text{Open subsets} \\ \text{of } \Spec(A_{\Q}(G))
            \end{Bmatrix}.
        \]
\end{corollary}
\begin{proof}
    By \cref{prop:qgsp_spc} and the defining property of the spectrum \eqref{eq:suppttclassification}, support composed with the comparison map $\rho$ provides a bijection between thick ideals of $\Sp_{G,\Q}^\omega$ and Thomason subsets of $\Spec(A_{\Q}(G))$. Since the latter space is zero-dimensional and thus profinite, the Thomason subsets coincide with the opens, see \cref{rem:openisThom}.
\end{proof}

\subsection{Weyl-equivariant sheaves}\label{ssec:algmodel_weylsheaves}

Now that we have discussed the category $\Sp_{G,\Q}$ for $G$ a profinite group, we will, in this subsection, introduce the relevant pieces needed for construction of the aforementioned algebraic model.

Recall from \cref{rem:qgsp_spc} that the geometric fixed point functor $\Phi^H_\Q$ for $H \leqslant G$ naturally takes values in naive $W_G(H)$-equivariant spectra.  We will need to record this Weyl action on the algebraic side, and to this end, we will require the notion of \emph{Weyl-$G$-sheaves} as introduced in \cite{barnessugrue_weylsheaves} and \cite{barnessugrue_spectra}. These sheaves are a particular subclass of rational $G$-equivariant sheaves over the profinite $G$-space $\Sub(G)$. In  \cite{barnessugrue_weylsheaves}, the category of Weyl-$G$-sheaves is shown to be equivalent to the category of rational $G$-Mackey functors and hence one concludes that chain complexes of Weyl-$G$-sheaves are Quillen equivalent to Fausk's model category of rational $G$-spectra. 

In this subsection, we introduce the category of Weyl-$G$-sheaves, and similar to $\Shv_{G,\Q}(X)$, show that it is a Grothendieck abelian category with enough injectives and projectives, and equip it with a monoidal model structure by using the techniques developed in \cref{sec:equivariantsheaves}. Let us begin with the definition.

\begin{definition}\label{defn:weylsheaf}
Suppose $G$ is a profinite group. A \emph{rational Weyl-$G$-sheaf} $E$ is a $G$-sheaf of $\Q$-modules over the $G$-space $\Sub (G)$ such that the stalk $E_K$ is $K$-fixed for all $K \in \Sub(G)$. 
  A map of Weyl-$G$-sheaves is a map of $G$-sheaves of $\Q$-modules over $\Sub (G)$.
We denote this category $\WeylSheaves_{G,\Q}(\Sub(G))$. 
\end{definition}

\begin{definition}
    Let $E$ be a rational $G$-sheaf over $\Sub (G)$. 
A \emph{Weyl section} of $E$
is a section $s \colon U \to E$ such that for each $K \in U$, $s_K$ is $K$-fixed.
\end{definition}

The definitions directly imply that every section of a Weyl-$G$-sheaf is a Weyl section. Moreover, by \cite[Corollary 10.9]{barnessugrue_gsheaves}, for any $G$-equivariant sheaf $E$, 
a $K$-fixed germ in $E_K$ can be represented by a Weyl section. 
This fact implies that the subspace of a $G$-sheaf $E$ 
given by the union of $E_K^K$ over all closed subgroups $K$, satisfies 
the local homeomorphism condition and is a Weyl-$G$-sheaf. 
This construction is part of an adjunction:

\begin{lemma}[{\cite[Lemma 10.11]{barnessugrue_gsheaves}}]\label{lem:weyladjunction}
Let $G$ be a profinite group. There is an adjunction
\[
\incfunctor
\colon 
\WeylSheaves_{G, \Q} (\Sub(G))
\rightleftarrows 
\Shv_{G, \Q} (\Sub(G))
\noloc 
{\weylfunctor},
\] 
where we have followed the usual convention of displaying the left adjoint on the top. Moreover, the counit of this adjunction
\[
\incfunctor (\weylfunctor(E)) \lra E
\]
is a monomorphism of rational $G$-sheaves,
and the unit map 
\[
F \lra \weylfunctor(\incfunctor F) 
\]
is a natural isomorphism. 
\end{lemma}

\begin{remark}
By \cite[Theorem 10.13]{barnessugrue_gsheaves}, the category of Weyl-$G$-sheaves is an abelian category with all small colimits and limits. Since the constant sheaf at $\Q$ is Weyl, and the tensor product of two Weyl-$G$-sheaves 
is also a Weyl-$G$-sheaf, the monoidal structure on equivariant sheaves 
descends to the category of Weyl-$G$-sheaves. The inclusion functor $\incfunctor$ of \cref{lem:weyladjunction} is a strong monoidal functor
and is fully faithful by the last statement of the lemma. 
Moreover, $\WeylSheaves_{G, \Q} (\Sub(G))$ is a closed monoidal category, with 
internal function object given by
\[
\weylfunctor (\underline{\hom} (\incfunctor -, \incfunctor - )). 
\]
\end{remark}

We prove that the functor $\weylfunctor$ is also left adjoint to
$\incfunctor$. The key result is a sheaf-level analogue of 
 \cref{rem:fixedorbitequivalence} below that allows us to construct 
Weyl-$G$-sheaves in terms of orbits. 

\begin{lemma}\label{lem:quotientweylsheaf}
Let $E$ be a rational $G$-sheaf over $\Sub(G)$. 
Let $\weylfunctor_Q E$ be the space given by 
\[
\coprod_{K \in \Sub(G)} (E_K)/K
\]
equipped with the quotient topology and with the induced map to $\Sub(G)$. Then ${\weylfunctor_Q} {E}$ is a Weyl-$G$-sheaf over $\Sub(G)$ with a natural map $E \to \weylfunctor_Q E$. Furthermore, the composite $\weylfunctor E \to E \to \weylfunctor_Q E$ is an isomorphism. Hence, $\weylfunctor E$ is a retract of $E$.
\end{lemma}
\begin{proof}
One can verify directly that $\weylfunctor_Q E \to \Sub(G)$ is a 
rational $G$-equivariant sheaf and that the claimed map exists. 

For the isomorphism statement one looks stalkwise and applies 
 \cref{rem:fixedorbitequivalence} to see that rationally, 
fixed points and orbits are isomorphic. The retraction statement follows. 
\end{proof}

Looking at stalks gives the first statement of the next lemma.
\begin{lemma}
Let $F$ be a Weyl-$G$-sheaf and $E$ any rational $G$-sheaf. 
A map of sheaves $E \to \incfunctor F$ must factor uniquely through the map from $E$ to
$\weylfunctor_Q E \cong \weylfunctor E$.

Hence, there is an adjunction
\[
\weylfunctor
\colon 
\Shv_{G, \Q} (\Sub(G))
\rightleftarrows 
\WeylSheaves_{G, \Q} (\Sub(G))
\noloc 
\incfunctor.
\] 
\end{lemma}

With our ambidextrous adjunction in hand, we can begin to prove the results that we require. The following proposition collects several important facts about the category $\WeylSheaves_{G, \Q} (\Sub(G))$. In particular, it provides analogues of the results of \cref{ssec:eqsheaves_abelian} and \cref{ssec:eqsheaves_derived} that we proved for the category of all equivariant sheaves.

\begin{proposition}\label{prop:weylsheaves_properties}
For $G$ a profinite group, the category of Weyl-$G$-sheaves:
$\WeylSheaves_{G, \Q} (\Sub(G))$, 
\begin{enumerate}
    \item has enough injectives, given by taking the equivariant skyscraper sheaves 
which are Weyl; 
    \item has enough projectives, given by applying the functor $\weylfunctor$ to the 
sheaves $\const_{G \cdot Y}(N)$ from  \cref{def:freeconstantsheaf};
    \item is a Grothendieck abelian category;
    \item has a monoidal projective model structure on its category of chain complexes, which is finitely generated, has a 
cofibrant unit, and whose weak equivalences are the quasi-isomorphisms, which may be characterized as the stalkwise homology isomorphisms;
    \item has an associated derived category $\sfD(\WeylSheaves_{G, \Q} (\Sub(G)))$, which is a compactly generated symmetric monoidal stable $\infty$-category.
\end{enumerate}
\end{proposition}
\begin{proof}
The first point is Corollary \cite[10.15]{barnessugrue_gsheaves}.
For the second, we note that $\weylfunctor (\const_{G \cdot Y}(N))$ represents the functor 
$E \to E(U)^{N \cap G_U}$ (where $G_U$ is the stabiliser of $U$ in $G$) 
from Weyl-$G$-sheaves to $\Q$-modules with a discrete action of $G_U / N \cap G_U$. 

The third point follows from  \cref{thm:gshvgrothendieck} using 
the objects $\weylfunctor (\const_{G \cdot Y}(N))$ to construct a family of generators. 
For the fourth point, we construct the projective model structure using 
the generating sets 
\begin{align*}
\weylfunctor I &= \{ (S^{n-1}\Q \to D^n \Q) \otimes \weylfunctor (\const_{G \cdot Y}(N)) \mid n \in \Z, \ Y \subseteq X 
\textnormal{ open compact}, N \trianglelefteqslant G \textnormal{ open}   \}, \\
\weylfunctor J &= \{ (0 \to D^n \Q) \otimes \weylfunctor (\const_{G \cdot Y}(N)) \mid n \in \Z, \ Y \subseteq X 
\textnormal{ open compact}, N \trianglelefteqslant G \textnormal{ open}   \}, 
\end{align*}
and apply the same argument as in  \cref{ssec:eqsheaves_derived}. In particular the weak equivalences are again the 
quasi-isomorphisms and the fibrations are the stalkwise epimorphisms. Moreover, 
quasi-isomorphisms are characterized as being stalkwise homology isomorphisms.

To see that the model structure is monoidal, consider the pushout product $\weylfunctor i \square \weylfunctor i'$
for $i, i' \in I$, which is a Weyl-$G$-sheaf as well. As $\weylfunctor i$ is a retract of $i$ by \cref{lem:quotientweylsheaf}, it is a cofibration in 
the projective model structure on rational $G$-sheaves over $\Sub(G)$. 
As that model structure is monoidal, $\weylfunctor i \square \weylfunctor i'$ 
is a cofibration in the projective 
model structure on rational $G$-sheaves $\Sub(G)$. The functor 
\[
\weylfunctor \colon 
\Ch (\Shv_{G, \Q} (\Sub(G)))
\longrightarrow
\Ch (\WeylSheaves_{G, \Q} (\Sub(G)))
\]
is a left Quillen functor, as witnessed by the choice of generating sets, hence
\[
\weylfunctor (\weylfunctor i \square \weylfunctor i') \cong
\weylfunctor i \square \weylfunctor i'
\]
is a cofibration in the projective model structure on Weyl-$G$-sheaves. 
Finally, the unit is the constant sheaf at $\Q$, which is cofibrant.

The final point of the proposition follows from the properties of the model structure proved in the fourth point using the observations of \cref{prop:gsheafhocompact}.
\end{proof}

One may observe that the $(\weylfunctor, \incfunctor)$ adjoint pair interacts well with the homotopy theory: indeed, in the proof of \cref{prop:weylsheaves_properties} we have already used the fact that $\Weyl$ is left Quillen. We obtain the following corollary.

\begin{corollary}
There are adjoint pairs $(\weylfunctor, \incfunctor)$ and 
$(\incfunctor,\weylfunctor)$ between $\sfD( \Shv_{G, \Q} (\Sub(G)))$ and 
$\sfD( \WeylSheaves_{G, \Q} (\Sub(G)))$. Moreover the inclusion functor $\incfunctor$ is strong monoidal.
\end{corollary}

\subsection{Constructing the equivalence}\label{ssec:algmodel_equivalence}

In this section we will construct the desired equivalence to provide an algebraic model for rational $G$-spectra via Weyl-$G$-sheaves.

\begin{remark}\label{rem:basechangeWeyl}
The base change functors of \cref{subsec:basechanged} can be extended to the Weyl setting by \cite[Construction 11.2]{barnessugrue_weylsheaves}. That is, whenever we have a surjection $\phi_{ij} \colon G_i \to G_j$ of finite quotient groups of $G$, we have an adjoint pair
\begin{equation}\label{eq:inducedonsheaves2}
{\phi_{ij}^\ast} \colon 
\WeylSheaves_{G_j, \Q} (\Sub(G_j))
\rightleftarrows 
\WeylSheaves_{G_i, \Q} (\Sub(G_i))
\noloc 
({\phi_{ij}})_*.
\end{equation}
This functor changes the base space via a pullback $\Sub(G_i) \to \Sub(G_j)$ and acts as an inflation functor on the stalks. 
\end{remark}

In preparation for proving the existence of the algebraic model, we highlight that there is a version of \cref{thm:eqsheaves_continuity} restricted to Weyl sheaves.

\begin{proposition}\label{thm:weyl_sheaves_continuous}
    Let $G = \lim_i G_i$ be a profinite group. Then there is a geometric equivalence
        \[
            \colim_{i}^\omega\sfD(\WeylSheaves_{G_i,\Q} (\Sub(G_i))) \xrightarrow[\quad]{\sim} \sfD(\WeylSheaves_{G, \Q}(\Sub(G))),
        \]
    where the transition maps are the inflation maps of \cref{rem:basechangeWeyl}.
\end{proposition}

\begin{proof}
The diagram of groups $G_\bullet$ acts on $\Sub(G_\bullet)$ in the appropriate sense so that we may form the required system. The result is then obtained using a Weyl sheaf version of \cref{thm:eqsheaves_continuity} which follows from the theory there, keeping in mind \cite[Corollary 11.8]{barnessugrue_weylsheaves} which provides the abelian version of the statement.
\end{proof}

In the case that $G$ is a finite group, we can relate \cref{thm:finitegroups_algmodel} to Weyl sheaves by choosing one representative of each conjugacy class
of $H$ in $\Sub(G)$:
\[
\sfD(\WeylSheaves_{G,\Q}(\Sub(G))) 
\simeq 
\prod_{(H) \in \Sub(G)/G} \sfD(\Mod_{\Q}(W_{G} (H))).
\]

We extend this to profinite groups and obtain an algebraic model for rational $G$-spectra in terms of derived categories of Weyl-$G$-sheaves by using the continuity result of \cref{thm:weyl_sheaves_continuous}. 

\begin{theorem}\label{thm:sheafalgebraicmodel}
Let $G$ be a profinite group. There is a geometric equivalence of $\infty$-categories
\[
\xymatrix{\Phi\colon \Sp_{G, \Q} \ar[r]^-{\sim} & \mathsf{D} (\WeylSheaves_{G,\Q}(\Sub(G))).}
\]
\end{theorem}
\begin{proof}
We prove that there are geometric equivalences
\begin{align}\label{eq:sheafalgebraicmodel}
    \begin{split}
\Sp_{G, \Q} 
& \simeq \colim_{i}^\omega \Sp_{G_i, \Q} \\
& \simeq \colim_{i}^\omega \mathsf{D} (\WeylSheaves_{G_i,\Q}(\Sub(G_i))) \\ 
& \simeq \mathsf{D} (\WeylSheaves_{G,\Q} (\Sub(G))). 
    \end{split}
\end{align}

The first equivalence of \eqref{eq:sheafalgebraicmodel} is \eqref{eq:qgsp_continuity}. The second step requires us to show that the two colimit systems are compatible. Suppose that we have finite quotients $G_i$ and $G_j$ of $G$ with a surjection $\phi_{ij} \colon G_i \to G_j$. Write $f_{ij}^*$ for the inflation functor on equivariant spectra from \eqref{eq:gsp_restriction}, and ${\phi}_{ij}^*$ for the functor of Barnes--Sugrue as discussed in \cref{rem:basechangeWeyl}. Consider the following square, in which $\Phi_{G_i}$ and $\Phi_{G_j}$ are the geometric equivalences of \cref{thm:finitegroups_algmodel}: 
\begin{equation}\label{dia:algmodsquare}
\begin{gathered}
\xymatrix@C=4em{
\Sp_{G_j, \Q}  \ar[r]^-{\Phi_{G_j}} \ar[d]_{f_{ij}^*} &
\mathsf{D} (\WeylSheaves_{G_j,\Q}(\Sub(G_j))) \ar[d]^{\phi_{ij}^*} \\
\Sp_{G_i, \Q}  \ar[r]_-{\Phi_{G_i}} &
\mathsf{D} (\WeylSheaves_{G_i,\Q}(\Sub(G_i)))
.}
\end{gathered}
\end{equation}
Write $G_i = G/U_i$ and let $K$ be an open normal subgroup of $G$ containing $U_i$, so that $K_i = KU_i/U_i = K/U_i$. The functor $\Phi_{G_i}$ sends a $G_i$-spectrum $X$ to the sheaf $\Phi_{G_i}(X)$ that at $K_i$ takes value $\Phi^{K_i} (X)$. As the category of (non-equivariant) rational spectra is equivalent to $\mathsf{D} (\Q)$, we may interpret the sheaf $\Phi^{K_i} (X)$ as an object of $\mathsf{D} (\Q)$ with Weyl group action as in \cref{rem:qgsp_spc}. 

We wish to prove that \eqref{dia:algmodsquare} commutes, and hence conclude that the two colimits coincide as the $\Phi_{G_i}$ are geometric equivalences by \cref{thm:finitegroups_algmodel}. To do so we can check equivalence of stalks: indeed, $\Sub(G_i)$ is a finite discrete $G_i$-space, and hence a sheaf over it is determined by its stalks. We obtain the following equivalences (omitting an algebraic inflation functor on the third and fourth terms)
    \begin{align}\label{eq:showingsquarecommutes}
        \begin{split}
            \left( \Phi_{G_i} ( f_{ij}^* X) \right)_{K_i}
            &\cong \Phi^{K_i} ( f_{ij}^* X) \\
            &\cong \Phi^{KU_j/U_j} (X) \\
            &\cong \left( \Phi_{G_j} (X) \right)_{KU_j/U_j} \\
            &\cong \left( \phi_{ij}^* (\Phi_{G_j} X) \right)_{K_i}.
        \end{split}
    \end{align}
The last equivalence of \eqref{eq:showingsquarecommutes} is an instance of \eqref{eq:simplestalks}. As such, the diagram commutes, and we obtain the second equivalence of \eqref{eq:sheafalgebraicmodel}. 

The final equivalence of \eqref{eq:sheafalgebraicmodel} is  given by \cref{thm:weyl_sheaves_continuous}.
\end{proof}

\section{The local structure of rational \texorpdfstring{$G$}{G}-spectra}\label{sec:min}

In this section we will describe the local structure of the category $\Sp_{G,\Q}$. For this, we first need to discuss the category of discrete modules for a profinite group, exhibiting them as examples of tt-fields when the coefficients are rational. We then identify the (co)stalk categories of rational $G$-spectra in terms of discrete modules and deduce that they are minimal and cominimal. This implies that $\Sp_{G,\Q}$ is von Neumann regular, and thus stratified and costratified if and only if the local-to-global principle holds for $\Sp_{G,\Q}$. Finally, we leverage our understanding of the local structure to prove that the telescope conjecture holds for rational $G$-spectra for any profinite group. 

While the algebraic part of this section can be simplified by appealing to the algebraic model of $\Sp_{G,\Q}$ constructed in the previous section directly, see for example \cref{rem:prop:qgsp_ttstalk_altproof}, we have chosen to give a self-contained treatment here.

\subsection{The category of discrete modules for a profinite group}\label{ssec:discretemodules}

We begin with a rapid review of the basic properties of discrete $\Gamma$-modules for a profinite group\footnote{Since we wish to specialize to the case $\Gamma= W_G(H)$ later, in this section we opted to write $\Gamma$ for a generic profinite group, reserving $G$ for the profinite group in the context of equivariant homotopy theory.} $\Gamma$, following \cite{ribes_zalesskii,barnessugrue_gsheaves}. Specializing to rational discrete modules in the next subsection, we then prove that the corresponding derived category forms a tt-field.  

\begin{definition}\label{def:discretemodules_general}
Let $\Gamma$ be a profinite group and $R$ a commutative ring. An $R$-linear $\Gamma$-module $M$ is called \emph{discrete} if the action map
\[
\Gamma \times M \to M
\]
is continuous; equivalently, $M$ is discrete if the stabilizers of all elements in $M$ are open in $\Gamma$. We write $\Mod_{R}^{\delta}(\Gamma)$ for the category of $R$-linear discrete rational $\Gamma$-modules and $R$-linear $\Gamma$-equivariant maps between them.
\end{definition}

Let $R[\Gamma]$ be the group algebra and let $\Mod_{R}(\Gamma)$ be the abelian category of $R[\Gamma]$-modules. For a module $M \in \Mod_{R}(\Gamma)$ and a closed subgroup $H \leqslant \Gamma$, write $M^H$ for the $H$-fixed points of $M$ with its residual $\Gamma/H$-action. By a mild abuse of notation, we may also view $M^H$ as a $\Gamma$-module with $\Gamma$ acting through the projection map. The canonical inclusion $\iota\colon \Mod_{R}^{\delta}(\Gamma) \to \Mod_{R}(\Gamma)$ admits a right adjoint $(-)^{\delta}$, given explicitly by the formula
\[
M^{\delta} = \colim_{U \leqslant_o \Gamma} M^{U}
\]
for any $M \in \Mod_{R}(\Gamma)$. The pair $(\iota, (-)^{\delta})$ exhibits $\Mod_{R}^{\delta}(\Gamma)$ as the colocalization of $\Mod_{R}(\Gamma)$ with respect to the collection $\{R[\Gamma/U]\mid U \leqslant_o \Gamma\}$; in particular, a module $M \in \Mod_{R}(\Gamma)$ is discrete if and only if $M \cong M^{\delta}$. Moreover, it follows that $\Mod_{R}^{\delta}(\Gamma)$ inherits the structure of a symmetric monoidal abelian subcategory from $\Mod_{R}(\Gamma)$ and that it contains enough injective objects. We denote by $\mathsf{D}\Mod_{R}^{\delta}(\Gamma)$ the corresponding derived category, which has the structure of a rigidly-compactly generated tt-category. Note that, when $\Gamma$ is finite, then every module is discrete, so $\mathsf{D}\Mod_{R}^{\delta}(\Gamma)$ coincides with the usual derived category of $R[\Gamma]$-modules. 

Consider a map $f\colon \Gamma \to \Gamma'$ of profinite groups, i.e., a continuous group homomorphism. Precomposing along $f$ induces a symmetric monoidal functor
\[
\xymatrix{f^*\colon \Mod_{R}^{\delta}(\Gamma') \ar[r] & \Mod_{R}^{\delta}(\Gamma).}
\]
This functor is symmetric monoidal, exact, and preserves all colimits, and hence gives rise to a geometric functor on the associated derived categories, which we will denote by the same symbol. If $f$ is given by the inclusion of a closed subgroup $i\colon H \to \Gamma$, then $\res_{H}^{\Gamma}\coloneqq i^*$ is referred to as \emph{restriction}. In this case, the right adjoint $f_*$ of $f^*$ is \emph{coinduction}. Moreover, if $i\colon U \leqslant \Gamma$ is open, then $i^*$ admits a left adjoint, namely \emph{induction} $\ind_{U}^{\Gamma}$. If $f$ instead arises as the projection $p\colon \Gamma \to \Gamma/H$ for some closed normal subgroup $H \leqslant \Gamma$, we will refer to $p^*$ as \emph{inflation} and write $\infl_{\Gamma/H}^{\Gamma} \coloneqq p^*$ in this case. The right adjoint of $p^*$ is $(-)^{H} \cong p_*$, the \emph{fixed point functor}, and the corresponding right derived functor $(-)^{hH}$ on the derived category, respectively. Finally, $p^*$ also admits a left adjoint given by the (derived) functor of $H$-\emph{orbits} $(-)_H$.

\begin{lemma}\label{lem:discretemodules_inflation}
The inflation functor $\infl_{\Gamma/U}^{\Gamma}\colon \mathsf{D}\Mod_{R}^{\delta}(\Gamma/U) \to \mathsf{D}\Mod_{R}^{\delta}(\Gamma)$ is fully faithful for any open normal subgroup $U \leqslant \Gamma$.
\end{lemma}
\begin{proof}
Let $p\colon \Gamma \to \Gamma/U$ be the projection map. We need to show that the unit of the adjunction, denoted $\eta\colon \id \to p_*p^*$, is a natural quasi-isomorphism, i.e., that the map
\[
M \to (\infl_{\Gamma/U}^{\Gamma}(M))^{hU}
\]
is a quasi-isomorphism for any $M \in \mathsf{D}\Mod_{R}(\Gamma/U)$. Since $p^*$ is a geometric functor, both $p^*$ and $p_*$ preserve colimits. Therefore, it suffices to prove the claim for a compact generator of $\mathsf{D}\Mod_{R}(\Gamma/U)$, say $M = R[\Gamma/U]$. We get $\infl_{\Gamma/U}^{\Gamma}(R[\Gamma/U]) \simeq  \ind_{U}^{\Gamma}R \simeq  \coind_{U}^{\Gamma}R$ as $\Gamma$-modules, where the second equivalence uses that $U$ is open and hence of finite index in $\Gamma$. The composite $U \hookrightarrow \Gamma \twoheadrightarrow \Gamma/U$ is trivial, so we obtain equivalences
\[
(\infl_{\Gamma/U}^{\Gamma}R[\Gamma/U])^{hU} \simeq (\coind_{U}^{\Gamma}R)^{hU} \simeq \coind_e^{\Gamma/U}R \simeq R[\Gamma/U].
\]
Unwinding the constructions, this is the unit evaluated on $M = R[\Gamma/U]$, as desired.
\end{proof}

\begin{proposition}\label{prop:discretemodules_continuity}
Let $\Gamma$ be a profinite group and consider the cofiltered system $(U_i \leqslant_o \Gamma)$ of open normal subgroups of $\Gamma$. Then the inflation functors induce a geometric equivalence
\[
\xymatrix{\colim_{U_i\leqslant_o \Gamma}^{\omega} \mathsf{D}\Mod_{R}(\Gamma/U_i) \ar[r]^-{\sim} & \mathsf{D}\Mod_{R}^{\delta}(\Gamma),}
\]
where the colimit is computed in the $\infty$-category $\Pr^{L,\omega}$ of compactly generated symmetric monoidal stable $\infty$-categories. 
\end{proposition}
\begin{proof}
By the universal property of the colimit in compactly generated $\infty$-categories, the family of inflation functors $(\infl_{\Gamma/U_i}^{\Gamma})$ induces a geometric comparison functor 
\[
\xymatrix{\psi^{*}\colon \colim_{U_i\leqslant_o \Gamma}^{\omega} \mathsf{D}\Mod_{R}(\Gamma/U_i) \ar[r] & \mathsf{D}\Mod_{R}^{\delta}(\Gamma),}
\]
whose right adjoint we denote as usual by $\psi_*$. In order to prove the proposition, it suffices to show that 
    \begin{enumerate}
        \item[(a)] $\psi^*$ is fully faithful; 
        \item[(b)] $\psi_*$ is conservative.
    \end{enumerate}
Indeed, (a) says that the unit of the adjunction is a natural isomorphism, while (b) together with the triangle identity for the adjunction then implies that the counit is also a natural isomorphism.

To see (a), we first observe that the restriction of $\psi^*$ to compact objects is fully faithful as a colimit of fully faithful functors, using \cref{lem:discretemodules_inflation}. Since $\psi^*$ is given by the ind-extension of a fully faithful functor, it is itself fully faithful, see \cite[Proposition 5.3.5.11]{htt}. Claim (b) is equivalent to the statement $\psi^*$ sends a set of generators to a set of generators. For this, we may take the objects $R[\Gamma/U]$ for varying $U \leqslant_o \Gamma$, thereby completing the proof. 
\end{proof}

\begin{remark}\label{rem:discretemodules_eqsheaves}
   As a tt-category, $\mathsf{D}\Mod_{R}^{\delta}(\Gamma)$ is equivalent to the derived category of $\Gamma$-equivariant sheaves on a point with coefficients in $\Mod(R)$. The continuity of equivariant sheaves proven in \cref{thm:eqsheaves_continuity} gives an alternative proof of \cref{prop:discretemodules_continuity}.
\end{remark}

\subsection{Rational discrete modules as a tt-field}\label{ssec:discretemodules_rational}

For the remainder of this section, we specialize to the case that $R = \Q$. Our next goal is to prove that the derived category of rational discrete $\Gamma$-modules is a tt-field in the sense of \cref{defn:ttfield}.

\begin{proposition}\label{prop:discretemodules_ttfields}
Let $\Gamma$ be a profinite group. The derived category of discrete rational $\Gamma$-modules $\mathsf{D}\Mod_{\Q}^{\delta}(\Gamma)$ is an abelian tt-field. 
\end{proposition}
\begin{proof}
We begin with a preliminary observation about the abelian category of rational discrete $\Gamma$-modules. By \cite[Proposition 3.1]{CastellanoWeigel2016}, every object in $\Mod_{\Q}^{\delta}(\Gamma)$ is both injective and projective. This implies that any short exact sequence of rational discrete modules splits, hence its derived category is abelian and split-exact as well, i.e., every object in $\Mod_{\Q}^{\delta}(\Gamma)$ decomposes into a direct sum of its homology groups. It follows that $\mathsf{D}\Mod_{\Q}^{\delta}(\Gamma)$ is pure-semisimple and satisfies Condition (F1), see for example \cite[Theorem 9.3]{Beligiannis2000} or \cite[Theorem 2.10]{Krause2000}.

In order to verify Condition (F2), first observe that it suffices to consider nonzero finite dimensional modules concentrated in a single degree by what we have already shown. Recall that the dimension $\dim(M)$ of such an $M \in \Mod_{\Q}^{\delta}(\Gamma)^{\omega}$ is defined as the composite of the coevaluation and evaluation maps:
\[
\xymatrix{\dim(M)\colon \Q \ar[r]^-{\text{coev}} & M \otimes M^{\vee} \ar[r]^-{\text{ev}} & \Q.}
\]
Since $\End(\Q) \cong \Q$, the dimension $\dim(M)$ is invertible as long as $M\neq 0$. If $f$ is any map in $\mathsf{D}\Mod_{\Q}^{\delta}(\Gamma)$ with $M \otimes f =0$, this retract diagram thus shows that $f = 0$; in other words, $M$ is $\otimes$-faithful, as desired.
\end{proof}

\begin{remark}\label{rem:discretemodules_semisimple}
    In fact, the tt-category $\mathsf{D}\Mod_{\Q}^{\delta}(\Gamma)$ is semisimple in the sense of Hovey, Palmieri, and Strickland \cite[\S8]{HoveyPalmieriStrickland1997}. Indeed, by \cref{lem:discretemodules_inflation} and \cref{prop:discretemodules_continuity}, this statement can be checked for the categories $\mathsf{D}\Mod_{\Q}^{\delta}(\Gamma/U)$ for $U \leqslant \Gamma$ open normal, where it holds by Maschke's theorem and Schur's lemma. It then follows from  \cite[Proposition 8.1.3]{HoveyPalmieriStrickland1997} that every object in $\mathsf{D}\Mod_{\Q}^{\delta}(\Gamma)$ is equivalent to a coproduct of suspensions of finite-dimensional simple $\Q[\Gamma/U]$-modules for varying $U\leqslant_o \Gamma$, thereby giving an alternative proof of \cref{prop:discretemodules_ttfields}. 
\end{remark}

\begin{remark}\label{rem:fixedorbitequivalence}
    As a special case of \cref{prop:discretemodules_ttfields}, the inclusion of fixed points $M^H \to M$ splits as a map of rational discrete $\Gamma$-modules for any closed subgroup $H \leqslant \Gamma$. If $H$ is also open, then one can see this directly: Indeed, in this case, the composite of the canonical maps
    \[
    \xymatrix{M^{H} \ar[r] & M \ar[r] & M_{H}}
    \]
    is an isomorphism for any $M \in \mathsf{D}\Mod_{\Q}^{\delta}(\Gamma)$. To see this, note first that if two $H$-fixed points are sent to the same orbit class, then they were equal. Hence, the composite is injective. Surjectivity comes from the fact that any $H$-orbit is finite. Taking the average of an orbit gives an $H$-fixed point, thus establishing surjectivity of the composite.
\end{remark}

\begin{corollary}\label{cor:discretemodules_biminimality}
    Let $\Gamma$ be a profinite group. Any nonzero object $M \in \mathsf{D}\Mod_{\Q}^{\delta}(\Gamma)$ generates the entire category both as a localizing ideal and as a colocalizing coideal. 
\end{corollary}
\begin{proof}
By \cite[Corollary 8.9]{BCHS_cosupport}, any tt-field $\sfF$ is stratified and costratified with spectrum $\Spc(\sfF^{\omega}) = \{(0)\}$. In particular, any nonzero object in $\sfF$ generates $\sfF$ both as a localizing ideal and as a colocalizing coideal. Therefore, the claim follows from \cref{prop:discretemodules_ttfields}.
\end{proof}

\begin{remark}\label{rem:discretemodules_nonmodular}
    Let $\Gamma$ be a profinite group and $m(\Gamma)$ its order, as defined in \eqref{eq:order}. Then the results of this subsection hold true for $\Q$ replaced by any field $k$ in which every prime $p$ dividing $m$ is invertible.
\end{remark}

\subsection{The tt-stalks of $\Sp_{G,\Q}$ and stratification}

We can now compute the tt-stalks of the category of rational $G$-spectra and show that they are tt-fields, thereby proving that $\Sp_{G,\Q}$ is von Neumann regular. The structural implications are then collected at the end of this section.

\begin{proposition}\label{prop:qgsp_ttstalk}
Let $G$ be a profinite group $H \in \Sub(G)/G$ a closed subgroup with profinite Weyl group $W_G(H)$ and with associated prime tt-ideal $\sfP(H) \in \Spc(\Sp_{G,\Q}^{\omega})$. The rational geometric fixed point functors (\cref{rem:qgsp_spc}) then induce an equivalence
\[
\xymatrix{(\Sp_{G,\Q})_{\sfP(H)} \ar[r]^-{\simeq} & \sfD\Mod_{\Q}^{\delta}(W_G(H))}
\]
of rigidly-compactly generated tt-categories. 
\end{proposition}
\begin{proof}
    Write $G = \lim_{i\in I}G_i$ as a cofiltered limit of finite groups $G_i$. Let $H_i \in \Sub(G_i)/G_i$ be (a representative of) the image of $H$ under the projection map $\Sub(G) \to \Sub(G_i)$. By virtue of \cref{prop:profiniteweyl}, there is a canonical isomorphism $W_G(H) \cong \lim_{i \in I}W_{G_i}(H_i)$ of profinite Weyl groups. For any group homomorphism $f_{ij}\colon G_i \to G_j$ in the given system, by naturality of stalks (\cref{def:ttstalkmap}) and the equivalences of \eqref{eq:finitegroup_stalks}, there is a commutative diagram of geometric functors
    \[
    \xymatrix{\Sp_{G_i,\Q} \ar[r] & (\Sp_{G_i,\Q})_{\sfP(H_i)} \ar[r]^-{\simeq} & \sfD\Mod_{\Q}(W_{G_i}(H_i)) \\
    \Sp_{G_j,\Q} \ar[r] \ar[u]^{f_{ij}^*} & (\Sp_{G_j,\Q})_{\sfP(H_j)} \ar[r]^-{\simeq} \ar[u]_{(f_{ij}^*)_{\sfP(H_i)}} & \sfD\Mod_{\Q}(W_{G_j}(H_j)) \ar[u]_{\infl_{G_j}^{G_i}}.}
    \]
    By \cref{thm:finitegroups_algmodel}, the horizontal composites in this diagram identify with the rational geometric fixed point functor $\Phi_{G_i,\Q}^{H_i}$ and $\Phi_{G_i,\Q}^{H_i}$, respectively. Passing to the colimit over $I$, we thus obtain a commutative diagram of geometric functors
    \[
    \xymatrix{\Sp_{G,\Q} \ar[r] & (\Sp_{G,\Q})_{\sfP(H)} \ar[r] & \sfD\Mod_{\Q}^{\delta}(W_{G}(H)) \\ 
    \colim_{i\in I}^{\omega}\Sp_{G_i,\Q} \ar[r] \ar[u]^{\simeq} & \colim_{i\in I}^{\omega}(\Sp_{G_i,\Q})_{\sfP(H_i)} \ar[r]^-{\simeq} \ar[u]_{\simeq}  & \colim_{i\in I}^{\omega}\sfD\Mod_{\Q}(W_{G_i}(H_i)), \ar[u]_{\simeq}}
    \]
    where the three vertical equivalences use, from left to right: \eqref{eq:qgsp_continuity}, \cref{cor:continuousttstalk}, and \cref{prop:discretemodules_continuity}, respectively. It follows that the right top horizontal functor is also an equivalence. Finally, the continuity of geometric fixed points (\cref{prop:gfp_properties}) implies that the top horizontal composite identifies with $\Phi_{G,\Q}^H$, as desired.
\end{proof}

\begin{remark}\label{rem:prop:qgsp_ttstalk_altproof}
    There is a different path to \cref{prop:qgsp_ttstalk}, which proceeds via the algebraic model for rational $G$-spectra (\cref{thm:sheafalgebraicmodel}). Recall the notion of equivariant skyscraper sheaf in the profinite setting from \cref{ex:skycraper}. From this, one can deduce that the equivariant tt-stalk $\sfD\Shv_{G,\Q}^{\Weyl}(\Sub(G))_{(H)}$ of the category of Weyl $G$-sheaves at a closed subgroup $H$ is equivalent to the category of Weyl $G$-sheaves on the $G$-orbit of $H$. The latter category in turn identifies with the derived category of rational discrete $W_G(H)$-modules, resulting in geometric equivalences: 
    \[
    (\Sp_{G,\Q})_{\sfP(H)} \simeq \sfD\Shv_{G,\Q}^{\Weyl}(\Sub(G))_{(H)} \simeq \sfD\Shv_{G,\Q}^{\Weyl}((H)_G) \simeq \sfD\Mod_{\Q}^{\delta}(W_G(H)).
    \]
    Here, $\sfD\Shv_{G,\Q}^{\Weyl}((H)_G)$ denotes the derived category of Weyl-$G$-sheaves on the orbit $(H)_{G}$ of $H$ in $\Sub(G)$ under the conjugation action by $G$. The composite equivalence is functorial in the group and hence compatible with reduction to finite groups as in the proof of \cref{prop:qgsp_ttstalk}.
\end{remark}

We can now deduce the local-global compatibility of the algebraic model of $\Sp_{G,\Q}$:

\begin{corollary}\label{cor:algmodel_localglobal}
    For any closed subgroup $H \leqslant G$, the following diagram commutes:
        \[
            \xymatrix{\Sp_{G,\Q} \ar[r]_-{\sim}^-{\Phi} \ar[d]_{\Phi^H} & \sfD\Shv_{G,\Q}^{\Weyl}(\Sub(G)) \ar[d]^{(-)_H} \\
            \sfD\Mod_{\Q}^{\delta}(W_G(H)) \ar[r]_-{\sim} & \sfD\Shv_{G,\Q}^{\Weyl}((H)_G).}
        \]
\end{corollary}
\begin{proof}
    Suppose first that $G$ is a finite group. In this case, $\Sub(G)$ is discrete and the result is the content of Wimmer's theorem \cref{thm:finitegroups_algmodel}. Keeping in mind \cref{prop:qgsp_ttstalk} and \cref{rem:prop:qgsp_ttstalk_altproof} the general claim then follows from the continuity of all constructions and functors in play.
\end{proof}

\begin{remark}
    As a consequence of this corollary, one can show that the induced equivalence on homotopy categories
        \[
            \xymatrix{\Ho(\Phi)\colon \Ho(\Sp_{G,\Q}) \ar[r] & \Ho(\sfD\Shv_{G,\Q}^{\Weyl}(\Sub(G)))}
        \]
    agrees with the (triangulated) equivalence constructed by Barnes and Sugrue~\cite{sugrue_thesis,barnessugrue_spectra}. We omit the details of this argument. 
\end{remark}

Recall the definition of von Neumann regular tt-category from \cref{def:vNr_ttcats}. Combining \cref{prop:discretemodules_ttfields} and \cref{prop:qgsp_ttstalk}, we obtain:

\begin{theorem}\label{thm:qgsp_vNr}
    The tt-category $\Sp_{G,\Q}$ is von Neumann regular for any profinite group $G$.
\end{theorem}

We are now ready to reap the harvest.

\begin{corollary}\label{cor:mainstrat}
    For any profinite group $G$, the following statements are equivalent:  
        \begin{enumerate}
            \item $\Sp_{G,\Q}$ is stratified;
            \item $\Sp_{G,\Q}$ is costratified;
            \item $\Sp_{G,\Q}$ satisfies the local-to-global principle.
        \end{enumerate}
    If these equivalent conditions are satisfied, then support and cosupport induce bijections 
        \[
            \begin{Bmatrix}
                \text{Localizing ideals} \\ \text{of } \Sp_{G,\Q}
            \end{Bmatrix} 
                \xymatrix@C=2pc{ \ar@<0.75ex>[r]^{\Supp} &  \ar@<0.75ex>[l]^{\Supp^{-1}}}
            \begin{Bmatrix}
                \text{Subsets} \\ \text{of } \Spec(A_{\Q}(G))
            \end{Bmatrix}
                \xymatrix@C=2pc{ \ar@<0.75ex>[r]^{\Cosupp} &  \ar@<0.75ex>[l]^{\Cosupp^{-1}}}
            \begin{Bmatrix}
                \text{Colocalizing coideals} \\ \text{of } \Sp_{G,\Q}
            \end{Bmatrix}.
        \]
\end{corollary}
\begin{proof}
    We have proven in \cref{thm:qgsp_vNr} that $\Sp_{G,\Q}$ is von Neumann regular. The stipulated equivalences are then a consequence of \cref{prop:vNr_stratification}, which in turn imply the classification of localizing ideals and colocalizing ideals by the theory of stratification, see \cref{ssec:ttstratification}. 
\end{proof}

\begin{corollary}\label{cor:telescope}
    Let $G$ be a profinite group satisfying the equivalent conditions of \cref{cor:mainstrat}.  Then $\Sp_{G,\Q}$ satisfies the telescope conjecture.
\end{corollary}
\begin{proof}
The spectrum of $\Sp_{G,\Q}$ is generically Noetherian by \cref{lem:weaklynoeth}, so the claim follows from \cref{prop:tc}.
\end{proof}

In the next subsection, we will establish an unconditional generalization of this result.

\subsection{The telescope conjecture}\label{ssec:qgsp_telescopeconjecture}

Our next goal is to prove the telescope conjecture for rational $G$-spectra for all profinite groups $G$, thereby removing the equivalent hypotheses from \cref{cor:telescope}. We note that this is a significant improvement because, as explained in the next section, these hypotheses do not hold for arbitrary profinite groups. We begin with a couple of preliminary observations.

\begin{lemma}\label{lem:qgsp_bijhyp}
    $\Sp_{G,\Q}$ satisfies the bijectivity hypothesis, i.e., the comparison map
        \[
        \xymatrix{\phi_{G,\Q}\colon \Spch(\Sp_{G,\Q}^{\omega}) \ar[r]^-{\cong} & \Spc(\Sp_{G,\Q}^{\omega})}
        \]
    is a homeomorphism.
\end{lemma}
\begin{proof}
    This holds by the same argument as for the integral case, see the proof of \cref{prop:bijectivityhyp}, using \cref{prop:spccontinuity,thm:spchcontinuity}, which are applicable thanks to \eqref{eq:qgsp_continuity}.
\end{proof}

\begin{lemma}\label{lem:qgsp_detection}
    The collection of geometric fixed point functors $(\Phi_{\Q}^H)_{H \in \Sub(G)/G}$ is jointly conservative on $\Sp_{G,\Q}$, i.e., they detect equivalences.
\end{lemma}
\begin{proof}
    By construction of $\sfD (\Shv_{G,\Q}^{\Weyl}((H)_G))$ via the model structure of \cref{prop:weylsheaves_properties}(4), the stalk functors 
    \[
        ((-)_H\colon \sfD\Shv_{G,\Q}^{\Weyl}(\Sub(G)) \to \sfD\Shv_{G,\Q}^{\Weyl}((H)_G))_H
    \]
    are jointly conservative for $H$ ranging through the (conjugacy classes of) closed subgroups of $G$. The claim then follows from \cref{cor:algmodel_localglobal}.
\end{proof}

The next lemma shows that, in $\Sp_{G,\Q}$, homological support and tensor-triangular support coincide and can be computed in terms of geometric isotropy. 

\begin{lemma}\label{lem:qgsp_supportidentifications}
    For any $t \in \Sp_{G,\Q}$, we have identifications
        \[
            \Supp(t) = \Supph(t) = \{H \in \Sub(G)/G\mid \Phi_{\Q}^H(t) \neq 0\}
        \]
    and all three notions of support have the detection property. Here, in the first equality we have implicitly used the homeomorphism $\phi_{G,\Q}$ to identify the homological with the triangular spectrum.
\end{lemma}
\begin{proof}
    The bijectivity hypothesis holds by \cref{lem:qgsp_bijhyp}, so we can use the map $\phi_{G,\Q}$ to identify the homological spectrum of $\Sp_{G,\Q}^{\omega}$ with its triangular spectrum. \Cref{cor:gsp_hsupport} implies that the homological support for $\Sp_{G,\Q}$ can be computed in terms of (rational) geometric isotropy, while \cref{lem:qgsp_detection} establishes the detection property in this context. Finally, the proof of \cref{prop:zerodim_telescopeconjecture} gives the identification of the triangular with the homological support.   
\end{proof}

\begin{theorem}\label{thm:qgsp_tc}
    $\Sp_{G,\Q}$ satisfies the telescope conjecture for any profinite group $G$.
\end{theorem}
\begin{proof}
    Since $\Sp_{G,\Q}$ is zero-dimensional, it suffices to verify the two conditions of our criterion \cref{prop:zerodim_telescopeconjecture}. This is the content of \cref{lem:qgsp_bijhyp} and \cref{lem:qgsp_supportidentifications}, respectively.
\end{proof}

\begin{remark}\label{rem:hrbek}
    The conclusion of \cref{thm:qgsp_tc} can also be obtained via recent work of Hrbek \cite{hrbek2023telescope} in place of our \cref{prop:zerodim_telescopeconjecture}. He establishes a stalkwise criterion for the telescope conjecture, provided that the tt-category under consideration satisfies the bijectivity hypothesis and (a weak form of) the detection property. Using that $\Sp_{G,\Q}$ is von Neumann regular, the telescope conjecture holds for each of its tt-stalks, so his result applies. 
\end{remark}

\section{The local-to-global principle and spaces of subgroups}\label{sec:qgsp_lgp}

In \cref{cor:mainstrat} we saw that the property of stratification for $\Sp_{G,\Q}$ reduces to the question of the local-to-global principle. In this last section we will describe when $\Sp_{G,\Q}$ satisfies the local-to-global principle in terms of the cardinality of $\Sub(G)/G$. To do so, we will begin in \cref{ssec:profiniteLGP} by linking the local-to-global principle of $\Sp_{G,\Q}$ to the local-to-global principle of the derived category of the rational Burnside ring, which will then feed into our main results in \cref{ssec:spacesofconj}. We will conclude in \cref{ssec:lgp_examples} with a plethora of examples.

\subsection{Profinite Burnside rings and the local-to-global principle}\label{ssec:profiniteLGP}

In this subsection we will tackle the question of when the local-to-global principle holds for the category $\Sp_{G,\Q}$, giving both a topological and an algebraic characterization. To this end, we will employ a descent type argument by comparing $\Sp_{G, \Q}$ to its `cellular' subcategory, which can then be understood via work of Stevenson. 

Let $\mathsf{D}(A_{\Q}(G))$ be the derived category of the rational Burnside ring of $G$. As suggested, $\mathsf{D}(A_{\Q}(G))$ should be thought of as a monogenic version of $\Sp_{G,\Q}$, indeed, it is the localizing subcategory generated by the unit:

\begin{lemma}\label{lem:morita}
There is a canonical fully faithful geometric functor
\[
\xymatrix{\iota\colon \mathsf{D}(A_{\Q}(G)) \ar[r] & \Sp_{G,\Q}}
\]
with essential image the localizing subcategory generated by $S_G^0$.
\end{lemma}
\begin{proof}
Let $\End_{\Sp_{G,\Q}}(S_G^0)$ be the rational commutative ring spectra of endomorphisms of $S_G^0$ and recall that the rational Burnside ring of $G$ is defined as $A_{\Q}(G) = \pi_0\End_{\Sp_{G,\Q}}(S_G^0)$. Since we are working rationally, $\End_{\Sp_{G,\Q}}(S_G^0) \cong HA_{\Q}(G)$ as rational commutative ring spectra or, equivalently, as rational commutative dg-algebras. The desired result is obtained as a composite of two geometric functors which we describe below:
\[
\xymatrix{\iota\colon \mathsf{D}(A_{\Q}(G)) \ar[r]^-{\sim} &  
\Mod(\End_{\Sp_{G,\Q}}(S_G^0))  \ar[r] & \Sp_{G,\Q}.}
\]
The first geometric equivalence follows from Shipley's theorem \cite{shipley_hz}. The second equivalence is an instance of symmetric monoidal Morita theory, in the form of \cite[\S7.1]{MathewNaumannNoel2017}: For any presentably symmetric monoidal stable $\infty$-category $\sfT = (\sfT,\otimes,\unit)$ and $R\coloneqq \End_{\sfT}(\unit)$, there is an adjunction
\[
-\otimes_R \unit \colon \Mod_{R} \rightleftarrows \sfT \noloc \Hom_{\sfT}(\unit,-).
\]
The left adjoint is geometric and promotes canonically to a symmetric monoidal functor. This adjunction induces a symmetric monoidal equivalence between $\Mod_R$ and the localizing subcategory of $\sfT$ generated by $\unit$. Specializing this to $\sfT = \Sp_{G,\Q}$ gives the second fully faithful and geometric functor above.
\end{proof}

The following result tells us that the local-to-global principle is completely determined by the topology of $\Sub(G)/G$ as well as by algebraic properties of the rational Burnside ring of $G$.

\begin{proposition}\label{prop:LGPiffscattered}    
Let $G$ be a profinite group. The following statements are equivalent:
        \begin{enumerate}
            \item $\Sp_{G,\Q}$ satisfies the local-to-global principle;
            \item $\mathsf{D}(A_{\Q}(G))$ satisfies the local-to-global principle;
            \item $\Sub(G)/G$ is scattered;
            \item $A_{\Q}(G)$ is semi-Artinian.
        \end{enumerate}
\end{proposition}
\begin{proof}\leavevmode
By \cref{prop:ratburnsidering}, the rational Burnside ring $A_{\Q}(G)$ of $G$ is von Neumann regular, so the equivalence of Conditions $(2)$, $(3)$, and $(4)$ follows from \cref{ex:absoluteflat}. 

For the implication $(1) {\implies} (2)$, we use \cref{lem:morita}: there is a fully faithful geometric functor $\mathsf{D}(A_{\Q}(G)) \to \Sp_{G,\Q}$. It then follows from \cite[Example 15.5]{BCHS_cosupport} that the local-to-global principle descends along this inclusion. Conversely, suppose that $(3)$ holds. Recall from \cref{prop:qgsp_spc} that $\Spc(\Sp_{G,\Q}^{\omega}) \cong \Sub(G)/G$ is Hausdorff. By \cref{rem:BDscattered}, if this space is scattered, the Cantor--Bendixson rank defines a convergent dimension function on it. This implies that the local-to-global principle holds for $\Sp_{G,\Q}$, see \cref{prop:scatteredimpl2g}.
\end{proof}

\begin{remark}\label{rem:maybetelescopes}
    In \cref{thm:qgsp_tc} we saw that $\Sp_{G,\Q}$ satisfies the telescope conjecture for any profinite $G$. Bazzoni--\v{S}t'ov\'{\i}\v{c}ek prove that the telescope conjecture holds for all $\sfD(R)$ for $R$ any absolutely flat ring \cite{BS_telescope}, generalizing an earlier result by Stevenson \cite[Theorem 4.25]{stevenson_absolutelyflatrings}. In particular, for any profinite group $G$, the category $\sfD(A_\Q(G))$ satisfies the telescope conjecture. Our result can thus be viewed as a multi-generator version of the one proven by Bazzoni--\v{S}t'ov\'{\i}\v{c}ek.
\end{remark}

\subsection{Spaces of conjugacy classes of closed subgroups}\label{ssec:spacesofconj}

In light of \cref{cor:mainstrat} and \cref{prop:LGPiffscattered}, it remains to determine when $\Sub(G)/G$ is scattered to resolve when $\Sp_{G,\Q}$ is stratified. 

\begin{theorem}\label{thm:scatterediffcountable}
    Let $G$ be a profinite group. Then $\Sub(G)/G$ is scattered if and only if $\Sub(G)/G$ is countable.
\end{theorem}

\begin{remark}
    It is not true in general for a profinite space $X$ that $X$ is scattered if and only if it is countable. In particular, one can consider the Tychonoff plank: this is a profinite space which is not countably based, and as such, not countable, but it is scattered \cite[\S 86]{counterexamples}. As such, the statement of \cref{thm:scatterediffcountable} strongly depends on the structure of subgroups of profinite groups. In particular we can conclude that there is no profinite group $G$ such that $\Sub(G)/G$ is the Tychonoff plank, for example.
\end{remark}

\begin{remark}
    By \cite[Proposition 2.12]{gartsidesmith1} we have a full characterization of the homeomorphism type of those $X = \Sub(G)/G$ which are countable. Consider the first limit ordinal $\omega$. Then for $\alpha$ any ordinal, and $n$ a positive integer, we can construct a space $\omega^\alpha n + 1$, which is formed by taking $n$ disjoint copies of the space $\omega^\alpha$, adding a disjoint point $\{\infty\}$ and then equipping the set with the order topology. For example, $\omega+1$ is simply the usual one-point compactification of $\N$. Then
    \[
    X \cong \omega^{(\cbrank(X)-1)} |X^{(\cbrank(X)-1)}| +1,
    \]
    where $\cbrank(X)$ is the Cantor--Bendixson rank of $X$ as introduced in \cref{defn:cbrank} and $|X^{(\cbrank(X)-1)}|$ denotes the cardinality of the finite space $X^{(\cbrank(X)-1)}$.
\end{remark}

To prove \cref{thm:scatterediffcountable} we will first of all prove some lemmas regarding the topology of $\Sub(G)/G$. Let $q \colon \Sub(G) \to \Sub(G)/G$ be the open quotient map, see \cref{rem:quotientisopen}. 

For a subgroup $H \leqslant G$ we write $(H)_G$ for the conjugacy class of $H$ viewed inside $\Sub(G)$, and $q(H)$ for the image of $H$ in $\Sub(G)/G$. The proof of the next lemma is similar to the argument used by Gartside and Smith for the corresponding statement in $\Sub(G)$, \cite[Lemma 5.4]{gartsidesmith1}.

\begin{lemma}\label{lem:isothenopen}
    Let $H$ be a closed subgroup of $G$ such that $q(H) \in \Sub(G)/G$ is isolated. Then $H$ is an open subgroup of $G$.
\end{lemma}

\begin{proof}
    Recall from \eqref{eq:nbhdbasis} that the collection of $q(\cB(H,N)) = \{q(K) \mid K \in \Sub(G)\colon NK = NH\}$ for $N$ ranging through the open normal subgroups of $G$ forms a neighbourhood basis for $q(H)$ in $\Sub(G)/G$. Now, if $q(H)$ is isolated, then there exists an open normal subgroup $N \leqslant G$ such that $\{ q(H)\} = q( \cB(H,N))$. Unpacking the definition, that means that any $K \in \Sub(G)$ which satisfies $NK = NH$ has the property that $q(H) = q(K)$, i.e., $K = H^g$ for some $g \in G$.  
    
    Consider $K = NH$. Because this is a closed subgroup of $G$ which satisfies $NK = NNH = NH$, there must be some $g \in G$ with $NH = K = H^g$. But $H \leqslant NH = H^g$, so in fact $NH = H$. This implies that $N \leqslant H$, which shows that $H$ must be open in $G$, as desired.
\end{proof}

\begin{lemma}\label{lem:perfection}
If $\Sub(G)$ is perfect (in the sense of \cref{rem:BDscattered}) then so is $\Sub(G)/G$.
\end{lemma}
\begin{proof}
    We prove the contrapositive: Assume that $\Sub(G)/G$ is not perfect, i.e., there exists some closed subgroup $H$ in $G$ such that $q(H)$ is isolated in $\Sub(G)/G$. On the one hand, it follows that the preimage $(H)_G = q^{-1}q(H) \subseteq \Sub(G)$, given by all closed subgroups of $G$ conjugate to $H$, is open in $\Sub(G)$. On the other hand, by \cref{lem:isothenopen}, $H$ is also open in $G$. Since open subgroups have only finitely many conjugates in $G$, the subset $(H)_G \subseteq \Sub(G)$ must be finite and non-empty. In summary, we have shown that $\Sub(G)$ contains a finite non-empty open subset, hence it cannot be perfect as a non-empty open subset of a perfect space is necessarily infinite.
\end{proof}

\begin{lemma}\label{lem:imperfection}
    If $\Sub(G)/G$ contains an isolated point, then $\Sub(G)/G$ is countably based.
\end{lemma}
\begin{proof}
    By hypothesis, $\Sub(G)/G$ is not perfect, so \cref{lem:perfection} implies that $\Sub(G)$ cannot be perfect either. Gartside and Smith prove that $\Sub(G)$ must then be countably based: Indeed, \cite[Theorem 5.6]{gartsidesmith1} shows in particular that $G$ is finitely generated, so countably based, see the discussion surrounding \cite[Proposition 4.1.3]{wilson}. It then follows from \cite[Corollary 2.4]{gartsidesmith1} that $\Sub(G)$ is countably based. We conclude that the quotient $\Sub(G)/G$ is countably based as well.
\end{proof}

Armed with the above lemmas we are now in a position to provide a proof of \cref{thm:scatterediffcountable}.

\begin{proof}[Proof of \cref{thm:scatterediffcountable}]
    As an application of the Cantor--Bendixson theorem (see \cref{prop:countbased}), we have that $\Sub(G)/G$ is countable if and only if it is scattered and countably based. Therefore, it suffices to prove that, for  $\Sub(G)/G$, being scattered implies being countably based. But this follows directly from \cref{lem:imperfection}: If $\Sub(G)/G$ is scattered, it contains an isolated point, so it has to be countably based. 
\end{proof}

\begin{remark}\label{rem:neccessaryGS}
    The proof of the theorem combined with \cite[Theorem 5.6]{gartsidesmith1} shows that a necessary condition for $\Sub(G)/G$ to be scattered is that $G$ is finite generated, virtually pronilpotent, and only finitely many primes divide $|G|$.
\end{remark}

\subsection{Examples and counterexamples}\label{ssec:lgp_examples}

Now that we have a full understanding of when $\Sub(G)/G$ is scattered, it remains to provide (counter-)examples. In the abelian case, we obtain the following full characterization.

\begin{proposition}\label{cor:abeliangroups}
    Let $G$ be an abelian group. Then $\Sub(G)/G$ is countable if and only if $G = A \times \prod_{i=1}^n \Z_{p_i}$ with $p_i \neq p_j$ for $i \neq j$ and $A$ a finite abelian group.
\end{proposition}
\begin{proof}
    This can be deduced from \cite[Theorem 3.7(ii),(iii)]{gartsidesmith1} with the usual observation that for $G$ abelian we have $Z(G)=G$ (see also \cite[Lemma 3.4]{gartsidesmith1}).
\end{proof}

\begin{example}\label{ex:zp}
    Let $G = \Z_p$. Then $G$ satisfies the conditions of \cref{cor:abeliangroups} and as such $\Sub(G)/G$ is countable. In fact, as observed before, $\Sub(G)/G \cong \N^*$ is the one-point compactification of the natural numbers, where the point at $\infty$ corresponds to the identity element of $G$. More generally if $G = \prod_{i=1}^n \Z_{p_i}$ for distinct primes $p_i$, then $\Sub(G)/G$ is homeomorphic to the countable space $(\N^*)^{\times n}$.
\end{example}

\begin{example}
    Let $G = \Z_p \times \Z_p$. Then $G$ does not satisfy the conditions of \cref{cor:abeliangroups} and hence $\Sub(G)/G$  is not countable. 

    Let us explore this example in some more detail as it is illustrative of what can happen. Consider lines $(a,b) \in \Q_p \times \Q_p$ (i.e., pairs $(a,b)$ up to scalar multiplication by $\Q_p^\ast$). We can use these to classify subgroups of $\Z_p \times \Z_p$ by intersecting. There are three cases to consider:
    \begin{enumerate}
        \item $a/b \in \Q$;
        \item $a/b \in (\Q_p \setminus \Q) \cup \{\infty\}$ where we say $a/b = \infty$ if $a \neq 0$ and $b = 0$;
        \item $a/b = 0/0$.
    \end{enumerate}
    The first collection of lines classify the countable family of finite index subgroups of $\Z_p \times \Z_p$. Indeed, these are the subgroups of the form $p^n \Z_p \times p^m \Z_p$. The third case corresponds to the unique finite subgroup of $\Z_p \times \Z_p$, namely the identity. Finally, the second class of subgroups constitutes the problematic uncountable family of subgroups of $\Z_p \times \Z_p$.

    We can moreover say something about the homeomorphism type of $\Sub(G)/G$ in this setting. It follows from \cite[Proposition 4.1]{gartsidesmith2} that $\Sub(G)/G$ has a countable collection of isolated points, but $\delta (\Sub(G)/G)$ contains no isolated points. Up to homeomorphism, there is only one countable profinite space with this property, namely the \emph{Pełczyński space}.

    A construction of this space can be obtained by the middle-thirds interval construction of the Cantor space where at each stage we reinsert the midpoints of the deleted intervals (see \cref{fig:pelspace}). 
    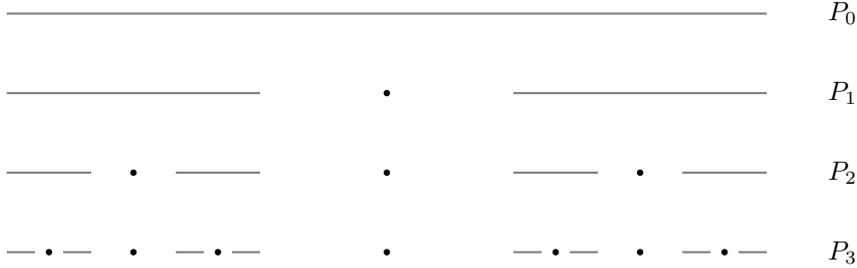
\begin{figure}[h]
    \begin{tikzpicture}
    \node at (11,0) {$P_0$};
    \draw[gray, thick] (0,0) -- (10,0);

    \begin{scope}[yshift=-30]
    \node at (11,0) {$P_1$};
        \draw[gray, thick] (0,0) -- (10,0);
        \draw[white, ultra thick] (10/3,0) -- (20/3,0);
        \filldraw[black] (10/2,0) circle (1pt);
    \end{scope}

    \begin{scope}[yshift=-60]
    \node at (11,0) {$P_2$};
        \draw[gray, thick] (0,0) -- (10,0);
        \draw[white, ultra thick] (10/3,0) -- (20/3,0);
        \draw[white, ultra thick] (10/9,0) -- (20/9,0);
        \draw[white, ultra thick] (70/9,0) -- (80/9,0);
        \filldraw[black] (10/2,0) circle (1pt);
        \filldraw[black] (10/6,0) circle (1pt);
        \filldraw[black] (50/6,0) circle (1pt);
    \end{scope}

    \begin{scope}[yshift=-90]
    \node at (11,0) {$P_3$};
        \draw[gray, thick] (0,0) -- (10,0);
        \draw[white, ultra thick] (10/3,0) -- (20/3,0);
        \draw[white, ultra thick] (10/9,0) -- (20/9,0);
        \draw[white, ultra thick] (70/9,0) -- (80/9,0);
        
        \draw[white, ultra thick] (10/27,0) -- (20/27,0);
        \draw[white, ultra thick] (70/27,0) -- (80/27,0);
        \draw[white, ultra thick] (190/27,0) -- (200/27,0);
        \draw[white, ultra thick] (250/27,0) -- (260/27,0);
        
        \filldraw[black] (10/2,0) circle (1pt);
        \filldraw[black] (10/6,0) circle (1pt);
        \filldraw[black] (50/6,0) circle (1pt);

        \filldraw[black] (10/18,0) circle (1pt);
        \filldraw[black] (50/18,0) circle (1pt);
        \filldraw[black] (130/18,0) circle (1pt);
        \filldraw[black] (170/18,0) circle (1pt);
    \end{scope}
    \end{tikzpicture}
    \caption{The first 3 steps of the construction of Pełczyński space via a middle-thirds method.}\label{fig:pelspace}
    \end{figure}
    In particular, the space should be thought of as being the Cantor set with a fuzz of isolated points surrounding it. By \cite[Theorem 5.6]{gartsidesmith1}, the isolated points correspond exactly to the open subgroups $p^n \Z_p \times p^m \Z_p$. Moreover, for any $n\geqslant 2$, a similar argument shows that $\Sub(\Z_p^{\times n})/\Z_p^{\times n}$ is abstractly homeomorphic to Pełczyński space, as is $\Sub(G)/G$ for $G$ a free $p$-group of rank $\geqslant 2$.
\end{example}

\begin{example}
    Recall the profinite group $G = \prod_{
    \N }\Z/p $ from \cref{ex:cantorspec}. As we observed there, this group has $\Sub(G) \cong \{0,1\}^{\N}$ the Cantor set. This space contains no isolated points and as such is not scattered. 
\end{example}

Now that we have addressed the abelian case, we turn our attention to the non-abelian case where the situation is more delicate. There appears to be no classification of when a profinite group $G$ has countably many conjugacy classes of closed subgroups in the literature.

We do know, however, that there exists exotic examples where $\Sub(G)$ is uncountable but  $\Sub(G)/G$ is countable. We provide such an example below which was communicated to us in a \texttt{MathOverflow} answer of Yves Cornulier \cite{mof_question}. We write $\operatorname{SL}_N(\Z_p)$ for the profinite group defined as $\lim_i \operatorname{SL}_N(\Z/p^i\Z)$. We note that $\operatorname{SL}_N(\Z_p)$ has the property that $\Sub(G)$ is uncountable as soon as $N > 1$. Indeed, \cite[Theorem 3.7]{gartsidesmith1} states that $\Sub(G)$ is countable if and only if $G$ is a central finite extension of $\prod_{i=1}^n \Z_{p_i}$ for $p_i$ distinct primes. We can deduce $\operatorname{SL}_N(\Z_p)$ is not of this form for $N > 1$ via a dimension argument using the observation that the center is a finite group. In light of this, the result of \cref{prop:slNzp} is somewhat surprising.

\begin{proposition}[Cornulier]\label{prop:slNzp}
    Let $G = \operatorname{SL}_N(\Z_p)$. Then $\Sub(G)/G$ is countable if and only if $N=1$ or $N=2$. 
\end{proposition}

\begin{proof}
    The result for $N=1$ is covered by \cref{ex:zp}. We reproduce the answer of \cite{mof_question} here for preservation reasons, which settle the remaining cases.
    
    Suppose that $N = 3$, then $\operatorname{SL}_3(\Z_p)$ contains the closed subgroup $\Z_p \times \Z_p$ in which the subgroups $L \cap (\Z_p \times \Z_p)$ as $L$ ranges over lines in $\Q_p \times \Q_p$ are distinct. This implies $\Sub(\Z_p \times \Z_p) \subseteq \Sub(G)/G$ and therefore $\Sub(G)/G$ is not countable. A similar argument, or the observation that there are inclusions $\operatorname{SL}_N(\Z_p) \hookrightarrow \operatorname{SL}_{N+1}(\Z_p) $, proves that $\operatorname{SL}_N(\Z_p)$ is not countable for any $N \geqslant 3$.

    We are left to consider the case of $N=2$.  We begin by observing that $G = \operatorname{SL}_2(\Z_p)$ is of countable index in $H = \operatorname{SL}_2(\Q_p)$. As such, it suffices to prove that $H$ has countably many compact subgroups up to conjugacy. We note that the Lie algebra of $\operatorname{SL}_2(\Q_p)$ is three-dimensional.
    
    Every closed subgroup of $H$ is $p$-analytic, and is therefore locally determined by its Lie algebra. The action of $H$ on its Lie algebra has countably many orbits on the Grassmannian of $d$ dimensional subspaces for $0 \leqslant d \leqslant 3$, and therefore it is sufficient to check that for a given Lie subalgebra $\mathfrak{K}$, there are only countably many closed subgroups $K \leqslant H$ with Lie algebra $\mathfrak{K}$.

    Write $d = \operatorname{dim}(\mathfrak{K}) = \operatorname{dim}(K)$. We shall consider all $\mathfrak{K}$ such that $\operatorname{dim}(\mathfrak{K}) = d$ for $0 \leqslant d \leqslant 3$ and verify that each one has only countably many closed subgroups associated to it.
    
    \begin{itemize}
    \item[$d=3$:] In this case $K$ is a compact open subgroup of $H$, of which there are only countably many.

    \item[$d=2$:] In this case the Lie algebra $\mathfrak{K}$ is a line stabilizer which is isomorphic to $\Q^\ast_p \ltimes \Q_p$ with the action given by $t \cdot x = t^2 x$. A compact open subgroup has to lie in $\Z^\ast_p \ltimes \Q_p$. It intersects $\Q_p$ in $p^n \Z_p$ for some $n \in \Z$. With the intersection being given, the subgroup $K$ is determined by some compact open subgroup of the quotient $\Z_p^\ast \ltimes \Q_p/p^n\Z_p$ and is necessarily contained in $\Z^\ast_p \ltimes p^m\Z_p/p^n \Z_p$ for some $m \leqslant n$. This is a semidirect product $\Z_p^\ast \ltimes \Z / p^k \Z$ for some $k$. In turn, this has to contain some congruence subgroup in $\Z_p$, and there are only finitely many possibilities.

    \item[$d=1$:] If $d=1$ and the Lie algebra is diagonalizable over an extension, then the normalizer of the Lie algebra is a one-dimensional group $M$ with a compact open normal subgroup isomorphic to $\Z_p$. The intersection with $\Z_p$ is some $p^n\Z_p$, the quotient being finite or virtually isomorphic to $\Z$, and hence has countably many subgroups. If $d=1$ and the Lie algebra is nilpotent, it is conjugate to the Lie algebra of strictly upper triangular matrices. So it is a compact  subgroup of its normalizer, which is the group of upper triangular matrices of determinant 1 (as in the case where $d=2$ above). So, as in the case of $d=2$, $K$ is some subgroup of $\Z_p^\ast \ltimes \Q_p$. Then $K$ contains the matrix
    \[
    \begin{pmatrix}
    1 & p^n \Z_p \\
    0 & 1
    \end{pmatrix} = 
    \begin{pmatrix}
    1 & 0 \\
    0 & 1
    \end{pmatrix}
    +
    \begin{pmatrix}
    0 & p^n \Z_p \\
    0 & 0
    \end{pmatrix}
    \]
    for some $n$. As such, $K$ is determined by its image in the quotient, a finite subgroup of $\Z_p^\ast \ltimes p^m \Z_p/ p^n \Z_p$ for some $m \leqslant n$. Since $\Z^\ast_p \cong \Z_p \times F$ with $F$ finite and $\Z_p$ is torsion free, the subgroup has to lie in the finite group $F \ltimes p^m \Z_p/ p^n \Z_p$. So for a given $m$ and $n$, this only leaves finitely many possibilities.
    
    \item[$d=0$:] In this case $K$ is finite. As every finite group has only finitely many representations on $\Q_p \times \Q_p$ modulo $\operatorname{SL}_2(\Q_p)$-conjugation, there are only finitely many conjugacy classes of finite subgroups. 
    \end{itemize}

    As such, collecting the above information for all $d$, we see that $\Sub(H)/H$ is countable as desired. This allows us to conclude that $\Sub(G)/G$ for $G = \operatorname{SL}_2(\Z_p)$ is countable.
\end{proof}

\begin{remark}
    \cref{prop:slNzp} highlights an interesting phenomenon. Consider $\operatorname{SL}_2(\Z_p)$, which we have just seen has countably many closed subgroups up to conjugacy but uncountably many closed subgroups. That is, $\Sub(G)$ is not scattered while $\Sub(G)/G$ is.
    The algebraic model of Barnes--Sugrue from \cref{sec:qgsp_algmodel} is described as Weyl-$G$-sheaves on the space $\Sub(G)$. The injective dimension of this algebraic category is given by the Cantor--Bendixson rank of the space $\Sub(G)$. In particular, for $\operatorname{SL}_2(\Z_p)$, the algebraic model is of infinite injective dimension. An approach to an algebraic model instead built from the Balmer spectrum, that is, $\Sub(G)/G$, would potentially have finite injective dimension due to the space having Cantor--Bendixson rank.
\end{remark}

\newpage

\addcontentsline{toc}{part}{Appendices}
\appendix

\bookmarksetupnext{level=part}
\section{Comparison of models for equivariant \texorpdfstring{$G$}{G}-spectra}\label{app:equivariantmodels}

There are (at least) three approaches to the genuine equivariant stable homotopy theory of a finite group $G$, to be recalled below: 
    \begin{itemize}
        \item orthogonal;
        \item Mackey;
        \item parametrized.
    \end{itemize}
These closely related perspectives have historically been developed in parallel, and later shown to be equivalent as symmetric monoidal homotopy theories. After a brief review of the extension of each of these models to the setting of profinite groups in \cref{ssec:threemodels}, we then provide the comparison between them. This leads to an essentially unique model of genuine equivariant $G$-spectra for a profinite group $G$, formed as an appropriate colimit of the corresponding categories for the finite quotient groups.

\subsection*{The three models...}\label{ssec:threemodels}
\subsubsection*{Orthogonal (Fausk)}\label{sssec:orthogonal}

The first model is the one of Fausk \cite{Fausk2008}, which is stated model categorically, and is a generalization of the usual construction of (equivariant) orthogonal spectra from \cite{mmss, mandell_may} to  profinite groups. The key property that allows for the extra topological complexity of profinite groups is that a continuous action of a profinite group $G$ on a finite dimensional $G$-inner product space $V$ factors through a finite quotient, see \cite[Lemma A.1]{Fausk2008}. In particular, the (conjugation) action of $G$ on the orthogonal group $O(V)$ factors through a finite quotient. 

As with equivariant orthogonal spectra for a compact Lie group, one indexes $G$-spectra on a complete $G$-universe $\cU$: a countably infinite  direct sum of a real $G$-inner product space $\cU'$, where $\cU'$ contains (up to isomorphism) a copy of each finite dimensional irreducible $G$-inner product space. The topology of $\cU$ arises from considering $\cU$ as the union of all finite dimensional $G$–subspaces of $\cU$ (each equipped with the norm topology). A finite dimensional sub-inner $G$-product space $V$ of $\cU$ is then called an indexing space of $\cU$.

\begin{definition}
    Let $G$ be a profinite group, a \emph{$G$-equivariant orthogonal spectrum} $X$ consists of a $G \ltimes O(V)$ space $X(V)$ for any indexing space $V$  of $\cU$ and a $G \ltimes \left( O(V) \times O(W) \right)$-equivariant structure map
    \[
    X(V) \wedge S^W \longrightarrow X(V \oplus W)    
    \]
    for all indexing spaces $V, W$ of $\cU$, where the domain is equipped with a $O(V) \times O(W)$ action via the block-sum inclusion to $O(V \oplus W)$.

    A map of \emph{$G$-equivariant orthogonal spectra} $f \colon X \to Y$ is a collection of $G \ltimes O(V)$-equivariant maps 
    \[
    f \colon X(V) \longrightarrow Y(V)
    \]
    that commute with the structure maps of $X$ and $Y$.
\end{definition}

As with (equivariant) orthogonal spectra, the weak equivalences of the model structure are defined in terms of homotopy groups. 

\begin{definition}
Let $X$ be a $G$-spectrum and $n$ a non-negative integer. We define the $H$-\emph{homotopy groups} for $H$ a closed subgroup of $G$ as 
\begin{align*}
\pi_n^H(X)  & = \colim_{V} \pi_n^H(\Omega^V X(V)) \\
\pi_{-n}^H(X)  & = \colim_{V \supset \R^n} \pi_0^H(\Omega^{V-\R^n} X(V)).
\end{align*}
In the first case the colimit runs over the indexing spaces of $\cU$, in the second case 
over the indexing spaces of $\cU$ that contain $\R^n$. 
\end{definition}

We state Fausk's main result:

\begin{theorem}[{\cite[Theorem 4.4]{Fausk2008}}]\label{thm:fausk}
Let $G$ be a profinite group. Then there is a model category of $G$-equivariant orthogonal spectra that is cofibrantly generated, 
stable and (closed) symmetric monoidal under the smash product of orthogonal $G$-spectra. 
The weak equivalences are defined in terms of isomorphisms of equivariant homotopy groups $\pi_*^H$ for each \emph{open} subgroup $H$ of $G$.
The associated homotopy category has a compact set of generators, given by the suspension spectra $\Sigma^\infty G/H_+$ of the
spaces $G/H_+$ for $H$ an open subgroup of $G$. 
\end{theorem}

We will write $\Spfausk_G$ for the $\infty$-category realised by the model structure in \cref{thm:fausk}.

\subsubsection*{Mackey (Barwick)}\label{sssec:mackey}
The second model is that of Barwick which generalizes the usual Mackey functor point of view on equivariant spectra. We require  a preliminary definition of the category of spans. For $G$ a profinite group, the category of spans $\mathsf{Span}(G)$ has objects the continuous $G$-sets with finitely many open orbits, and morphisms are given by spans, that is, a morphism $X \to Y$ is a diagram 
\[
\xymatrix{&Z \ar[dl] \ar[dr] \\ X && Y}
\]
for $Z$ some object of the category. Composition in this category is given by taking the pullback of the resulting diagram. A description of $\mathsf{Span}(G)$ as a complete Segal space is given by Barwick \cite[\S 5]{Barwick}. Relevant for the construction is to highlight that $\mathsf{Span}(G)$ is a monoidal category with respect to disjoint union.

\begin{definition}
    Let $G$ be a profinite group, then the $\infty$-category of \emph{$G$-Mackey functors} is the category $\SpBarwick_G$ of additive functors $\Fun^{\mathrm{add}}(\mathsf{Span}(G),\Sp)$.
\end{definition}

\subsubsection*{Parametrized (Bachmann--Hoyois)}\label{sssec:parametrized}

Our third model of interest is the one of Bachmann--Hoyois \cite{bachmannhoyois_norms} which starts with the same basic set-up as the model of Barwick. In particular we begin by considering the additive functors $\Fun^\mathrm{add}(\mathsf{Span}(G), \mathsf{Top})$ now to the $\infty$-category of spaces. The idea is to then form a certain localization of this category (essentially, one inverts the representation spheres).

In particular, let  $G$ be a profinite group, and $H \leqslant G$ a subgroup giving rise to a finite cover $p \colon BH \to BG$, then there is a unique functor
\[
p_\otimes \colon \Fun^\mathrm{add}(\mathsf{Span}(G), \mathsf{Top}) \to \Fun^\mathrm{add}(\mathsf{Span}(G), \mathsf{Top})
\]
whose construction is implicit in \cite[\S 9.1]{bachmannhoyois_norms}. When $G$ is a finite group, $p_\otimes$ retrieves the Hill--Hopkins--Ravenel norm $N_H^G$ \cite{HHR}.

\begin{definition}
    Let $G$ be a profinite group, then the $\infty$-category of \emph{parametrized $G$-spectra} is the category obtained from $\Fun^{\mathrm{add}}(\mathsf{Span}(G),\Top)$ by inverting $p_\otimes(S^1)$ as $p$ runs through the finite coverings $p \colon BH \to BG$ for $H < G$. We denote this $\infty$-category as $\SpBH$.
\end{definition}

\begin{remark}
    The way that $\SpBH$ is constructed in \cite{bachmannhoyois_norms} naturally gives rise to  geometric fixed point functors which are coherently compatible with restriction and inflation. In light of \cref{prop:equiv_of_models_profinite}, it follows that the geometric fixed point functors that we have constructed in \cref{def:gfp} also enjoy these properties.
\end{remark}

\subsection*{...and their comparison}\label{ssec:comparison}

Now that we have introduced the three models of interest, we now show that all of these models are equivalent. We begin with that case that $G$ is a finite group.

\begin{proposition}\label{prop:equiv_of_models_finite}
    Let $G$ be a finite group. Then there are geometric equivalences
    \[
    \Spfausk_G \simeq \SpBarwick_G \simeq \SpBH_G.
    \]
\end{proposition}

\begin{proof}
    The equivalence of $\Spfausk_G$ and $\SpBarwick_G$ can be found, for instance, in \cite{GuillouMay}, while the equivalence with $\SpBH_G$ is in \S9.2 of \cite{bachmannhoyois_norms}.
\end{proof}

Encouraged by \cref{prop:equiv_of_models_finite}, the following result tells us that the equivalences remain true when $G$ is not finite, but profinite.

\begin{proposition}\label{prop:equiv_of_models_profinite}
       Let $G$ be a profinite group. Then there are equivalences
    \[
    \Sp_G^{\mathrm{cont}} \simeq \Spfausk_G \simeq \SpBarwick_G \simeq \SpBH_G.
    \]
    Moreover these equivalences are symmetric monoidal.
\end{proposition}

\begin{proof}
    By writing $G = \lim_i G_i$, we can deduce the desired result from \cref{prop:equiv_of_models_finite} by showing that all of the models are continuous, that is, $\Sp_G^{(-)} \simeq \colim_i^{\omega} \Sp_{G_i}^{(-)}$. The continuity of Fausk's model is given by \cref{thm:gsp_contfausk}. The continuity of  the Barwick and the Bachmann--Hoyois models is the subject of \cite[Proposition 9.11, Example 9.12]{bachmannhoyois_norms}, keeping in mind \cref{rem:gsp_limitmodel}.

    For the claim regarding the symmetric monoidal equivalences, we appeal to \cref{thm:gsp_contfausk}, \cite[Proposition 9.11]{bachmannhoyois_norms}, and \cite[Lemma 9.6]{bachmannhoyois_norms} for Fausk's model, Barwick's model, and the model of Bachmann--Hoyois, respectively.
\end{proof}

In conclusion, we have shown that when $G$ is a profinite group the three models appearing in previous literature are all equivalent to the continuous model $\Sp_G^\mathrm{cont}$ that we have constructed in \cref{def:gsp_contmodel}.

\newpage

\addcontentsline{toc}{part}{References}

\let\oldaddcontentsline\addcontentsline
\renewcommand{\addcontentsline}[3]{}
\bibliographystyle{alpha}\bibliography{bibliography}
\let\addcontentsline\oldaddcontentsline

\end{document}